%% file: fiberedcusp7.tex
\documentclass[12pt]{amsart}
\usepackage[margin = 1in]{geometry}
\usepackage{amsfonts, amsmath,amscd, amssymb,amsthm, cite, latexsym, mathrsfs, mathtools,
slashed, stmaryrd, verbatim, wasysym }
\usepackage[usenames,dvipsnames]{color}
\usepackage[all]{xy}
\usepackage[pdftex]{graphicx}
\usepackage{hyperref}
\usepackage[applemac]{inputenc}
\usepackage[cyr]{aeguill}

\newtheorem{theorem}{Theorem}
\newtheorem{corollary}[theorem]{Corollary}
\newtheorem{definition}[theorem]{Definition}
\newtheorem{example}[theorem]{Example}
\newtheorem{lemma}[theorem]{Lemma}
\newtheorem{proposition}[theorem]{Proposition}
\newtheorem{remark}[theorem]{Remark}
\newtheorem{assumption}{Assumption}

\numberwithin{equation}{section}
\numberwithin{theorem}{section}

\let\oldsqrt\sqrt
\def\sqrt{\mathpalette\DHLhksqrt}
\def\DHLhksqrt#1#2{%
\setbox0=\hbox{$#1\oldsqrt{#2\,}$}\dimen0=\ht0
\advance\dimen0-0.2\ht0
\setbox2=\hbox{\vrule height\ht0 depth -\dimen0}%
{\box0\lower0.4pt\box2}}

\newcommand\IH{\operatorname{IH}}
\newcommand\bm{\overline{m}}
\newcommand\bM{\overline{M}}

\newcommand\ord{\operatorname{ord}}

\newcommand\RC{\operatorname{RC}}
\newcommand\euc{\operatorname{euc}}
\newcommand\sfc{\epsilon\operatorname{, d}}

\newcommand\ff{ff}
\newcommand\phibf{\bhs{\phi bf}}
\newcommand\bs{sb}
\newcommand\sm{\bhs{sm}}
\newcommand\ms{sm}
\newcommand\fbface{\bhs{bf}}
\newcommand\bff{bf}
\newcommand\fbf{\phi bf}
\newcommand\sus{\operatorname{sus}}
\newcommand\GL{\operatorname{GL}}
\newcommand\SL{\operatorname{SL}}

\newcommand\lf{lf}
\newcommand\mf{mf}

\newcommand\rf{rf}

\newcommand\ran{\operatorname{ran}}

\newcommand\sma{\operatorname{small}}
\newcommand\Span{\operatorname{span}}
\newcommand{\extu}{{\overline{\cup}}}
\newcommand{\tX}{\widetilde{X}}
\newcommand{\tR}{\widetilde{R}}

\newcommand{\bhs}[1]{\mathfrak B_{#1}}
\renewcommand{\tilde}{\widetilde}
\renewcommand{\bar}{\overline}
\renewcommand{\Re}{\operatorname{Re}}
\renewcommand{\hat}[1]{\widehat{#1}}

\newcommand{\wt}[1]{\widetilde{#1}}
\newcommand{\rest}[1]{\big\rvert_{#1}} 

\newcommand\lra{\longrightarrow}
\newcommand\xlra[1]{\xrightarrow{\phantom{x} #1 \phantom{x}}}

\newcommand\pa{\partial}

\newcommand\cf{cf\@. }

\newcommand\eg{e\@.g\@. }

\newcommand\eps\varepsilon
\renewcommand\epsilon\varepsilon

\newcommand{\Cl}{\mathbb{C}l}

\newcommand\ephi{\operatorname{\eps, \phi}}
\newcommand\ed{\operatorname{\eps, d}}

\newcommand\dR{\operatorname{dR}}

\newcommand\Ephi{{}^{\ephi}}
\newcommand\Ed{{}^{\ed}}

\newcommand\dCI{\dot{\mathcal{C}}^{\infty}}

\newcommand\CI{{\mathcal{C}}^{\infty}}

\newcommand{\lrpar}[1]{\left( #1 \right)}
\newcommand{\lrspar}[1]{\left[ #1 \right]}
\newcommand\ang[1]{\langle #1 \rangle}

\DeclareMathOperator*{\btimes}{\times} 
\newcommand\fib{\operatorname{---}} 

\renewcommand\det{\operatorname{det}}

\newcommand\diag{\operatorname{diag}}
\newcommand\Diff{\operatorname{Diff}}

\newcommand\End{\operatorname{End}}
\newcommand\ev{\operatorname{even}}

\DeclareMathOperator*{\FP}{\operatorname{FP}}

\newcommand\Hom{\operatorname{Hom}}
\newcommand\Id{\operatorname{Id}}
\renewcommand\Im{\operatorname{Im}}

\newcommand\lAT{\operatorname{LAT}} 

\newcommand\odd{\operatorname{odd}}

\newcommand\phg{\operatorname{phg}}
\newcommand\pt{\operatorname{pt}}

\renewcommand\Re{\operatorname{Re}}
\newcommand\Res{\operatorname{Res}}

\newcommand\RTr[1]{{}^R\operatorname{Tr}\left( #1 \right)}

\newcommand\Spec{\operatorname{Spec}}

\newcommand\Tr{\operatorname{Tr}}

\newcommand\tr{\operatorname{tr}}

\newcommand\pr{\operatorname{pr}}

\newcommand\can{\operatorname{can}}

\newcommand\Mand{\text{ and }}
\newcommand\Mas{\text{ as }}
\newcommand\Mat{\text{ at }}

\newcommand\Mforall{\text{ for all }}

\newcommand\Mif{\text{ if }}

\newcommand\Mon{\text{ on }}
\newcommand\Mor{\text{ or }}

\newcommand\Mwhere{\text{ where }}
\newcommand\Mwith{\text{ with }}

\newcommand\bp{\overline{p}}
\newcommand\bq{\overline{q}}
\newcommand\br{\overline{r}}
\newcommand\bt{\overline{t}}
\newcommand\IC{\operatorname{IC}}

\newcommand\paperintro%
        {%
         }
\newcommand\paperbody%
        {%
         }

\newcommand\bbB{\mathbb{B}}
\newcommand\bbC{\mathbb{C}}

\newcommand\bbH{\mathbb{H}}

\newcommand\bbN{\mathbb{N}}

\newcommand\bbQ{\mathbb{Q}}
\newcommand\bbR{\mathbb{R}}
\newcommand\bbS{\mathbb{S}}

\newcommand\bbZ{\mathbb{Z}}

\newcommand\bN{\mathbf{N}}
\newcommand\bR{\mathbf{R}}

\newcommand\cA{\mathcal{A}}
\newcommand\cB{\mathcal{B}}
\newcommand\cC{\mathcal{C}}

\newcommand\cE{\mathcal{E}}
\newcommand\cF{\mathcal{F}}
\newcommand\cG{\mathcal{G}}
\newcommand\cH{\mathcal{H}}

\newcommand\cJ{\mathcal{J}}
\newcommand\cK{\mathcal{K}}
\newcommand\cL{\mathcal{L}}
\newcommand\cM{\mathcal{M}}

\newcommand\cO{\mathcal{O}}
\newcommand\cP{\mathcal{P}}
\newcommand\cQ{\mathcal{Q}}
\newcommand\cR{\mathcal{R}}
\newcommand\cS{\mathcal{S}}

\newcommand\cU{\mathcal{U}}
\newcommand\cV{\mathcal{V}}

\newcommand\sC{\mathscr{C}}

\newcommand\sE{\mathscr{E}}

\newcommand\sR{\mathscr{R}}

\newcommand\sT{\mathscr{T}}

\DeclareMathAlphabet{\mathpzc}{OT1}{pzc}{m}{it}

\begin{document}

\title[Fibered  cusp degeneration]{Resolvent, heat kernel, and torsion under degeneration to fibered  cusps}
\author{Pierre Albin}
\author{Fr\'ed\'eric Rochon}
\author{David Sher}
\address{Department of Mathematics, University of Illinois at Urbana-Champaign}
\email{palbin@illinois.edu}
\address{D\'epartement de Math\'ematiques, UQÀM}
\email{rochon.frederic@uqam.ca }
\address{Department of Mathematics, University of Michigan}
\email{dsher@umich.edu}

\begin{abstract}
Manifolds with fibered cusps are a class of complete non-compact Riemannian manifolds including many examples of locally symmetric spaces of rank one.
We study the spectrum of the Hodge Laplacian with coefficients in a flat bundle on a closed manifold undergoing degeneration to a manifold with fibered cusps. We obtain precise asymptotics for the resolvent, the heat kernel, and the determinant of the Laplacian. Using these asymptotics we obtain a topological description of the analytic torsion on a manifold with fibered cusps in terms of the R-torsion of the underlying manifold with boundary.
\end{abstract}

\maketitle

\tableofcontents

\section{Introduction}

Reidemeister torsion, or $R$-torsion, a combinatorial invariant of a finite simplicial complex and a choice of representation of its fundamental group, was originally introduced by Reidemeister \cite{Reidemeister1935} and Franz \cite{Franz1935} to distinguish lens spaces that are homotopic but not homeomorphic. 
For manifolds, an analytic analogue, the analytic torsion, was introduced by Ray and Singer \cite{Ray-Singer}, who showed that it depends only on the smooth structure of the manifold and conjectured that it was equal to the Reidemeister torsion.  This conjecture was subsequently established by Cheeger \cite{Cheeger1979} and M\"uller \cite{Muller1978} for orthogonal representations, by M\"uller \cite{Muller1993} for unimodular representations, and the difference between them was computed for general representations by Bismut-Zhang \cite{Bismut-Zhang}.  

This deep and subtle result has many important applications in fields ranging from topology and number theory to mathematical physics.  For instance, on hyperbolic manifolds,  it was used by Fried \cite{Fried1986} to express the $R$-torsion in terms of a Ruelle zeta function, a dynamical quantity defined in terms of closed geodesics.  More recently, it has been used by various authors \cite{BV,Calegari-Venkatesh, Muller:AsympRSTHyp3Mfds,Marshall-Muller,Muller-Pfaff:OnAsympRSATCmptHypMfds, Muller-Pfaff:ATL2TorsionCmptLocSymSpaces, Muller-Pfaff:GrowthTorsionCohoArithGps,BMZ} to study the torsion of the homology of arithmetic groups.

Already in the context of arithmetic groups, but also in algebraic geometry and in the study of various moduli spaces, singular and non-compact spaces naturally appear and have often the structure of a stratified pseudomanifold.  It is therefore a natural question to ask if the Cheeger-M\"uller theorem has any analogue on compact stratified spaces.  Part of the question of course consists in finding suitable analogues of the $R$-torsion and analytic torsion.  In 1987, Dar \cite{Dar} introduced a natural candidate for the  $R$-torsion on stratified spaces by using the intersection homology of Goresky and MacPherson \cite{GM1980}.  Dar also suggested that this $R$-torsion, called the intersection $R$-torsion, should be related to the analytic torsion of an appropriate incomplete iterated edge metric encoding the singularities of the stratified space.  Despite many important recent advances \cite{Hartmann-Spreafico2010, Vertman, Mazzeo-Vertman, Hartmann-Spreafico2011, Lesch2013, Sher:ConicDeg, Guillarmou-Sher, Dai-Huang}, a relation is still not known to hold even in the simplest case where the compact stratified space has only isolated conical singularities.  

In this paper and its companion \cite{ARS2}, we propose to look for an analytic counterpart to the intersection $R$-torsion by replacing incomplete iterated edge metrics by the iterated fibered cusp metrics of \cite{DLR2011}.  Those metrics are complete, but as for iterated edge metrics, they geometrically encapsulate  the information about the singularities of the stratified space.  For instance, as shown in \cite{HR2012}, for Witt stratified spaces  the $L^2$-cohomology of iterated fibered cusp metrics can naturally be identified with the upper middle perversity intersection cohomology of the stratified space. In this paper, we will however  immediately restrict ourselves to the case where the stratified space is of depth one, a situation where those metrics correspond to the fibered cusp metrics considered in \cite{v,hhm}, a class which includes many examples of locally symmetric spaces of rank one \cite{Muller:CuspIndex}.

More precisely, let $M$ be the interior of a compact manifold $\bar M$ with boundary endowed with a fiber bundle structure
\begin{equation*}
	Z \fib \pa M \xlra{\phi} Y.
\end{equation*}
Let $x\geq 0$ be a smooth function on $\bar M$ that vanishes precisely on $\pa M$ and such that $dx\neq 0$ on $\pa M.$
A complete metric $g_d$ on $M$ is a fibered  cusp metric if near $\pa M$ it is asymptotically of the form
\begin{equation*}
	\frac{dx^2}{x^2} + x^2 g_Z + \phi^*g_Y.
\end{equation*}
Thus for example, if $Y$ is a point, $(M,g_d)$ is a manifold with cusps and, if $Z$ is a point, $(M,g_d)$ is a manifold with cylindrical ends.
Examples include many locally symmetric spaces $\Gamma \backslash G /K$ of rank one with the metric induced by the Killing form, e.g., non-compact hyperbolic manifolds of finite volume.
In this setting, notice that the associated stratified space $\hat M$ can be obtained from $\bar M$ by collapsing the fibers of $\phi.$  

For such metrics, the heat kernel is typically not trace class, but one can still define analytic torsion using the renormalized trace of Melrose \cite{MelroseAPS}.  To relate analytic torsion with intersection $R$-torsion, our strategy, strongly inspired by the work of Dai, Hassell, Mazzeo and Melrose\cite{mame1, hmm, Hassell, Dai-Melrose} in their study of analytic torsion on fibrations and manifolds with cylindrical ends, is to deduce a relation by starting with the standard Cheeger-Müller theorem on a compact manifold and understanding what happens to the analytic and Reidemeister torsion when a fibered cusp is forming.  Thus,   assume now that $M$ is a closed manifold (of arbitrary dimension).  Let $\alpha:\pi_1(M)\to \GL(k,\bbR)$ be a unimodular representation and let $F\to M$ be the corresponding flat vector bundle.  Equip $F$ with a metric $g_F$.  Notice that unless $\alpha$ is an orthogonal representation, this metric cannot be chosen to be compatible with the flat connection.   Suppose there is a two-sided hypersurface $H$ endowed with a fiber bundle structure
\begin{equation*}
	Z \fib H \xlra{\phi} Y.
\end{equation*}
We consider a family of metrics $g_{\ed}$ parametrized by $\eps \in (0,1)$ that in a tubular neighborhood of $H$ are asymptotically of the form
\begin{equation*}
	\frac{dx^2}{x^2+\eps^2} + (x^2+\eps^2) g_Z + \phi^*g_Y.
\end{equation*}
This can be visualized as stretching the manifold $M$ in the direction normal to the hypersurface $H$ until it has two infinite cusp-like ends in place of the hypersurface.  
In fact, while for $\eps>0$ the metrics $g_{\ed}$ are smooth Riemannian metrics on $M,$
as $\eps\to0$ the metric degenerates along $H$ and gives a fibered cusp metric $g_d$ on $M \setminus H$.  Related degenerations have a long history see, e.g., \cite{Cheeger:EtaInvs, Seeley:ConicDeg, Seeley-Singer, McDonald, wol87}.

Understanding what is happening to the analytic torsion under such a degeneration is a delicate question and requires a uniform understanding of the resolvent and the heat kernel of the Hodge Laplacian.  As in \cite{mame1}, this can be done by constructing the resolvent and the heat kernel on suitable manifolds with corners that encode the degeneration.  The upshot is that the de Rham operator $\eth_{\dR} = d+\delta$ associated to $g_{\ed}$ has two important model operators that capture its behavior as $\epsilon$ approaches zero.  The first one is the corresponding de Rham operator   $D_{d}$ on $M\setminus H.$ The other model operator relates the two sides of $H$ and is actually defined on $Y \times \bbR$.  It is a de Rham operator $D_b$ associated to a metric with cylindrical ends and a  twisted differential acting on forms taking values in the vertical harmonic forms bundle,
\begin{equation*}
	\mathrm \cH^*(H/Y;F)\lra Y\times \bbR,
\end{equation*}
pulled-back from $Y.$

Our construction can roughly be seen as a suitable fusion of the resolvent and heat kernel constructions of Mazzeo-Melrose\cite{mame1} and Vaillant \cite{v}.  There are however many delicate details and new twists.  Moreover, except for some small steps in the construction, we have had for the most part  to start from the beginning and reprove in our settings many of the results appearing in \cite{mame1} and \cite{v}.  Thus, for the sake of completeness and for the convenience of the reader, we provide a detailed construction that systematically avoids relying on results of \cite{mame1} and \cite{v}.  It is also written in great generality and applies to a wide class of elliptic differential operators including Dirac-type operators.  When applied to the de Rham operator, this leads to the following result, see Theorem~\ref{rc.7}, Remark~\ref{Hodge.1}, Corollary~\ref{rc.29}, Theorem~\ref{thm:HeatKernelConst}, Theorem~\ref{thm:FinalAsymp} and Corollary~\ref{aat.1} for the general and detailed statements.

\begin{theorem}[Spectrum of the twisted de Rham operator under fibered cusp degeneration] 
Let $(M,g_{\ed},F)$ be as above and assume that $F$ is a `Witt bundle' in that either $\dim Z$ is odd or $\mathrm H^{\dim Z/2}(\phi^{-1}(y);F)=0$ for any fiber of $\phi.$
\begin{itemize}
\item [1)] 
The resolvent $(\eth_{\dR}-\lambda)^{-1}$ extends uniformly in $\epsilon$ from $\lambda\in \bbC\setminus \bbR$ to a meromorphic family of bounded operators on a small neighborhood of the origin in $\bbC$ with only simple poles. Furthermore, its integral kernel can be described uniformly down to $\eps=0$ as a polyhomogeneous distribution on an appropriate manifold with corners.

\item [2)]
In particular, there are only finitely many eigenvalues converging to zero as $\eps \to 0,$ the small eigenvalues. The projection onto the corresponding eigenspace is a polyhomogeneous family of pseudodifferential operators on an appropriate manifold with corners and it is identified with the projection onto $\ker_{L^2} D_{d}\oplus \ker_{L^2}D_b$ at $\epsilon=0$.

\item [3)]
The integral kernel of the heat operator $e^{-t\eth_{\dR}^2}$ can be described uniformly down to $\eps=0$ as a polyhomogeneous distribution on an appropriate manifold with corners. In particular, with $\rho = \sqrt{\eps^2+t},$ its trace satisfies
\begin{equation*}
	\Tr(e^{-t\eth_{\dR}^2}) = 
	\begin{cases}
	\RTr{e^{-tD_{d}^2}} + \RTr{e^{-tD_{b}^2}} + B \log\eps + \cO(\eps) & \Mas \eps \to 0, t>0,\\
	t^{-\dim M/2}\sum\limits_{k\geq 0} a_k t^k & \Mas t\to 0, \eps >0, \\
	\rho^{-m}\sum\limits_{j\geq 0}A_j\rho^j + \rho^{-(\dim Y+1)/2}\sum\limits_{j\ge 0} \wt A_j \rho^j \log \rho
	& \Mas t\to 0, \eps\to 0.
	\end{cases}
\end{equation*}

\item [4)]
If $\dim Y$ is even, if the product of the positive small eigenvalues $\det (\eth_{\dR}^2)_{\sma}$ is polyhomogeneous in $\epsilon$ and if the metric is of `product-type', then the determinant of the Laplacian satisfies
\begin{equation*}
	\FP_{\eps=0} \log \det \eth_{\dR}^2
	= \log \det D_b^2 + \log \det D_{d}^2 - \FP_{\eps=0} \log\det (\eth_{\dR}^2)_{\sma}.
\end{equation*}
\end{itemize}
\label{int.1}\end{theorem}

As in \cite{Guillarmou-Sher} for conic degenerations, this analysis of the spectrum indicates that a computable description of the limit of analytic torsion as $\eps\searrow 0$ passes in general through a fine understanding of the asymptotic behavior of the small eigenvalues, a discussion that we postpone to \cite{ARS2}.  Notice on the other hand that this theorem determines the number of small eigenvalues with multiplicity.  In particular, there will be no small eigenvalues if and only if $\ker_{L^2}D_b=\ker_{L^2}D_{d}=\{0\}$.  For the operator $D_b$, a simple way to ensure that it has no $L^2$-kernel is to require that $H^q(Z;F)=0$ for all $q.$  For $D_{d}$, we then know by the Kodaira decomposition of Corollary~\ref{l2.3} below that its $L^2$-kernel is naturally identified with the $L^2$-cohomology groups $H^*_{(2)}(M\setminus H;F)$.  Thus, requiring the strong acyclicity condition
\begin{equation*}
	\mathrm H^*(Z;F)=0, \quad \mathrm H^*_{(2)}(M\setminus H ;F)=0,
\end{equation*}    
will ensure that there are no small eigenvalues and that the operator $D_b$ is in fact trivial.  As we will see, we do not need to require that $\mathrm H^*_{(2)}(M\setminus H ;F)=0$, since a `strongly acyclic at infinity' condition 
\begin{equation}\label{eq:Acyclicity}
	\mathrm H^*(Z;F)=0
\end{equation} 
implies that there are no \textbf{positive} small eigenvalues and that $H^*_{(2)}(M\setminus H ;F)\cong H^*(M;F)$.  Condition \eqref{eq:Acyclicity} also leads to some analytical simplifications: by Remark~\ref{ds.1}, the operator $D_{d}$ only has discrete spectrum, while by Remark~\ref{tc.1}, its heat kernel is trace class for positive time.  In particular, we do not need to use the regularized trace of Melrose \cite{MelroseAPS} to define analytic torsion in this setting.  On the topological side, this condition is clearly restrictive, since it cannot be satisfied if $F$ is trivial or if $F$ is required to be defined on the stratified space associated to $(M\setminus H, g_d)$.    A simple way to ensure that condition \eqref{eq:Acyclicity} holds  is to take $M=M'\times \bbS^1$ and $H=H'\times \bbS^1$ with $H'$ a two sided hypersurface in $M'$ equipped with a fiber bundle $\phi': H'\to Y$.  Then taking $\phi= \phi'\circ \pr_L: H\to Y$ to be the fiber bundle structure on $H$, where $\pr_L:H'\times \bbS^1\to H'$ is the projection on the left factor, condition \eqref{eq:Acyclicity} will be automatically satisfied if we take $F$ to be the pull-back of a flat Euclidean vector bundle $F'$ on $\bbS^1$ such that $H^*(\bbS^1;F')=\{0\}$.  For examples involving finite volume hyperbolic $3$-manifolds, see Example~\ref{acyc.2} below.  Alternatively, following Menal, Ferrer and Porti \cite{Menal-Ferrer-Porti:TwistedCohoHyp3Mfds}, let $N$ be a three-dimensional orientable complete hyperbolic manifold with cusps. The holonomy representation $\mathrm{Hol}: \pi_1(N) \lra \mathrm{PSL}(2,\bbC)$ of $\pi_1(N)$ acting as isometries of $\bbH^3$ lifts to a representation $\wt{\mathrm{Hol}}:\pi_1(N) \lra \mathrm{SL}(2,\bbC).$ Let
\begin{equation*}
	\alpha_n:\pi_1(N) \lra \mathrm{SL}(n,\bbC)
\end{equation*}
be the composition of $\wt{\mathrm{Hol}}$ with the unique complex irreducible representation $V_n$ of $\mathrm{SL}(2,\bbC)$ of dimension $n.$  Assuming that $N$ has a single cusp and that the holonomy representation is irreducible, Menal, Ferrer, and Porti prove that 
\begin{equation*}
	\mathrm H^*(\pa \bar N;\alpha_{2n}) = 0
\end{equation*}
for all $n>0.$ This applies in particular to hyperbolic knot exteriors in $\bbS^3.$

\begin{remark}
Notice that if $N$ is a finite volume odd dimensional hyperbolic manifold with $Z=H$ a disjoint union of tori, we see by \cite[Remarks~3.5 (3) of Chapter VII]{Borel-Wallach} that if the holonomy representation of $F$ is orthogonal (or more generally becomes a direct sum of irreducible representations when restricted to each connected component of $H$), then the Witt condition is satisfied if and only if the strong acyclicity condition at infinity \eqref{eq:Acyclicity} is satisfied.
\end{remark}

With condition \eqref{eq:Acyclicity},  we immediately deduce from Theorem~\ref{int.1} that the limiting behavior of the analytic torsion $\mathrm{AT}(M,g_{\ed},F)$ is given by
\begin{equation}
 \FP_{\eps=0} \log \mathrm{AT}(M,g_{\ed},F)
	= \log \mathrm{AT}(M\setminus H, g_{d},F).
\label{int.2}\end{equation}

The final ingredient in proving our extension of the Cheeger-M\"uller theorem is to analyze how the R-torsion $\tau(M,\alpha;\mu^*)$ is affected by the degeneration for $\mu^*$ a choice of basis for $H^*(M;F)$.   Under the condition \eqref{eq:Acyclicity}, this is a simple consequence of the formula of Milnor \cite{Milnor1966} relating the $R$-torsion of a short exact sequence of complexes, yielding the following relation,
\begin{equation}
     \tau(M,\alpha,\mu^*)=  \tau(\overline{M\setminus H}, \pa \overline{M\setminus H}, \alpha,\mu^*)\tau(H,\alpha),
\label{int.3}\end{equation}
where $\overline{M\setminus H}$ is the manifold with boundary obtained by cutting $M$ along $H$ and $\mu^*$ is a basis of $\mathrm H^*(M;F)$, see Theorem~\ref{ac.5} below for more details. 

Finally, combining \eqref{int.2} and \eqref{int.3} with the Cheeger-M\"uller theorem on compact manifolds gives our main result.  

\begin{theorem}[A Cheeger-M\"uller theorem for fibered  cusps] \label{thm:IntroCM}$ $

Let $(M, g_d)$ be an odd-dimensional manifold with fibered cusps as above, with base $Y$ even-dimensional and with boundary compactification $\bar M$.
Let $\alpha:\pi_1(M) \lra \GL(k,\bbR)$ be a unimodular representation whose associated flat bundle $F\lra M$ satisfies the condition
\eqref{eq:Acyclicity}.  Suppose $F$ is equipped with a smooth bundle metric $g_F$ on $\bar M$ having an even expansion at $\pa\bar M$ in terms of the boundary defining function.  
Then the analytic torsion and the R-torsion are related by
\begin{equation}
	\mathrm{AT}(M,g_d, F) = \tau(\bM,\pa\bM,\alpha,\mu^*)\tau(H,\alpha)^{\frac12},
\label{int.4}\end{equation}
where $\mu^*$ is an orthonormal basis of $L^2$ harmonic forms with respect to the metrics $g_d$ and $g_F$ inducing a basis for $ H^*(\bar M, \pa \bM;F)$.
\label{int.5}\end{theorem}
\begin{remark}
If $\alpha$ is an orthogonal representation and $g_F$ is chosen to be compatible with the flat connection, then $\tau(H,\alpha)=1$ and the formula simplifies to
$$
\mathrm{AT}(M,g_d, F) =\tau(\bM,\pa\bM,\alpha,\mu^*).
$$ 
\end{remark}
\begin{remark}
The assumption that $Y$ is even ensures that some contribution in the limiting behavior of the analytic torsion vanishes.  If $Y$ is odd, we expect from \cite{Dai-Melrose} that there should be another term appearing in formula \eqref{int.4}.
\end{remark}
\begin{remark}
The definition of the intersection $R$-torsion of Dar \cite{Dar} requires that $F$ be defined on the stratified space associated to $(M,g_d)$, which means in this case that $F$ cannot be strongly acyclic at infinity.  Thus, a relation between analytic torsion and  intersection $R$-torsion is only obtained in our companion paper \cite{ARS2}, where the assumption that $F$ be strongly acyclic at infinity is weakened to the assumption that $F$ be Witt.    
\end{remark}

In the special case where the metric is hyperbolic of finite volume, there are many related recent works involving analytic torsion.  In \cite{Park}, Park extended the result of Fried \cite{Fried1986} by relating the analytic torsion of unitary representations on hyperbolic manifolds with cusps to some corresponding Ruelle zeta function.  Combining with our result, this gives a relation between the $R$-torsion and the Ruelle zeta function when the flat unitary vector bundle is strongly acyclic at infinity \eqref{eq:Acyclicity}.   In \cite{Muller-Pfaff:ATComHypMfdsFinVol}, M\"uller and Pfaff generalized their study of the asymptotics of analytic torsion along a family of flat bundles (corresponding to unimodular representations) to non-compact hyperbolic manifolds of finite volume.

They also extended the results of Bergeron-Venkatesh to these spaces in \cite{Pfaff:ExpGrowthHomTorTowerCongSubgpsBianchi, Muller-Pfaff:ATAsympBhvSeqHypMfdsFinVol} (see also \cite{Raimbault:Asymp, Raimbault:ARHtorsion}). The paper of Bergeron-Seng\"un-Venkatesh \cite{BSV} mentioned above treats certain hyperbolic 3-manifolds with cusps.

In \cite{Pfaff:SelbergZetaFunOddDHypMfdsFinVol}, Pfaff extends the analysis of the Selberg and Ruelle zeta functions of Bunke-Olbrich \cite{Bunke-Olbrich} to finite volume non-compact hyperbolic manifolds, twisted by representations. He applies this analysis in \cite{Pfaff:ATRTHyp3MFdsCusps}, and the work of Menal-Ferrer-Porti \cite{Menal-Ferrer-Porti:HigherDRTCuspedHyp3Mfds}, to relate analytic torsion and Reidemeister torsion of non-compact hyperbolic $3$-manifolds with cusps for the representations coming from symmetric powers of the standard representation of $\mathrm{SL}_2(\bbC).$

Most related to our result but coming from a different direction, there is  a Cheeger-M\"uller-type theorem of Pfaff \cite{Pfaff:GluingFormATHypMfdsCusps} on non-compact hyperbolic manifolds of finite volume $M=\Gamma\setminus \bbH^m$ with $\Gamma$  neat in the sense of Borel and with flat bundle $F$ having holonomy representation induced by a finite dimensional irreducible representation of $\mathrm{SO}^{\circ}(m,1)$ or $\mathrm{Spin}(m,1)$ that is not invariant under the standard Cartan involution.  The bundle $F$ is then unimodular and possesses a canonical metric $g_{F, \can}$. Let $\cC = \cup \{ [1,\infty) \times T_i \},$ where $T_i$ ranges among the tori that are the links of the cusps in $M,$ be endowed with the usual warped product metric.
Finally, let $\bar M$ denote the compactification of $M$ to a manifold with boundary $\cup \{ T_i\}.$ Pfaff uses Eisenstein series to define a canonical Reidemeister torsion for $\tau_{Eis}(\bar M, F)$ and is then able to compute the difference
\begin{equation*}
	\log \tau_{Eis}(\bar M, F) - \log \lrpar{ \frac{ AT(M; F) }{AT(\cC, \pa \cC; F)} }
\end{equation*}
in terms of the rank of $F,$ the Betti numbers and volume of $\pa \cC,$ and some weights associated to the holonomy representation of $F.$  Notice that there is no intersection with our result, since the bundle $F$ in this setting is such that $H^k(Z;F)$ is non-trivial in each degree, so that $F$ is not Witt and  the acyclicity condition \eqref{eq:Acyclicity} does not hold.

A more subtle, but very important difference between our results and the work of Pfaff is that in \cite{Pfaff:GluingFormATHypMfdsCusps}, as well as in \cite{Muller-Pfaff:ATComHypMfdsFinVol,Pfaff:ExpGrowthHomTorTowerCongSubgpsBianchi,Pfaff:SelbergZetaFunOddDHypMfdsFinVol}, it is the canonical metric $g_{F,\can}$ which is used to define analytic torsion.  This metric turns out to be quite different from the bundle metrics we consider, since as for the hyperbolic metric, it degenerates near $\pa\bar M$ and does not extend to a smooth bundle metric on the boundary compactification $\pa \bar M$.  This has a drastic impact on the spectrum of the de Rham operator.  In particular, despite the fact that $F$ is not Witt, the de Rham operator is nevertheless Fredholm when defined with the metric $g_{F,\can}$.

Finally, we mention that analytic and $R$-torsion of knot complements have also been the subject of recent applications in knot theory where it is related to the twisted Alexander polynomial see, e.g., \cite{Friedl, Dunfield} for details.\\

Now let us indicate in more detail the content of this paper. It is roughly divided into four parts, together with three appendices containing technical results used in the body of the paper. The first part, \S\ref{sec:Metrics}, describes the main object of study, namely the family of  de Rham operators associated to a fibered cusp surgery metric. In order to analyze the degeneration of the metrics smoothly we introduce a {\em surgery space} $X_s$
\begin{equation*}
	X_s = [M \times [0,1]_{\eps}; H \times \{ 0\}]
\end{equation*}
where the notation indicates that we perform a `radial blow-up' of the submanifold $H$ at $\eps=0.$ We also replace the tangent bundle with a `stretched tangent bundle' adapted to the geometry, $\Ed TX_s.$ This has the effect of desingularizing the metrics $g_{\ed}$ described above.

The second part, consisting of \S\S\ref{doublespsec}-\ref{sec:SmallEigen}, is devoted to understanding the effect of this degeneration on the spectrum of the Hodge Laplacian. Analogously to how the singular limit of the metric is understood by passing to the blown-up surgery space $X_s,$ we resolve the singular behavior of the resolvent as $\eps\to0$ at the level of its Schwartz kernel by constructing a `double surgery space'  $X^2_s$ out of $M^2 \times [0,1]_{\eps}.$ The fruit of these efforts is a description of the Schwartz kernel as a polyhomogeneous distribution all the way down to $\eps=0.$ In particular we are able to read off important aspects of the spectrum under degeneration and especially about the small eigenvalues.

Part three, made up of sections \ref{sec:SurgeryHeatSpace} and \ref{sec:SolveHeat}, is where we analyze the heat kernel under degeneration. As with the resolvent, the Schwartz kernel of the heat operator has various singularities as $\eps\to0$ which we resolve geometrically, replacing the space $M^2 \times [0,1]_\eps \times \bbR^+_t$ with a `surgery heat space' $HX_s.$ On this space the heat kernel is smooth with polyhomogeneous expansions at the boundary faces. We use this refined description to obtain precise asymptotics of the trace of the heat kernel throughout the degeneration. 

The sections \ref{sec:RTorsion}-\ref{sec:CheegerMuller} make up the final part of the paper devoted to the behavior of analytic and Reidemeister torsion under degeneration. First we recall the definition of R-torsion on a closed manifold and intersection R-torsion on a stratified space.  We also extend some of the results of \cite{hhm} to allow for coefficients in a flat vector bundle.   Next we turn to analytic torsion in \S~\ref{sec:ATConventions} and use our results about the resolvent and the heat kernel to deduce the behavior of the determinant of the Laplacian under degeneration in \S~\ref{sec:AsympAT}. Finally combining these results yields our Cheeger-M\"uller theorem for manifolds with fibered cusps.\\

{\bf Acknowledgements.}
P. A. was supported by NSF grant DMS-1104533 and Simons Foundation grant \#317883.
F. R. was supported by a Canada Research Chair, NSERC and FRQNT.
D. S. was supported by a CRM postdoctoral fellowship and by NSF EMSW21-RTG 1045119.
The authors are happy to acknowledge useful conversations with Steven Boyer, Dan Burghela, Nathan Dunfield, Colin Guillarmou, Rafe Mazzeo, Richard Melrose, Werner M\"uller and Jonathan Pfaff.

\section{Fibered  cusp surgery metrics} \label{sec:Metrics}

Let $M$ be a closed manifold of dimension $m$ with a hypersurface $H \subseteq M$ that participates in a fiber bundle
\begin{equation*}
	Z \fib H \xlra{\phi} Y
\end{equation*}
of closed manifolds (with $v = \dim Z$ and $h = \dim Y$). We will assume that $M$ is oriented and $H$ has trivial normal bundle (i.e., has `two sides').
In this section we will discuss a class of metrics depending on a parameter $\eps$ that for $\eps>0$ are Riemannian metrics on $M$ and for $\eps=0$ are `fibered  cusp metrics' on $M \setminus H.$ Examples of the latter are the natural metrics on many locally symmetric spaces of $\bbQ$-rank one, see \cite{Muller:CuspIndex, v}. We call these `fibered  cusp surgery metrics' or more concisely $\ed$-metrics, $d$ being the moniker used for fibered  cusp metrics by Vaillant \cite{v}.
We will also recall the definition of analytic torsion for closed manifolds and for manifolds with fibered  cusps.

\subsection{Fibered  cusp surgery metrics} \label{sec:MetricDesc}

Let $x$ be a defining function for $H,$ meaning that $x$ is a smooth function defined in a neighborhood of $H$ such that  $H = \{ x =0 \}$ and with no critical points on $H.$
The fibered  cusp structure depends on the choice of $x$ slightly, and we shall simply fix $x$ once and for all.
We will work in the category of manifolds with corners and $b$-maps, as described in \cite{MelroseAPS}. We refer the reader to this reference for more on these concepts as well as the notions of {\em radial blow-up} and {\em $b$-fibration}, which we will use repeatedly.\\

Let $\sT$ be a tubular neighborhood of $H$ in $M$ consistent with $x$ in that
\begin{equation*}
	\sT \cong (-1,1)_x \times H.
\end{equation*}
On $H$ we fix a choice of connection for $\phi$ and choose a compatible submersion metric of the form
\begin{equation*}
	\phi^*g_Y + g_{H/Y}
\end{equation*}
where $g_Y$ is a Riemannian metric on $Y$ and $g_{H/Y}$ restricts to a metric on each fiber of $\phi.$
Finally, we introduce a parameter $\eps \in [0,1].$
For positive $\eps,$ a {\em product-type} $\ed$-metric is one that on $\sT \times (0,1)_{\eps}$ takes the form
\begin{equation}
	\frac{dx^2}{x^2+\eps^2} + (x^2+\eps^2)g_{H/Y} + \phi^*g_Y.
\label{metric.1}\end{equation}

This is a family of smooth metrics on $M$ that on $M\setminus H$ limits to a Riemannian metric that near $H$ has the form
\begin{equation}\label{eq:PtFc}
	\frac{dx^2}{x^2} + x^2g_{H/Y} + \phi^*g_Y,
\end{equation}
that is, a (product-type) fibered  cusp metric.
On the other hand, this family of metrics has a singular limit as $(x,\eps) \to (0,0)$ and to resolve this limit we perform a radial blow-up of $H \times \{0\}.$
The `single surgery space' is 
\begin{equation*}
	X_{s} = [M \times [0,1]_{\eps}; H \times \{ 0 \} ]
\end{equation*}
and is pictured in Figure \ref{fig:singlespace}.
\begin{figure}
	\centering
	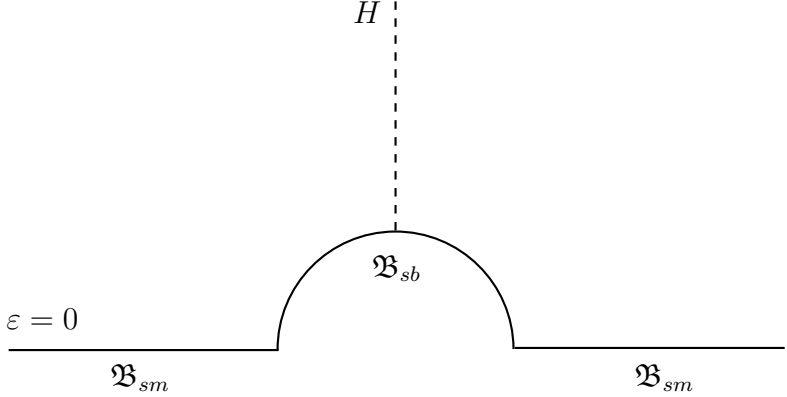
	\caption{The single surgery space $X_s.$}
	\label{fig:singlespace}
\end{figure}
Recall \cite{MelroseAPS} that this means that we replace $H \times \{0\}$ in $M \times [0,1]_{\eps}$ with its inward-pointing spherical normal bundle.
We single out two natural maps
\begin{equation*}
	\xymatrix{
	X_{s} \ar[rd]^{\beta_{(1)}} \ar[rr]^{\pi_{\eps}} & & [0,1]_{\eps} \\ & M \times [0,1]_{\eps}. \ar[ru] & }
\end{equation*}
Here, $\beta_{(1)}$ is the blow-down map that undoes the blow-up by collapsing the inward pointing spherical normal bundle of $H\times \{0\}$ back down to $H \times \{0\},$ and $\pi_{\eps}$ is the composition of $\beta_{(1)}$ with the projection onto the right factor of $M \times [0,1]_{\eps}.$ The second map is a  $b$-fibration, since it is a $b$-submersion and its target is a manifold with boundary.

The boundary hypersurface of $X_{s}$ resulting from the blow-up will be denoted $\bhs{\bs},$ and referred to as the `surgery boundary'. It fibers over $H$ with fiber an interval $[-\pi/2,\pi/2].$ Indeed, the existence of a global defining function implies that the normal bundle of $H$ in $M$ is trivial, so we can identify $\bhs{\bs}$ with the product $[-\pi/2,\pi/2] \times H.$ We will use a natural extension of $\phi$ to a fiber bundle on $\bhs{\bs},$
\begin{equation}\label{rc.1}
	Z \fib \bhs{\bs} \xlra{\phi_+} [-\pi/2,\pi/2] \times Y.
\end{equation}
The lift of the function $\sqrt{x^2+\eps^2}$ is a natural choice of boundary defining function, or `bdf', for $\bhs{\bs}$ and will be denoted $\rho$ or $\rho_{\bs}.$

The other boundary hypersurface of $X_{s}$ over $\eps = 0$ will be denoted $\bhs{\ms},$ for `surgery $M$', as this face is the `interior lift' of $M\times \{0\}$ under the surgery blow-up. 
We can identify it with 
\begin{equation*}
	\bhs{\ms} = [M;H].
\end{equation*}
Both of these faces are manifolds with boundary, 
\begin{equation*}
	\pa \bhs{\ms} = \pa \bhs{\bs} = \wt H,
\end{equation*}
equal to the double sheeted orientation cover $\wt H$ of $H.$ It inherits a fibration from $H$ which we continue to denote $\phi.$

There is also a boundary hypersurface at $\eps =1,$ but this bhs will not enter our considerations and we will simply ignore it.\\

The blow-down map $\beta_{(1)}$ is a diffeomorphism between $X_s\setminus \bhs{\bs}$ and $M\times [0,1] \setminus H \times \{ 0 \},$ and we think of $X_s$ as a compactification of $M \times (0,1)$ that is well adapted to the surgery degeneration that we are studying. 
It will be very convenient to similarly replace the tangent bundle of $X_s$ with a different bundle that coincides over $X_s^{\circ}$ but is well adapted to the degeneration we are studying.

Since our main object of study is the behavior of the spectral invariants of the level sets of $X_s \lra [0,1]_{\eps},$ let us start by setting
\begin{equation*}
	{}^\eps TX_s = \ker \pi_{\eps *} \subseteq TX_s.
\end{equation*}

We will define a vector bundle $\Ed TX_s$ whose sections are naturally identified with those sections of ${}^\eps TX_s$ that have bounded pointwise length with respect to any $\ed$-metric. It is convenient to define this bundle in two steps: first, let us set
\begin{equation*}
	\cV_{\ephi} = \{ V \in \CI(X_s;{}^\eps TX_s) : (\phi_+)_*( V\rest{\bhs{\bs}} )=0, \quad V\rho \in \cO(\rho^2) \}
\end{equation*}
and use the Serre-Swan theorem (or argue directly as in \cite{MelroseAPS}) to find a vector bundle $\Ephi TX_s$ over $X_s$ and a bundle map 
\begin{equation*}
	j:\Ephi TX_s \lra {}^\eps TX_s
\end{equation*}
with the property that it is the identity over $X_s^{\circ}$ and that $\cV_{\ephi} = j_*\CI(X_s;\Ephi TX_s);$ second, let us rescale this bundle to define
\begin{equation*}
	\Ed TX_s = \frac1\rho \Ephi TX_s
\end{equation*}
as in \cite[Chapter 8]{MelroseAPS}.
Sections of $\Ed TX_s$ are locally spanned by 
\begin{equation*}
	\rho \pa_x, \quad \tfrac1\rho \pa_z, \quad \pa_y.
\end{equation*}

We denote the dual bundle of $\Ed TX_s$ by $\Ed T^*X_s$ and refer to it as the {\em $\ed$ cotangent bundle}.
Its sections are locally spanned by
\begin{equation*}
	\frac{dx}\rho, \quad 
	\rho \; dz, \quad
	dy.
\end{equation*}
The bundles $\Ed T^*X_s$ and ${}^\eps T^*X_s$ are canonically isomorphic over $X_s^{\circ}.$ They are also isomorphic over $X_s$, but not naturally. Note in particular that $\rho \; dz$ vanishes at $\rho=0$ as a section of ${}^\eps T^*X_s$ and does not vanish as a section of $\Ed T^*X_s.$  Notice that the restriction of $\Ed T^*X_s$ to $\bhs{\ms}$, denoted  ${}^{d}T\bhs{\ms}$, is just the $d$-tangent bundle of \cite{v}.

Associated to the bundle  ${}^{\ephi} TX_s$ is a space of differential operators, which we denote  $\Diff^k_{\ephi}(X_s)$ and call  the space of $\ephi$-differential operators of order $k$.   They consist of  all $C^{\infty}$-linear combinations of at most $k$ elements of $\{\rho^2\pa_x,\pa_z,\rho\pa_y\}$. If $E\to X_s$ is a Euclidean vector bundle, we may define $\Diff^k_{\ephi}(X_s;E)$ in the usual way.  For the bundle  $\Ed T X_s$, the space of smooth sections does not have a natural structure of Lie algebra, but we can define the space of $\epsilon,d$-differential operators of order $k$ by $\Diff^{k}_{\Ed}(X_s;E)=\rho^{-k}\Diff^{k}_{\ephi}(X_s;E)$.\\

A product-type $\ed$-metric as described above over $X_s^{\circ}$ naturally extends to $X_s$ as a {\em non-degenerate} metric on $\Ed TX_s.$
A general $\ed$-metric is a non-degenerate bundle metric on $\Ed TX_s.$ Since the analysis of such a general metric is rather involved, we will restrict our attention to a class of better behaved metrics. Let us say that an $\ed$-metric is {\bf product-type to order $k$} if there is a choice of tubular neighborhood, connection, etc. as above and a corresponding product-type $\ed$-metric $g_{\ed, \pt}$ such that
\begin{equation*}
	g_{\ed}-g_{\ed,\pt} \in \rho^k\; \CI(X_s;S^2(\Ed T^*X_s)),
\end{equation*}
where $S^2(\Ed T^*X_s)$ denotes the space of symmetric two-tensors on $\Ed T^*X_s.$

\subsection{The de Rham operator of a surgery metric} \label{sec:DeRhamOp}

Let $g_{\ed}$ be an $\ed$-metric.
We are interested in the corresponding de Rham operator $\eth_{\dR} = d + \delta$ on differential forms.
Actually, it is convenient to replace the space of differential forms with forms adapted to the geometry.
Specifically, if $\Ed T^*X_{s}$ is the dual bundle to the $d$-surgery tangent bundle $\Ed TX_{s},$ then 
we are interested in the action of $\eth_{\dR}$ on the sections of the exterior powers of $\Ed T^*X_{s},$ $\Lambda^* (\Ed T^*X_{s}).$
Note that a differential form on $X_s$ is a section of $\Lambda^* (\Ed T^*X_{s})$ if, locally near $\bhs{\bs},$ it can be written as a linear combination
of wedge products of 
\begin{equation*}
	\beta^*(\frac{dx}\rho), 
	\quad dy, \quad \rho \; dz
\end{equation*}
with coefficients smooth up to $\bhs{\bs}.$
The bundle $\Ed TX_{s}$ is canonically isomorphic to ${}^\eps TX_{s}$ over the interior of $X_{s},$ so there is no loss in considering the de Rham operator acting on fibered  cusp surgery forms.
On the other hand, this is an advantageous point of view as it makes the model operators tractable, as we will see below.

We specialize to the de Rham operator twisted by a flat bundle following \cite{Muller1993}. Let $F \lra X_s$ be a flat vector bundle over $X_s$ with flat connection $\nabla^F.$
Then the dual bundle $F^*\to X_s$ is also flat when equipped with the dual connection $\nabla^{F^*}$.  For both bundles, the flat connection induces exterior derivatives
\begin{gather*}
      d_F: \CI(X_s; \Lambda^*(\Ed T^* X_s)\otimes F)\to \CI(X_s; \Lambda^{*+1}(\Ed T^* X_s)\otimes F), \\ 
      d_{F^*}: \CI(X_s; \Lambda^*(\Ed T^* X_s)\otimes F^*)\to \CI(X_s; \Lambda^{*+1}(\Ed T^* X_s)\otimes F^*).
\end{gather*}
Let $g_F$ be a bundle metric on $F,$ not necessarily compatible with the flat structure.  Then the metrics $g_F$ and $g_{\ed}$ induce the maps
\begin{equation*}
	\#:\Lambda^p ( \Ed T^*X_s ) \otimes F \lra \Lambda^p ( \Ed T^*X_s ) \otimes F^*, \quad
	*:\Lambda^p ( \Ed T^*X_s ) \otimes F \lra \Lambda^{m-p} ( \Ed T^*X_s ) \otimes F.
\end{equation*}
In terms of these maps, recall from \cite{Muller1993} that the formal adjoint $d^*_F$  of $d_F$ on $\Lambda^p (\Ed T^*X_s) \otimes F$ is given by
\begin{equation}\label{eq:MullerDelta}
	d^*_F = (-1)^{mp+m+1} * \#^{-1} d_{F^*} \# *.
\end{equation}
We will not incorporate the twisting into our notation for the de Rham operator,
\begin{equation*}
	\eth_{\dR}= d_F + d_F^*: \Omega^*(X_{s}; F) \lra \Omega^*(X_{s};F)
\end{equation*}
and trust that this will not lead to confusion. We set
\begin{equation*}
	E = \Lambda^*(\Ed T^*X_{s}) \otimes F.
\end{equation*}

In order to find an expression for the de Rham operator, let us recall its behavior on the total space of a fibration of smooth manifolds from \cite{hhm}.
Consider $H$ endowed with the submersion metric $\phi^*g_Y + g_Z$ and let $\Omega^{p,q}(H)$ be the forms of horizontal degree $p$ and vertical degree $q.$
The exterior derivative decomposes into
\begin{equation*}
\begin{gathered}
	d_F^H = d^{H/Y} + \hat d^Y + \mathrm R, \quad \Mwhere 
	\xymatrix{ & \Omega^{p,q}(H;F) \ar[ld]^{d^{H/Y}} \ar[d]^{\hat d^Y} \ar[rd]^{\mathrm R} & \\
	\Omega^{p, q+1}(H;F) & \Omega^{p+1, q}(H;F) & \Omega^{p+2,q-1}(H;F) }
\end{gathered}
\end{equation*}
are the non-zero projections of $d_F^H.$
We can identify these pieces geometrically in a way that justifies the notation; namely,
$d^{H/Y}$ is the vertical exterior derivative, $\hat d^Y$ is related to the exterior derivative on $Y$ and the second fundamental form of the fibers of the fibration, and $\mathrm R$ is obtained from the curvature of the fibration.

Now let us assume that $g_{\ed}$ is product-type in a tubular neighborhood $\sT$ of $H$ in $M,$ so that
\begin{equation*}
	g_{\ed,pt} = \frac{dx^2}{x^2+\eps^2} + (x^2+\eps^2) g_{H/Y} + \phi^*g_Y \Mon \sT,
\end{equation*}
where $g_{H/Y}$ and $\phi^*g_Y$ are independent of both $x$ and $\eps.$  Assume similarly that $g_F$ is constant in $x$ and $\epsilon$.
With respect to the splitting of differential forms
\begin{multline}\label{eq:SplittingEdTub}
	\Lambda^{\ell}(\Ed T^*\sT) \cong \\
	\lrpar{ \bigoplus_{j+k=\ell}\Lambda^j\phi^*T^*Y \wedge \rho^k \Lambda^kT^*H/Y }
	\oplus \frac{dx}\rho \wedge 
	\lrpar{ \bigoplus_{j+k=\ell-1}\Lambda^j\phi^*T^*Y \wedge \rho^k \Lambda^kT^*H/Y },
\end{multline}
the exterior derivative is given by
\begin{equation*}
	d_F = 
	\begin{pmatrix}
	\tfrac1\rho d^{H/Y} + \hat d^Y + \rho \mathrm R & 0 \\
	\rho\pa_x + \bN_{H/Y} \tfrac x\rho & -(\tfrac1\rho d^{H/Y} + \hat d^Y + \rho \mathrm R) 
	\end{pmatrix}
\end{equation*}
where $v = \dim Z$ and $\bN_{H/Y}$ is the `vertical number operator' which multiplies a form by its vertical degree. 
%
%
%
Its formal adjoint is thus given by
\begin{equation*}
	d_F^* = 
	\begin{pmatrix}
	\tfrac1\rho (d^{H/Y})^*  + (\hat d^Y)^*  + \rho \mathrm R^* & -\rho\pa_x + (\bN_{H/Y} - v) \tfrac x\rho \\
	0 & -(\tfrac1\rho (d^{H/Y})^*  + (\hat d^Y)^* + \rho \mathrm R^*) 
	\end{pmatrix}.
\end{equation*}
Thus, if we write $\eth_{\dR}^{H/Y}=d^{{H/Y}} + (d^{H/Y})^*$, $\hat\eth_{\dR}^{Y}= \hat d^Y+(\hat d^Y)^*$ and $\bR = \mathrm R + \mathrm R^*,$
$\eth_{\dR}$ is given by a two-by-two matrix,
\begin{equation*}
	\eth_{\dR} = 
	\begin{pmatrix}
	\tfrac1\rho \eth_{\dR}^{H/Y} + \hat \eth_{\dR}^Y + \rho \bR & -\rho\pa_x + (\bN_{H/Y} - v) \tfrac x\rho  \\
	\rho\pa_x + \bN_{H/Y} \tfrac x\rho & -(\tfrac1\rho \eth_{\dR}^{H/Y} + \hat \eth_{\dR}^Y + \rho \bR) 
	\end{pmatrix}.
\end{equation*}
Note that $\eth_{\dR}\in\Diff^1_{\Ed}(X_s;E)$. We are interested in this operator as an unbounded operator on the natural $L^2$ space of sections $L^2_{\ed}(M;E)$; however for some constructions it is convenient to work on the $L^2$ space corresponding to a `$b$-density' (see \cite{MelroseAPS} and section \ref{sec:Densities} below), such as
\begin{equation*}
	L^2_{\eps,b}(M;E) = \rho^{v/2} L^2_{\ed}(M;E).
\end{equation*}
The action of $\eth_{\dR}$ on $L^2_{\ed}(M;E)$ is equivalent to the action of 
\begin{equation}\label{eq:ShiftToL2b}
	D_{\dR} = \rho^{v/2} \eth_{\dR} \rho^{-v/2}
\end{equation}
on $L^2_{\eps,b}(M;E),$ so we will from now on consider $D_{\dR}$ as our main object of interest.
Note that, on $\sT \times (0,1)_{\eps},$ $D_{\dR}$ is given by
\begin{equation*}
	D_{\dR} = 
	\begin{pmatrix}
	\tfrac1\rho\eth_{\dR}^{H/Y}  + \hat \eth_{\dR}^Y  + \rho \bR 
		& -\rho\pa_x + (\bN_{H/Y} - \tfrac v2) \tfrac x\rho  \\
	\rho\pa_x + (\bN_{H/Y}-\tfrac v2) \tfrac x\rho 
		& -(\tfrac1\rho \eth_{\dR}^{H/Y}  + \hat \eth_{\dR}^Y + \rho \bR) 
	\end{pmatrix}.
\end{equation*}

A key r\^ole in understanding the behavior of this operator is played by its leading and `subleading' terms. For the former, note that while $\eth_{\dR}$ is singular at $\bhs{\bs},$ $\rho\eth_{\dR}$ can be restricted to this face. We call 
\begin{equation}\label{eq:VerticalOpDR}
	\rho D_{\dR}\rest{\bhs{\bs}} = 
	\begin{pmatrix}
	 \eth_{\dR}^{H/Y}  & 0 \\
	0 & -\eth_{\dR}^{H/Y} 
	\end{pmatrix}
\end{equation}
the {\bf vertical operator} at $\bhs{\bs}.$
Note that each fiber $Z$ of $\phi$ inherits a bundle $F\rest{Z}$ with flat connection $\nabla^F\rest{Z}$ and bundle metric $g_F\rest{Z}.$
From \cite[Proposition 3.7]{Bismut-Lott} we see that $\eth_{\dR}^{H/Y} \rest{Z}$ is precisely the de Rham operator $d_{F\rest{Z}} + d_{F\rest{Z}}^*$ corresponding to this data. 
It follows that the null space of the vertical operator is the space of vertical harmonic forms with respect to the induced flat connection on the fibers.
These null spaces fit together into a bundle over $Y$ and we pull-back this bundle to $\bhs{sb}$ and denote it
\begin{equation*}
	\rho^{\bN} \cH^*(H/Y;F) \lra \bhs{\bs}.
\end{equation*}
$ $

The second, `subleading' term is the leading term at $\bhs{\bs}$ once we restrict to extensions from $\bhs{\bs}$ of sections of $\ker \;  \rho D\rest{\bhs{\bs}} .$ Thus, denoting by $\Pi_{h}$ the projection onto fiberwise harmonic forms, the {\bf horizontal operator} is given by
\begin{multline*}
	D_b = D_{\dR}\rest{\bhs{\bs}, \ker \;  \rho D\rest{\bhs{\bs}} }  \\
	= 
	\Pi_{h}
	\begin{pmatrix}
	\hat \eth_{\dR}^Y   & \lrspar{-\beta^*(\rho\pa_x)+ (\bN_{H/Y} - \tfrac12 v)\frac x\rho}\rest{\bhs{sb}} \\
	\lrspar{\beta^*(\rho\pa_x) + (\bN_{H/Y}-\tfrac12 v)\frac x\rho}\rest{\bhs{\bs}} & -\hat \eth_{\dR}^Y 
	\end{pmatrix} \Pi_{h}.
\end{multline*}
Note that, by our assumptions on the metric, the diagonal entries commute with the off-diagonal entries.\\

To work with this operator, it is convenient to introduce projective coordinates near $\bhs{\bs}$ such as
\begin{equation*}
	X = \frac{x}\eps, \quad y, \quad z, \quad \eps
\end{equation*}
in which $\eps$ is a boundary defining function for $\bhs{\bs}.$
In these coordinates,
\begin{equation*}
	\beta^*(\rho\pa_x) = \sqrt{X^2+1} \pa_X = \ang X \pa_X, \quad
	\frac x\rho = \frac X{\ang X}
\end{equation*}
and hence
\begin{equation}\label{eq:dRHorOp}
	D_b = 
	\begin{pmatrix}
	\hat \eth_{\dR}^Y    & -\ang X \pa_X + (\bN_{H/Y} - \tfrac12 v)\frac X{\ang X} \\
	\ang X \pa_X + (\bN_{H/Y}-\tfrac12 v) \frac X{\ang X} & -\hat \eth_{\dR}^Y  
	\end{pmatrix} 
\end{equation}
acting on $\CI(Y \times \bbR^+_X; \Lambda^*Y \otimes \rho^{\bN} \cH^*(H/Y;F) \oplus \frac{dX}{\ang{X}} \wedge \Lambda^*Y \otimes \rho^{\bN}\cH^*(H/Y;F)).$ 
The bundle $\rho^{\bN} \cH^*(H/Y;F)$ inherits a metric $g_\cH$ and a connection $\nabla^{\cH}$ by composing with $\Pi_h.$
From \cite[Proposition 3.14]{Bismut-Lott}, the connection $\nabla^{\cH}$ is flat and $\hat\eth_{\dR}^Y$ is the de-Rham operator $d_{\cH} + d_{\cH}^*$ corresponding to this data (cf. \cite[Proposition 15]{hhm}, \cite[\S3.1]{ALMP:Novikov}).
We will denote this operator by $\eth_{\dR}^{\cH}.$


The horizontal operator is a $b$-operator in the sense of \cite{MelroseAPS} and so it will be Fredholm when its indicial family is invertible. In this case the indicial family corresponds to the ends $X\to \pm \infty,$ each of which we compactify using $\ang{X}^{-1}.$
The indicial family is equal to
\begin{equation*}
	I_b(D_b;\zeta) =
	\begin{cases}
	\begin{pmatrix}
	\eth_{\dR}^{\cH}    & -i\zeta - (\bN_{H/Y} - \tfrac12 v) \\
	i\zeta - (\bN_{H/Y}-\tfrac12 v) & -\eth_{\dR}^{\cH}  
	\end{pmatrix} 
	& \Mat X\to -\infty \\ 
	& \\
	\begin{pmatrix}
	\eth_{\dR}^{\cH}    & i\zeta + (\bN_{H/Y} - \tfrac12 v) \\
	-i\zeta + (\bN_{H/Y}-\tfrac12 v) & -\eth_{\dR}^{\cH}  
	\end{pmatrix} 
	& \Mat X\to \infty
	\end{cases}
\end{equation*}
and so $\zeta$ will be an indicial root precisely when
\begin{equation*}
	(\eth_{\dR}^{\cH})^2 + \zeta^2 + (\bN_{H/Y}-\tfrac12 v)^2
\end{equation*}
is not invertible.
We know that $D_b$ will be a Fredholm operator on $L^2_b$ when there are no indicial roots with imaginary part equal to zero.
This will be the case unless the operator $(\eth_{\dR}^{\cH})^2$ has non-trivial null space acting on the bundle $\rho^{v/2} \cH^{v/2}(H/Y;F) \lra Y$; the simplest way to rule this out is to ask that the bundle be of rank $0$.
Let us say that the flat bundle $F$ is {\bf Witt} if 
\begin{equation}\label{eq:DefWitt}
	\text{H}^{v/2}(H/Y;F) =0.
\end{equation}

\begin{lemma}\label{lem:WittIsFredholm}
Let $F \lra X_s$ be a bundle with flat connection $\nabla^F$ and bundle metric $g_F,$ not necessarily compatible.
If the bundle $F \lra X_s$ is Witt, then the horizontal operator is Fredholm on $L^2_b.$
Otherwise the horizontal operator is Fredholm on $L^2_b$ if and only if
\begin{equation*}
	\ker \eth_{\dR}^{\cH}\rest{\rho^{v/2} \cH^{v/2}(H/Y;F)} = \{ 0 \}.
\end{equation*}
\end{lemma}

Next, let us record the following two extreme cases. If $Z$ is a point, so that $H=Y,$ we are in the setting of `$b$-surgery' and  the de Rham operator at the boundary takes the form
\begin{equation*}
	\text{ $b$-surgery } \eth_{\dR} = 
	\begin{pmatrix}
	\eth_{\dR}^Y    & -\rho\pa_x   \\
	\rho \pa_x  & - \eth_{\dR}^Y 
	\end{pmatrix}.
\end{equation*}
At the other extreme, if $Y$ is a point so that $H=Z,$ what we will call `cusp-surgery,' the de Rham operator near $\bhs{\bs}$ is
\begin{equation*}
	\text{ cusp-surgery } \eth_{\dR} = 
	\begin{pmatrix}
	 \tfrac1\rho \eth_{\dR}^{Z}   & -\rho\pa_x + (\bN_{H/Y} - \tfrac v2) \tfrac x\rho  \\
	\rho\pa_x + (\bN_{Z}-\tfrac12 v)\tfrac x\rho &  - \tfrac1\rho \eth_{\dR}^{Z}
	\end{pmatrix}.
\end{equation*}
$ $

Finally, for more general metrics, it is easy to see that if $g_{\ed}$ is product-type to second order and if $g_F$ is constant in $x$ and $\epsilon$ modulo terms of order $\rho^2$, then it has the same model operators as a product-type metric.
If $g_{\ed}$ is product-type to first order then it will have the same vertical operator, but the horizontal operator will generally be substantially more complicated (cf. \cite[Lemma 5.32]{v}).

\section*{Resolvent under degeneration}
\section{Pseudodifferential operator calculi}\label{doublespsec}

In this section we consider $H \subseteq M$ a hypersurface of $M$ with a fixed boundary defining function $x$ and fiber bundle structure
\begin{equation*}
	Z \fib H \xlra{\phi} Y.
\end{equation*}
We endow $M$ with a fibered cusp surgery metric and wish to 
give a precise description of the behavior of the resolvent of the Hodge Laplacian twisted by a flat bundle.
We will do this by describing the asymptotics of the Schwartz kernel of the resolvent by first constructing an appropriate pseudodifferential calculus and then proving that it contains the resolvent.
In keeping with the geometric approach to microlocal analysis of Melrose, we start from $M^2 \times [0,1]_{\eps}$ by constructing a `double space'  on which the Schwartz kernel of the resolvent will be practically smooth (i.e., polyhomogeneous).

\subsection{Double space}
The first step in our resolvent construction is to define a surgery double space, where the integral kernels of the operators in our various pseudodifferential calculi will be defined. As in \cite{mame1}, the double space associated to $X_s$ is an iterated radial blow-up of $M^2\times[0,1]_{\epsilon}.$ We let
\begin{equation*}
	X^2_{b,s}=
	[M\times M\times [0,1]_{\epsilon}; H\times H\times\{0\}; H\times M\times\{0\};
	M\times H\times\{0\}]
\end{equation*}
and denote the blow-down map by $\beta_b:X^2_{b,s} \lra M^2\times[0,1].$
The space $X^2_{b,s}$ is exactly the same as the space which is called $X_s^2$ in \cite{mame1}. There are four boundary hypersurfaces (though some may be disconnected, depending on the topology). 
The closure of the remains of the original $\{ \eps=0\}$ boundary is denoted $\bhs{\mf},$ this is the interior lift of $\{\eps=0\},$
\begin{equation*}
	\bhs{\mf} = \beta_b^{\sharp}(\{\eps=0\}) = \overline{ \beta_b^{-1}\lrpar{ \{ \eps=0\} \setminus 
	H\times M\times\{0\} \cup
	M\times H\times\{0\} } }.
\end{equation*}
The boundary hypersurfaces produced by the blow-ups are denoted
\begin{equation*}
	\bhs{\bff} = \beta_b^{-1}(H\times H \times \{ 0\}), \quad
	\bhs{\lf} = \beta_b^{-1}(H \times M \times \{ 0 \}), \quad
	\bhs{\rf} = \beta_b^{-1}(M\times H \times \{ 0 \}).
\end{equation*}
Our notation is inspired by \cite{mame1,v}. Let $x$ denote the boundary defining function for $H$ in the left factor of $M$ in $M^2 \times [0,1]_{\eps}$ and $x'$ denote the corresponding function on the right factor (we will generally use primes to denote functions or coordinates on the second factor). The lift of $x$ to $X^2_{b,s}$ vanishes at both $\bhs{\lf}$ and $\bhs{\bff},$ while the lift of $x'$ vanishes at $\bhs{\rf}$ and $\bhs{\bff},$ and the lift of $\epsilon$ vanishes at all boundary hypersurfaces.  We denote by $\Delta_b$ the interior lift of the diagonal $\Delta_M\times [0,1]_{\epsilon}\subset M\times M\times [0,1]_{\epsilon}$ to $X^2_{b,s}$.

The surgery double space we are interested in requires one additional blow-up which involves the fibration structure of the boundary. 
Note that $D_x = \bhs{\bff} \cap \beta_b^{\sharp}\{ x =x' \}$ does not intersect $\bhs{\lf} \cup \bhs{\rf}$ and it is easy to see that $\beta_b$ gives $D_x$ the structure of a fiber bundle over $H^2.$ The submanifold we are interested in is
\begin{equation}
	D_{fib} = (\beta_b\rest{D_x})^{-1}( \{ (h,h') \in H^2 : \phi(h) = \phi(h') \}) \subseteq \bhs{\bff}.
\label{fib.1}\end{equation}
This is a p-submanifold of $X^2_{b,s}$ and we define the surgery double space by 
\begin{equation*}
	X_s^2
	=[X_{b,s}^2; D_{fib}].
\end{equation*}
We label the new boundary hypersurface $\bhs{\ff}$ and relabel the lift of $\bhs{\bff}$ under the blow-down map as $\phibf$; the other boundary hypersurfaces ($\bhs{\mf},$ $\bhs{\lf},$ and $\bhs{\rf}$) keep their names under the lift. The space $X_s^2$ is illustrated in Figure \ref{fig:doublespace}. 
\begin{figure}
	\centering
	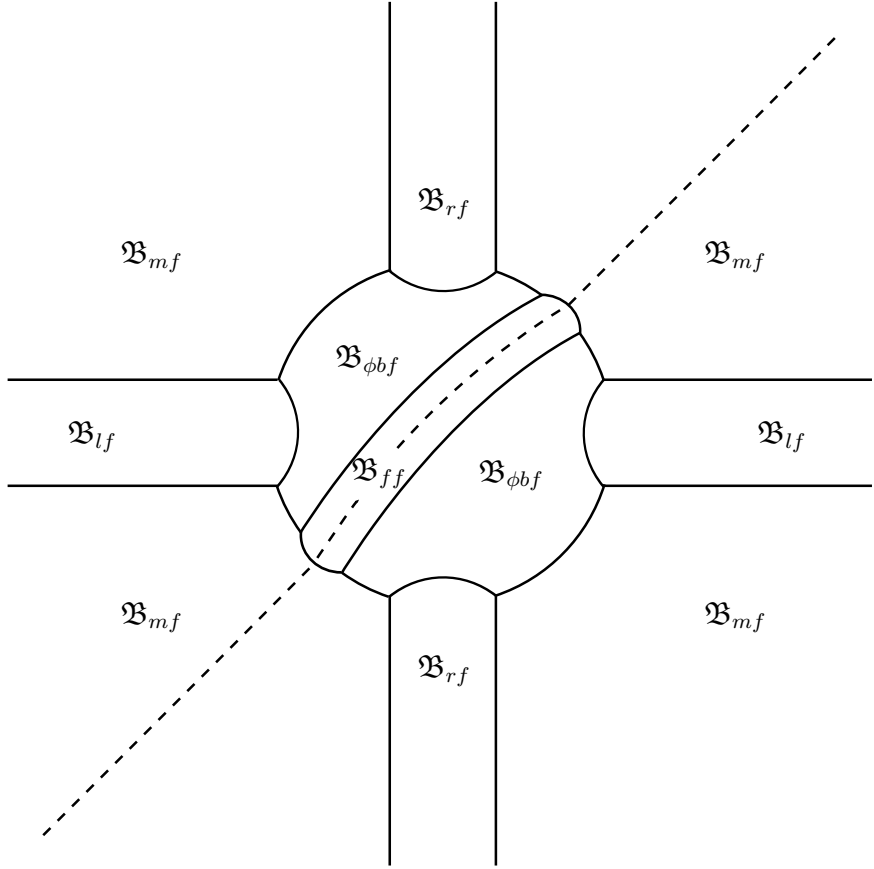
	\caption{The surgery double space $X_s^2.$}
	\label{fig:doublespace}
\end{figure}
Let $\beta_{(2)}$ be the blow-down map,
\begin{equation*}
	\beta_{(2)}:X^2_s \lra 
	M^2\times[0,1]_{\epsilon}.
\end{equation*}
There are also well-defined left and right `projection maps' from $X_s^2$ to $X_s,$ given by taking the projection on the interior and extending by continuity; see Proposition \ref{propc1}. We label these maps 
\begin{equation*}
	\xymatrix{
	& X^2_s \ar[ld]_{\beta_{(2),L}} \ar[rd]^{\beta_{(2),R}} & \\
	X_s & & X_s }
\end{equation*}
Similar maps $\beta_{(2),b,L}$ and $\beta_{(2),b,R}$ may be defined from $X_{b,s}^2$ to $X_s.$ Finally, we denote the interior lift of the diagonal of $M$ to $X_s^2$ by $\Delta_s.$

Note that $\bhs{\mf}$ may be identified with the $\phi$-double space corresponding to the $\phi$-manifold $\bhs{\ms}=[M;H]$ with boundary defining function $x.$ As in \cite{mame1}, we may identify $\fbface$ with the overblown $b$-double space corresponding to the face $\bhs{\bs}$ of the single surgery space. On the other hand, $\phibf$ is $\fbface$ after the fiber diagonal is blown up. Local coordinates help us understand the geometry. 
If $(x,y,z,x',y',z',\eps)$ are local coordinates in $M^2 \times [0,1]_{\eps}$ near the fiber diagonal in $H^2 \times \{0\},$ then a convenient choice of local coordinates in a neighborhood of the fiber diagonal on $\fbface,$ before the final blow-up, is
\begin{equation}\label{eq.lc1}
	\lrpar{ \rho'=\sqrt{(x')^2+\epsilon^2},s=\frac{x-x'}{\rho'},y,y',z,z',\theta'=\arctan \left(\frac{x'}{\epsilon}\right) }.
\end{equation}
In these coordinates, the submanifold blown-up to obtain $\bhs{\ff}$ is given by $\{\rho'=0;\ s=0,\ y=y'\}.$
After the final blow-up, coordinates near the interior of $\bhs{\ff}$ are:
\begin{equation}\label{eq.lc2}
	\lrpar{ \rho',\hat s=\frac{s}{\rho'},\hat y=\frac{y-y'}{\rho'},y',z,z',\theta' }.
\end{equation}
The lifted diagonal is the set of points with $\hat s=\hat y=0$ and $z=z'.$

This is related to the vector bundle 
\begin{equation}\label{eq:DefPhiNB}
	\Ephi N\bhs{\bs}\lra \bhs{\bs}
\end{equation}
obtained as the kernel of the inclusion $\Ephi TX_s \lra {}^\eps TX_s$ after restricting to $\bhs{\bs}.$  Notice first that as in \cite{mame2}, see also \cite[(2.10) and (2.11)]{DLR2011}, the vector bundle $\Ephi N\bhs{\bs}$ may be naturally identified with
$\phi_+^*\Ephi NY$ for some vector bundle $\Ephi NY\to Y\times [-\pi/2,\pi/2]$, where $\phi_+$ is defined in \eqref{rc.1}.  Now, 
before it is blown up, the fiber diagonal on $\fbface$ is naturally isomorphic to $\bhs{\bs}\times_{Y\times[-\pi/2,\pi/2]}\bhs{\bs},$ with coordinates $(y,z,z',\theta').$ Therefore, after blowing up at $\rho'=0,$ $s=1,$ and $y=y',$ we see that 
\begin{equation*}
	\bhs{\ff}\cong\bar{\Ephi N\bhs{\bs}}\times_{Y\times[-\pi/2,\pi/2]} \bhs{\bs} \cong    \bar{\Ephi NY}\times_{Y\times[-\pi/2,\pi/2]} \left(\bhs{\bs}\times_{Y\times[-\pi/2,\pi/2]}\bhs{\bs} \right).
\end{equation*}
In local coordinates, $y'$ and $\theta'$ are coordinates on the product, $z'$ is the extra coordinate on $\bhs{\bs},$ and $\hat s,$ $\hat y,$ and $z$ are extra coordinates on $\Ephi N\bhs{\bs}$; the $z$-coordinate comes from the base of the bundle, and $\hat s$ and $\hat y$ are Euclidean coordinates on the fibers $\mathbb R\times T_q Y.$ 

In order to check that we have the right double space, we need to show the following:
\begin{proposition}\label{liftedvfs} 
The lift of the Lie algebra $\mathcal V_{\phi,\epsilon}$ from the left and from the right is transversal to the lifted diagonal. 
 \end{proposition}
\begin{proof} By symmetry, it suffices consider the lift from the left.  We first prove this for the basis vector fields $\rho^2\pa_x,$ $\rho\pa_y,$ and $\pa_z.$ Away from the face $\bhs{\ff},$ the result is immediate.

To begin the analysis near $\bhs{\ff},$ we compute the lifts of these vector fields in the coordinates $(\rho',s,y,y',z,z',\theta')$ of \eqref{eq.lc1}. From \eqref{eq.lc1}, we compute
\begin{equation*}
	\beta_{(2),L}^*\pa_x=\frac{\pa s}{\pa x}\pa_s=\frac{1}{\rho'}\pa_s, \quad
	\beta_{(2),L}^*\pa_y=\pa_y, \quad
	\beta_{(2),L}^*\pa_z=\pa_z.
\end{equation*}
On the other hand, using $x'=\rho'\sin\theta',$ $x=\rho's+x',$ and $\rho=\sqrt{x^2+\epsilon^2},$ we compute that  $\rho=\rho'\sqrt{s^2+2s\sin\theta'+1}.$ So
\begin{equation*}
	\beta_{(2),L}^*(\rho^2\pa_x)=\rho'(s^2+2s\sin\theta'+1)\pa_s, \quad
	\beta_{(2),L}^*(\rho\pa_y)=\rho'\sqrt{s^2+2s\sin\theta'+1}\pa_y, \quad
	\beta_{(2),L}^*\pa_z=\pa_z.
\end{equation*}

Now we compute the lifts in the coordinates of $\eqref{eq.lc2}.$ We have $\beta_{(2),L}^*\pa_s=(\rho')^{-1}\pa_{\hat s}$ and $\beta_{(2),L}^*\pa_y = (\rho')^{-1}\pa_{\hat y},$ so
\begin{equation*}
	\beta_{(2),L}^*(\rho^2\pa_x)=((\rho')^2\hat s^2+2\rho'\hat s\sin\theta'+1)\pa_{\hat s}, \quad
	\beta_{(2),L}^*(\rho\pa_y)=\sqrt{(\rho')^2\hat s^2+2\rho'\hat s\sin\theta'+1}\pa_{\hat y}, 
\end{equation*}
and $\beta_{(2),L}^*\pa_z=\pa_z.$
Restricting to $\bhs{\ff},$ where $\rho'=0,$ we find that
\begin{equation*}
	\beta_{(2),L}^*(\rho^2\pa_x)|_{\bhs{\ff}}=\pa_{\hat s}, \quad
	\beta_{(2),L}^*(\rho\pa_y)|_{\bhs{\ff}}=\pa_{\hat y}, \quad
	\beta_{(2),L}^*\pa_z|_{\bhs{\ff}}=\pa_z.
\end{equation*}

The proposition now follows by $C^{\infty}(X_s)$-linearity, since multiplying by a smooth function of $(\rho,\theta,y,z)$ changes none of the properties we want to show. 
\end{proof}

\begin{corollary}\label{cor:normdiagonal} The normal bundle $N\Delta_s$ to the lifted diagonal $\Delta_s$  is canonically identified with ${}^{\ephi}TX_s.$
\end{corollary}

\subsection{Densities}\label{sec:Densities}
When discussing pseudodifferential operator calculi, we must make a choice of convention  with densities. Here we adopt the convention that our operators act on sections, and correspondingly have Schwarz kernels which are `densities lifted from the right.' This has the advantage that if $A$ is a differential operator and $B$ is a pseudodifferential operator with kernel $K_B,$ the operator $A\circ B$ has kernel $\beta_L^*(A)K_B$; that is, lifting $A$ from the left and applying  it to $K_B,$ gives the kernel of $A\circ B.$

Since objects from the $\phi$-calculus and the $b$-calculus will both appear in our construction, as will local invariants of $d$-metrics, we will first define several different density bundles. On any manifold with corners $W,$ there is a canonical density bundle $\Omega(W),$ obtained by using the local coordinate charts. There is also an associated $b$-density bundle
\[\Omega_b(W)=f^{-1}\Omega(W),\]
where $f$ is a product of boundary defining functions for each boundary hypersurface of $W.$ Throughout, we let $\nu(W)$ and $\nu_b(W)$ denote spanning sections for $\Omega(W)$ and $\Omega_b(W)$ respectively.

We now define several density bundles on $X_s.$ Recall that the hypersurface $H$ of $M$ is the total space of a fibration $Z \fib H \xlra{\phi} Y$ and we denote the dimension of $Z$ by $v$ and that of $Y$ by $h.$
 The metric objects we are interested in are $\ed$-densities; we call the appropriate bundle $\Omega_{\epsilon, d}(X_s),$ which is defined by the property that $\nu_{\epsilon, d}\equiv |dg_{\epsilon}|$ is a smooth nonvanishing section. In local coordinates, \[\nu_{\epsilon, d}\sim |\rho^{v-1}dx d\phi^*g_Y dg_Z|.\]
There is a similar surgery density bundle for the conformally related metric $g_{\ephi},$ which we call $\Omega_{\ephi}(X_s).$ It is spanned over $C^{\infty}(X_s)$ by $\nu_{\ephi}\equiv|dg_{\ephi}|,$ which in local coordinates reads
\[|\rho^{-h-2} dx d\phi^*g_Y dg_Z|.\]
Finally, there is a bundle of $b$-surgery densities $\Omega_{\epsilon,b}(X_s),$ as in \cite{mame1}, spanned by
\[\nu_{\epsilon, b}\equiv |\rho^{-1} dx d\phi^*g_Y dg_Z|.\]
None of these density bundles are multiples of the canonical density bundle on $X_s,$ as they all lack a $|d\epsilon|.$ However, directly from the definitions, they are related by \[\nu_{\ephi}=\rho^{-n}\nu_{\epsilon,d}=\rho^{-(h+1)}\nu_{\epsilon,b}.\]

\subsection{Operator calculi}
Now we define our pseudodifferential operator calculi. As before, let $E\to X_s$ be a smooth vector bundle.  Denote also by $E\to X_s^2$ its lift $\beta_{(2),L}^*E$ from the left  and let $E'\to X_s^2$ be the lift from the right 
$\beta_{(2),R}^* E^*$ of the dual vector bundle $E^*\to X_s$.  On $X_{b,s}^2$, let us also use the notation $E_b:= \beta_{(2),b,L}^*E$ and $E_b':= \beta_{(2),b,R}^*E^*$.   Our Schwarz kernels for the $\phi$- and $b-$surgery calculi will live on the double spaces $X_s^2$ and $X_{b,s}^2$ respectively. We will view them as sections of the partial density bundles $(\beta_{(2),R})^*\Omega_{\ephi}(X_s)$ for the $\phi$-surgery calculus and $(\beta_{(2),b,R})^*\Omega_{\epsilon,b}(X_s)$ for the $b$-surgery calculus. Let $\kappa_{\phi}$ and $\kappa_{b}$ be smooth nonvanishing sections of the former and latter bundles respectively.

First we define the small $b$- and $\phi$-surgery calculi. We follow the notation of Vaillant. The small $b$-surgery calculus is the union over all $m\in\mathbb R$ of
\[\Psi^{m}_{\epsilon,b}(X_s;E)=\dot{C}^{\infty}_{\bff,\mf}I^m(X_{b,s}^2,\Delta_b;(E_b\otimes (E_b'\otimes(\beta_{(2),b,R})^*\Omega_{\epsilon,b}(X_s)))).\]
These are the distributions which have a (one-step polyhomogeneous) conormal singularity of order $m$ at the lifted diagonal $\Delta_b$ in $X_{b,s}^2$ and are polyhomogeneous conormal on $X_{b,s}^2,$ with index set $C^{\infty}$ at $\bhs{\bff}$ and $\bhs{\mf},$ and with infinite-order vanishing at $\bhs{\lf}$ and $\bhs{\rf}.$ Similarly, the small $\phi$-surgery calculus is the union over $m \in \bbR$ of
\[\Psi^{m}_{\ephi}(X_s;E)=\dot{C}^{\infty}_{\ff, mf}I^m(X_s^2,\Delta_s;(E\otimes(E'\otimes (\beta_{(2),R})^*\Omega_{\ephi}(X_s)))).\]

Distributions in the small $\phi$-surgery calculus have (one-step polyhomogeneous) conormal singularities at the lifted diagonal $\Delta_s$ in $X_s^2,$ are polyhomogeneous on $X_s^2$ with $C^{\infty}$ index sets at $\bhs{\ff}$ and $\bhs{\mf}$ and vanishing to infinite order at $\bhs{\lf},$ $\bhs{\rf},$ and $\bhs{\fbf}.$ Note that by our choice of density bundles, the identity operator on $(M,g_{\ephi})$ is an element of $\Psi^0_{\ephi}(X_s).$

The small-calculus definitions generalize in the usual way to full calculus definitions. For example, let $\mathcal E$ be an index family for $X_{b,s}^2$; then
\[\Psi^{m,\mathcal E}_{\epsilon,b}(X_s;E)=\mathcal A^{\mathcal E}I^m(X_{b,s}^2,\Delta_b;(E_b\otimes (E_b'\otimes(\beta_{(2),b,R})^*\Omega_{\epsilon,b}(X_s)))).\]
The small calculus is the subset of the full calculus with $C^{\infty}$ index sets at $\bhs{\mf}$ and $\bhs{\bff}$ and empty index sets at all other faces. A similar definition holds for the full $\phi$-calculus; given an index family $\mathcal F$ for $X_s^2,$
\[\Psi^{m,\mathcal F}_{\ephi}(X_s;E)=\mathcal A^{\mathcal F}I^m(X_s^2,\Delta_\phi;(E\otimes(E'\otimes (\beta_{(2),R})^*\Omega_{\ephi}(X_s)))).\]

We also have a symbol map for the $\phi$-surgery calculus. The kernel of a $\phi$-surgery operator is a conormal distribution, for which there is the usual symbol map
\begin{equation*}
	\sigma:I^m(X_s^2,\Delta_\phi;(E\otimes(E'\otimes (\beta_{(2),R})^*\Omega_{\ephi}(X_s))))\to S^m(N^*\Delta_s; \textrm{End}(E)),
\end{equation*}
with null space the set of conormal distributions of one lower order. Using the identification between $N\Delta_s$ and ${}^{\phi}TX_s$ from Corollary \ref{cor:normdiagonal} (note also that, by an argument similar to that in Proposition \ref{liftedvfs}, the density factor $(\rho')^{-h-2}\ dx'\ dy'\ dz'$ lifts to be nonsingular all the way down to $\bhs{\ff}$), this may be interpreted as a map
\begin{equation}\label{eq:symbolmap}{}^{\ephi}\sigma_m:\Psi^{m}_{\ephi}(X_s;E)\rightarrow S^m({}^{\ephi}TX_s;\textrm{End}(E)).\end{equation}
This is the $\phi$-calculus symbol map; it is surjective and its nullspace is $\Psi^{m-1}_{\ephi}(X_s).$ It has a natural extension to the full $\phi$-surgery calculus (vanishing on distributions supported away from $\Delta_s$), which we also denote ${}^{\ephi}\sigma_m.$

\subsection{Mapping properties and composition}

In the remainder of this work, we will need a formula concerning the action of surgery operators on functions.

\begin{theorem}\label{mappingprops} Let $f\in\mathcal A^{\mathcal F}(X_s;E)$ and let $A\in\Psi^{m,\mathcal E}_{\ephi}(X_s;E).$ Then $g=Af\in\mathcal A^{\mathcal G}(X_s;E),$ with
\[G_{\ms}=(E_{\mf}+F_{\ms})\extu(E_{\rf}+F_{\bs}-(h+1));\ \ G_{\bs}=(E_{\lf}+F_{\ms})\extu(E_{\ff}+F_{\bs})\extu(E_{\fbf}+F_{\bs}-(h+1)).\]
\end{theorem}

There is also a composition formula for operators in the $\phi$-surgery calculus:
\begin{theorem}\label{composition}[Composition] Let $A\in\Psi^{m,\mathcal E}_{\ephi}(X_s;E)$ and $B\in\Psi^{m',\mathcal F}_{\ephi}(X_s;E).$ Then $C=A\circ B\in\Psi^{m+m',\mathcal G}_{\ephi}(X_s;E),$ where
\[G_{\ff}=(E_{\ff}+F_{\ff})\extu(E_{\fbf}+F_{\fbf}-(h+1))\extu(E_{\lf}+F_{\rf});\] 
\[G_{\fbf}=(E_{\fbf}+F_{\ff})\extu(E_{\ff}+F_{\fbf})\extu(E_{\fbf}+F_{\fbf}-(h+1))\extu (E_{\lf}+F_{\rf});\] 
\[G_{\lf}=(E_{\ff}+F_{\lf})\extu(E_{\fbf}+F_{\lf}-(h+1))\extu(E_{\lf}+F_{\mf});\]
\[G_{\rf}=(E_{\rf}+F_{\ff})\extu(E_{\rf}+F_{\fbf}-(h+1))\extu(E_{\mf}+F_{\rf});\]
\[G_{\mf}=(E_{\mf}+F_{\mf})\extu(E_{\rf}+F_{\lf}-(h+1)).\]
\end{theorem}
The proof may be found in Appendix \ref{appendix:composition}. Note that in contrast with the usual full $\phi$-calculus, two operators in the full $\phi$-surgery calculus may always be composed for $\epsilon>0$. However, if $E_{\mf}=F_{\mf}=0$ and $E_{\rf}+F_{\lf}>(h+1),$ then the restrictions of each operator to $\bhs{\mf}$ may be composed \cite{v}, and the index sets in the composition rule in \cite{v} match with the index sets in Theorem \ref{composition}.

\subsection{Normal operators}\label{normalops} For certain $A\in\Psi^{m,\mathcal F}_{\ephi}(X_s;E),$ we can define \textbf{normal operators}, which are essentially the leading order coefficients of the kernel of $A$ at the various boundary hypersurfaces of the double space. First, assuming $F_{mf}\ge 0$, we define $N_{\mf}(A)$ by restriction to $\bhs{\mf},$ and note that $N_{\mf}(A)$ may be viewed as a $\phi$-operator on the manifold with fibered boundary $\bhs{\ms} = [M;H]$; indeed, $\bhs{\mf}$ is an overblown version of the $\phi$-calculus double space for $\bhs{\ms}.$

Secondly, for any $A\in\Psi^{m,\mathcal F}_{\ephi}(X_s;E)$ with $F_{\ff}\geq 0$ (that is, with the kernel bounded as we approach $\bhs{\ff}$), we let the normal operator $N_{\ff}(A)$ be the restriction of $A$ to $\bhs{\ff}.$     Recall that 
\begin{equation*}
	\bhs{\ff}^{\circ}\cong \Ephi N\bhs{\bs}\times_{Y\times[-\pi/2,\pi/2]} \bhs{\bs}  \cong \Ephi NY\times_{Y\times[-\pi/2,\pi/2]}  \left( \bhs{\bs} \times_{Y\times[-\pi/2,\pi/2]}\bhs{\bs} \right),
	\end{equation*}
so this face fibers over $Y \times [-\pi/2,\pi/2]$ with typical fiber $Z^2 \times \bbR^{h+1}.$
The normal operator at this face corresponds in each fiber to a conormal distribution with respect to $\Delta_Z\times \{0\}\subset Z^2\times \bbR^{h+1}$.  If furthermore $F_{\phi bf}=\infty$, this conormal distribution decays rapidly at infinity on $\bbR^{h+1}$.  This means that this is a family of suspended operators of order $m$ in the sense of \cite{mel95,mame2}.  We denote the space of such \textbf{families of suspended operators} by 
 $\Psi^{m}_{\sus(\Ephi NY)}(\bhs{\bs}/(Y\times[-\frac{\pi}2, \frac{\pi}2]);E)$.  Recall that    the action of an element $B$ of \linebreak $\Psi^{m}_{\sus(\Ephi NY)}(\bhs{\bs}/(Y\times[-\frac{\pi}2, \frac{\pi}2]))$ on a function defined on $\Ephi N\bhs{\bs}\cong\Ephi NY \times_{Y\times[-\pi/2,\pi/2]}   \bhs{\bs} $ is given by
\[Bu(y',\theta',z,\hat s,\hat y)=\int K_B(y',\theta',z,z',\hat s-\hat s',\hat y-\hat y')u(y',\theta',z',\hat s',\hat y')\ dz'\ d\hat s'\ d\hat y'\]
for some kernel $K_B(y',\theta',z,z',\hat s'',\hat y'')$ with a conormal singularity of order $m$ at $\{z=z',\hat s''=\hat y''=0\}$; the kernel acts as a convolution operator because of the translation invariance of $B.$ In particular, we let the normal map take $A\in\Psi^{m,\mathcal F}_{\ephi}(X_s;E)$ to the suspended operator whose kernel is $A|_{\bhs{\ff}}.$ Note that if $A$ is an element of the calculus $\Psi^{m,\cF}_{\ephi}(X_s;E),$ but with $\cF_{\phi bf}\ne \infty$,  the normal operator can still be considered a family of  suspended operators, though this time its Schwartz kernel does not decay rapidly at infinity in $\bbR^{h+1}$.  Also observe that, restricting $\theta'$ to $\pm\pi/2,$ we obtain precisely the same space of suspended operators as in the $\phi$-calculus of \cite{mame2}.

As indicated in \cite{mame2}, see also \cite{DLR2011}, the normal map in the usual small $\phi$-calculus is a homomorphism into the corresponding space of suspended operators, which forms an order-filtered algebra. Although the same is true for the small calculus in our case, all we need is the following proposition:

\begin{proposition} Let $K_P \in \Psi^*_{\ephi}(X_s;E)$ be the Schwartz kernel of a differential operator $P$.
For any $Q\in\Psi^{m,\mathcal F}_{\ephi}(X_s;E),$ continuous down to $\bhs{\ff}$ (that is, with leading order at worst zero), we have $N_{\ff}(P\circ Q)=N_{\ff}(P)\circ N_{\ff}(Q).$
\label{comp.1}\end{proposition}
\begin{proof} For simplicity, consider $V=f_1(\rho,\theta,y,z)\rho^2\pa_x+f_2(\rho,\theta,y,z)\rho\pa_y+f_3(\rho,\theta,y,z)\pa_z,$ where the functions $f_1,$ $f_2,$ and $f_3$ are smooth on $X_s$; the extension to higher-order surgery differential operators is immediate. Let $Q$ have kernel, in a neighborhood of $\bhs{\ff},$ given by $q(\rho',y',\theta',z,z',\hat s,\hat y).$

Recall that since our densities lift from the right factor, the action of differential operators is given by a lift from the left factor. In a neighborhood of $\bhs{\ff},$ the lift from the left of $V,$ as a direct result of the computation of the lifted vector fields in Proposition \ref{liftedvfs}, is
\[f_1(0,\theta',y',z)\pa_{\hat s}+f_2(0,\theta',y',z)\pa_{\hat y}+f_3(0,\theta',y',z)\pa_z+\mathcal O(\rho').\]
Therefore the kernel of $P\circ Q$ is given by
\[f_1(0,\theta',y',z)q_{\hat s}+f_2(0,\theta',y',z)q_{\hat y}+f_3(0,\theta',y',z)q_{z}+\mathcal O(\rho').\]
So
\[N_{\ff}(P\circ Q)=f_1(0,\theta',y',z)q_{\hat s}|_{\rho'=0}+f_2(0,\theta',y',z)q_{\hat y}|_{\rho'=0}+f_3(0,\theta',y',z)q_{z}|_{\rho'=0}.\]
Now take a compactly supported test function $u(y',\theta',z,\hat s,\hat y)$ on $\Ephi N\bhs{\bs}.$ We have
\[N_{\ff}(Q)u(y',\theta',z,\hat s,\hat y)=\int q(0,y',\theta',z,z',\hat s-\hat s',\hat y-\hat y')u(y',\theta',z',\hat s',\hat y')\ dz' d\hat s'\ d\hat y'.\]
We want to apply $N_{\ff}(P)$ to this function, but $N_{\ff}(P)$ is simply the suspended operator whose kernel is the kernel of the vector field $f_1(0,\theta',y',z)\pa_{\hat s}+f_2(0,\theta',y',z)\pa_{\hat y}+f_3(0,\theta',y',z)\pa_z.$ So we apply the vector field to $N_{\ff}(Q)u$ directly. We can move the derivatives inside the integral since $u$ has compact support. The derivatives in $\hat s,$ $\hat y,$ and $z$ just hit the kernel $q,$ and we conclude that $N_{\ff}(P)\circ N_{\ff}(Q)$ is the suspended operator with kernel
\[f_1(0,\theta',y',z)q_{\hat s}|_{\rho'=0}+f_2(0,\theta',y',z)q_{\hat y}|_{\rho'=0}+f_3(0,\theta',y',z)q_{z}|_{\rho'=0}.\]
This completes the proof. \end{proof}

\subsection{Compatibility}

The normal operators, as defined here, are compatible with each other and with the usual symbol. In particular, let $A$ be an element of the $\phi$-surgery calculus with order zero at $\bhs{\ff}$ and $\bhs{\mf}.$ At any component of the intersection $\bhs{\ff}\cap\bhs{\mf},$ the restriction of $N_{\ff}(A)$ to $\theta=\pm\pi/2$ is the same as the $\phi$-calculus normal operator of $N_{\mf}(A)$; they are both restrictions to $\bhs{\ff}\cap\bhs{\mf},$ just done in a different order.  The restriction of the lifted diagonal $\Delta_s$ to $\bhs{\mf}$ is the lifted diagonal $\Delta_{\phi}$ of $\bhs{\mf}$ seen as the $\phi$-double space of $\bhs{\ms}$, and its conormal bundle restricts to give the conormal bundle of $\Delta_{\phi}$ in $\bhs{\mf}$.  This means that the principal symbol ${}^{\ephi}\sigma_m(A)$ of a $\phi$-surgery operator $A$ of order $m$ naturally restricts on $N^*\Delta_{\phi}$ to give the principal symbol ${}^{\phi}\sigma_m(N_{\mf}(A))$ of the normal operator $N_{\mf}(A)$.  Similarly, the restriction of the conormal bundle $N^*\Delta_s$ to $\Delta_s\cap \bhs{\ff}$ is simply the conormal bundle of $\Delta_s\cap \bhs{ff}$ in $\bhs{ff}$, so that the restriction of ${}^{\ephi}\sigma_m(A)$ to $\Delta_s\cap \bhs{\ff}$ is naturally identified with the principal symbol of the family of suspended operators $N_{\ff}(A)$.

\section{Resolvent construction}  \label{rc.0}

In this section we carry out the construction of the resolvent using the pseudodifferential calculi described in \S\ref{doublespsec}. For our application to analytic torsion we will need information about the resolvent of the twisted de Rham operator, $\eth_{\dR}.$ However as it requires no more effort, we will construct the resolvent for a  Dirac-type operator associated to an $\ed$-metric.

Thus together with a fixed boundary defining function $x$ for $H$ and an  $\ed$-metric $g_{\sfc}$ product-type to order 1, we assume that we have an Euclidean vector bundle $E \lra X_s$ and a formally self-adjoint $\ed$-elliptic differential operator $\eth_{\ed}\in \Diff^1_{\sfc}(X_s;E)$, that is, $\rho\eth_{\sfc} \in \Diff^1_{\ephi}(X_s;E),$ where $\rho = \sqrt{x^2+\eps^2}$.  An example to keep in mind is the situation where $E$ is in fact a Clifford module for the Clifford bundle of  the $\ed$-tangent bundle and that $\eth_{\ed}$ is the Dirac-type operator associated to a choice of Clifford connection.  
As in \eqref{eq:ShiftToL2b}, the operator $\eth_{\ed}$ acting on $L^2_{\ed}(X_s;E)$ is equivalent to the operator
\begin{equation*}
	D_{\sfc}
	= \rho^{v/2}\eth_{\sfc}\rho^{-v/2} \text{ acting on } L^2_{\eps,b}(X_s;E),
\end{equation*}
and it will be convenient for us to work with the latter. It is also convenient to introduce the abbreviation 
\begin{equation*}
	P(\lambda) 
	= \rho( D_{\sfc} - \lambda ).
\end{equation*}

We now define two model operators associated to $D_{\sfc}.$ Noting that $\rho D_{\sfc}\in\Diff^1_{\ephi}(X_s;E),$ so its restriction to $\bhs{\bs}$ is tangent to the fibers of $\phi_{+},$ we have the following definition:
\begin{definition}
The \textbf{vertical family} is the family of operators 
\begin{equation*}
	D_v\in \Diff^1(\bhs{\bs}/(Y\times [-\pi/2,\pi/2]);E)
\end{equation*}
obtained by restricting the action of $\rho D_{\sfc}$ to the boundary face $\bhs{\bs}.$ Notice that we obtain the same family by restricting the action of $P(\lambda)$ or   $\rho\eth_{\sfc}$ to $\bhs{\bs}.$
\label{rc.2}\end{definition}

Writing $D_v = \rho^{1/2} D_{\sfc} \rho^{1/2}\rest{\bhs{sb}},$ notice that the formal self-adjointness of $D_{\sfc}$ on $L^2_{\eps,b}(X_s;E)$ implies the formal self-adjointness of $D_v.$
The vertical operator is closely related to the normal operator at $\bhs{\ff}$ of $\rho D_{\ed}.$ Indeed, a direct computation in local coordinates or an appeal to naturality shows that 
\begin{equation}\label{eq:NffDed}
	Y \times [-\pi/2,\pi/2] \ni p \mapsto N_{\ff}(\rho D_{\ed})_p = D_v\rest{Z_p} + \eth_h
\end{equation}
where $\eth_h$ is a family of elliptic translation invariant operators in the fibers of ${}^\phi N \bhs{\bs}$ over $\phi^{-1}(p) \in \bhs{\bs}.$  We will assume that this family is in fact a family of Euclidean Dirac-type operators, which is automatic if $\eth_{\ed}$ is a Dirac-type operator or if it is the de Rham operator of section~\ref{sec:Metrics}.

In order to be able to construct the resolvent, we will make the following crucial assumption:
\begin{assumption}\label{kernelbundle}
The nullspaces of the various fiberwise operators of the family $D_v$ form a vector bundle 
\begin{equation*}
	\ker D_v \lra Y\times [-\pi/2,\pi/2].
\end{equation*} 
\label{rc.3}\end{assumption}

Using the restriction of the $\phi$-surgery metric $g_{\ephi}=\rho^{-2}g_{\sfc}$ to the fibers of the fiber bundle \eqref{rc.1} and the Hermitian metric on  $E,$ we can define an $L^2$ norm of the sections of $E$ over each fiber of \eqref{rc.1}.  Thus, we can define a smooth family of fiberwise projections onto the bundle $\ker D_v,$ given by
\begin{multline}
	\Pi_h:  \CI(Y\times [-\pi/2,\pi/2]; L^2(\bhs{\bs}/(Y\times [-\pi/2,\pi/2]);E))\\
	\lra \CI(Y\times [-\pi/2,\pi/2]; \ker D_v).
\label{rc.4}\end{multline} 
This can be used as follows to define the horizontal operator already mentioned in \S~\ref{sec:DeRhamOp}.
\begin{definition}
The $b$-operator $D_b\in \Diff^1_b(Y\times [-\pi/2,\pi/2];\ker D_v)$ associated to $\eth_{\sfc}$ is defined by
$$
           D_b u:= \Pi_h  \left( \left( D_{\sfc} \widetilde{u}\right)\rest{\bhs{\bs}}\right), \quad u\in \CI(Y\times [-\pi/2,\pi/2];\ker D_v),
$$ 
where $\widetilde{u}\in \CI(X_s; E)$ is chosen so that $\left. \widetilde{u}\right|_{\bhs{\bs}}=u.$
\label{rc.4a}\end{definition}
\begin{lemma}
The $b$-operator $D_b$ is well-defined; that is, $D_b u$ does not depend on the choice of extension $\widetilde{u}.$  Moreover, it is formally self-adjoint as a $b$-operator.   
\label{rc.5}\end{lemma}
\begin{proof}
Since $D_{\sfc}$ is a Dirac-type operator associated to an $\ed$-metric product-type to order 1,  a simple computation in local coordinates shows that  $[D_{\sfc},\rho]\in \rho\CI(X_s, \End(E)).$  Thus, if $\widetilde{u_1}$ and $\widetilde{u}_2$ are two different choices of extensions of $u,$ so that $\widetilde{u}_1-\widetilde{u}_2=\rho w$ for some $w\in\CI(X_s;E),$ then
 $$
     \left. \left( D_{\sfc}(\widetilde{u}_1-\widetilde{u}_2)\right)\right|_{\bhs{\bs}}= \left. \left(D_{\sfc}( \rho w) \right)\right|_{\bhs{\bs}}=\left. \left(\rho D_{\sfc} w+ [D_{\sfc}, \rho] w\right)\right|_{\bhs{\bs}}= D_v w.
 $$
 Since $D_v$ is family of formally self-adjoint operators,  $D_v w$ is orthogonal to $\ker D_v$ so that 
 $$
 \Pi_h \left( \left. \left( D_{\sfc}(\widetilde{u}_1-\widetilde{u}_2)\right)\right|_{\bhs{\bs}} \right)= \Pi_h (D_v w)=0,
 $$
 showing that the definition of $D_b u$ does not depend on the choice of extension of $u.$  
 
Moreover, the fact that $D_b$ is formally self-adjoint follows from the corresponding assertion for $D_{\sfc}.$  Indeed, given $u_1, u_2\in \CI_c(Y\times [-\pi/2,\pi/2];\ker D_v),$ choose smooth extensions $\widetilde{u}_1,\widetilde{u}_2\in \CI(X_s;E)$ supported away from $\sm.$  Then 
 $$
 \langle u_1, D_b u_2\rangle_{L^2_b}= \left. \langle \widetilde{u}_1, D_{\sfc} \widetilde{u}_2\rangle_{L^2_b}\right|_{\epsilon=0}=   \left. \langle D_{\sfc} \widetilde{u}_1, \widetilde{u}_2\rangle_{L^2_b}\right|_{\epsilon=0}=   \langle D_b u_1 ,u_2\rangle_{L^2_b},
 $$
 which shows that $D_b$ is formally self-adjoint.  
 \end{proof}
 The following related lemma will also be useful for the construction of the heat kernel of $D^2_{\ed}$.
\begin{lemma}
If $\widetilde{u}\in \CI(X_s;E)$ satisfies $D_{\ed}^{\ell}\widetilde{u}\in \CI(X_s;E),$ then
\begin{equation*}
	\Pi_h (D^{\ell}_{\ed}\wt u)\rest{\bhs{sb}} = D_b^{\ell} \lrpar{ \wt u\rest{\bhs{sb}}}.
\end{equation*}
Conversely, given $u\in \CI(Y\times[-\pi/2,\pi/2];\ker D_v)$ there exists an extension $\widetilde{u}\in \CI(X_s;E)$ such that 
\begin{equation*}
	(D_{\ed}^{\ell}\widetilde{u})\rest{\bhs{sb}}= \Pi_h (D_{\ed}^{\ell}\widetilde{u})\rest{\bhs{sb}} = D_b^{\ell}u \quad \forall\; \ell\in\bbN_0.
\end{equation*}
\label{rc.5a}\end{lemma}

\begin{proof}
For $\ell=1$ note that 
since the leading term of $D_{\ed} \wt u$ at $\bhs{sb}$ is $\tfrac1\rho D_v \wt u,$ the fact that $D_{\ed}\wt u \in \CI(X_s;E)$ implies that
\begin{equation*}
	\wt u\rest{\bhs{sb}} = \Pi_h \wt u\rest{\bhs{sb}}
\end{equation*}
and so
\begin{equation*}
	\Pi_h (D_{\ed}\wt u)\rest{\bhs{sb}} = D_b \lrpar{ \wt u\rest{\bhs{sb}}}
\end{equation*}
follows directly from Definition \ref{rc.4a}.

For $\ell>1,$ 
assume inductively that
\begin{equation*}
	\Pi_h (D_{\ed}^{\ell-1}\wt u)\rest{\bhs{sb}} = D_b^{\ell-1}\wt u\rest{\bhs{sb}}.
\end{equation*}
We have that
$\wt v = D_{\ed}^{\ell-1}\wt u$ satisfies
\begin{equation*}
	\wt v \in \CI(X_s;E), \quad D_{\ed}\wt v \in \CI(X_s;E)
\end{equation*}
and hence 
\begin{equation*}
	\wt v\rest{\bhs{sb}} = \Pi_h \wt v\rest{\bhs{sb}}, \quad 
	\Pi_h (D_{\ed} \wt v)\rest{\bhs{sb}} = D_b(\wt v\rest{\bhs{sb}}).
\end{equation*}
It follows that
\begin{equation*}
	\Pi_h (D^{\ell}_{\ed}\wt u)\rest{\bhs{sb}}
	= \Pi_h (D_{\ed} \wt v)\rest{\bhs{sb}}
	= D_b(\wt v\rest{\bhs{sb}})
	= D_b(\Pi_h \wt v\rest{\bhs{sb}})
	= D_b(D_b^{\ell-1}\wt u\rest{\bhs{sb}})
	= D_b^{\ell}\wt u\rest{\bhs{sb}},
\end{equation*}
as required.\\

Now given $u \in \CI(Y \times [-\pi/2, \pi/2]; \ker D_v)$ we want to find an extension $\wt u$ to $X_s$ as in the statement of the lemma. By what we have shown so far, this is equivalent to finding $\wt u \in \CI(X_s; E)$ such that 
\begin{equation*}
	\wt u\rest{\bhs{sb}}=u, \quad D_{\ed}^{\ell}\wt u \in \CI(X_s;E) \Mforall \ell \in \bbN.
\end{equation*}

Assume inductively that we have found $\wt u_N \in \CI(X_s;E)$ satisfying
\begin{equation*}
	\wt u_N \rest{\bhs{sb}}=u, \quad
	D_{\ed}^\ell \wt u_N \in \CI(X_s;E) \Mforall \ell \leq N.
\end{equation*}
Let us show that we can find $\wt w_N \in \CI(X_s;E)$ such that $D_{\ed}^{N+1}(\wt u_N + \rho^N \wt w_N) \in \CI(X_s;E).$
First we point out that for any smooth $\wt w_N$ we have that 
\begin{equation*}
	\rho D_{\ed}^{N+1}(\rho^N \wt w_N) \in \CI(X_s;E), \quad
	\rho D_{\ed}^{N+1}(\rho^N \wt w_N)\rest{\bhs{sb}} = D_v^{N+1}(\wt w_N\rest{\bhs{sb}}).
\end{equation*}
On the other hand, for any $\wt u_N$ satisfying our inductive hypothesis we have that 
\begin{equation*}
	\rho D_{\ed}^{N+1}(\wt u_N) \in \CI(X_s;E), \quad
	\rho D_{\ed}^{N+1}(\wt u_N)\rest{\bhs{sb}} = D_v( (D_{\ed}^N\wt u_N)\rest{\bhs{sb}}).
\end{equation*}
Now, since $D_v$ is formally self-adjoint, the null space of $D_v^{N+1}$ coincides with the null space of $D_v$ and hence the image of $D_v$ coincides with the image of $D_v^{N+1}.$ Thus there exists $w_N \in \CI(\bhs{sb};E)$ such that 
\begin{equation*}
	D_v^{N+1}(w_N) = -D_v( (D_{\ed}^N\wt u_N)\rest{\bhs{sb}})
\end{equation*}
and if we let $\wt w_N$ be any smooth extension of $w_N$ off of $\bhs{sb}$ we have
\begin{equation*}
	\rho D_{\ed}^{N+1}(\wt u_N + \rho^N\wt\omega_N) \in \CI(X_s;E), \quad
	\rho D_{\ed}^{N+1}(\wt u_N + \rho^N\wt\omega_N)\rest{\bhs{sb}} =0,
\end{equation*}
i.e., $D_{\ed}^{N+1}(\wt u_N + \rho^N\wt\omega_N) \in \CI(X_s;E)$ so we have finished the induction.

Note that the $N^{\text{th}}$ step in this construction does not change the first $N$ terms in the Taylor expansion of $\wt u_N$ at $\bhs{sb}.$
By Borel's Lemma \cite{MelroseAPS} there exists a smooth function $\wt u \in \CI(X_s;E)$ such that $\wt u - \wt u_N \in \rho_{sb}^{N}\CI(X_s;E)$ for all $N \in \bbN.$ It follows that
\begin{equation*}
	\wt u\rest{\bhs{sb}}=u, \quad
	D_{\ed}^{\ell}\wt u = D_{\ed}^{\ell}(\wt u_{\ell+1}) + D_{\ed}^{\ell}(\wt u - \wt u_{\ell+1}) \in \CI(X_s;E) \Mforall \ell \in \bbN
\end{equation*}
as required.
\end{proof}

To construct the resolvent of $D_{\sfc}$ near $\lambda=0,$ we will make the following assumption on the $b$-operator $D_b.$

\begin{assumption}\label{dbfredholm}
For $m\in \bbN_0,$ the operator $D_b$ is Fredholm as a map
$$
     D_b : H^{m+1}_b(Y\times [-\pi/2,\pi/2]; \ker D_v) \to H^m_b(Y\times[-\pi/2,\pi/2];\ker D_v).
$$
If $\widehat{D}_b$ is the indicial family of $D_b,$ this is equivalent to requiring that $0\notin \Spec_b(\widehat{D}_b)$ by the Fredholm criterion of \cite[Theorem~5.60]{MelroseAPS}. 
\label{rc.6}\end{assumption}
Along with Assumption \ref{kernelbundle}, this will allow us to construct the resolvent of $D_{\sfc}.$\\

Let us point out that both of these assumptions hold for our main application, the de Rham operator of a fibered  cusp surgery metric that is of product-type up to order two, twisted by a Witt flat bundle $F.$ Indeed, in this case the vertical operator is given in \eqref{eq:VerticalOpDR} and it clearly satisfies Assumption 1, while the $b$-operator is given in \eqref{eq:dRHorOp} and Assumption 2 is guaranteed by Lemma \ref{lem:WittIsFredholm}.\\

To state the main theorem of this section, it will be convenient,  for an index set $\cE$, to write $\inf \cE>a$ if any $(z,k)\in \cE$ is such that $\Re z>a$.  Similarly, the notation $\inf\cE \ge a$ will mean that 
$$
         (z,k)\in \cE, \; \Re z\le a \; \Longrightarrow \; (z,k)=(a,0).
$$  

\begin{theorem}
Let $E \lra X_s$ be a  Euclidean vector bundle  and let $\eth_{\ed}\in \Diff^1_{\sfc}(X_s;E)$ be a formally self-adjoint $\ed$-elliptic differential operator 
satisfying assumptions \ref{rc.3} and \ref{rc.6}.  Suppose also that $\eth_h$ in \eqref{eq:NffDed} is a family of Euclidean Dirac-type operators. \\
i) 
For any bounded open set $V \subseteq \bbC$ such that
\begin{equation*}
	\bar V \cap \lrpar{ \Spec(D_b) \cup \Spec(N_{\mf}(D_{\ed})) } = \emptyset,
\end{equation*}
there is an $\eps_0(V)>0$ such that $D_{\ed}-\lambda$ is invertible for all $\lambda \in V$ and $\eps < \eps_0(V).$\\
ii) 
For any bounded open set $V \subseteq \bbC$ containing the origin such that
\begin{equation*}
	\bar V \cap \lrpar{ \Spec(D_b) \cup \Spec(N_{\mf}(D_{\ed})) } \subseteq \{ 0 \},
\end{equation*}
the resolvent $(D_{\ed}-\lambda)^{-1}$ extends from $V \cap \{ \Im \lambda \neq 0\}$ to a meromorphic family of bounded operators on $V$ with only simples poles.

More precisely,
there is:\\
a bounded scalar function $f(\eps,\lambda),$ polyhomogeneous in $\eps$ and holomorphic in $\lambda,$ \\
a pair of  families of index sets, $\cJ(\lambda),$ $\cK(\lambda),$ defined and holomorphic in $V,$
satisfying $\inf \mathcal{J}\ge 0,$
$\inf \mathcal{K}\ge 0$ and 
$$
\begin{gathered}
\inf \left. \mathcal{J}\right|_{\lf}>0, \quad \inf \left. \mathcal{J}\right|_{\rf}>h+1, \quad \inf \left. \mathcal{J}\right|_{\fbf}\ge h+1,\quad \inf\left.\mathcal{J}\right|_{\ff}\ge 1, \\
 \inf \left. \mathcal{K}\right|_{\lf}>0, \quad \inf \left. \mathcal{K}\right|_{\rf}>0,
\end{gathered}
$$  
a family of operators,
\begin{equation*}
	\Res_H(\lambda) \in \Psi^{-1,\cJ(\lambda)}_{\ephi}(X_s;E),
\end{equation*}
holomorphic in $V,$\\
and a family of operators, $\Res_M(\lambda),$ such that
\begin{equation}
	f(\lambda, \epsilon)\Res_{M}(\lambda)\in \Psi^{-\infty,\mathcal{K}(\lambda)}_{\eps, b}(X_s;E)
\label{merp.1}\end{equation}
is holomorphic in $V$ and of uniformly finite rank with 
$$
  N_{\mf}( \Res_M(\lambda)) = -\frac{\Pi_{\ker_{L^2}N_{\mf}(D_{\sfc})}}{\lambda} \quad \mbox{and} \quad   N_{\bff}( \Res_M(\lambda))= -\frac{\Pi_{\ker_{L^2}(D_b)}}{\lambda},
 $$
for which we have
\begin{equation*}
	(D_{\ed}-\lambda)^{-1} = \Res_H(\lambda) + \Res_M(\lambda)
\end{equation*}
as meromorphic families on $V.$  

\label{rc.7}\end{theorem}

Before proving this result, let us make few remarks.

\begin{remark}
Restricting to $\bhs{\mf},$ we recover in particular the result of Vaillant \cite[Theorem 3.25]{v},  for $\lambda$ in a sufficiently small neighborhood of 0.  
\end{remark}

\begin{remark}
Using the formula $(\eth_{\ed}-\lambda)^{-1}= \rho^{-\frac{v}2}(D_{\ed}-\lambda)\rho^{\frac{v}2}$ gives a corresponding result for the resolvent of the operator $\eth_{\ed}$.  The effect of the factors $\rho^{\pm\frac{v}2}$ is to shift the index sets at $\bhs{\lf}$ by $-\frac{v}2$ and the index sets at $\bhs{\rf}$ by $+\frac{v}2$.
\end{remark}

\begin{remark}
Using the formula $(D^2_{\ed}-\lambda^2)= (D_{\ed}-\lambda)^{-1}(D_{\ed}+\lambda)^{-1}$ and Theorem~\ref{composition} gives a corresponding result of the resolvent of the operator $D_{\ed}^2$.  For instance, for $V$ as in part (ii) of Theorem~\ref{rc.7} such that $V=-V$, we see that 
$$
    V\ni \lambda \mapsto (D^2_{\ed}-\lambda^2)^{-1}\in \Psi^{-2,\widetilde{\mathcal{J}}(\lambda)}_{\ephi}(X_s;E)
$$
is a meromorphic family such that $f(\lambda,\epsilon)f(-\lambda,\epsilon)(D^2_{\ed}-\lambda^2)^{-1}$ is holomorphic, where $\widetilde{\mathcal{J}}(\lambda)$ is a family of index sets  holomorphic on $V$ such that $\inf \widetilde{\mathcal{J}}\ge 0$ and
$$
\inf \left. \widetilde{\mathcal{J}}\right|_{\lf}>0, \quad \inf \left. \widetilde{\mathcal{J}}\right|_{\rf}>h+1, \quad \inf \left. \widetilde{\mathcal{J}}\right|_{\fbf}\ge h+1,\quad \inf\left.\widetilde{\mathcal{J}}\right|_{\ff}\ge 1.
$$
\end{remark}

\begin{remark}
Using the formula $(\eth^2_{\ed}-\lambda^2)^{-1}= \rho^{-\frac{v}2}(D_{\ed}^2-\lambda^2)^{-1}\rho^{\frac{v}2}$ gives a corresponding result of the resolvent of $\eth_{\ed}^2$.  For instance, for $V$ as in part (ii) of Theorem~\ref{rc.7} such that $V=-V$, we see that 
$$
    V\ni \lambda \mapsto (\eth^2_{\ed}-\lambda^2)^{-1}\in \Psi^{-2,\widehat{\mathcal{J}}(\lambda)}_{\ephi}(X_s;E)
$$
is a meromorphic family such that $f(\lambda,\epsilon)f(-\lambda,\epsilon)(D^2_{\ed}-\lambda^2)^{-1}$ is holomorphic, where $\widehat{\mathcal{J}}(\lambda)$ is a family of index sets holomorphic on $V$ such that $\inf \left. \widehat{\mathcal{J}}\right|_{\mf}\ge 0$ and
$$
\inf \left. \widehat{\mathcal{J}}\right|_{\lf}>-\frac{v}2, \quad \inf \left. \widehat{\mathcal{J}}\right|_{\rf}>h+1+\frac{v}2, \quad \inf \left. \widehat{\mathcal{J}}\right|_{\fbf}\ge h+1,\quad \inf\left.\widehat{\mathcal{J}}\right|_{\ff}\ge 1.
$$
Notice in particular that taking $\eth_{\ed}$ to be the de Rham operator yields a result of the resolvent of the Hodge Laplacian $\Delta:= \eth_{\ed}^2$, and hence for the scalar Laplacian when one restricts to forms of degree zero.

\label{Hodge.1}\end{remark}

\begin{remark}
The construction of the resolvent near the origin as $\eps \to 0$ could be carried out with essentially no change near any other isolated element of the point spectrum of the model operators of $D_{\ed}.$ Thus the same construction yields a description of the resolvent on any bounded open set whose closure does not intersect either the essential spectrum of $D_b$ or the essential spectrum of $N_{\mf}(D_{\ed}).$
\label{rc.8}\end{remark}

The proof of Theorem~\ref{rc.7} proceeds via a detailed construction of the resolvent which involves many steps. We focus on ($ii$) of the theorem and prove ($i$) in the process. For the sake of readability, we first outline the construction and then discuss each step in detail.
\begin{itemize}
\item[\textbf{Step 0}:]  This step is the standard symbolic inversion at the diagonal. It consists of finding a holomorphic family $\bbC \ni \lambda\mapsto Q_0(\lambda)\in \Psi^{-1}_{\ephi}(X_s;E)$ such that 
\begin{equation}
       (\rho D_{\sfc}- \rho \lambda)Q_0(\lambda)= \Id -R_0(\lambda),
\label{rc.9}\end{equation}
with $\bbC \ni \lambda\mapsto R_0(\lambda)\in \Psi^{-\infty}_{\ephi}(X_s;E)$ a holomorphic family.  

\item[\textbf{Step 1}:] We improve the error term in \eqref{rc.9} so that it vanishes to first order at $\bhs{\ff}.$  This is achieved by inverting $P(\lambda)=\rho(D_{\sfc}-\lambda)$ at the face $\bhs{\ff}.$  Now, as an element of  $\Psi^{1}_{\sus(\Ephi NY)}(\bhs{\bs}/(Y\times[-\frac{\pi}2, \frac{\pi}2]);E)$,  the normal $N_{\ff}(P(\lambda))$  decomposes as  a family of Dirac operators in the fibers of $\phi_+:\bhs{\bs}\to Y\times [-\pi/2,\pi/2]$ and a family of Euclidean Dirac operators in the fibers of $\Ephi NY \to Y\times [-\pi/2,\pi/2]$.   We can thus invert each family separately and apply Corollary~\ref{euc.4} to the Euclidean part. This yields a holomorphic family $\widetilde{Q}_1(\lambda)$ such that 
 $$
      (\rho D_{\sfc}- \rho \lambda)\widetilde{Q}_1(\lambda)= \Id -\widetilde{R}_1(\lambda)
 $$
where $\wt R_1$ vanishes to first order at $\bhs{\ff}$ but with generally a non-trivial asymptotic expansion at $\phibf.$  Multiplying by $\rho$ on the right and by $\rho^{-1}$ on the left, this gives
 \begin{equation*}
        (D_{\sfc}-\lambda) Q_1(\lambda)= \Id -R_1(\lambda)
\end{equation*}
 with $Q_1(\lambda)= \widetilde{Q}_1(\lambda)\rho$ and $R_1(\lambda)= \rho^{-1}\widetilde{R}_1 \rho.$  
 
 \item[\textbf{Step 2}:] The leading order term at $\phibf$ of the error from step $1,$ $R_1(\lambda),$ turns out to be of order $h+1;$ once the densities are taken into account this means that $R_1(\lambda)$ is not compact unless this leading order term vanishes. In this step, we find $Q_2(\lambda)$ so that 
 \begin{equation*}
   (D_{\sfc}-\lambda) Q_2(\lambda) = \Id -R_2(\lambda),
\end{equation*} 
and arrange for the part of the coefficient of the leading order term of $R_2(\lambda)$ at $\phibf$ which is in the range of $\Pi_h$ to vanish to infinite order at $\phibf\cap \bhs{\ff}.$ This allows us to view that coefficient as a function on $\fbface$ rather than $\phibf.$
 
 \item[\textbf{Step 3}:] Using Assumption~\ref{rc.6}, we find $Q_3(\lambda)$ with
 \begin{equation*}
          (D_{\sfc}-\lambda) Q_3(\lambda) = \Id -R_3(\lambda),
\end{equation*}
where the term of order $h+1$ of $R_3(\lambda)$ at $\phibf$ is precisely the projection $\Pi_b$ onto the kernel of $D_b.$
 
\item[\textbf{Step 4}:] We show that the restriction of $R_3(\lambda)$ to $\bhs{\mf}$ is compact as an operator on $L^2_b([M;H])$.  In particular, $N_{\mf}(D_{\sfc})$ is Fredholm, which could also have been deduced from the Fredholm criterion of \cite{v}.  Using this fact, we improve the error term in the previous step so that its restriction to $\bhs{\mf}$ is just the projection onto the kernel of $N_{\mf}(D_{\sfc}).$ 

\item[\textbf{Step 5}:] We ensure that  the new error is uniformly of finite rank. Moreover, at the cost of introducing a pole at $\lambda=0,$ we further improve the error so that its restriction to $\bhs{\mf}$ and its term of order $h+1$ at $\phibf$ both vanish.    
 
\item[\textbf{Step 6}:] The error term has now sufficiently nice properties and it can be removed using analytic Fredholm theory, producing the full resolvent.

\item[\textbf{Step 7}:] Analyzing the projection onto the eigenspace of small eigenvalues allow us to be more specific about the meromorphic part of the resolvent.  
 
\end{itemize} 

Notice that if $D_v$ is invertible, then the resolvent construction greatly simplifies.  In particular, in \textbf{Step 1}, there is no need to invert a family of Euclidean Dirac operators, so one can skip \textbf{Step 2} and \textbf{Step 3}.  We now discuss each step in full detail.

\subsection*{Step 0: Symbolic inversion}

\begin{proposition}\label{rc.13}
There exist holomorphic families $\bbC \ni \lambda\mapsto Q_0(\lambda)\in \Psi^{-1}_{\ephi}(X_s;E)$ and $\bbC\ni\lambda\mapsto R_{0}(\lambda)\in \Psi^{-\infty}_{\ephi}(X_s;E)$ so that
$$ \rho(D_{\sfc}-\lambda)Q_0(\lambda)= \Id - R_0(\lambda). $$
\end{proposition}
\begin{proof}
This is the usual symbolic parametrix construction. We take advantage of the fact that $\rho(D_{\sfc}-\lambda)$ is an elliptic $\phi$-surgery operator to find $Q_0' \in\Psi_{\ephi}^{-1}(X_s;E)$ with principal symbol given by 
${}^{\ephi}\sigma_{-1}(Q_0')= ({}^{\ephi}\sigma_1(\rho (D_{\sfc}-\lambda)))^{-1}=({}^{\ephi}\sigma_1(D_{\sfc}))^{-1},$ so that 
$$
      \rho(D_{\sfc}-\lambda)Q_0'= \Id - R_0'(\lambda),
$$ 
with $R_0'(\lambda)\in \Psi^{-1}_{\ephi}(X_s;E)$ holomorphic in $\lambda.$ Then add $Q_0''(\lambda):= Q_0'R_0'(\lambda)$ to $Q_0',$ so that 
 $$
      \rho(D_{\sfc}-\lambda)(Q_0'+Q_0''(\lambda))= \Id - R_0''(\lambda)
$$
with $R_0''(\lambda) \in \Psi^{-2}_{\ephi}(X_s;E)$ and also holomorphic in $\lambda.$ Proceeding inductively, we find holomorphic families $Q^{(k)}_0(\lambda):= Q_0'R_0^{(k-1)}(\lambda)\in \Psi^{-k}_{\ephi}(X_s;E)$ and $R^{(k)}_0(\lambda)\in \Psi^{-k}_{\ephi}(X_s;E)$ such that 
$$
        \rho(D_{\sfc}-\lambda) \lrpar{ \sum_{j=1}^kQ_0^{(k)}(\lambda) }= \Id - R_0^{(k)}(\lambda).
$$
Taking an asymptotic sum over the $Q^{(k)}_0(\lambda),$ we get the desired $Q_0(\lambda).$ The asymptotic sum can be taken so that $Q_0(\lambda)$ is also holomorphic.
\end{proof}

\subsection{Step 1: Removing the error at $\bhs{\ff}$}
In this step, we improve the remainder term so that its restriction to $\bhs{\ff}$ is trivial. This involves a careful analysis of the mapping properties of the Euclidean Dirac operator which we carry out in Appendix \ref{euc.0}. 

\begin{proposition}
There exist families of operators, defined and holomorphic for $\lambda \in \bbC,$ $\widetilde{Q}_1(\lambda)\in \Psi^{-1,\widetilde{\mathcal{Q}}_1}_{\ephi}(X_s;E)$ and $\widetilde{R}_1(\lambda)\in \Psi^{-\infty, \widetilde{\mathcal{R}}_1}_{\ephi}(X_s;E)$ such that 
\begin{equation}
      \rho(D_{\sfc}-\lambda)\widetilde{Q}_1(\lambda)= \Id -\widetilde{R}_{1}(\lambda).
\label{rc.14a}\end{equation}
Here the index families $\widetilde{\mathcal{Q}}_1$ and $\widetilde{\mathcal{R}}_1$ are empty at all boundary faces, except at $\bhs{\mf}$ where they are both $\bbN_0,$ and at $\bhs{\ff}$ and $\phibf,$ where we have
$$
     \left. \widetilde{\mathcal{Q}}_1\right|_{\ff}= \bbN_0, \quad  \left. \widetilde{\mathcal{Q}}_1\right|_{\fbf}= [\circ](J_{h+1}), \quad \left. \widetilde{\mathcal{R}}_1\right|_{\ff}= \bbN_0+1,  \quad   \left. \widetilde{\mathcal{R}}_1\right|_{\fbf}\subset (J_{h+1}+1),
     $$
where $J_{h+1}$ is the index set of Corollary~\ref{euc.4}.   Moreover, $\inf \left.  \widetilde{\mathcal{Q}}_1\right|_{\fbf}= (h,0)$  and   $\inf \left.  \widetilde{\mathcal{R}}_1\right|_{\fbf}= (h+1,0).$  \label{rc.14}\end{proposition}
Here the notation $[\circ]$ means that the leading order term $A$ in the expansion at  $\fbf$ is such that $\Pi_hA\Pi_h=A$.   

\begin{proof}
We need to find $\widetilde{Q}_1'(\lambda)$ such that 
\begin{equation}
         N_{\ff}(\rho(D_{\sfc}-\lambda)\widetilde{Q}_1'(\lambda))= N_{\ff}(R_0(\lambda)),
\label{rc.15}\end{equation}
for then it suffices to take $\widetilde{Q}_1(\lambda)= Q_0(\lambda)+ \widetilde{Q}_1'(\lambda).$  To solve \eqref{rc.15}, we can decompose $N_{\ff}(R_0(\lambda))$ using the fiberwise projection $\Pi_h$ onto $\ker D_v$ on the left (note that $\Pi_h$ is an element of $\Psi^{0}_{\sus(\Ephi NY)}(\bhs{\bs}/(Y\times[-\frac{\pi}2, \frac{\pi}2]);E)$ trivial in the Euclidean directions):
$$
    N_{\ff}(R_0(\lambda))= \Pi_h(N_{\ff}(R_0(\lambda)) + (\Id-\Pi_h)(N_{\ff}(R_0(\lambda)).
$$
Now recall from \eqref{eq:NffDed} that we have the decomposition
\begin{equation}
    N_{\ff}(\rho(D_{\sfc}-\lambda)) = N_{\ff}(\rho(D_{\sfc}))= D_v+ \eth_{h}.
\label{decom.1}\end{equation}
Here $D_v$ is a vertical operator in the fibers and $\eth_{h}$ is a family of Euclidean Dirac operators on the fibers of the vector bundle $ \Ephi NY\to Y\times [-\pi/2,\pi/2].$ 
On the range of $(\Id - \Pi_h)$ the operator $D_v + \eth_{h}$ is, on each fiber, an invertible suspended operator in the sense of \cite{mel95, mame2} and hence has an inverse in $\Psi^{-1}_{\sus(\Ephi NY)}(\bhs{\bs}/(Y\times[-\frac{\pi}2, \frac{\pi}2]);E)$. On the range of $\Pi_h,$ we can apply Corollary~\ref{euc.4} to invert $\eth_{h}.$ In particular, it suffices to take $\widetilde{Q}_1'(\lambda)$ such that
$$
    N_{\ff}(\widetilde{Q}_1'(\lambda))
    = (\eth_{h})^{-1}\Pi_h(N_{\ff}(R_0(\lambda)))+ (D_v+\eth_{h})_{\perp}^{-1}(\Id-\Pi_h)  (N_{\ff}(R_0(\lambda))).
$$
Note from Corollary~\ref{euc.4} that the image of $(\eth_{h})^{-1}$ has an expansion at infinity in the Euclidean coordinates, with index family $J_{h+1}.$ This expansion corresponds to a nontrivial expansion of $\widetilde{Q}'_1(\lambda)$ at $\phibf.$ In particular, we see immediately that $\widetilde{Q}'_1(\lambda)\in \Psi^{-\infty, \widetilde{\mathcal{Q}}_1'}_{\ephi}(X_s;E),$ with $\widetilde{\mathcal{Q}}'_1$ the index family specified by the proposition. Then $\widetilde{Q}_1(\lambda)= Q_0(\lambda)+ \widetilde{Q}_1'(\lambda)$ will be as claimed with
$$
        \widetilde{R}_1(\lambda)= R_0(\lambda)-\rho(D_{\sfc}-\lambda)\widetilde{Q}_1'(\lambda).
$$
Since $\widetilde{Q}_1'(\lambda)$ vanishes at order $\rho_{\fbf}^h$ at $\phibf,$ we know $\widetilde{R}_1$ will vanish to at least at this order at $\phibf.$  However, since we have chosen the top order term in the asymptotic expansion of $\widetilde{Q}'_1(\lambda)$ at $\phibf$ so that it maps into the range of $\Pi_h,$ this means that $\rho(D_{\sfc}-\lambda)\widetilde{Q}_1'(\lambda)$ and therefore $\widetilde{R}_1(\lambda)$ vanish at order $\rho^{h+1}_{\fbf}$ at $\phibf.$  
\end{proof}

By multiplying by $\rho^{-1}$ on the left and by $\rho$ on the right in \eqref{rc.14a}, and by setting $Q_1(\lambda):= \widetilde{Q}_1(\lambda)\rho$ and $R_1(\lambda)=\rho^{-1}\widetilde{R}_1(\lambda)\rho,$ we obtain a formulation of the previous proposition directly in terms of $D_{\sfc}-\lambda$:

\begin{corollary}\label{rc.16}
There are holomorphic families parametrized by $\lambda \in \bbC,$ 
\begin{equation*}
	Q_1(\lambda)\in \Psi^{-1,\mathcal{Q}_1}_{\ephi}(X_s;E)\Mand 
	R_1(\lambda)\in \Psi^{-\infty, \mathcal{R}_1}_{\ephi}(X_s;E)
\end{equation*}
such that 
\begin{equation}
      (D_{\sfc}-\lambda)Q_1(\lambda)= \Id -R_{1}(\lambda).
\label{rc.16a}\end{equation}
The Schwartz kernels of $Q_1$ and $R_1$ vanish to infinite order at most boundary hypersurfaces, the exceptions being $\bhs{\mf}$ where their index sets are both $\bbN_0,$ and at  $\bhs{\ff}$ and $\phibf,$ where we have
\begin{equation*}
	     \left. \mathcal{Q}_1\right|_{\bhs{\ff}}= \bbN_0+1, \quad  
	     \left. \mathcal{Q}_1\right|_{\fbf}= [\circ](J_{h+1}+1), \quad 
	     \left. \mathcal{R}_1\right|_{\bhs{\ff}}= \bbN_0+1,  \quad   
	     \left. \mathcal{R}_1\right|_{\fbf}=(J_{h+1}+1).
\end{equation*}
In particular, we have that $\inf \left.  \mathcal{Q}_1\right|_{\fbf}=\inf \left.  \mathcal{R}_1\right|_{\fbf}= (h+1,0).$ 
\end{corollary}

\subsection*{Step 2: Error at $\phibf,$ preliminary step}

The error term $R_2(\lambda)$ is not good enough to apply analytic Fredholm theory yet. Indeed, since it has a term of order $h+1$ at $\phibf,$ hence possibly at $\phibf\cap\bhs{\mf},$ it may not even be compact as an operator on $L^2_b(X_s;E)$ uniformly down to $\epsilon=0.$ The reason is that vanishing at order $h+1$ in terms of $\phi$-densities corresponds to vanishing at order zero in terms of $b$-densities, and $b$-operators of this form are compact only if their restriction to the front face is zero. Therefore, our next objective is to remove this term of order $\rho^{h+1}_{\fbf}.$  In this step, we first remove the expansion of this term at $\bhs{\ff}\cap \phibf$ (technically, just the part which is in the range of $\Pi_h$), which enables us to regard the removal of the remaining term as solving a $b$-problem on the face $\fbface.$

\begin{proposition}
There exist families of operators, defined and holomorphic for $\lambda\in \bbC,$ $Q_2(\lambda)\in \Psi^{-1,\mathcal{Q}_2}_{\ephi}(X_s;E)$ and $R_2(\lambda)\in \Psi^{-\infty, \mathcal{R}_2}_{\ephi}(X_s;E)$ such that 
\begin{equation}
      (D_{\sfc}-\lambda)Q_2(\lambda)= \Id -R_{2}(\lambda).
\label{rc.16b}\end{equation}
The Schwartz kernels of $Q_2$ and $R_2$ vanish to infinite order at most boundary hypersurfaces, the exceptions being $\bhs{\mf}$ where their index sets are both $\bbN_0,$ and at  $\bhs{\ff}$ and $\phibf,$ where we have
$$
     \left. \mathcal{Q}_2\right|_{\ff}= G_{h+1}\cup(\bbN_0+1), \quad  \left. \mathcal{Q}_2\right|_{\fbf}= [\circ](J_{h+1}+1), \quad \left. \mathcal{R}_2\right|_{\ff}= \bbN_0+1,  \quad   \left. \mathcal{R}_2\right|_{\fbf}= (J_{h+1}+1), 
     $$
where $G_{h+1}$ is the index set of Proposition~\ref{bu.5} and $J_{h+1}$ is the index set of Corollary~\ref{euc.4}.   In particular $\inf \left.  \mathcal{Q}_2\right|_{\fbf}=\inf \left.  \mathcal{R}_2\right|_{\fbf}=(h+1,0).$  Finally and most importantly, if $A$ is the term of order $\rho_{\fbf}^{h+1}$ at $\phibf$ of $R_2(\lambda),$ then $\Pi_hA$ vanishes to infinite order at the boundary face $\phibf\cap \bhs{\ff}.$    
\label{rc.17}\end{proposition}
\begin{proof}
Writing $Q_2(\lambda)= Q_1(\lambda)+ \widetilde{Q}_2(\lambda),$ where $Q_1(\lambda)$ is given by Corollary~\ref{rc.16}, we need to find a holomorphic family $\widetilde{Q}_2(\lambda)$ such that the restriction of
\begin{equation}
          (D_{\sfc}-\lambda) \widetilde{Q}_2(\lambda)-R_1(\lambda)
\label{rc.18a}\end{equation}
at order $h+1$ at $\phibf$ vanishes to infinite order at $\phibf\cap \bhs{\ff}.$  Let $r_1^{\circ}(\lambda)$ denote the part of the restriction (at order $h+1$) to $\phibf$ of $R_1(\lambda)$ whose image is in the range of $\Pi_h.$   We need to find $q^{\circ}_2(\lambda)$ such that  
\begin{equation}
                (D_b-\lambda)(q_2^{\circ}(\lambda))- r_1^{\circ}(\lambda)
\label{rc.18}\end{equation}
vanishes to infinite order at $\phibf\cap \bhs{\ff}.$

Using Proposition~\ref{bu.5} we can find such a
$q_2^{\circ}(\lambda).$ Indeed, the face $\phibf$ is obtained from the face $\fbface$ by blowing up $D_{fib}.$ If $F$ is the normal bundle of $D_{fib}$ in $\fbface,$ then in $\phibf,$ the boundary face $\phibf\cap \bhs{\ff}$ has a tubular neighborhood modelled on $\widetilde{F},$ the blow-up of $F$ at the zero section.  On the other hand, the lift from the left of $(D_b-\lambda)$ corresponds near $D_{fib}$ to a family of operators as in Proposition~\ref{bu.5} for the vector bundle $F.$ 
Therefore, we can find a holomorphic family
$q^{\circ}_2(\lambda)$ on $\phibf$  with index set $G_{h+1}$ at $\phibf\cap\bhs{\ff},$ index set $\bbN_0$ at $\bhs{\mf}\cap \phibf$ and vanishing to infinite order at all other boundary hypersurfaces (since we may assume that $q^{\circ}_2(\lambda)$ is identically zero away from $\phibf\cap\bhs{\ff}$), so that \eqref{rc.18} vanishes to infinite order at $\phibf\cap\bhs{\ff}.$

Extending $q^{\circ}_2 (\lambda)$ smoothly to all of $X^2_s,$ we get a holomorphic family $\widetilde{Q}^{\circ}_2(\lambda)$ with index set $\bbN_0$ at $\bhs{\mf},$  $(\bbN_0+h+1)$ at $\phibf,$ $G_{h+1}$ at $\bhs{\ff}$ and vanishing to infinite order at all other boundary hypersurfaces such that the restriction 
$$
(D_{\sfc}-\lambda) \widetilde{Q}^{\circ}_2(\lambda)+R_1(\lambda)$$
 at order $h+1$ at $\phibf$ vanishes to infinite order at $\bhs{\ff}\cap \phibf$ after we apply $\Pi_h$ on the left.  More precisely, if $T(\lambda)$ denotes this restriction, then $\Pi_h T(\lambda)$ vanishes to infinite order at $\bhs{\ff}\cap\phibf.$
 
 \end{proof}

\subsection*{Step 3: Removing the error at $\phibf$}
   
The next step is to get rid of the term of order $h+1$ of $R_2(\lambda)$ at $\phibf.$  This is where our assumption that $D_b$ is Fredholm will be used.  Recall from Lemma~\ref{rc.5} that $D_b\in\Diff^1(Y\times [-\pi/2,\pi/2];\ker D_v)$ is formally self-adjoint. By Assumption~\ref{rc.6}, zero is not an indicial root of the indicial family of $D_b.$  Consequently, by \cite[Proposition~5.64]{MelroseAPS}, \cite[Theorem 4.20]{maz91}, $D_b$ has a generalized inverse which is polyhomogeneous conormal on the $b$-double space corresponding to $Y\times[-\pi/2,\pi/2],$ and hence on the corresponding overblown $b$-double space, which we have identified with $\bhs{bf}.$ More specifically, let $G_b$ be the generalized inverse of $D_b$ on $L^2_b,$ so that 
\begin{equation*}
	D_bG_b = G_bD_b = \Id - \Pi_b
\end{equation*}
with $\Pi_b$ the orthogonal projection onto the $L^2_b$ null space of $D_b.$
If $\wt \cE$ is the smallest index set satisfying
\begin{equation*}
	\{ (z,p) \in \bbC \times \bbN_0 : (iz,p) \in \Spec_b(D_b), \quad  \Re z>0 \} \subseteq \wt\cE
\end{equation*} 
and $\cE = \wt\cE \bar\cup \wt \cE$
then we have that 
\begin{equation*}
\begin{gathered}
	\Pi_b \in \Psi^{-\infty, (\wt\cE, \wt\cE)}(Y \times [-\pi/2, \pi/2];\ker D_v), \\
	G_b \in 
	\Psi_b^{-1, (\cE, \cE, \bbN_0)}(Y \times [-\pi/2, \pi/2];\ker D_v)
	+\Psi^{-\infty, (\cE, \cE)}(Y \times [-\pi/2, \pi/2];\ker D_v).
\end{gathered}
\end{equation*}
Here, $\Psi^{-\infty, (\cE, \cE )}(W;\ker D_v)$  with $W=Y\times [\pi/2,\pi/2]$ is the space of operators with polyhomogeneous Schwartz kernel in $\cA^{ (\cE, \cE )}(W\times W; \ker D_v\otimes (\ker D_v)^*\otimes \pi_R^*{}\Omega_b(W))$ where \linebreak $\pi_R:W\times W\to W$ is the projection on the right factor.


More generally, for $\lambda\ne 0$ sufficiently close to zero or with non-vanishing imaginary part,  $D_b-\lambda$ is invertible with inverse
\begin{equation}
	(D_b-\lambda)^{-1}\in 
	\Psi_b^{-1, (\cE,(\lambda) \cE(\lambda), \bbN_0)}(Y \times [-\pi/2, \pi/2];\ker D_v).
\label{rc.19}\end{equation}
where $\cE(\lambda)$ is the index family, holomorphic in $\lambda,$ obtained by replacing $\Spec_b(D_b)$ by $\Spec_b(D_b-\lambda)$ in the definition of $\wt\cE$ above.
The family \eqref{rc.19} is meromorphic in a small neighborhood of zero as, e.g., a family of bounded operators acting on $L^2_b(Y\times [-\pi/2,\pi/2]; \ker D_{v}),$ with a pole of order one at $\lambda=0$ with residue given by $\Pi_b.$ 

We now use this knowledge about $(D_b-\lambda)^{-1}$ to remove the term of order $h+1$ at $\phibf.$

\begin{proposition}
Let $\cU$ be a bounded neighborhood of $0$ in $\bbC$ such that $\Spec(D_b)\cap \cU\subseteq \{ 0\}.$
There are index families $\mathcal{Q}_3(\lambda)$ and $\mathcal{R}_3(\lambda),$ defined and holomorphic in $\cU,$
and families of operators
$Q_3(\lambda)\in \Psi^{-1,\mathcal{Q}_3(\lambda)}_{\ephi}(X_s;E)$ and 
$R_3(\lambda)\in \Psi^{-\infty, \mathcal{R}_3(\lambda)}_{\ephi}(X_s;E),$ defined in $V_{\delta}$ and holomorphic as operators acting on $L^2_b(X_s;E),$
such that 
\begin{equation}
      (D_{\sfc}-\lambda)Q_3(\lambda)= \Id -R_{3}(\lambda).
\label{rc.16c}\end{equation}
Moreover, the index sets satisfy 
$\inf \mathcal{Q}_3(\lambda)\ge 0,$
$\inf \mathcal{R}_3(\lambda) \ge 0,$ (meaning they are $\geq 0$ at all faces) and 
\begin{equation*}
\begin{gathered}
	\inf \left. \mathcal{Q}_3(\lambda)\right|_{\lf}> 0, \quad 
	\inf \left. \mathcal{Q}_3(\lambda)\right|_{\rf}> h+1, \quad 
	\inf \left. \mathcal{Q}_3( \lambda)\right|_{\fbf}\ge h+1, \quad
	\inf \left. \mathcal{Q}_3( \lambda)\right|_{\ff}\ge 1,  \\
	\inf \left. \mathcal{R}_3(\lambda)\right|_{\lf}> 0, \quad 
	\inf \left. \mathcal{R}_3(\lambda)\right|_{\rf}> h+1, \quad 
	\inf \left. \mathcal{R}_3( \lambda)\right|_{\fbf}\ge h+1, \quad
	\inf \left. \mathcal{R}_3( \lambda)\right|_{\bhs{\ff}}> 0.
\end{gathered}
\end{equation*}
Finally and most importantly, the term of order $h+1$  of $R_3(\lambda)$ at $\phibf$ is $\Pi_b.$          
\label{rc.17b}\end{proposition}

Note that, without changing the proof, we could replace $\cU$ with an unbounded open set consisting of a neighborhood of the origin as in the statement of the proposition together with any open set in $\bbC$ bounded away from the real axis. This will be true throughout the construction, but in order to have uniform bounds it is convenient to work with bounded open sets.

\begin{proof}
Let $A(\lambda)$ be the term of order $h+1$ at $\phibf$  of  the remainder term $R_2(\lambda)$ from Proposition \ref{rc.17}, and write 
$A(\lambda)=a(\lambda)+t^{\perp}(\lambda)+\Pi_b,$ 
where $a(\lambda)=\Pi_hA(\lambda)-\Pi_b$ and $t^{\perp}(\lambda)=(Id-\Pi_h)A(\lambda).$

We first use \eqref{rc.19} to remove the term $a(\lambda).$
Note that applying $\Pi_h$ to the term of order $h+1$ at $\phibf$ of \eqref{rc.16b} at $\lambda=0$ shows that the range of 
\begin{equation*}
	\Id - \Pi_b - a(0)
\end{equation*}
is in the range of $D_b.$ Since $D_b$ is self-adjoint, this shows that $\Id - \Pi_b - a(0)$ takes values orthogonal to the kernel of $D_b$ and hence the range of $a(0)$ is contained in the range of $\Id-\Pi_b.$
Since $a$ is holomorphic, it is thus of the form
$$
     a(\lambda)= \lambda \Pi_b a_1(\lambda) +(\Id-\Pi_b) a(\lambda)
$$
with $a_1(\lambda)$ holomorphic in $\lambda.$  This suggests that we set, for $\lambda \in \cU,$
$$
q_3^{\circ}(\lambda):=  \Pi_b a_1(\lambda) - (\Id-\Pi_b)(D_b-\lambda)^{-1}(\Id-\Pi_b) a(\lambda),
$$
so that 
$$
	(D_b-\lambda) q_3^{\circ}(\lambda)= -a(\lambda).
$$
Now let $Q_3^{\circ}(\lambda)$ be a smooth extension off $\phibf$ with term of order $h+1$ at $\phibf$ given by $q_3^{\circ}(\lambda).$ Then $(D_{\sfc}-\lambda)Q_3^{\circ}(\lambda)$ will have $-a(\lambda)$ as its term of order $h+1$ at $\phibf$ (which is in the range of $\Pi_h$), so subtracting $Q_3^{\circ}(\lambda)$ from $Q_2(\lambda)$ will remove the $a(\lambda)$ term.
  
If $\widetilde{A}(\lambda)$ denotes the term of order $h+1$ at $\phibf$ of the new remainder term, then 
$$
                 \widetilde{A}(\lambda)= u^{\perp}(\lambda) +\Pi_b \quad \mbox{with} \; \; u^{\perp}(\lambda)= (\Id-\Pi_h)\widetilde{A}(\lambda).
$$

To get rid of $u^{\perp}(\lambda),$ it suffices to take a holomorphic family $Q_3^{\perp}(\lambda)\in\Psi^{-\infty,\mathcal{J}}_{\ephi}(X_s;E)$ with restriction of order $h+1$ at $\phibf$ given by 
 $$
                         q^{\perp}_3(\lambda)= -(D_v^{-1} u^{\perp}(\lambda)), 
 $$   
where the index family $\mathcal{J}$ is such that
$$
     \left.\mathcal{J}\right|_{\fbf}=\bbN_0+h+1,  \quad \left.\mathcal{J}\right|_{\bhs{\ff}}= \bbN_0,  \quad \left.\mathcal{J}\right|_{\mf}= \bbN_0, \quad 
     \inf \left.\mathcal{J}\right|_{\lf}>0, \quad \inf \left.\mathcal{J}\right|_{\rf}>h+1.$$
Using the fact that $[D_{\sfc},\rho]\in \rho\Psi^{0}_{\ephi}(X_s;E),$ we then have  that
$$
        (D_{\sfc}-\lambda)\rho Q_3^{\perp}(\lambda)
$$
has top term at $\phibf$ of order $h+1$ given by $-u^{\perp}(\lambda).$  This means that we can take 
$$
Q_3(\lambda)= Q_2(\lambda)-Q_3^{\circ}(\lambda)-\rho Q_3^{\perp}(\lambda) \in\Psi^{-\infty, \mathcal{Q}_3(\lambda)}_{\ephi}(X_s;E)
$$
with holomorphic index family $\mathcal{Q}_3(\lambda)$ as in the statement of the proposition.  By construction, we have that  
$$
     (D_{\sfc}-\lambda)Q_3(\lambda)=\Id - R_3(\lambda)
$$
with remainder $R_3(\lambda)\in\Psi_{\ephi}^{-\infty, \mathcal{R}_3(\lambda) }(X_s;E)$ as in the statement of the proposition.
\end{proof}

\subsection*{Step 4: Removing the error at $\bhs{\mf}$} Here we improve the error at $\bhs{\mf}$ and also deduce consequences for the operator $\left. D_{\sfc}\right|_{\bhs{\mf}}$ which recover some of the results of \cite{v}. The first step is to improve the behavior of the restriction of the error at $\bhs{\mf}.$

\begin{proposition}\label{improvenormalop}
Let $\cU \subseteq \bbC$ be as above, and $\cU'\subseteq \bbC$ an open set containing the origin such that $\bar{\cU'}\subseteq \cU.$
There exist index families $\mathcal Q_4(\lambda)$ and $\mathcal R_4(\lambda),$ and families of operators
$Q_4(\lambda)\in \Psi^{-1,{\mathcal{Q}}_4}_{\ephi}(X_s;E)$ and $R_4(\lambda)\in \Psi^{-1,{\mathcal{R}}_4}_{\ephi}(X_s;E),$ defined and holomorphic in $\cU',$
satisfying $(D_{\sfc}-\lambda)Q_4=Id-R_4$ and all of the properties in the statement of Proposition~\ref{rc.17b}. However, in addition, $N_{\mf}(R_4(\lambda))$ decays to infinite order at $\bhs{\mf}\cap\bhs{\lf},$ $\bhs{\mf}\cap\bhs{\ff},$ and $\bhs{\mf}\cap\phibf.$ 
\end{proposition}

In particular, note that $N_{\mf}(R_4(\lambda))$ is actually in the $b$-calculus since it decays to infinite order at $\bhs{\mf}\cap\bhs{\ff}$.  Moreover, it is very residual in the sense of \cite[p.20]{maz91}.
\begin{proof}Consider $N_{\mf}(R_3(\lambda))$ for $\lambda \in \cU.$ 
From Proposition~\ref{rc.17b}, the leading order term of $R_3(\lambda)$ at $\phibf$ is $\Pi_b$ which decays to positive order at $\phibf\cap\bhs{\mf},$ so the term of order $h+1$ in the expansion of $N_{\mf}(R_3(\lambda))$ at $\bhs{\fbf}$ is zero. Thus it is an element of the $\phi$-surgery calculus on the face $\bhs{\mf},$ with positive leading order at $\bhs{\lf}$ and $\bhs{\ff}$ and leading order greater than $h+1$ at $\phibf$ and $\bhs{\rf}.$ 

We first claim that there exists a holomorphic family of kernels $B_1(\lambda)$ on $\bhs{\mf},$ supported in a neighborhood of $\bhs{\lf}$ and with positive order at $\bhs{\lf}$ and order $>h+1$ at $\phibf,$  such that $N_{\mf}(D_{\sfc}-\lambda)B_1(\lambda)- R_3(\lambda)$ decays to infinite order at $\bhs{\lf}.$ Indeed, this follows precisely as in the proof of \cite[Lemma 5.44]{MelroseAPS} and is obtained by solving for $B(\lambda)$ in Taylor series at $\bhs{\lf}.$ 
Setting $B_2(\lambda)=N_{\mf}(Q_3(\lambda))+B_1(\lambda),$ we have
\[N_{\mf}(D_{\sfc}-\lambda)B_2(\lambda)=\Id_{\mf}-S_1(\lambda),\]
and $S_1(\lambda)$ decays to infinite order at $\bhs{\lf},$ positive order at $\bhs{\ff},$ and order $>h+1$ at $\phibf$ and $\bhs{\rf}.$  
By restricting to $\cU'$ we can find $\gamma>0$ for which 
the leading orders of $S_1(\lambda)$ are greater than $\gamma$ at $\bhs{\lf}$ and $\bhs{\ff}$ and greater than $h+1+\gamma$ at $\phibf$ and $\bhs{\rf}.$ 
Note also that $B_2(\lambda)$ has the same leading terms at $\bhs{\ff}$ and $\phibf$ (order 1 at $\bhs{\ff}$ and $h+1$ at $\phibf$) as $N_{\mf}(Q_3(\lambda)).$

Next we use a Neumann series argument to remove all terms of the error at $\phibf$ and $\bhs{\ff}.$ Choose $S_2(\lambda)$ such that 
$$
\Id_{\mf}+S_2(\lambda)\sim\Id_{\mf}+\sum_{i=1}^{\infty}(S_1(\lambda))^i.
$$ This is possible, since from the composition rules of Theorem~\ref{composition} restricted to $\bhs{\mf}$, the orders of $(S_1(\lambda))^i$ iterate away (that is, go to infinity with $i$) except at $\bhs{\rf},$ where they stabilize, so that the series may be asymptotically summed by Borel's lemma. Now let $B_3(\lambda)=B_2(\lambda)(\Id_{\mf}+S_2(\lambda))$; then we have that 
\begin{equation}\label{betterparam}N_{\mf}(D_{\sfc}-\lambda)B_3(\lambda)=\Id_{\mf}-S_{\infty}(\lambda),\end{equation}
where $S_{\infty}(\lambda)$ decays to infinite order at $\bhs{\lf},$ $\phibf,$ and $\bhs{\ff},$ and order $>h+1$ at $\bhs{\rf}.$ In particular, $S_{\infty}(\lambda)$ is maximally residual  in the $b$-calculus. Moreover, by the composition rules, $B_3(\lambda)$ again has the same leading terms as $N_{\mf}(Q_3(\lambda))$ at $\bhs{\ff}$ and $\phibf.$ Therefore, everything is consistent: we may take $Q_4(\lambda)$ to be an operator which has the same leading orders as $Q_3(\lambda)$ at $\bhs{\ff}$ and $\phibf,$ positive orders at $\bhs{\lf}$ and $\bhs{\rf},$ and leading order $B_3(\lambda)$ at $\bhs{\mf}.$ Note that $\Pi_b$ decays at $\bhs{\mf}\cap\phibf,$ so that there is no consistency problem there. Thus $\cU' \ni \lambda \mapsto Q_4(\lambda)$ satisfies the conditions of Proposition \ref{improvenormalop}, which completes the proof.
\end{proof}

We now use the previous result to analyze the operator $\left. D_{\sfc}\right|_{\mf},$ recovering some of the results of \cite{v}. Recall that $\bhs{\ms} = [M;H]$ is a manifold with fibered boundary.
\begin{corollary}
The operator $N_{\mf}(D_{\sfc})$ is self-adjoint and Fredholm on its natural domain as an unbounded operator on $L^2_b(\sm;E).$ Equivalently, the same is true for $N_{\mf}(\eth_{\ed})$ as an unbounded operator on $L^2_{g_{\sfc}}(\sm; E).$
Furthermore, elements of the kernel of $N_{\mf}(D_{\sfc})$ are polyhomogeneous on $\sm$ with indicial set $\mathcal{W}$ such that $\inf{\mathcal{W}}>0.$  In particular, the projection $\Pi_{\ker N_{\mf}(D_{\sfc})}$ is an element of $\Psi^{-\infty,\cR}_{b}(\sm;E)$ for some index family $\cR$ with $\inf\cR>0.$  
\label{rc.20}\end{corollary}
\begin{proof}
Taking $\lambda=0$ in the previous result and applying $N_{\mf},$ we get
\begin{equation}
      N_{\mf}(D_{\sfc}) N_{\mf}(Q_4(0))= \Id- N_{\mf}(R_4(0)).
\label{rc.20a}\end{equation}
From the construction above, $N_{\mf}(R_4(0))$ is a  $b$-operator of order $-\infty$ which has positive order decay at $\bhs{\rf}$ and rapid decay at the other faces. Therefore, by the usual properties of $b$-operators, it is compact when acting on $L^2_b(\sm;E).$ This implies that the self-adjoint extension of $N_{\mf}(D_{\sfc})$ on $L^2_b(\sm;E)$ must be Fredholm. Moreover, if $f$ is in the $L^2_b$-kernel of $N_{\mf}(D_{\sfc}),$ we see from taking adjoints in \eqref{rc.20a}, then applying both sides to $f,$ that
 $$
     f= (N_{\mf}(R_4(0)))^*f.
 $$  
 Since $(N_{\mf}(R_4(0)))^*$ is a $b$-operator which decays to infinite order at $\bhs{\ff}$ and $\bhs{\rf}$ and has positive order at $\bhs{\lf},$ it is an immediate consequence of the theory of $b$-pseudodifferential operators that $(N_{\mf}(R_4(0)))^*f$ is polyhomogeneous for any $f\in L^2_b(\sm;E).$ Hence $f$ itself is polyhomogeneous. It has positive order because $(N_{\mf}(R_4(0)))^*$ has positive order at $\bhs{\lf}.$ This completes the proof. 
\end{proof}
\begin{remark}
If $\ker D_v=0$, notice from \cite{mame2} that $N_{mf}(\rho D_{\sfc})$ is Fredholm as $\phi$-operator.  In particular, one can find a generalized inverse $Q\in\Psi_{\phi}^{-1}(\bhs{sm};E)$ such that 
$$
               Q xN_{mf}(D_{\sfc})= \Id- \Pi_V
$$
where $V\subset\dot{\mathcal{C}}^{\infty}(\bhs{sm};E)$ is the kernel of $|x|N_{mf}(D_{\sfc})$.  Consequently, $N_{mf}(D_{\sfc})$ has a compact generalized inverse given by $Q|x|\in |x|\Psi_{\phi}^{-1}(\bhs{sm};E)$, which implies in particular that its spectrum is discrete, \cf \cite{v,Moroianu}.  
\label{ds.1}\end{remark}

In fact the proof of Corollary~\ref{rc.20} shows that $N_{\mf}( (D_{\sfc}-\lambda))$ is Fredholm for any $\lambda \in \cU'.$
In turn, this Fredholmness result allows us to further improve the error at $\bhs{\mf}$ so that the leading order is just the projection onto the kernel of the normal operator.

\begin{proposition} \label{prop:Q5}
Let $\cU' \subseteq \bbC$ be as above.
There exist index families $\mathcal Q_5(\lambda)$ and $\mathcal R_5(\lambda),$ and families of operators
$Q_5(\lambda)\in \Psi^{-1,{\mathcal{Q}}_5}_{\ephi}(X_s;E)$ and $R_5(\lambda)\in \Psi^{-1,{\mathcal{R}}_5}_{\ephi}(X_s;E),$ defined and holomorphic in $\cU',$
satisfying $(D_{\sfc}-\lambda)Q_5=Id-R_5$ and all of the properties in the statement of Proposition~\ref{improvenormalop}. However, in addition, the term of order $h+1$ in the expansion of $R_5(\lambda)$ at $\phibf$ is $\Pi_b,$ and most importantly, so that
\begin{equation}
	N_{\mf}(R_5(\lambda))= \Pi_{\ker N_{\mf}(D_{\sfc})}.
\label{rcmm.1b}\end{equation}
\end{proposition}

\begin{proof} 
Since $N_{\mf}(D_{\sfc}-\lambda)$ is self-adjoint and Fredholm from $L^2_b$ to $L^2_b$ for $\lambda \in \cU',$ we can find a holomorphic family of bounded operators on $L^2_b,$ $V_{\delta}' \ni\lambda \mapsto G(\lambda),$ satisfying
\[N_{\mf}(D_{\sfc}-\lambda)G(\lambda)=G(\lambda)N_{\mf}(D_{\sfc}-\lambda)=\Id_{\mf}-\Pi_{\ker N_{\mf}(D_{\sfc})}.\]
Indeed, $G(\lambda)$ is the inverse of $N_{\mf}(D_{\sfc}-\lambda)\rest{(\ker N_{\mf}(D_{\sfc}))^{\perp}}$ extended by zero to the rest of $L^2_b.$
Moreover, we have constructed a parametrix in Proposition~\ref{improvenormalop}, and  taking the adjoint gives 
\begin{equation}\label{betterparamadj}(N_{\mf}(Q_4(\overline{\lambda})))^*N_{\mf}(D_{\sfc}-\lambda)=\Id_{\mf}-(N_{\mf}(R_4(\overline{\lambda})))^*.\end{equation}
We then use the usual trick, as in \cite[Section 4]{maz91}, to show that $G(\lambda)$ is an element of our calculus. In particular, we can evaluate $(N_{\mf}(Q_4(\overline{\lambda})))^*N_{\mf}(D_{\sfc}-\lambda)G(\lambda)$ in two different ways; one yields $G(\lambda)-(N_{\mf}(R_4(\overline{\lambda})))^*G(\lambda)$ and the other yields $(N_{\mf}(Q_4(\overline{\lambda})))^*-(N_{\mf}(Q_4(\overline{\lambda})))^*\Pi_{\ker N_{\mf}(D_{\sfc})}.$ We therefore have
\[G(\lambda)=(N_{\mf}(Q_4(\overline{\lambda})))^*+(N_{\mf}(R_4(\overline{\lambda})))^*G(\lambda)-(N_{\mf}(Q_4(\overline{\lambda})))^*\Pi_{\ker N_{\mf}(D_{\sfc})}.\]
The same trick applied to $G(\lambda)N_{\mf}(D_{\sfc}-\lambda)N_{\mf}(Q_4(\lambda))$ shows that
\[G(\lambda)=N_{\mf}(Q_4(\lambda))+G(\lambda)N_{\mf}(R_4(\lambda))-\Pi_{\ker N_{\mf}(D_{\sfc})}N_{\mf}(Q_4(\lambda)).\]
Plugging the first expression into the second expression shows that
\[G(\lambda)=N_{\mf}(Q_4(\lambda))+(N_{\mf}(Q_4(\overline{\lambda})))^*N_{\mf}(R_4(\lambda))+(N_{\mf}(R_4(\overline{\lambda})))^*G(\lambda)N_{\mf}(R_4(\lambda))\]
\[-(N_{\mf}(Q_4(\overline{\lambda})))^*\Pi_{\ker N_{\mf}(D_{\sfc})}N_{\mf}(R_4(\lambda))-\Pi_{\ker N_{\mf}(D_{\sfc})}N_{\mf}(Q_4(\lambda)).\]

We now examine the right-hand side. For the third term on the right-hand side, note that $N_{\mf}(R_4(\lambda))$ is a very residual element of the $b$-calculus in the sense of \cite[p.20]{maz91} and $G(\lambda)$ is a bounded operator from $L^2_b$ to $L^2_b$; therefore, as explained in \cite[Section 4]{maz91}, $N_{\mf}(R_4(\overline{\lambda}))^*G(\lambda)N_{\mf}(R_4(\lambda))$ is a very residual element of the $b$-calculus. Moreover, as an element of the $\phi$-calculus, it has positive order at $\bhs{\lf}$ and order $>h+1$ at $\phibf$, $\bhs{\ff}$ and $\bhs{\rf}.$ Each other term may be examined directly. We see that $G(\lambda)-N_{\mf}(Q_4(\lambda))$ is an element of the $\phi$-calculus with positive order at $\bhs{\lf}$  and order $>h+1$ at $\phibf$, $\bhs{\ff}$ and $\bhs{\rf}.$ This shows that $G(\lambda)$ is in the $\phi$-calculus with positive order at $\bhs{\lf},$ order one at $\bhs{\ff},$ order $h+1$ at $\phibf,$ and order $>h+1$ at $\bhs{\rf},$ and that its leading orders at $\phibf$ and $\bhs{\ff}$ are the same as those of $N_{\mf}(Q_4(\lambda)).$

Finally, let $Q_5(\lambda)$ be an operator in our calculus which satisfies the requirements of Proposition~\ref{improvenormalop}, agrees with $Q_4(\lambda)$ to leading order at $\bhs{\ff}$ and $\phibf,$ and has normal operator at $\bhs{\mf}$ equal to $G(\lambda).$ Because the leading orders of $G(\lambda)$ at $\bhs{\ff}$ and $\phibf$ are consistent with those of $Q_4(\lambda)$ (and the leading orders are all that matter for Proposition~\ref{improvenormalop}), such an operator exists. This completes the proof.
\end{proof}

\begin{remark}

Notice that for any open set $W\subseteq \bbC$ with closure disjoint from the spectrum of $D_b$ and the spectrum of $N_{\mf}(D_{\sfc}),$ these constructions allow us to construct a holomorphic family of operators $W \ni \lambda \mapsto Q_5(\lambda) \in \Psi^{-1,{\mathcal{Q}}_5}_{\ephi}(X_s;E)$ with $R_5 = \Id - (D_{\sfc}-\lambda)Q_5(\lambda)$ such that
$N_{\mf}(R_5(\lambda))= 0$, since $N_{\mf}(D_{\sfc}-\lambda)$ is invertible in this case. Similarly, since $D_b-\lambda$ is invertible, we can additionally choose $Q_5(\lambda)$ so that $R_5(\lambda)$ has no term of order $h+1$ at $\phibf$.
\label{rcim.1}\end{remark}

\subsection*{Step 5: Solution up to finite rank error}

With slightly more work, we can also ensure that the remainder term is also uniformly of finite rank.  We will need the following lemma.

\begin{lemma}
Let $V\subset \bbC$ be open and let $W$ be a bounded open set with $\overline{W}\subset V.$  Let $\lambda\mapsto R(\lambda)\in \Psi^{-\infty,\cR}_{\ephi}(X_s;E)$ be a smooth family of operators which is holomorphic as a family of operators acting on $L^2_b(X_s;E),$ where $\cR(\lambda)$ is a holomorphic family of index sets such that  
$$
     \inf \cR(\lambda)>0, \quad \inf \left.  \cR(\lambda)\right|_{\rf}>h+1, \quad  \inf \left.  \cR(\lambda)\right|_{\fbf}>h+1 \quad \forall\; \lambda\in V.
$$
Then there is an $\epsilon_0>0$ and a holomorphic family of index sets $\cS(\lambda)$ with $\cR(\lambda)\subset \cS(\lambda)$ and 
$$
     \inf \cS(\lambda)>0, \quad \inf \left.  \cS(\lambda)\right|_{\rf}>h+1, \quad  \inf \left.  \cS(\lambda)\right|_{\fbf}>h+1 \quad \forall\; \lambda\in W,
$$
such that $\Id-R(\lambda)$ is invertible for $\epsilon\le \epsilon_0$ and $\lambda\in W,$ with inverse of the form $\Id-S(\lambda),$  where $S(\lambda)\in \Psi^{-\infty,\cS(\lambda)}_{\ephi}(X_s;E)$ is a smooth family which is holomorphic in $\lambda\in W$ as a family of bounded operators acting on $L^2_b(X_s;E).$  Furthermore, if $\inf \left.  \cR(\lambda)\right|_{ff}>h+1$ for all $\lambda \in V$, then $\inf \left.  \cS(\lambda)\right|_{ff}>h+1$ for all $\lambda\in W$.
\label{rcmm.2}\end{lemma}
\begin{proof}
Let $W'$ be a bounded open set containing $W$ and contained in $V.$  By compactness, we can find $\delta>0$ such that
$$
     \inf \cR(\lambda)>\delta, \quad \inf \left.  \cR(\lambda)\right|_{\rf}>h+1+\delta, \quad  \inf \left.  \cR(\lambda)\right|_{\fbf}>h+1+\delta \quad \forall\; \lambda\in W'.
$$
Thus, by Theorem~\ref{composition}, we know that for $j\in\bbN,$  
$$
     R(\lambda)^{j}\in \Psi^{-\infty,\cR_j(\lambda)}_{\ephi}(X_s;E)
$$
with $\cR_j(\lambda)$ a holomorphic family of index sets with 
$$
     \inf \cR_j(\lambda)>j\delta, \quad \inf \left.  \cR_j(\lambda)\right|_{\rf}>h+1+j\delta, \quad  \inf \left.  \cR_j(\lambda)\right|_{\fbf}>h+1+j\delta \quad \forall\; \lambda\in W'.
$$
Setting $\cS=\bigcup_{j=1}^{\infty} \cR_j,$ this means that we can find a family $\widetilde{S}({\lambda})\in \Psi^{-\infty,\cS(\lambda)}_{\ephi}(X_s;E)$ holomorphic in $\lambda$ as a family of bounded operators acting on $L^2_b(X_s;E)$ such that 
$$
   \widetilde{S}(\lambda)\sim  \sum_{j=1}^{\infty}R(\lambda)^j \quad  \mbox{for}   \; \; \lambda\in W'.  
$$
Consequently, for $\lambda \in W',$ we have that
$$
     (\Id-R(\lambda))(\Id+\widetilde{S}(\lambda))= \Id +T(\lambda)
$$
with $T(\lambda)\in \dot{\Psi}^{-\infty}(X_s;E),$ the space of smooth Schwartz kernels vanishing to infinite order at all boundary hypersurfaces of $X_s.$ Again by compactness, we can find $\epsilon_0>0$ such that for $\epsilon\in [0,\epsilon_0]$ and $\lambda\in W,$
$T(\lambda)$ has small norm (say less than $1/2$) as an operator acting on $L^2_b(X_s;E).$ In this case, $\Id+T(\lambda)$ is invertible with inverse of the form $\Id+T_1(\lambda),$ and we must have $T_1(\lambda)\in \dot{\Psi}^{-\infty}(X_s;E)$.  Consequently,
$$
    (\Id-R(\lambda))^{-1}= (\Id+\widetilde{S}(\lambda))(\Id+ T_1(\lambda))= \Id-S(\lambda)
$$
with 
$$
    S(\lambda)= -\widetilde{S}(\lambda) -T_1(\lambda)- \widetilde{S}(\lambda)T_1(\lambda)
$$
of the desired form.  
\end{proof}

Now let $\phi_1^0,\ldots, \phi^0_{N_b}$ be an orthonormal basis for the $L^2$ kernel of $D_b,$ so that 
$$
     \Pi_b= \sum_{j=1}^{N_b}  \phi^0_{j}\cdot \overline{\phi^0_j} \nu_b,
$$
where $\nu_b$ is the $b$-density used to define the $L^2$ inner product on $L^2_b(\bhs{\bs}).$  We know from \cite{MelroseAPS} and the fact that $D_b$ is Fredholm that each $\phi_j^0$ is polyhomogeneous on $\bhs{\bs}$ with positive index set.  This means that  the $\phi^0_j$'s can be extended smoothly from $\bhs{\bs}$ to $X_s$ to give polyhomogeneous sections $\phi_1,\ldots,\phi_{N_b}$ of $E$ with positive index set on $\sm$ and index set $\bbN_0$ on $\bhs{\bs}.$   In this way, setting
$$
     \Pi_{b,s}= \sum_{j=1}^{N_b}  \phi_{j}\cdot \overline{\phi_j} \frac{dg_{\sfc}}{\rho^{v}},
$$
we obtain a smooth extension $\Pi_{b,s}\in \Psi^{-\infty,\cP_b}_{\eps, b}(X_s;E)$ of $\Pi_b$ to $X_s$ with index family $\cP_b$ such that $ \left.\cP_b\right|_{\bhs{\bff}}=\bbN_0,$ with $\cP_b$ having strictly positive index sets at the other faces.  By construction, $\Pi_{b,s}$ is uniformly of finite rank and  with range having dimension bounded by $N_b$ for all $\epsilon.$

Similarly, if $\varpi_1^0,\ldots,\varpi^0_{N_m}$ is an orthonormal basis of the $L^2_b$-kernel of $N_{\mf}(D_{\sfc}),$  then 
$$
      \Pi_{\ker N_{\mf}(D_{\sfc})}= \sum_{j=1}^{N_m} \varpi^0_j\cdot \overline{\varpi^0_j} \left. \frac{d g_{\sfc}}{\rho^{v}}\right|_{\sm}.
$$ 
Since the $\varpi^0_j$'s are polyhomogeneous with positive index set, we can extend them smoothly to obtain polyhomogeneous sections $\varpi_1,\ldots, \varpi_{N_m}$ $X_s$ so that 
$$
\Pi_{m,s}= \sum_{j=1}^{N_m}  \varpi_j\cdot \overline{\varpi_j}  \frac{d g_{\sfc}}{\rho^{v}}  \in \Psi^{-\infty,\cP_m}_{\eps, b}(X_s;E)
$$
is a smooth extension of $\Pi_{\ker N_{\mf}(D_{\sfc})}$ with $\left. \cP_m\right|_{\mf}=\bbN_0$ and $\cP_m$ having positive index sets elsewhere. By construction,  $\Pi_{m,s}$ is uniformly of finite rank and $N_{\mf}(\Pi_{m,s})= \Pi_{\ker N_{\mf}(D_{\sfc})}.$

Now set $\Pi'= \Pi_{b,s}+\Pi_{m,s}.$ Then, with $Q_5(\lambda)$ from Proposition \ref{prop:Q5} we can write
\begin{equation}
     (D_{\sfc}-\lambda)Q_5(\lambda)= \Id-\Pi' -S_3(\lambda)
\label{rcmm.4}\end{equation}
with $S_3(\lambda)\in \Psi^{-\infty,\mathcal S_3}_{\ephi}(X_s;E)$; here $\mathcal S_3$ is a holomorphic family of index families with 
$$
  \inf \mathcal S_3>0, \quad \inf \left.  \mathcal S_3\right|_{\rf}>h+1, \quad \inf \left.  \mathcal S_3\right|_{\fbf}> h+1.
$$
By Lemma~\ref{rcmm.2}, we know that  we can find $\epsilon_0>0$ such that $\Id-S_3(\lambda)$ is invertible for $\lambda\in V_{\delta}'$ and $\epsilon\le \epsilon_0$ with inverse of the form $\Id-S_4(\lambda)$ where $S_4(\lambda)\in \Psi^{-\infty,\cS_4}_{\ephi}(X_s;E)$ with $\cS_4$ a holomorphic family of index families such that 
$$
     \inf \cS_4(\lambda)>0, \quad \inf \left.  \cS_4(\lambda)\right|_{\rf}>h+1, \quad  \inf \left.  \cS_4(\lambda)\right|_{\fbf}>h+1 \quad \forall\; \lambda\in \cU.
$$
Therefore, we see from \eqref{rcmm.4} that if we set 
$$
     Q_6(\lambda)= Q_5(\lambda)(\Id-S_4(\lambda)) - \frac{\Pi'}{\lambda},
$$
we have that 
\begin{equation}
      (D_{\sfc}-\lambda)Q_6(\lambda)= \Id-R_6(\lambda)
\label{rcmm.5}\end{equation}
with $R_6(\lambda)=-\Pi'S_4(\lambda)+\frac{D_{\sfc}\Pi'}{\lambda}.$  

By construction and the composition formula, $Q_6(\lambda)\in \Psi_{\epsilon,\phi}^{-1,\cQ_6(\lambda)}(X_s;E)$ is now a meromorphic family with only a simple pole at $\lambda=0,$ where $\cQ_6(\lambda)$ is a holomorphic family of index families such that $\inf \cQ_6(\lambda)\ge 0$ and 
$$
\inf \left. \cQ_6(\lambda)\right|_{\lf}>0, \quad \inf \left. \cQ_6(\lambda)\right|_{\rf}>h+1, \quad \inf \left. \cQ_6(\lambda)\right|_{\fbf}\ge h+1, \quad \inf \left. \cQ_6(\lambda)\right|_{\ff}\ge 1.$$ 
Similarly, $R_6(\lambda)\in \Psi^{-\infty,\cR_6(\lambda)}_{\ephi}(X_s;E)$ is now a meromorphic family with only possibly a simple pole at $\lambda=0,$ where $\cR_6(\lambda)$ is a holomorphic family of index families such that $\inf \cR_6(\lambda)>0$ with
$$
\inf \left. \cR_6(\lambda)\right|_{\rf}>h+1, \quad \inf \left. \cR_6(\lambda)\right|_{\fbf}>h+1,  \quad \inf \left. \cR_6(\lambda)\right|_{ff}>h+1.
$$ 
Notice also that $R_6(\lambda)$ is uniformly of finite rank, and the dimension of its range is bounded by $2N_b+2 N_m.$ 

The choice of extensions $\varpi_i$ and $\phi_i$ above was arbitrary, but we will need to be a little more specific for the purposes of Step 6.

\begin{proposition} We can choose the extensions $\varpi_i$ and $\phi_i$ so that for each $i,$ $D_{\sfc}\varpi_i$ and $D_{\sfc}\phi_i$ are polyhomogeneous on $X_s,$ and so that, for some $\alpha_0>0,$ each of these extensions has leading order $\alpha_0$ (with nontrivial coefficient) at $\bhs{\ms}$ and greater than $\alpha_0$ at $\bhs{\bs}.$
Moreover, we may assume that after multiplying by $\epsilon^{-\alpha_0},$ the $D_{\sfc}\varpi_i$ and $D_{\sfc}\phi_i$ are linearly independent for $\epsilon\leq\epsilon_0.$
\end{proposition}

\begin{proof}First we just want to choose the extensions  $\varpi_1,\ldots,\varpi_{N_m}$ so that $D_{\sfc}\varpi_1,\ldots,D_{\sfc}\varpi_{N_m}$ will be polyhomogeneous on $X_s$ with some positive order of decay at the boundaries. Fortunately, this can be done. Note first that, using the explicit form of $D_{\sfc}$ near the boundary, all terms within one order of the leading term in the expansion of an element of $\ker N_{\mf}(D_{\sfc})$ at the boundary $\sm\cap\bhs{\bs}$ must be sections of $\ker D_v$; since the leading order is positive, this covers all terms of order $\leq 1.$ Therefore, we may choose the extensions $\varpi_i$ so that each term of order less than or equal to one in the expansion at $\bhs{\bs}$ is a section of $\ker D_v$ on $\bhs{\bs}.$ Thus $D_{\sfc}\varpi_i$ is polyhomogeneous (automatic) and decays to positive order at $\bhs{\bs}$ (by the extra condition); in fact it has positive order at both $\sm$ and $\bhs{\bs}.$ 

To make sure the same is true for $D_{\sfc}\phi_i$ for each $i,$ we just make sure that the term of order 1 at $\bhs{\bs}$ in the expansion of $\phi_i$ is also a section of $\ker D_v$; then the result is immediate.

Next, let $\alpha_1$ be the infimum of the set of exponents $\alpha$ for which $\epsilon^{-\alpha}D_{\sfc}\varpi_i$ and $\epsilon^{-\alpha}D_{\sfc}\phi_i$ are bounded for all $i$ (i.e. the minimum order of growth of such an element); it might be infinite, but in any case is positive. Then let $\alpha_0=\min\{1/2,\alpha_1/2\}.$ It is now easy to see that there exist $N_m+N_b$ (in fact, as many as we want) compactly supported sections $\eta_i$ on $\bhs{\ms}\setminus \pa \bhs{\ms}$ with mutually disjoint support and which are not in the kernel of $N_{\mf}(D_{\ed})$ Each such $\eta_i$ may be extended to a smooth function on $X_s,$ also called $\eta_i$ which is zero outside of a small neighborhood of the original support. Then, to each $\varpi_i$ and $\phi_i,$ we add $\epsilon^{\alpha_0}\eta_i,$ using a different $\eta_i$ for each different element. The results satisfy all the same properties as before, and moreover $D_{\sfc}\phi_i$ and $D_{\sfc}\varpi_i$ have leading order precisely $\alpha_0$ at $\bhs{\ms}$ and greater than $\alpha_0$ at $\bhs{\bs}$ In addition, all of their leading orders at $\bhs{\ms},$ after multiplying by $\epsilon^{-\alpha_0},$ are linearly independent and in fact mutually orthogonal.   Hence the functions themselves are linearly independent for suitably small $\epsilon.$
\end{proof}

\subsection*{Step 6: Analytic Fredholm theory and the resolvent}

We are now ready to apply analytic Fredholm theory.  First, if $\bbB_{\delta}(0)\subset\bbC$ is a small ball of radius $\delta>0$ centered at $0$ such that $\overline{\bbB_\delta(0)}\subset \cU',$ then we know by Lemma~\ref{rcmm.2} applied to $\Id-R_6(\lambda)$ that taking $\epsilon_0$ smaller if needed, we can assume that $\Id-R_6(\lambda)$ is invertible for $\lambda\in \cU'\setminus \bbB_{\delta}(0)$ and $\epsilon\in [0,\epsilon_0].$  By analytic Fredholm theory, since $R_6(\lambda)$ is uniformly of finite rank, we know that for a fixed $\epsilon\in [0,\epsilon_0],$ $\Id-R_6(\lambda)$ is invertible except at a finite number of points corresponding to the zeroes of a holomorphic function.  In fact, for $\epsilon$ fixed, the inverse of  $\Id-R_6(\lambda)$ will be meromorphic with poles of finite rank.  

The goal is now to obtain a more precise description of the inverse. First observe that \begin{equation}
      \phi_1, \epsilon^{-\alpha_0}D_{\sfc}\phi_1,\ldots,  \phi_{N_b}, \epsilon^{-\alpha_0}D_{\sfc}\phi_{N_b}, \varpi_1, \epsilon^{-\alpha_0}D_{\sfc}\varpi_1,\ldots,  \varpi_{N_m},\epsilon^{-\alpha_0}D_{\sfc}\varpi_{N_m}
\label{rcmm.6}\end{equation}
are linearly independent sections for $\epsilon\le\epsilon_0$; the reason is that they are linearly independent as $\epsilon$ goes to zero, since the restrictions of the sections with a factor of $\epsilon^{-\alpha_0}$ to $\bhs{\ms}$ have mutually disjoint support and are orthogonal to $\ker N_{\mf}(D_{\ed}).$

Now let $\Pi_1\in \Psi^{-\infty, \cP_1}_{\eps, b}(X_s;E)$ be the projection on the span of \eqref{rcmm.6}.   Note that the index family $\cP_1$ is such that $\inf \cP_1\ge 0$ and $\inf\left. \cP_1\right|_{\lf}>0,$ $\inf\left. \cP_1\right|_{\rf}>0.$  Using the decomposition 
$$
    L^2(X_s;E)= \ran(\Id-\Pi_1)+ \ran(\Pi_1),
$$
we have that 
$$
       \Id-R_6(\lambda)= \left( \begin{array}{cc}
              \Id & 0 \\
              C(\lambda) & D(\lambda) 
              \end{array}
       \right)
$$
with $C(\lambda)= \Pi_1(\Id-R_6(\lambda))(\Id-\Pi_1),$ $D(\lambda)= \Pi_1(\Id-R_6(\lambda))\Pi_1.$  By the definition of $R_6(\lambda)$ and $\Pi_1,$ $C(\lambda)$ is holomorphic while $D(\lambda)$ is meromorphic with possibly only a simple pole at $\lambda=0.$  Since the Fredholm determinant of $\Id-R_6(\lambda)$ is clearly equal to the determinant of $D(\lambda),$ we see that $\Id-R_6(\lambda)$ is invertible if and only if $D(\lambda)$ is.  Furthermore, when this is the case, the inverse of $\Id-R_6(\lambda)$ is given by 
\begin{equation}
     (\Id-R_6(\lambda))^{-1}= \left( \begin{array}{cc}
              \Id & 0 \\
              -D(\lambda)^{-1}C(\lambda) & D(\lambda)^{-1} 
              \end{array}
  \right).
\label{rcmm.6a}\end{equation}
\begin{lemma}
There exists a bounded function $f(\epsilon,\lambda),$ polyhomogeneous in $\epsilon\in [0,\epsilon_0]$ and holomorphic in $\lambda\in \cU',$ such that 
$$
         f(\epsilon,\lambda)D(\lambda)^{-1}\in \Psi^{-\infty,\cB(\lambda)}_{\eps, b}(X_s;E)
$$
is a holomorphic family when acting on $L^2_b(X_s;E),$ where $\cB(\lambda)$ is a holomorphic family of index sets with $\inf \cB(\lambda)\ge 0$ and $\inf \left. \cB(\lambda)\right|_{\lf}> 0,$ $\inf \left. \cB(\lambda)\right|_{\rf}> 0.$ 
\label{rcmm.7}\end{lemma}
\begin{proof}
Note first that $D(\lambda)=\Pi_1+\Pi'S_4(\lambda)\Pi_1-\frac{D_{\ed}\Pi'}{\lambda}.$ Clearly  each element of the matrix representing $D(\lambda)$ is a sum of three terms, each of which is polyhomogeneous and bounded in $\epsilon$ with possibly a pole at $\lambda=0$.  More precisely, using the decomposition
\begin{multline*}
	\ran \Pi_1= \\
	\Span\langle \phi_1, \ldots,\phi_{N_b},\varpi_1,\ldots,\varpi_{N_m}\rangle 
	\oplus \Span\langle \frac{D_{\sfc}\phi_1}{\epsilon^{\alpha_0}}, \ldots, \frac{D_{\sfc}\phi_{N_b}}{\epsilon^{\alpha_0}},\frac{D_{\sfc}\varpi_1}{\epsilon^{\alpha_0}},
		\ldots,\frac{D_{\sfc}\varpi_{N_m}}{\epsilon^{\alpha_0}}\rangle,
\end{multline*}
we have that
$$
     D(\lambda)=  \left( \begin{array}{cc}
              E(\lambda) & F(\lambda) \\
              \frac{G(\lambda)}{\lambda}  & H(\lambda)
              \end{array}
  \right)
$$
with $E,F,G,H$ holomorphic in $\lambda$ with $G=0$ when $\epsilon=0.$  Therefore, this means that 
$$
         \det D(\lambda)= \frac{f(\epsilon,\lambda)}{\lambda^{N_b+N_m}} \; \Longrightarrow \; \det(D(\lambda))^{-1}= \frac{\lambda^{N_b+N_m}}{f(\epsilon,\lambda)},
$$ 
where $f(\epsilon,\lambda)$ is holomorphic in $\lambda.$  The cofactor matrices of $D(\lambda)$ are also meromorphic with at most poles of order $N_b+N_m$ at $\lambda=0.$  This means that $f(\epsilon,\lambda)D(\lambda)^{-1}$ is indeed holomorphic in $\lambda$ as a family of bounded operators on $L^2_b(X_s;E).$   To check that $f(\epsilon,\lambda)= \lambda^{N_b+N_m}\det(D(\lambda))$ is polyhomogeneous in $\epsilon$ and that $f(\epsilon,\lambda)D(\lambda)^{-1}\in \Psi^{-\infty,\cB(\lambda)}_{\eps, b}(X_s;E),$ it suffices to use the formula for $D(\lambda)^{-1}$ in terms of the cofactor matrix.
\end{proof}

From \eqref{rcmm.6a} and Lemma~\ref{rcmm.7}, it follows that there is a function $f(\epsilon, \lambda),$ polyhomogeneous in $\epsilon$ and holomorphic in $\lambda,$ such that 
$$
      (\Id-R_6(\lambda))^{-1}= \Id-S_6(\lambda)  
$$ 
with $f(\epsilon,\lambda) S_6(\lambda)\in \Psi^{-\infty,\cS_6(\lambda)}_{\ephi}(X_s;E)$ holomorphic as a family of bounded operators on $L^2_b(X_s;E),$ where $\cS_6(\lambda)$ is a holomorphic family of index families  with $\inf \cS_6(\lambda)>0$ and $\inf \left. \cS_6(\lambda)\right|_{\rf}>h+1,$ $\inf \left. \cS_6(\lambda)\right|_{\fbf}>h+1,$  $\inf \left. \cS_6(\lambda)\right|_{ff}>h+1.$ From \eqref{rcmm.6a},  $S_6(\lambda)$ is also uniformly of finite rank.  

This completes the resolvent construction modulo a small detail.

\begin{proposition}
The statement of Theorem~\ref{rc.7} holds with \eqref{merp.1} replaced by the slightly weaker statement that 
$$
  f(\lambda, \epsilon)\Res_{M}(\lambda)\in  \Psi^{-\infty,\cL(\lambda)}_{\epsilon,\phi}(X_s;E)  
$$
for $\lambda\mapsto \cL(\lambda)$ a holomorphic family of families of index sets with $\inf \cL\ge 0$ and 
\begin{equation}
\left. \inf\mathcal{L}\right|_{\lf}>0, \quad\left. \inf\mathcal{L}\right|_{\rf}>h+1, \quad \inf \left. \mathcal{L}\right|_{\fbf} \ge h+1,\quad \inf\left.\mathcal{L}\right|_{\ff}\ge h+1.
\label{es.1b}\end{equation}
\label{es.1}\end{proposition}

\begin{proof}
Composing on the right by $\Id-S_6(\lambda)$ in \eqref{rcmm.5}, we finally obtain that 
$$
    (D_{\sfc}-\lambda)^{-1}= Q_6(\lambda)(\Id-S_6(\lambda))
$$
for $\lambda\in V.$  From the properties of $S_6(\lambda)$ and $Q_6(\lambda),$ the decomposition 
$$
    (D_{\sfc}-\lambda)^{-1}= \Res_{H}(\lambda) + \Res_{M}(\lambda)
$$
follows immediately; we just let 
\begin{equation}
\Res_{M}(\lambda)=-\frac{\Pi'}{\lambda}-Q_6(\lambda)S_6(\lambda)
\label{es.2}\end{equation} and let $\Res_{H}(\lambda)$ be $Q_6(\lambda)+\frac{\Pi'}{\lambda}=Q_5(\lambda)(\Id-S_4(\lambda)).$ Since $\Res_{M}(\lambda)$ equals $\Pi'/\lambda$ plus a correction which is lower-order at the boundary faces of $X_s^2$ (since $\inf\mathcal S_6(\lambda)>0$), it has the desired normal operators.  This completes the construction for $\lambda\in V$ and $\epsilon\le \epsilon_0.$  However, for $\epsilon\ge \epsilon_0,$ we can just use the standard construction of the resolvent of self-adjoint elliptic operators on compact manifolds. Moreover, for $\lambda$ away from the spectra of $N_{\mf}(D_{\ed})$ and $D_b,$ we know by Remark~ \ref{rcim.1} that we can choose $Q_5(\lambda)$ so that its right hand side can be immediately inverted using Lemma~\ref{rcmm.2}, giving an expression of the desired form for the resolvent.    
\end{proof}

The last step of the resolvent construction will be carried at the end of the next section.

\section{Projection onto the eigenspace of small eigenvalues}\label{sec:SmallEigen}

We have shown in the previous section that the operators $N_{\mf}(D_{\ed})$ and $D_b$ have discrete spectra near the origin and that if $\bbB_{\delta}(0)$ is a ball around the origin in the complex plane with the property that the only element of the spectrum of $N_{\mf}(D_{\ed})$ or $D_b$ in $\bbB_{\delta}(0)$ is zero, then there is an $\eps_0>0$ such that 
\begin{equation*}
	D_{\ed}-\lambda \text{ is invertible for all $\lambda \in \bbS_{\delta}(0)$ and $\eps<\eps_0.$}
\end{equation*}
In particular, $\bbB_{\delta}(0)$ should contain all zeroes of $f(\lambda,\epsilon)$ which converge to $0$ as $\epsilon\rightarrow 0,$ for suitably small $\epsilon,$ and no others.

We refer to the elements of $\Spec(D_{\sfc})\cap \bbB_{\delta}(0)$ as the {\bf small eigenvalues}. It follows from the discussion above that all of the small eigenvalues converge to zero as $\eps \to 0.$ In fact the construction of the resolvent gives us more refined information about the small eigenvalues and the corresponding eigensections.\\

We can express the projection $\Pi_{\sma}$ onto the eigenspace of the small eigenvalues in terms of the resolvent
\begin{equation}
    \Pi_{\sma}= \frac{i}{2\pi} \int_{\Gamma} (D_{\sfc}-\lambda)^{-1} d\lambda= \frac{i}{2\pi} \int_{\Gamma} \Res_M(\lambda) d\lambda,  
\label{rc.27}\end{equation}
where $\Res_M(\lambda)$ is given by Proposition~\ref{es.1} and  $\Gamma$ is a sufficiently small contour integral going anti-clockwise around the origin, e.g., $\bbS_{\delta}(0).$  Since the index family of $(D_{\sfc}-\lambda)^{-1}$ depends on $\lambda,$  it is not obvious from \eqref{rc.27} that $\Pi_{\sma}$ is also polyhomogeneous.  We can, however, deduce from \eqref{rc.27}, by integrating the two parts of $\Res_M(\lambda)$ in \eqref{es.2} separately, that 
\begin{equation}
\Pi_{\sma}\in \Psi^{-\infty}_{\eps,b}(X_s;E) +\Psi^{-\infty,\tau}_{\ephi}(X_s;E)
\label{poly.1}\end{equation} 
for some $\tau>0$, where $\Psi^{-\infty,\tau}_{\ephi}(X_s;E)= \cA^{\tau}_{-}(X^2_{s};E\otimes \beta_{(2),R}^*\Omega_{\eps,b}(X_s))$ is the space of conormal functions vanishing at order $\tau$ (with respect to $b$-surgery densities) at each boundary face of $X^2_{s},$ \cf \cite[Equation~(78)]{mame1}.  Thus, written in terms of the pull-back from the right of a $\ephi$-density, this means that, in agreement with \eqref{es.1b},  it must in fact vanish at order $\tau+h+1$ at $\bhs{rf}$, $\bhs{ff}$ and $\bhs{\phi bf}$.
This can be used to obtain the following rough upper bound on the growth of small eigenvalues near $\epsilon=0.$  

\begin{corollary}\label{rc.28}
There exist positive constants $c,$ $\tau$ and $\epsilon_0$ such that if $\lambda_{\epsilon}$ is an eigenvalue of $D_{\sfc}$ with $\lambda_{\epsilon}\to 0$ as $\epsilon\to 0,$ then 
$$
   | \lambda_{\epsilon}| < c \epsilon^{\tau} \quad  \mbox{for}  \quad \epsilon\in (0,\epsilon_0].
$$\end{corollary}
\begin{proof}
The small eigenvalues correspond to the zeros of the polynomial
$$
     P(\lambda,\epsilon)= \det (\Pi_{\sma}(D_{\sfc}-\lambda)\Pi_{\sma}).
$$
If $N$ is for fixed $\epsilon$ the dimension of the range of $\Pi_{\sma}$, then we have immediately that
$$
     P(\lambda,\epsilon)= c_N(\epsilon)\lambda^{N}+ c_{N-1}(\epsilon)\lambda^{N-1}+\cdots + c_1(\epsilon)\lambda +c_0(\epsilon)
$$
with $c_N(0)\ne 0$ and $c_j(0)=0$ for $j<N.$   Taking $\epsilon_0$ sufficiently small, we may thus assume $c_N(\epsilon)\ne 0$ for $0\le \epsilon\le \epsilon_0.$  Dividing by $c_N(\epsilon),$ we thus need to consider the zeros of the polynomial
$$
Q(\lambda,\epsilon)= \lambda^{N}+ q_{N-1}(\epsilon)\lambda^{N-1}+\cdots + q_1(\epsilon)\lambda +q_0(\lambda)
$$
with $q_j(\epsilon)= \frac{c_j(\epsilon)}{c_N(\epsilon)}.$  In particular, 
$$
      q_{N-1}(\epsilon)^2-2q_{N-2}(\epsilon)= \sum_{i=1}^N \lambda_i^2(\epsilon)
$$ 
is the sum of the squares of the small eigenvalues of $D_{\sfc}$ counted with multiplicity.  Since 
\begin{equation}
\begin{aligned}
    (D_{\sfc}-\lambda)\Pi_{\sma} &= \frac{i}{2\pi} \int_{\Gamma} (D_{\sfc}-\zeta+\zeta-\lambda)(D_{\sfc}-\zeta)^{-1}d\zeta = \frac{i}{2\pi} \int_{\Gamma} (\zeta-\lambda)(D_{\sfc}-\zeta)^{-1}d\zeta \\
    &=  \frac{i}{2\pi} \int_{\Gamma} (\zeta-\lambda)\Res_M(\zeta)d\zeta ,
\end{aligned}    
\label{rc.28a}\end{equation}
we see that $(D_{\sfc}-\lambda)\Pi_{\sma}$ and $\Pi_{\sma}(D_{\sfc}-\lambda)\Pi_{\sma}$ are also in $\Psi^{-\infty}_{\eps, b}(X_s;E)+ \Psi^{-\infty,2\tau}_{\ephi}(X_s;E)$ for some small $\tau>0$.
Since $q_{i}(0)=0$ for $i\le N-1,$ this implies that 
$$
q_{N-1}^2(\epsilon)-2q_{N-2}(\epsilon)= \mathcal{O}(\epsilon^{2\tau})
$$  
as $\epsilon\searrow 0$.  Thus, $\sum_{i=1}^N\lambda_i^2=\mathcal{O}(\epsilon^{2\tau}).$    Since the $\lambda_i$ are real, the result follows.    
\end{proof}

With this upper bound, we can now deduce that $\Pi_{\sma}$ is polyhomogeneous.  

\begin{corollary}
There is an index family $\cK$ with $\inf \cK\ge 0$ and 
$$
 \inf \left. \mathcal{K}\right|_{\lf}>0, \quad \inf \left. \mathcal{K}\right|_{\rf}>0,
 $$
 such that $\Pi_{\sma}\in \Psi^{-\infty,\cK}_{\eps, b}(X_s;E).$  
\label{rc.29}\end{corollary}
\begin{proof}
Recall from \eqref{rcmm.4} that $\Pi'$ is a finite rank operator which restricts to $\Pi_{\ker N_{\mf}(D_{\sfc})}$ on $\bhs{\mf}$ and $\Pi_{\ker D_b}$ on $\fbface$; we have $\Pi'\in \Psi^{-\infty,\cK}_{\eps, b}(X_s;E)$ for some index family $\cK$ with 
$\inf \cK\ge 0.$ By Cauchy's formula and the definition of $\Pi_{\sma},$
$$
   \Pi_{\sma}-\Pi' =\frac{i}{2\pi} \int_{\Gamma} \left(\Res_M(\lambda)+\frac{\Pi'}{\lambda}\right) d\lambda.
$$
For $\epsilon>0,$  we can now change the contour $\Gamma$ to $\widetilde{\Gamma}(\epsilon),$ where
\begin{equation}
   \widetilde{\Gamma}(\epsilon)= \{ \lambda\in  \bbC \; | \;  |\lambda|= 2c\epsilon^{\frac{\tau}{k}} \},
\label{rc.29a}\end{equation}
with $c$ and $\tau$ being the positive constants of Corollary~\ref{rc.28} and $k$ a large natural number. We then have
\begin{equation}
 (\Pi_{\sma}- \Pi')_{\epsilon} = \frac{i}{2\pi} \int_{\widetilde{\Gamma}(\epsilon)} (\Res_M(\lambda)_{\epsilon} +\frac{\Pi'_{\epsilon}}{\lambda}) d\lambda
\label{rc.30}\end{equation}
for $\epsilon>0,$ since the small eigenvalues are contained inside $\widetilde{\Gamma}(\epsilon)$ by our choice of contour. Notice however that \eqref{rc.30} makes sense all the way down to $\epsilon=0.$  Indeed, if $\lambda_{\epsilon}$ is a small eigenvalue continuous in $\epsilon$ and $\Pi_\epsilon$ is the projection (continuous in $\epsilon$) onto an eigenspace of dimension 1, then taking $\tau>0$ smaller if needed, we know from Corollary~\ref{rc.28} that  $\lambda_{\epsilon}=\mathcal{O}(\epsilon^{\tau})$ and $\Pi_{\epsilon}-\Pi_0=\mathcal{O}(\epsilon^{\tau}).$  Since for $\lambda\in \widetilde{\Gamma}(\epsilon),$ we have that $\lambda=\mathcal{O}(\epsilon^{\frac{\tau}{k}}),$ this means that
$$
         \frac{\Pi_{\epsilon}}{\lambda-\lambda_{\epsilon}}- \frac{\Pi_0}{\lambda}= \Pi_{\epsilon}\left( \frac{1}{\lambda-\lambda_{\epsilon}} -\frac{1}{\lambda} \right) + \frac{\Pi_{\epsilon}-\Pi_0}{\lambda}= \Pi_{\epsilon} \left( \frac{\lambda_{\epsilon}}{\lambda(\lambda-\lambda_{\epsilon})}  \right) + \frac{\Pi_{\epsilon}-\Pi_0}{\lambda}= \mathcal{O}(\epsilon^{\frac{k-2}{k}\tau}).
$$
Consequently, this means that $\Res_{M}(\lambda)+\frac{\Pi'}{\lambda}$ is uniformly bounded on $\widetilde{\Gamma}(\epsilon)$ as $\epsilon\to 0,$ so we can say that
\begin{equation}
\Pi_{\sma}- \Pi' = \frac{i}{2\pi} \int_{\widetilde{\Gamma}} (\Res_M(\lambda) +\frac{\Pi'}{\lambda}) d\lambda
\label{rc.31}\end{equation}
where $\widetilde{\Gamma}$ is the family of contour integrals \eqref{rc.29a}.  Now, $(\Res_M(\lambda) +\frac{\Pi'}{\lambda})= Q_6(\lambda)S_6(\lambda)$ is in  $\Psi^{-\infty,\cE(\lambda)}_{\ephi}(X_s;E)$  for some holomorphic family $\cE(\lambda)$ of index families with $\inf \cE(\lambda)> 0$ and
$$
\left. \inf\mathcal{E}(\lambda)\right|_{\rf}> h+1, \quad \inf \left. \mathcal{E}(\lambda)\right|_{\fbf} > h+1,\quad \inf\left.\mathcal{E}(\lambda)\right|_{\ff} > h+1.
$$

  On the other hand, for $\lambda\in \widetilde{\Gamma},$ we have
$\lambda= e^{i\theta}2c\epsilon^{\frac{\tau}{k}},$ where $\arg \lambda =\theta.$  Thus, if $\varrho$ is one of the boundary defining functions for a particular face of $X^2_{s},$ then for such a $\lambda\in\widetilde{\Gamma},$
$$
   \varrho^{\lambda}= \exp\left( \lambda \log \varrho   \right) = \exp\left( e^{i\theta}2c\epsilon^{\frac{\tau}{k}} \log \varrho   \right) = \sum_{j=0}^{\infty} \frac{e^{ij\theta}}{j!} (2c\epsilon^{\frac{\tau}{k}} \log \varrho)^j
$$
is clearly polyhomogeneous on $X^2_{s}$ with index set independent of $\theta.$  More generally, if $g(\lambda)$ is holomorphic, then $\varrho^{g(\lambda)}$ is polyhomogeneous on $X^2_{s}$ when $\lambda=e^{i\theta}2c\epsilon^{\frac{\tau}{k}}$ with index set independent of $\theta.$  Expanding in this way each term in the polyhomogeneous expansions of $\Res_M(\lambda)+\frac{\Pi'}{\lambda}$ when $\lambda=e^{i\theta}2c\epsilon^{\frac{\tau}{k}},$ we see that
$\widetilde{\Gamma}$ parametrizes a circle family of operators in $\Psi^{-\infty, \cF}_{\ephi}(X_s; E)$ with index set $\cF$ independent of $\theta,$ with $\inf \cF> 0$ and
$$
\left. \inf\mathcal{F}\right|_{\rf}> h+1, \quad \inf \left. \mathcal{F}\right|_{\fbf} > h+1,\quad \inf\left.\mathcal{F}\right|_{\ff} > h+1,
$$
 by the uniform boundedness.  
We thus deduce from \eqref{rc.31} that $(\Pi_{\sma}-\Pi')\in \Psi^{-\infty,\cF}_{\ephi}(X_s;E).$  Since we already know that $\Pi'$ is polyhomogeneous, this means that $\Pi_{\sma}\in \Psi^{-\infty,\cG}_{\ephi}(X_s;E)$ for some other index family $\cG$ with $\inf \cG\ge 0$ and    
$$
\left. \inf\mathcal{G}\right|_{\lf}> 0, \quad\left. \inf\mathcal{G}\right|_{\rf}> h+1, \quad \inf \left. \mathcal{G}\right|_{\fbf} \ge h+1,\quad \inf\left.\mathcal{G}\right|_{\ff} \ge h+1.
$$
Finally, using Theorem~\ref{mappingprops}, we see that the range of $\Pi_{\sma}$ can be written as the span of bounded polyhomogeneous sections on $X_s$ either vanishing on $\bhs{sm}$ or $\bhs{sb}$.  Since $\Pi_{\sma}$  is a projection, this means that in fact $\Pi_{\sma} \in \Psi^{-\infty,\cK}_{\epsilon,b}(X_s;E)$ for some index family $\cK$  as claimed.    
\end{proof}
\begin{remark}
In light of \cite[Remark~6.9a]{Kato}, it is not clear that small eigenvalues themselves must necessarily be polyhomogeneous in $\epsilon$.  
\label{pse.1}\end{remark}

Note also that since $\Res_{M}(\lambda)$ is $-\frac{\Pi'}{\lambda}$ plus a term which is lower order at the boundaries of $X_{b,s}^2,$ the leading orders of $\Pi_{\sma}$ are $\Pi_{\ker N_{\mf}(D_{\sfc})}$ and $\Pi_{\ker D_b}$ at $\bhs{\mf}$ and $\fbface$ respectively, and zero at all other boundary hypersurfaces of $X_{b,s}^2.$

This result about $\Pi_{\sma}$ allows us to complete the construction of the resolvent of the previous section.

\subsection*{Step 7: Sharper statement concerning the meromorphic part of the resolvent}

Using the resolvent $(D_{\eps,d}-\lambda)^{-1}$ of Proposition~\ref{es.1}, we can use $\Pi_{\sma}$ to redefine the holomorphic and meromorphic parts of the resolvent by
$$
     \Res_M(\lambda)= \Pi_{\sma}(D_{\eps,d}-\lambda)^{-1}\Pi_{\sma}  \quad \mbox{and} \quad \Res_H(\lambda)= (\Id-\Pi_{\sma})(D_{\eps,d}-\lambda)^{-1}(\Id-\Pi_{\sma}). 
$$
Using Corollary~\ref{rc.29}, Proposition~\ref{es.1}, the fact that $\Pi_{\sma}$ is a projection and Theorem~\ref{mappingprops}, we see that \eqref{merp.1} follows with this new definition of $\Res_M(\lambda)$, completing the proof of Theorem~\ref{rc.7}.

\section*{Heat kernel under degeneration}
\section{Surgery heat space} \label{sec:SurgeryHeatSpace}

As we have pointed out above, our study of analytic torsion proceeds by understanding its behavior on a compact manifold as it undergoes degeneration to a space with fibered  cusps. In particular we need to understand the asymptotic behavior of analytic torsion as a function of the degeneration parameter $\eps.$ We will achieve this by obtaining a precise description of the asymptotic behavior of the heat kernel of the Hodge Laplacian directly and then deriving the asymptotics of the torsion. Following Melrose \cite{MelroseAPS, MelroseScat} we will understand these asymptotics geometrically by replacing the space $M^2 \times [0,1]_{\eps} \times \bbR^+_t,$ on which the heat kernel is a priori defined, with another space on which the singularities of the heat kernel are `resolved'. Precisely, we will construct a space $HX_s$ from $M^2\times[0,1]_{\eps}\times \bbR^+$ and then, in section \ref{sec:SolveHeat}, show that the heat kernel lifts to $HX_s$ to a polyhomogeneous density.\\

We begin our construction of the heat space by first identifying a compactification of 
$M^2 \times (0,1)_{\eps} \times (0,\infty)_t$  having nice left and right pushforward maps onto the single surgery space $X_s$, the space where the sections on which the heat kernel acts live.  A natural candidate is to take
\begin{equation}
       X^2_{b,s}\times \bbR^+_t,
\label{hk.1}\end{equation}
since then we have left and right projections $\pr_1\circ \pi_{b,s,L}^2$ and $\pr_1\circ \pi_{b,s,R}^2$ onto $X_s$, where $\pr_1: X^2_{b,s}\times \bbR^+_t$ is the projection on the first factor and 
$$
\pi^2_{b,s,R}: X^2_{b,s}\to X_s, \quad \pi^2_{b,s,L}: X^2_{b,s}\to X_s,
$$ 
are the $b$-fibrations of \cite[(71)]{mame1}.  As composite of $b$-fibrations,  $\pr_1\circ \pi_{b,s,L}^2$ and $\pr_1\circ \pi_{b,s,R}^2$ are automatically $b$-fibrations.

Recall that the heat kernel of the Euclidean Laplacian on $\bbR^n$ is 
\begin{equation*}
	\frac1{(4\pi t)^{n/2}}\exp\lrpar{ - \frac{|\zeta-\zeta'|^2}{4t} }.
\end{equation*}
The simplest analogue on a Riemannian manifold is obtained by replacing the Euclidean distance with the Riemannian distance.  To obtain the right heat space, we need to further blow up $X^2_{b,s}\times \bbR^+_t$ in such a way that
the functions $\sqrt t$ and $e^{-d(\zeta,\zeta')^2/4t}$ extend `smoothly'.
The first is easily achieved by declaring $\tau = \sqrt t$ to be the smooth global function on $\bbR^+.$  Technically, this is a change of the smooth structure on $X^2_{b,s} \times \bbR^+_t$.    Alternatively, one could insist on working with $t$, but then the blow-ups need to be don parabolically, see  \cite[Chapter 7]{MelroseAPS} for this approach. To ensure that  $e^{-d(\zeta,\zeta')^2/4t}$ extend `smoothly', we need to blow up twice, which leads us to consider the following heat space,
\begin{equation}\label{eq:DefHX}
	HX_{s} 
	= [ X^2_{b,s} \times \bbR^+_{\tau}; D_{fib} \times \{  \tau =0 \} ;  \Delta_s ].
\end{equation}
Here, $D_{fib}$ is the lifted fibered diagonal defined in \eqref{fib.1}, so after the first blow-up, the lift of the face $\tau=0$ is naturally identified with the surgery double space $X^2_s$.  The second blow-up corresponds to blowing up the lifted diagonal
$$
\Delta_s\subset X^2_s\subset   [ X^2_{b,s} \times \bbR^+_{\tau}; D_{fib} \times \{  \tau =0 \} ]
$$
seen as p-submanifold of the lift of the face $\{\tau=0\}$.

The space $HX_{s}$ has natural blow-down maps
\begin{equation*}
	\xymatrix @C=50pt
	{ & HX_s \ar[ld]_-{\beta_H'} \ar[d]^-{\beta_H} \ar[rd]^-{\beta_{H,b}} \\
	[X_{b,s}^2\times \bbR^+_\tau ; D_{fib} \times \{ \tau=0\} ] \ar[r]^-{\beta_H''} & X^2_{b,s} \times \bbR^+_{\tau} \ar[r]^-{\beta_b\times\Id_\tau} 
		& M^2 \times [0,1]_{\eps} \times \bbR^+_{\tau} }
\end{equation*}
and seven boundary hypersurfaces,
\begin{equation*}
\begin{gathered}
	\bhs{tb} = \beta_H^{\sharp}( \{ \tau=0 \} )\\
	\bhs{hmf} = \beta_H^{\sharp}(  \bhs{\mf}\times \bbR^+_{\tau}  )  \\
	\bhs{hbf} =  \beta_H^{\sharp}( \bhs{\bff} \times \bbR^+_{\tau} )  \\
	\bhs{tff} = (\beta_H')^{\sharp}(\beta_H'')^{-1}( D_{fib} \times \{ \tau=0 \} )  \\
	\bhs{tf} = (\beta_H')^{\sharp}( \Delta_s ) \\
	\bhs{hlf}= \beta_H^{\sharp}( \bhs{\lf}) \\
	\bhs{hrf}= \beta_H^{\sharp}( \bhs{\rf}) \\
\end{gathered}
\end{equation*}
where we innocuously ignore the lift of the face  $\eps = 1$.
There is also a natural $b$-fibration
\begin{equation*}
	\pi_{\eps}:HX_{s} \lra [0,1]_{\eps}.
\end{equation*}
We depict the heat space at $\eps=0$ in Figure \ref{fig:heatspace}.
\begin{figure}
\centering
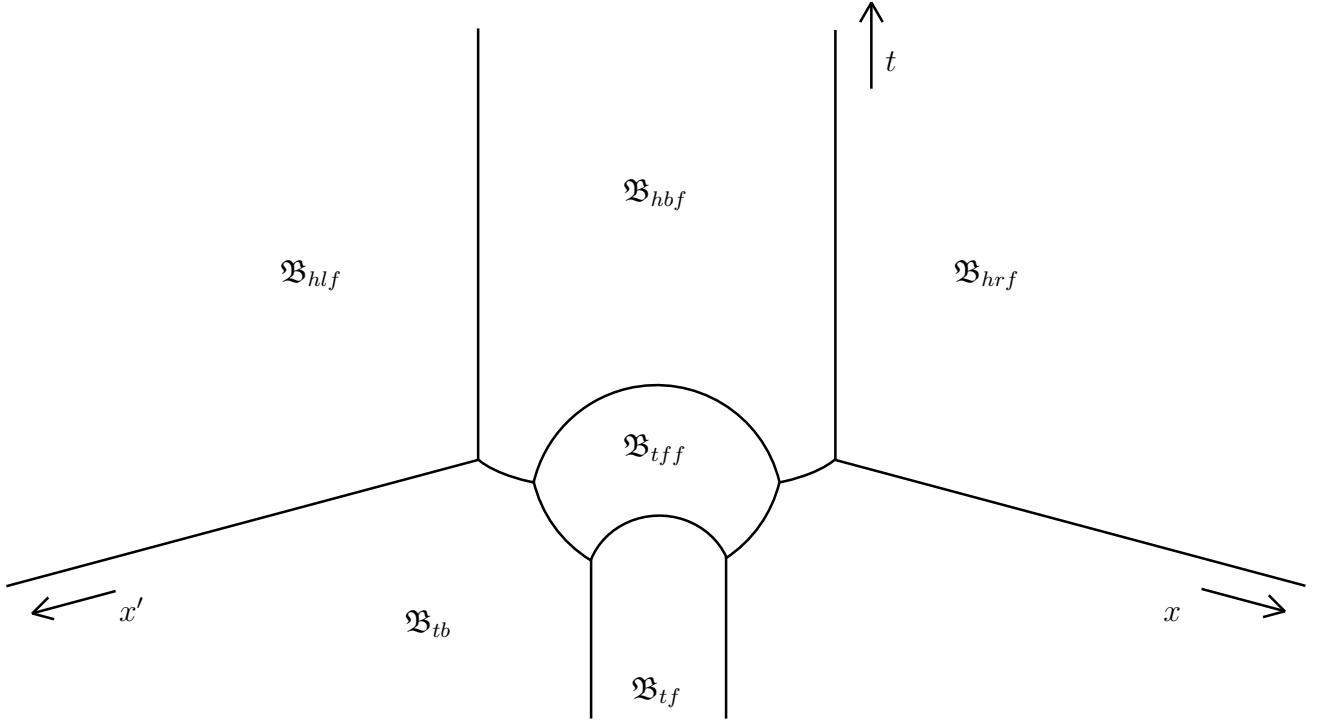
\caption{The heat surgery space $HX_s$ at $\epsilon=0$. }
\label{fig:heatspace}
\end{figure}

$ $

It will be very useful to have clear descriptions of the bhs's of $HX_{s}.$
As our constructions will only involve those that intersect the lift of the diagonal of $M,$ we focus on these.

First note that the boundary hypersurface $\bhs{tf}$ is the front face corresponding to the blow-up of $\Delta_s\subset [ X^2_{b,s} \times \bbR^+_{\tau}; D_{fib} \times \{  \tau =0 \} ]$, so it is naturally identified with the radial compactification of the bundle $N\Delta_s\to \Delta_s$, that is, the radial compactification of the bundle $\Ephi TX_{s}\to X_s$ by Corollary~\ref{cor:normdiagonal}.  This identification will allow appropriate surgery differential operators to act through their $\ephi$ principal symbol on models at this face.

Next, there is a natural identification 
\begin{equation}
\bhs{tff}\setminus(\bhs{tb}\cap\bhs{hbf}\cap\bhs{tff})\cong [(\bhs{\ff}\setminus\phibf)\times [0,\pi/2]_{\hat\theta}; \Delta_s\cap \bhs{\ff}\times \{0\}],
\label{tff.1}\end{equation}
where $\hat\theta= \arctan \left(\tau/\rho_{\bff} \right)$.  In fact, making use of the rescaled square root of time $\tau/\rho,$ we will conveniently think of  $\bhs{tff}(HX_{s})$ as heat space for the normal operators defined on  $\bhs{\ff}(X^2_{s})$.

Similarly the face $\bhs{hbf}$ is obtained from $\bhs{\bff}\times \bbR^+_{\tau}$ by blowing-up $D_{fib}\times \{0\}$. In particular, for $\tau>0$ $\bhs{hbf}$ can be identified with the $b$-front face of the heat space of \cite[\S6.1]{mame1}.
Correspondingly, the model operator at this face will be a $b$-heat kernel.  As expected, the face $\bhs{hmf}(HX_{s})$ is easily seen to coincide with the $d$-heat space of Vaillant \cite[\S4]{v} of the manifold with fibered boundary $\bhs{\ms} = [M;H].$  
$ $\\


We can define left and right projections $\beta_{H,L}$ and $\beta_{H,R}$ onto $X_s$ by precomposing the $b$-fibrations $\pr_1\circ \pi_{b,s,L}^2$ and $\pr_1\circ \pi_{b,s,R}^2$ with the blow-down map $\beta_H$,
\begin{equation*}
	\xymatrix{ & & X_s  \\
	HX_{s} \ar@/^15pt/[rru]^{\beta_{H,L}} \ar@/_15pt/[rrd]_{\beta_{H,R}} \ar[r]^-{\beta_H}  & X^2_{b,s}\times \bbR^+_{\tau} 
		\ar[ru]_{\pr_1\circ \pi_{b,s,L}^2} \ar[rd]^{\pr_1\circ \pi_{b,s,R}^2} & \\
	&  & X_s.}
\end{equation*}
The maps $\beta_{H,L}$ and $\beta_{H,R}$ are $b$-fibrations as can be readily seen from \cite[Lemma~12]{hmm}.

Given a bundle $E \lra X_s$, we define a homomorphism bundle of $E$ over $HX_s$ by
\begin{equation*}
	\Hom(E) = \beta_{H,L}^*E \boxtimes \beta_{H,R}^*E^*.
\end{equation*}

To define the action of a section of this bundle, consider first the case of a closed manifold $M$ without any degeneration.
Here we have
\begin{equation*}
	\xymatrix{
	 & & M \times \bbR^+ \\
	HM= [M^2\times \bbR^+_{\tau};\diag_M\times \{ 0 \}] \ar@/^15pt/[rru]^{\bar\beta_{H,L}} \ar@/_15pt/[rrd]_{\bar\beta_{H,R}} \ar[r]^-{\bar\beta_H} 
		& M^2 \times \bbR^+ \ar[rd]^{\bar\pi_{R}} \ar[ru]^{\bar\pi_L} \\
	& & M\times \bbR^+ \ar[r]^-{\bar \pi_M} & M}
\end{equation*}
and every $\cK \in \CI(M^2\times \bbR^+; \Hom(E) \otimes \bar\pi_R^*\Omega(M))$ defines an operator on polyhomogeneous sections of $E$
\begin{equation*}
	\begin{array}{ccc}
	\cA^*(M;E) & \longrightarrow & \cA^*(M\times \bbR^+;E) \\
	f & \longmapsto & (\bar\beta_{H,L})_*\lrpar{\cK \cdot \bar\beta_{H,R}^*\bar \pi_M^*f}
	\end{array}
\end{equation*}
Recall that $\bar \beta_{H,R}^*$ acts on polyhomogeneous functions whenever $\bar\beta_{H,R}$ is a $b$-map, and that $(\bar \beta_{H,L})_*$ acts on polyhomogeneous functions whenever $\bar\beta_{H,L}$ is a $b$-fibration. We can also allow $\cK$ to be a polyhomogeneous section of $\Hom(E) \otimes \bar\pi_R^*\Omega(M)$ and the same formula defines an operator.\\

For the case of $M$ with degeneration, we can pull-back a polyhomogeneous section of $E$ from $X_s$ to a polyhomogeneous section over $HX_s.$
The left projection to $X_s \times \bbR^+_\tau$ is not a $b$-fibration, since the boundary hypersurface $\bhs{tff}$ of $HX_s$ is sent to a corner, so we introduce the space
\begin{equation}\label{eq:DefDeltaHX}
\begin{gathered}
	\Delta_{HX}   
	= [ X_{s} \times \bbR^+_{\tau}; \bhs{\bs}\times \{\tau=0\}]
\end{gathered}
\end{equation}
pictured in Figure \ref{fig:deltahx}, as well as the extended heat space
$$
\wt{HX}_s := [HX_s; \bhs{hbf}\cap \bhs{tb}; \bhs{lf}\cap \bhs{tb}; \bhs{rf}\cap\bhs{tb}].
$$
Indeed, since $D_{fib}\times \{0\}\subset \bhs{bf}\times \{0\}$, we see, using the commutativity of blow-ups which are nested or transverse (see for instance  \cite[Lemma 5]{hmm}),  that there is a canonical isomorphism
\begin{equation}
\wt{HX}_s \cong [X^2_{b,s}\times \bbR_{\tau}^+; \bhs{bf}\times \{0\}; \bhs{lf}\times \{0\}; \bhs{rf}\times \{0\}; D_{fib}\times \{0\}; \Delta_s ].
\label{nbo.1}\end{equation}
By \cite[Lemma~10]{hmm}, we know that the $b$-fibration $\pi^2_{b,s,L}: X^2_{b,s}\times \bbR^+\to X_s\times \bbR^+$ lifts to a $b$-fibration
$$
[X^2_{b,s}\times \bbR_{\tau}^+; \bhs{bf}\times \{0\}; \bhs{lf}\times \{0\}; \bhs{rf}\times \{0\}]\to \Delta_{HX}.
$$ 
Hence, by \cite[Lemma~12]{hmm} and equation \eqref{nbo.1}, this lifts to a $b$-fibration $\wt{HX}_s\to \Delta_{HX}$.  
On $\Delta_{HX}$, we denote the blow-down maps by
\begin{equation*}
	\xymatrix{ 
	& \Delta_{HX} \ar[ld]_-{\beta_{\Delta}} \ar[rd]^-{\beta_{\Delta, (1)}} \\
	X_s \times \bbR^+_\tau \ar[rr]^{\beta_{(1)}\times\Id_{\tau}} &
	& M \times [0,1]_{\eps}\times \bbR^+_{\tau} }
\end{equation*}
\begin{figure}
\centering
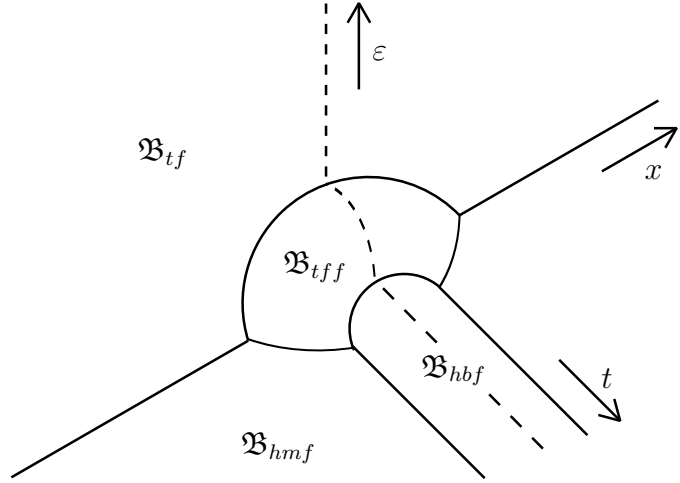
\caption{The auxiliary space $\Delta_{HX}$; the dotted line represents $x=0$.}
\label{fig:deltahx}
\end{figure}
and  label the boundary hypersurfaces of $\Delta_{HX}$ (ignoring, as usual, $\eps=1$) by
\begin{equation*}
\begin{gathered}
	\bhs{hmf} = \beta_{\Delta,(1)}^{\sharp}(\{\eps = 0\}), \quad
	\bhs{hbf} = \beta_{\Delta}^{\sharp}( \bhs{\bs} \times \bbR^+_{\tau} ) \\
	\bhs{tf} = \beta_{\Delta}^{\sharp}(\{\tau = 0 \}), \quad
	\bhs{tff} = \beta_{\Delta}^{-1}(\bhs{\bs}\times \{\tau=0\}),
\end{gathered}
\end{equation*}
consistently with the labels on $HX_s,$ and  extend the diagram above to
\begin{equation*}
	\xymatrix @C=1pt{
	& \Delta_{HX} \ar[rr]^-{\beta_{\Delta}}  & & X_s \times \bbR^+_{\tau} \\
	\wt{HX}_{s} \ar[ur]^-{\wt\beta_{H,L}} \ar[rr]^-{\wt\beta_H} \ar@/_15pt/[rrrd]_{\wt\beta_{H,R}} & & X^2_{b,s} \times \bbR^+_{\tau} 
		\ar[ru]_{\pi^2_{b,s,L}\times \Id_{\tau}} \ar[rd]^{\pr_1\circ \pi_{b,s,R}^2} \\
	&  & & X_s. & &  }
\end{equation*}
Then every density $\cK \in \cA^*(HX_s; \Hom(E) \otimes \beta_{H,R}^*\Omega(M))$ defines an operator on (appropriate) polyhomogeneous sections of $E,$
\begin{equation*}
	\xymatrix @R=1pt {
	\cA^*(X_s;E) \ar[r] & \cA^*(\Delta_{HX};E) \\
	f \ar@{|->}[r] & (\wt\beta_{H,L})_*\lrpar{\wt\beta^*\cK \cdot \wt\beta_{H,R}^*f,}
	}
\end{equation*}
where $\wt\beta: \wt{HX}_s \to HX_s$ is the natural blow-down map.  
In the next section we show that the solution operator of the heat equation on $X_s$ is an example of such an operator, with $\cK$ the heat kernel.

\section{Solving the heat equation} \label{sec:SolveHeat}

Our study of analytic torsion proceeds through a careful understanding of the asymptotics of the heat kernel of the Hodge Laplacian as the metric on $M$ degenerates.
In \S\ref{sec:SurgeryHeatSpace} we have described the `surgery heat space' $HX_s.$ Our refined description of the asymptotics of the heat kernel is that it is a weighted smooth density on $HX_s$ (with explicit weights). In particular this will allow us to deduce the behavior of the trace of the heat kernel as a function of $\eps$ and $t.$\\

\subsection{The heat kernel of a Laplace-type operator} \label{sec:HeatKerLaplaceType}
In this section, we will construct the heat kernel of the operator $\eth^2_{\ed}$, where $\eth_{\ed}$ is the operator of \S\ref{rc.0}.  As in \S\ref{rc.0}, in order to work with $b$-densities, it is convenient to work instead with the shifted operator
$$
       \Delta_{\ed}:= \rho^{v/2} \eth^2_{\ed} \rho^{-v/2}= D^2_{\ed}  
$$
seen as an unbounded operator acting on $L^2_b(X_{\eps};E).$  

Note that the operator $\Delta_{\ed}$ is singular at $\bhs{\bs}.$
We define
\begin{equation*}
	\Delta_v = \rho_{\bs}^2\Delta_{\ed}\rest{\bhs{\bs}}= D_v^2 \in \Diff^2(\bhs{\bs}/Y\times[-\pi/2,\pi/2]; E), \quad
	\Delta_d = \Delta_{\ed}\rest{\bhs{\ms}} \in \rho_{\bs}^{-2}\Diff^2_{\phi}([M;H]).
\end{equation*}
As in \S\ref{rc.0}, we make the fundamental  constant rank assumption (\textbf{Assumption~1}) that 
\begin{equation*}
	 \ker\Delta_v= \ker D_v \text{ forms a bundle over }Y \times [-\pi/2,\pi/2].
\end{equation*}
Notice however that in the construction of the heat kernel, we will not require the operator $D_b$ of Definition~\ref{rc.4a} to be Fredholm, but we will make use of the corresponding second order operator $\Delta_b:=D_b^2$.
\\

Given a section $f \in \CI(X_{s};E),$ a solution to the heat equation is a section $u \in \CI(\Delta_{HX}; E)$ satisfying
\begin{equation*}
	\begin{cases}
	(\pa_t  + \Delta_{\ed})u = 0,\\
	u\rest{t=0} = f.
	\end{cases}
\end{equation*}
The heat operator of $\Delta_{\ed}$ is the map $f \mapsto u$ and we will show that its distributional kernel is a smooth section of an appropriate bundle over $HX_{s}.$

%
%

\begin{theorem}\label{thm:HeatKernelConst}
There exists an operator $A:\CI(X_s;E) \lra \CI(\Delta_{HX};E)$ solving the heat equation in that
\begin{equation*}
	\begin{cases}
	\beta_{\Delta}^*(t\pa_t + t\Delta_{\ed})(Af) = 0 \\
	Af\rest{\bhs{tf}(\Delta_{HX})} = f.
	\end{cases}
\end{equation*}
The integral kernel of $A$ is a weighted smooth density on $HX_s,$
\begin{equation}\label{eq:DefAlpha}
\begin{gathered}
	\cK_A \in \rho^{\alpha}\CI \lrpar{ HX_s; \beta_{H,R}^*\lrpar{ \Omega_{\epsilon,b}(X_s)} \otimes \Hom(E) }\\
	\Mwith \alpha = (\alpha_{tf} = -m, \alpha_{tff} = -h-1, \alpha_{hbf} =0, \alpha_{hmf} = 0, \alpha_{tb}=\alpha_{hlf}=\alpha_{hlr}=\infty)
\end{gathered}
\end{equation}
and leading term at each boundary hypersurface
\begin{equation}\label{eq:AllNormalOps}
\begin{gathered}
	N_{tf}(A) = \frac1{(4\pi)^{m/2}}\exp\lrpar{ - \frac{|\cdot|_{\Ephi TX_s}^2}4 } \mu_{\Ephi TX_s/X_s}\\
	N_{tff}(A) = e^{-\sigma^2 \Delta_v}\frac1{(4\pi\sigma^2)^{(h+1)/2}}\exp\lrpar{ -\frac{|\cdot|^2_{\Ephi NY} }{2\sigma^2}} \mu_{\Ephi NY} \\
	N_{hbf}(A) = e^{-\sigma^2 \Delta_{b}}, \quad
	N_{hmf}(A) = e^{-t\Delta_d}
\end{gathered}
\end{equation}
where $\sigma = \tau/\rho$ is a rescaled time variable, 
$\mu_{\Ephi TX_s/X_s}$ is a vertical density on the fibers of $\Ephi TX_s \lra X_s,$
$\mu_{\Ephi NY}$ is a vertical density on the fibers of $\Ephi NY\to Y\times [-\pi/2,\pi/2],$ and $e^{-t\Delta_d}$ is the heat kernel of the operator $N_{\mf}(\Delta_{\ed})$ acting on $L^2_b$-sections. 
\end{theorem}

The operator $A$ in this theorem will be denoted
\begin{equation*}
	H_{\ed} = e^{-t \Delta_{\ed}}.
\end{equation*}
\begin{remark}
In particular, notice that if $D_b=0$, then $N_{hbf}(e^{-t \Delta_{\ed}})=0$, which implies that for $t>0$, the $b$-operator $e^{-t \Delta_{\ed}}$ is trace class.
\label{tc.1}\end{remark}

\begin{proof}
As in the statement of the theorem, it is convenient to start by replacing the operator $\pa_t + \Delta_{\ed}$ with the operator $t(\pa_t + \Delta_{\ed})$ as the latter is made up from vector fields tangent to $\bhs{tf}$ and $\bhs{tff}.$ In particular $t(\pa_t + \Delta_{\ed})(Af)=0$ will induce model problems at these boundary hypersurfaces and we can begin our construction by solving these simpler problems.

The solution to the problem at $\bhs{tf}$ proceeds identically to the treatment in \cite[Chapter 7]{MelroseAPS}. Nevertheless we will go through the details as it illustrates the procedure we follow at the other boundary hypersurfaces, and as we will need to verify that our model problem solutions are mutually compatible.

Let us start by considering the action of $A$ on $f$ near $\bhs{tf}.$
Recall that the interior of $\bhs{tf}$ can be identified with the vector bundle $\Ephi TX_s \lra X_s.$
Away from the other boundary faces, a convenient choice of coordinates is
\begin{equation}\label{eq:CoordChartTf}
	\zeta = (x,y,z), \quad \Theta' = \frac{\zeta'-\zeta}{\tau}, \quad \eps, \quad \tau = \sqrt t
\end{equation}
in which $\tau$ is a bdf for $\bhs{tf}.$
In these coordinates, with $\Theta_x' = \frac{x'-x}{\tau},$ we have for $\zeta,\epsilon$ and $\tau$ fixed that  
\begin{equation*}
	\beta_{H,R}^*(d\zeta') = \tau^m d\Theta',
\end{equation*}
so let us write $\cK_A = \wt \cK_A \tau^m d\Theta'$  so that the action of $A$ on $f$ is through
\begin{equation}\label{eq:LocalCoordActTfA}
	(\bar \beta_{H,L})_* \cK_A f(\zeta,\tau) \\
	= \int_{\bbR^m} \wt\cK_A(\zeta, \Theta', \eps,\tau) f(\tau\Theta' + \zeta) \; \tau^m d\Theta'
\end{equation}
for $f$ supported in the coordinate chart of $\zeta=(x,y,z)$.  Clearly, the restriction of $Af$ at $t=0$ makes sense as long as $\tau^m \wt\cK_A\rest{\bhs{tf}(HX)}$ makes sense.
Thus we set $\alpha_{tf} = -m,$ and let us write $\wt\cK_A = \tau^{-m} \wt\cK_{tf} + \tau^{1-m}\wt\cK_A'$ with 
$\wt\cK_{tf}$ independent of $\tau$ and $\wt\cK_A' = \cO(\tau^0).$
The restriction of $Af$ to $\bhs{tf}(\Delta_{HX})$ is then given by
\begin{equation*}
	f(\zeta, \eps) \int_{\bbR^n} \wt\cK_{tf}(\zeta, \Theta', \eps) 
	\; d\Theta'
\end{equation*}
and so the requirement that $Af\rest{t=0}=f$ comes down to asking that the fiberwise integral of $\cK_A\rest{\bhs{tf}}$ is equal to one, at least away from the other boundary hypersurfaces.

Let us consider the action of $\beta_{\Delta}^*(t\pa_t + t\Delta)$ on $Af$ at $\beta_{tf}(\Delta_{HX}).$
First $\beta_{\Delta}^*(t\pa_t)=\frac12 \tau\pa_{\tau}$ will act by (in the interest of space we will not write the arguments of $\cK_A$ and $f$ when they coincide with those in \eqref{eq:LocalCoordActTfA})
\begin{equation*}
	\int_{\bbR^m} \lrpar{ \frac12\tau\pa_{\tau} (\tau^m\wt \cK_A(\cdot)) } f(\cdot) \; d\Theta' 
	+\frac12 \int_{\bbR^m} \tau^m\wt \cK_A(\cdot)  \lrpar{\lrpar{ \Theta'\cdot\pa_{\Theta'} } f(\cdot)} \; d\Theta'
\end{equation*}
the first integral is $\cO(\tau),$ and
integrating by parts in the second integral yields
\begin{equation*}
	\int_{\bbR^m} \lrpar{ \frac12\lrpar{-m - \sR}(\tau^m\wt \cK_A(\cdot))} f(\cdot) \; d\Theta'
\end{equation*}
where $\sR= \Theta'\cdot\pa_{\Theta'}$ denotes the radial vector field on $\bbR^m.$
Restricting this expression at $\tau =0$, we find
\begin{equation*}
	\beta_{\Delta}^*(t\pa_t)(Af)\rest{\bhs{tf}} = 
	f(x,y,z) \int_{\bbR^m} -\frac12(m+\sR)\cK_{tf}(\zeta,\Theta', \eps) \; d\Theta',
\end{equation*}
so that
$$
  \cK_{\beta_{\Delta}^*(t\pa_t) \circ A}\rest{\bhs{tf}(HX)}= \bar \beta_{H,L}^*(t\pa_t)\cK_A\rest{\bhs{tf}(HX)}= -\frac12(m+\sR)\cK_{tf}(\zeta,\Theta', \eps) \; d\Theta'
$$
in the coordinates \eqref{eq:LocalCoordActTfA}.

Next consider the action of $t\Delta_{\ed}.$
First note that if $V$ is any vector field on $X_s,$ the action of $\tau V$ is equal to
\begin{multline}\label{eq:StartingSymbol}
  \tau V Af(\zeta,\tau)= 
	\int_{\bbR^m} (\tau V_{\zeta})\wt\cK_A(\zeta, \Theta', \eps) f(\tau \Theta' + \zeta,\eps) \tau^{m}  \; d\Theta'  \\
	+ \int_{\bbR^m} \wt\cK_A(\zeta, \Theta', \eps) (\tau V_{\zeta})f(\tau\Theta' + \zeta,\eps) \tau^{m} \; d\Theta',
\end{multline}
where $V_{\zeta}$ indicates that  $V$ acts through the left coordinates $\zeta$.  
The first term is $\cO(\tau)$ as $\tau \to 0,$ while in the second term we can replace $\tau V_{\zeta}$ with $V_{\Theta'}=v^j(\zeta)\frac{\pa}{\pa (\Theta')^j}$ if $V_\zeta=v^j(\zeta)\frac{\pa}{\pa \zeta^j}$.  Thus $V_{\Theta'}$ is a vector field on $\Ephi TX_s,$ tangent to the fibers of the projection onto $X_s,$ with constant coefficients on each fiber. 
Integrating by parts, the second term is thus equal to
\begin{equation*}
	\int_{\bbR^m} \wt \cK_A(\zeta, \Theta', \eps) V_{\Theta'}f(\tau\Theta' + \zeta,\eps) \tau^{m} \; d\Theta' \\
	= - \int_{\bbR^m} V_{\Theta'} \cK_A(\zeta, \Theta',\eps) f(\tau\Theta'+\zeta, \eps) \tau^m \; d\Theta'
\end{equation*}
and so has leading term at $\bhs{tf}(\Delta_{HX})$ given by
\begin{equation*}
	- f(\zeta, \eps) \int_{\bbR^m} V_{\Theta'} \cK_{tf}(\zeta, \Theta',\eps) d\Theta'.
\end{equation*}
Thus 
$$
   \cK_{\tau V\circ A}\rest{\bhs{tf}(HX)}=   -V_{\Theta'} \cK_{tf}(\zeta, \Theta',\eps) d\Theta'$$
with $V_{\Theta'}$ a family of translation invariant vector fields on the fibers of $\Ephi TX_s\cong \bhs{tf}\setminus \bhs{tb}$.  In particular, notice that the principal symbol of $V$ is given by the Fourier transform of $V_{\Theta'}$ in the fibers of $\Ephi TX_s$.     
Thus a vector field acts on $\cK_A$ at $\bhs{tf}(HX)$ through the inverse Fourier transform of its principal symbol. 
This same reasoning shows that the action on the left of $(t\Delta_{\ed})$ on the leading term of $K_{A}$ at $\bhs{tf}$ is through the inverse Fourier transform of its principal symbol  $\Ephi\sigma_2(\rho^2\Delta_{\ed})$, which is just a family of Euclidean Laplacians
$\Delta_{\Ephi TX_s}= \cF^{-1}(\Ephi\sigma_2(\rho^2\Delta_{\ed}))$ on the fibers of $\Ephi TX_s\cong  \bhs{tf}\setminus \bhs{tb}.$
Hence altogether at $\bhs{tf}$ away from other boundary hypersurfaces we need to solve fiberwise
\begin{equation}
	\begin{cases}
	(\Delta_{\Ephi TX_s} - \tfrac12(m+ \sR))N_{tf}(A)=0,\\
	\int_{\Ephi TX_s/X_s} N_{tf}(A) = 1.
	\end{cases}
\label{moeq.1}\end{equation}

To see that this makes sense all the way down to $\bhs{tf}\cap \bhs{tff}$, we need to consider coordinates near $\bhs{tf}\cap \bhs{tff}$ of the form
\begin{equation}
          \rho=\sqrt{x^2+\epsilon^2}, \quad
        \theta = \arctan \frac{x}{\eps}, \quad
	y, 
	z, \quad
	\wt S' = \frac{x' -x}{\rho^2\sigma}, \quad
	\wt U' = \frac{y'-y}{\rho\sigma}, \quad
	\wt Z' = \frac{z'-z}{\sigma}, \quad
	\sigma = \frac{\tau}\rho, \quad	
\label{coor.1}\end{equation}
in which $\sigma$ is a bdf for $\bhs{tf}(HX)$ and $\rho$ is a bdf for $\bhs{tff}(HX).$
Note that since $f$ is a function on $X_s$ near $\bhs{\bs}$, we can evaluate it on, e.g., $(\rho,\theta, y, z).$ 
In these coordinates and with $\theta,y,z,\rho,\sigma$ fixed, 
\begin{equation*}
	\beta_{H,R}^*\lrpar{ \frac{dxdydz}{\sqrt{x^2 + \eps^2}} } = \frac{\sigma^m\rho^{h+1} d\wt S'd\wt U'd\wt Z'}{\sqrt{(\sin\theta+\rho\sigma\wt S')^2 + \cos^2\theta}}
\end{equation*}
so we will write $\cK_A = \wt \cK_A \lrpar{\sigma^m\rho^{h+1} d\wt S'd\wt U'd\wt Z'}$ and $\wt\cK_{tf} = \sigma^{m}\wt\cK_A\rest{\bhs{tf}}.$  For $f$ supported in the coordinate chart of $(\rho,\theta,y,z)$ in $X_s$,  the action of $A$ on $f$ is then given by
\begin{multline}\label{eq:LocalCoordActTfB}
	(\wt\beta_{H,L})_*(\cK_A\beta_{H,R}^*f)(\rho,\theta,y,z,\sigma)= \\
	 \int_{\bbR^m} \wt \cK_A(\rho, \theta,y, z, \wt S',\wt U',\wt Z',\sigma) f(\rho', \theta',\rho\sigma \wt U'+y,\sigma\wt Z'+z) \;
	\lrpar{ \sigma^m\rho^{h+1} d\wt S'd\wt U'd\wt Z'},
\end{multline}
where $\rho'=\rho\sqrt{(\sin\theta+\rho\sigma\wt S')^2 + \cos^2\theta}$ and $\theta'=\arctan\left( \frac{\sin\theta+ \rho\sigma\wt S'}{\cos\theta} \right)$.  Thus, the restriction at $\sigma=0$ is equal to
\begin{equation*}
	\rho^{h+1}f(\rho,\theta,y,z) \int_{\bbR^m} \wt \cK_{tf}(\rho,\theta,y,z,\wt S', \wt U', \wt Z') \; d\wt S'd\wt U'd\wt Z'.
\end{equation*}
Thus requiring the fiberwise integral of $\cK_A\rest{\bhs{tf}}$ to be equal to one necessitates $\wt \cK_{tf} = \cO(\rho^{-h-1})$ and we see that we should set
$\alpha_{tff} = -h-1.$\\

Considering the action $t\pa_t = \frac12 \tau\pa_{\tau}= \frac12 \sigma\pa_{\sigma}$ on $Af$, we can integrate by parts as before to obtain this time 
$$
  \cK_{t\pa_t \circ A}\rest{\bhs{tf}}=  -\frac12(m+\sR)\wt \cK_{tf}(\rho,\theta,y,z,\wt S', \wt U', \wt Z') \; \rho^{h+1}d\wt S' d\wt U' d \wt Z',
$$
where $\sR= \wt S'\pa_{\wt S'} + \wt U'\cdot \pa_{\wt U'} + \wt Z\cdot\pa_{\wt Z'}$ in the coordinate $\wt S',\wt U',\wt Z'$.  To compute the action of $t\Delta_{\ed}$, consider first an $\ed$-vector field
\begin{equation}
        V= a\rho\pa_x + b^i\pa_{y^i} + \frac{c^j}{\rho}\pa_{z^j}
\label{vect.1}\end{equation}
with $a,b^i$ and $c^j$ smooth sections of $\End(E)$ on $X_s$.  We want to see how $\tau V=\sigma\rho V$ acts on $K_{A}$ from the left.  Now, $\rho V$ is a $\ephi$-vector field,
so keeping in mind that $\rho$ and $\sigma$ in \eqref{coor.1} depend on $x$ and integrating by parts as in \eqref{eq:StartingSymbol}, a careful computation shows that 
\begin{equation}
\tau V Af(\rho,\theta,y,z)=  \int_{\bbR^m}  V'\cK_A(\rho, \theta,y, z, \wt S',\wt U',\wt Z',\sigma) f(\rho', \theta',\rho\sigma \wt U'+y,\sigma\wt Z'+z) 
	 \sigma^m\rho^{h+1} d\wt S'd\wt U'd\wt Z',
\label{cc.1}\end{equation}
with 
\begin{multline}
V'= a \left(\sigma \rho^2\pa_x- (1+2x\sigma\wt S')\pa_{\wt S'} -\sigma x\sum (\wt U')_i\pa_{(\wt U')_i} + \sigma x(m-h-2-\sR) \right) \\
+ \sum b^i\pa_{(\wt U')^i} + \sum c^j \pa_{(\wt Z')^j}.
\label{cc.2}\end{multline}
Hence, restricting to $\bhs{tf}$, that is, to $\sigma=0$, this gives
$$
\tau V Af\rest{\sigma=0}(\rho,\theta,y,z)=  -\rho^{h+1}f(\rho,\theta,y,z) \int_{\bbR^m} \lrpar{ V'\rest{\sigma=0}} \wt\cK_{tf}(\rho,\theta,y,z,\wt S', \wt U', \wt Z') \; d\wt S'd\wt U'd\wt Z',
$$
where 
$$
 V'\rest{\sigma=0}= a(\rho,\theta,y,z)\pa_{\wt S'} + b^i(\rho,\theta,y,z)\pa_{(\wt U')^i} + c^j(\rho,\theta,y,z)\pa_{(\wt Z')^j}
$$
is as before the family of translation invariant vector fields in the fibers of $\Ephi TX_s\cong \bhs{tf}\setminus \bhs{tb}$ corresponding to the inverse Fourier transform of the principal symbol $\Ephi\sigma_1(\rho V)$.  Thus, 
$$
  \cK_{\tau\pa_{\tau} \circ A}\rest{\bhs{tf}}= \bar \beta_{H,L}^*(\tau\pa_{\tau})\cK_A\rest{\bhs{tf}}= -V'\wt \cK_{tf}(\rho,\theta,y,z,\wt S', \wt U', \wt Z') \; \rho^{h+1}d\wt S' d\wt U' d \wt Z'.$$

Similarly, we deduce that  $\Delta_{\ed}$ acts on the left on $K_A\rest{\bhs{tf}}$ via the inverse Fourier transform  $\Delta_{\Ephi TX_s}= \cF^{-1}(\Ephi\sigma_2(\rho^2\Delta_{\ed}))$, so that the equation  \eqref{moeq.1}  is also satisfied by $\cK_A\rest{\bhs{tf}}$ near $\bhs{tf}\cap\bhs{tff}$.  This equation is readily solved by using the Fourier transform on the fibers of $\Ephi TX_s$ and the solution is
\begin{equation}
	N_{tf}(A) = \frac1{(4\pi)^{m/2}}\exp\lrpar{ - \frac{|\cdot|_{\Ephi TX_s}^2}4 } \mu
\label{ntf.1}\end{equation}
with $\mu$ the vertical Euclidean density on the fibers of   $\Ephi TX_s \lra X_s$ induced by $\rho^2g_{\ed}$.  
\\

Next, let us consider the action of $(t\pa_t + t\Delta_{\ed})$ from the left on the interior of $\bhs{tff}.$
Recall that the identification \eqref{tff.1} allows to regard $\bhs{tff}$ as the heat space for the normal operator $N_{\ff}(\rho^2\Delta_{\ed})$.  In fact, notice that the restriction of \eqref{ntf.1} to $\bhs{tf}\cap\bhs{tff}$ (the term of order $\rho^{-h-1}$ when we use $b$-densities) is the same as the restriction of the heat kernel of $N_{ff}(\rho^2\Delta_{\ed})$ seen as Schwartz kernel on 
$$
\bhs{tff}\setminus(\bhs{tb}\cap\bhs{hbf}\cap\bhs{tff})\cong [(\bhs{\ff}\setminus\phibf)\times [0,\pi/2]_{\theta}; \Delta_s\cap \bhs{\ff}\times \{0\}].
$$ 
This suggests that the term of order $\rho_{tff}^{-h-1}$ at $\bhs{tff}$ of the heat kernel of $\Delta_{\ed}$ should be precisely the heat kernel of $N_{\ff}(\rho^2\Delta_{\ed})$.  To see this, we need to investigate what equation must be satisfied by the heat kernel on $\bhs{tff}$ away from $\bhs{tf}$.  For this purpose, let us introduce the coordinates 

\begin{equation}\label{eq:CoordsTff}
         \rho= \sqrt{x^2+\epsilon^2}, \quad
	\theta = \arctan\left(\frac x\eps\right), \quad
	y, \quad
	z, \quad
	S' = \frac{x' -x}{\rho^2}, \quad
	U' = \frac{y'-y}{\rho}, \quad
	z', \quad
	\sigma = \frac{\tau}\rho,
\end{equation}
near $\bhs{tff}$, where $\rho$ is a bdf for $\bhs{tff}.$  
In these coordinates, 
\begin{equation*}
	\beta_{H,R}^*\lrpar{ \frac{dxdydz}{\sqrt{x^2 + \eps^2}} } = \frac{\rho^{h+1} d S'd U'dz'}{\sqrt{(\sin\theta+\rho S')^2 + \cos\theta}},
\end{equation*}
so we will write $\cK_A = \wt \cK_A \lrpar{\rho^{h+1} d S'd U'dz'}$ and $\wt\cK_{tff} = \rho^{h+1}\wt\cK_A\rest{\bhs{tff}}.$
The action of $A$ on $f$ is given by
\begin{equation*}
	\int_{\bbR^{h+1}\times Z} \wt\cK_A(\rho,\theta, y, z, S', U', z', \sigma) f(\rho',x+\rho^2S', \rho U' + y, z') \lrpar{\rho^{h+1} d S'd U'dz'}
\end{equation*}
with this time $\rho'=\rho\sqrt{(\sin\theta+ \rho S')^2+\cos^2\theta}$ and $\theta'=\arctan\left(  \frac{\sin\theta + \rho S'}{\cos\theta}\right)$.  Restricting to $\bhs{\bs}$ gives
\begin{equation*}
	(Af)\rest{\bhs{sb}} = 
	\int_{\bbR^{h+1}\times Z} \wt\cK_{tff}(\theta, y, z, S', U', z', \sigma) f(0, \theta, y, z') \;  d S'd U'dz'.
\end{equation*}

Let us compute the action of $(t\pa_t + t\Delta_{\ed})$ on $Af$ at $\bhs{tff}(\Delta_{HX}).$
The action of $t\pa_t =\frac12\sigma\pa_{\sigma}$ on $(Af)\rest{\bhs{tff}}$ is straightforward.
To compute the action of $t\Delta_{\ed},$ first consider the action of $\tau V = \sigma\rho V$ with $V$ given by \eqref{vect.1} in local coordinates.
Keeping in mind that $\sigma$ and $\rho$ in \eqref{eq:CoordsTff} depend on $x$ and integrating by parts as in \eqref{eq:StartingSymbol}, a careful computation shows that
\begin{equation}
 \tau V Af(\theta,y,z)= \int_{\bbR^{h+1}\times Z} (V'\wt\cK_{tff}(\theta, y, z, S', U', z', \sigma)) f(\rho, \theta, y, z') \;  d S'd U'dz'
 \label{cc.3}\end{equation}
 with
 \begin{multline}
    V'= a\lrpar{\sigma \rho^2 \pa_x -\sigma(1+2xS') \pa_{S'} -x\sigma \sum U_i'  \pa_{U_i'} -x\sigma \left( \sigma\pa_{\sigma}+ h+2 \right)}  \\
    +\sigma\sum b^i(\rho\pa_{y^i}-\pa_{(U')^i}) + \sigma c^j\pa_{z^j}.
 \label{cc.4}\end{multline}
 Hence, restricting to $\bhs{\bs}$, that is, setting $\rho=0$,  gives
 $$
\lrpar{\tau V Af}\rest{\bhs{\bs}}(\theta,y,z)= \int_{\bbR^{h+1}\times Z} (V'\rest{\rho=0}\wt\cK_{tff}(\theta, y, z, S', U', z', \sigma)) f(0, \theta, y, z') \;  d S'd U'dz'
 $$
 with 
 $$
    V'\rest{\rho=0}= \sigma\lrpar{-a(0,\theta,y,z)\pa_{S'}  -b^i(0,\theta,y,z)\pa_{(U')^i} + c^j(0,\theta,y,z)\pa_{z^j}}
 $$
 
corresponding to $\sigma N_{\ff}(W)$ for the $\ephi$ vector field
$$
    W= a(\rho,\theta,y,z)\rho^2\pa_x  -b^i(\rho,\theta,y,z)\rho\pa_{y} + c^j(0,\theta,y,z)\pa_{z^j}.
$$
Clearly, this implies that
$$
 \tau^2\Delta_{\ed} Af\rest{\bhs{\bs}}(\theta,y,z)= \int_{\bbR^{h+1}\times Z} (\sigma^2 N_{\ff}(\rho^2\Delta_{\ed})\wt\cK_{tff}(\theta, y, z, S', U', z', \sigma)) f(0, \theta, y, z') \;  d S'd U'dz',$$
so that away from $\bhs{tf}$,
$$
  \cK_{t\Delta_{\ed}\circ A}\rest{\bhs{tff}(HX)}=  \sigma^2 N_{\ff}(\rho^2\Delta_{\ed})\circ N_{tff}(A).
$$
Thus, the model problem at $\bhs{tff}(HX_s)$ is
\begin{equation*}
	\begin{cases}
	\lrpar{\frac12\sigma\pa_{\sigma} + \sigma^2 N_{\ff}(\rho^2\Delta_{\ed})} N_{tff}(A) =0\\
	N_{tff}(A)\rest{\bhs{tf}\cap \bhs{tff}}
	= N_{tf}(A)\rest{\bhs{tf}\cap \bhs{tff}}
	\end{cases}
\end{equation*}
with solution
\begin{equation}\label{eq:SolHeatTff}
	N_{tff}(A) = e^{-\sigma \Delta_v}\frac1{(4\pi\sigma^2)^{(h+1)/2}}\exp\lrpar{ -\frac{|\cdot|^2_{\Ephi NY} }{4\sigma^2}} \mu
\end{equation}
the heat kernel of $N_{ff}(\rho^2\Delta_{\ed})$,
where $\mu$ is the density on the fibers of 
$$
\Ephi NY \btimes_{Y\times[-\pi/2,\pi/2]} \bhs{\bs} \to Y\times [-\pi/2,\pi/2]
$$
induced by the metric $\rho^2 g_{\ed}$. \\

Next we want to solve the model problem at $\bhs{hbf}$ and the first thing to do is to determine the compatibility condition at $\bhs{tff}\cap \bhs{hbf}.$
A convenient choice of coordinates near $\bhs{tff}\cap \bhs{hbf}$ is
\begin{equation*}
	\theta= \arctan \left( \frac{x}{\epsilon} \right), \quad
	y, \quad
	z, \quad
	\bar S' = \frac{x' -x}{\eta\tau^2}, \quad
	\bar U' = \frac{y'-y}{\tau}, \quad
	z', \quad
	\eta = \frac{\rho}\tau= \frac{\sqrt{x^2+\eps^2}}{\tau}, \quad
	\tau
\end{equation*}
in which $\tau$ is a bdf for $\bhs{tff}$ and $\eta$ is a bdf for $\bhs{hbf}.$
Let us write the solution to the heat equation at $\bhs{tff}$ in the coordinates \eqref{eq:CoordsTff} and then change coordinates to this set:
\begin{multline*}
	N_{tff}(A) 
	= e^{-\sigma \Delta_v}\frac1{(4\pi\sigma)^{(h+1)/2}}\exp\lrpar{ -\frac{|(S',U')|^2_{\Ephi NY} }{4\sigma^2}} dS'dU'dz' \\
	= e^{-\frac1\eta \Delta_v}\frac{\eta^{h+1}}{(4\pi)^{(h+1)/2}}\exp\lrpar{ -\frac{|(\bar S',\bar U')|^2_{\Ephi NY} }{4}} \lrpar{ \frac{d\bar S'd\bar U'dz'}{\eta^{h+1}} }.
\end{multline*}
It follows that the restriction of $N_{tff}(A)$ to $\bhs{tff}\cap\bhs{hbf} = \{ \eta = 0 \}$ is given by
\begin{equation}
	\cP_{\ker \Delta_v}\frac{1}{(4\pi)^{(h+1)/2}}\exp\lrpar{ -\frac{|(\bar S',\bar U')|^2_{\Ephi N\bhs{\bs}} }{4}} \; d\bar S'd\bar U'dz'.
\label{ini.1}\end{equation}
Significantly, this tells us that $\alpha_{hbf}=0$ and suggests that we set up the model problem at $\bhs{hbf}$ entirely in terms of  sections of $ \ker \Delta_v.$  Indeed, using the natural identification
$$
     \bhs{hbf}\cong [\bhs{bf}\times \bbR^+; D_{fib}\times\{0\}],
$$
notice that the fibration $\phi:H\to Y$ naturally induces a fibration $$\phi_{hbf}: \bhs{hbf}\to H_{ob}Y$$ over the overblown $b$-heat space of $Y\times [-\pi/2,\pi/2]$,
$$
    H_{ob}Y=  [(Y\times [-\pi/2,\pi/2]])^2_{ob}\times \bbR^+; D_b\times\{0\}],
$$
where
$$
    (Y\times [-\pi/2,\pi/2])^2_{ob}= [(Y\times [-\pi/2,\pi/2])^2; (\pa(Y\times [-\pi/2,\pi/2]))^2]
 $$
is the overblown $b$-double space of $Y\times [-\pi/2,\pi/2]$ and $D_b$ is the lift of the diagonal to this overblown $b$-double space.  From this point of view,  \eqref{ini.1} is precisely the restriction of the $b$-heat kernel of $\Delta_b= D^2_b$ to the temporal front face of $H_{ob}Y$.  This suggests that the restriction of the heat kernel of $\Delta_{\ed}$ to $\bhs{hbf}$ should be to the heat kernel of $\Delta_b$.  To see this, we need to see what equation is satisfied away from $\bhs{tff}$.     
Clearly , $\cK_{\frac12\tau\pa_{\tau}A}\rest{\bhs{hbf}}=\frac12\tau\pa\tau \cK_A\rest{hbf}$ away from $\bhs{tff}$.  To compute the action of $t\Delta_{\ed}$, we need to make the ansatz that $\cK_{hbf}=K_A\rest{\bhs{hbf}}=\Pi_h K_A\rest{\bhs{hbf}}\Pi_h$ comes from a Schwartz kernel on $H_{ob}Y$ acting on sections  $\ker D_v$ and that $\cK_{A}$ is obtained by extending $\cK_{hbf}$ off $\bhs{hbf}$ as in Lemma~\ref{rc.5a} , so that 
$$
          \cK_{D_{\ed}^{\ell}A}\rest{\bhs{hbf}}=   \Pi_h \bar\beta_{H_b,L}^*(D_b^{\ell})\cK_{hbf} \Pi_h,  \quad \forall \ell\in\bbN_0,
$$
where $\bar\beta_{H_b,L}: \pr_L\circ \beta_{H_b}$ with $\beta_{H_{ob}}: H_b Y\to (Y\times [-\pi/2,\pi/2])^2\times \bbR^+$ the blow-down map and $\pr_L: (Y\times [-\pi/2,\pi/2])^2\times \bbR^+\to Y\times [-\pi/2,\pi/2]$ the projection on the left factor.  
With this ansatz, the model problem to solve is
\begin{equation*}
	(\frac12\tau\pa_{\tau} + \tau^2\Delta_b) \Pi_h \cK_{hbf} \Pi_h = 0.
\end{equation*}
Combining with the initial condition \eqref{ini.1}, this gives the solution
\begin{equation*}
	N_{hbf}(A) = \Pi_he^{-\tau^2 \Delta_b}\Pi_h
\end{equation*}
as claimed.

Before proceeding with the construction of the heat kernel, let us point out that the computations above show that 
\begin{equation*}
	Af \in \CI(\Delta_{HX};E).
\end{equation*}
The smoothness is a consequence of the matching of the solutions at the various corners; the reason that the solution is evaluated on $\Delta_{HX}$ rather than $X_s \times \bbR^+_\tau$ is that the boundary hypersurface $\bhs{tff}(HX_s)$ corresponds to the blow-up of $\bhs{\bs}(X_s)$ at $\{ \tau =0 \}.$\\

Notice also that all our model heat kernels so far vanish rapidly at the faces $\bhs{tb}, \bhs{hlf}$ and $\bhs{hrf}$, which means we can look for a heat kernel that vanishes rapidly at these faces.  

%

Given consistent solutions of the model problems at $\bhs{tf},$ $\bhs{tff}$ and $\bhs{hbf},$ it is now relatively straightforward to complete the construction of the heat kernel (cf. \cite{MelroseAPS}).
First note that smooth extension of these solutions to $HX_{s}$ which are as in Lemma~\ref{rc.5a} for the extension off $\bhs{hbf}$ yields a polyhomogeneous function $A_0 \in \rho^{\alpha}\CI(HX_{s};\Hom(E))$ with $\alpha$ given by \eqref{eq:DefAlpha} whose leading terms at each boundary hypersurface solve the corresponding model problem. These computations show that
\begin{equation*}
	t(\pa_t + \Delta_{\ed})A_0 = - R_1, \quad 
	R_1 \in \frac{\eps \rho_{tf}}{\rho_{hmf}}\rho^{\alpha}\CI(HX_{s};\Hom(E)).
\end{equation*}
Proceeding inductively, we point out that the knowledge of the normal operators allows us to solve
\begin{equation*}
	t(\pa_t + \Delta_{\ed})A_j =R_j-  R_{j+1}, \quad \Mwith A_j, R_j \in \left(\frac{\eps\rho_{tf}}{\rho_{hmf}}\right)^{j}\rho^{\alpha}\CI(HX_{s};\Hom(E))
\end{equation*}
for all $j \in \bbN.$  Indeed, always using Lemma~\ref{rc.5a} to extend off $\bhs{hbf}$,  this amounts to solving inductively the equations 
\begin{gather}
	(\Delta_{\Ephi TX_s} - \tfrac12(m+ \sR))N_{tf}(A_j)=N_{tf}(R_j), \\
	\lrpar{\frac12\sigma\pa_{\sigma} + \sigma^2 N_{\ff}(\rho^2\Delta_{\ed})} N_{tff}(A_j) =N_{tff}(R_j),\\
	(\frac12\tau\pa_{\tau} + \tau^2\Delta_b) \Pi_h N_{hbf}(A_j) = N_{hbf}(R_j)= \Pi_hN_{hbf}(R_j), 
\label{problem.1}\end{gather}
using the model heat kernel at each of these faces, where in the last equation, we can reduce to the case $N_{hbf}(R_j)= \Pi_hN_{hbf}(R_j)$ by first adding to $A_j$ a term of order $2$, $\rho_{hbf}^2A_j'$, with the property that $N_{hbf}(R_j)-D_v^2N_{hbf}(A_j')= \Pi_hN_{hbf}(R_j)$.

Now let $H_0$ be an asymptotic sum of the series $\sum \left(\frac{\eps \rho_{tf}}{\rho_{hmf}}\right)^j A_j;$ that is, let \linebreak $H_0 \in \rho^{\alpha}\CI(HX_{s};\Hom(E))$ coincide in Taylor expansion with this formal sum at each boundary hypersurface. The existence of $H_0$ follows from a natural extension of Borel's lemma.
By construction we have
\begin{equation*}
	t(\pa_t+\Delta_{\ed})H_0 = -R_{\infty} \in \dCI_{hmf}(HX_{s};\Hom(E)).
\end{equation*}
where the dot in $\dCI_{hmf}$ indicates vanishing to infinite order at all boundary hypersurfaces except $\bhs{hmf}$.  Restricting to $\bhs{hmf}$, this gives the equation
\begin{equation*}
	t(\pa_t+\Delta_{d})N_{hmf}(H_0) = -N_{hmf}(R_{\infty}) \in \dCI(\bhs{hmf};\Hom(E)),
\end{equation*}
so that $N_{hmf}(H_0)$ can be seen as a parametrix for the heat equation associated to the operator $\Delta_d$ on $[M;H]$.  
To further improve $N_{hmf}(H_0)$ and obtain the actual heat kernel on $\bhs{hmf}$, it is convenient to think of a heat operator on $\bhs{hmf}$ as a map
\begin{equation*}
	\dCI([M;H] \times \bbR^+;E)\lra \dCI([M;H]\times \bbR^+;E)
\end{equation*}
via convolution in $t.$ The identity for this product is the kernel with the Dirac distribution over the diagonal at time $t=0.$
Since $N_{tf}(H_0)=\Id,$ it is easy to see, as in \cite[Proposition 7.17]{MelroseAPS}, \cite[\S7.1]{mame1} that 
\begin{equation*}
	(\pa_t+\Delta_{d})N_{hmf}(H_0) = \Id - t^{-1}N_{hmf}(R_{\infty})
\end{equation*}
and that the inverse of this operator has the form $\Id - S,$ with $S \in \dCI(\bhs{hmf};\Hom(E)).$  Thus, adding a smooth extension of $N_{hmf}(H_0)(-S)$ off $\bhs{hbf}$ to $H_0$ we can assume that $H_0$ in fact satisfies the equation
\begin{equation*}
	t(\pa_t+\Delta_{\ed})H_0 = -R_{\infty} \in \dCI(HX_{s};\Hom(E))= \dCI(X^2_{b,s}\times\bbR^+;E).
\end{equation*}
As before, in terms of the convolution product, this equation can be written
$$
(\pa_t+\Delta_{\ed})H_0 = \Id - t^{-1}R_{\infty}$$
where now $\Id - t^{-1}R_{\infty}$ has inverse of the form $\Id-S$ with $S\in \dCI(HX_{s};\Hom(E))$.
It follows that the heat kernel satisfies
\begin{equation*}
	H_{\ed} = H_0(\Id - S) \in \rho^{\alpha} \CI(HX_{s};\Hom(E))
\end{equation*}
as required.

\end{proof}

\subsection{The trace of the heat kernel} \label{sec:TraceHeatKer}

Lidskii's theorem \cite[Proposition 4.55]{MelroseAPS} shows that the trace of an operator with kernel 
$\cK_G \in \rho^{\gamma}\CI(HX_{s};\Hom(E)\otimes \beta_{H,R}^*(\tfrac1{\sqrt{x^2+\eps^2}}\Omega(M)))$ acting by $f\mapsto (\wt\beta_{H,L})_*(\cK_G\cdot\wt\beta_{H,R}^*f),$ when it exists, is given by
\begin{equation*}
	\Tr(G) = \int_M \tr (\beta_{H,b*}\cK_G)\rest{\diag_M\times [0,1]_{\eps}\times \bbR^+_{\tau}} 
\end{equation*}
where $\tr$ denotes the trace in $\Hom(E)$ and the integral is the push-forward along 
\begin{equation*}
	p_{\eps,\tau}: M\times [0,1]_{\eps}\times \bbR^+_\tau \lra [0,1]_{\eps}\times \bbR^+_\tau.
\end{equation*}
The trace is a function of $\eps$ and $\tau,$ but will not generally be smooth in $\eps$ and $\tau.$

We can use our description of the heat kernel to determine the regularity of its trace.
First note that 
\begin{equation*}
	\beta_{H,b}^{\sharp}(\diag_M\times [0,1]_{\eps\times \bbR^+_{\tau}}) = \Delta_{HX}
\end{equation*}
so that the trace is equal to
\begin{equation*}
	p_{\eps,\tau*}\tr (\beta_{H,b*}\cK_G)\rest{\diag_M} 
	= (p_{\eps,\tau}\circ\beta_{\Delta, (1)})_*\tr (\cK_G\rest{\beta_{H,b}^{\sharp}(\diag_M)}).
\end{equation*}
Since the map $p_{\eps,\tau}\circ\beta_{\Delta,(1)}:\Delta_{HX} \lra [0,1]_\eps\times \bbR^+_{\tau}$ is not a $b$-fibration, we will need to resolve it.
Let us set
\begin{equation*}
\begin{gathered}
	\wt \Delta_{HX} = [ \Delta_{HX}; \bhs{hmf}\cap\bhs{tf} ] = [\Delta_{HX}; \beta_{\Delta,b}^{\sharp}(\{ \eps=0, \tau=0 \})] \\
	\sE\sT = \lrspar{ [0,1]_{\eps} \times \bbR^+_{\tau}; \{ \eps, \tau =0 \} }
\end{gathered}
\end{equation*}
so that the map $p_{\eps,\tau}$ lifts to a map
\begin{equation*}
	\pi_{\eps,\tau}:\wt \Delta_{HX} \lra \sE\sT
\end{equation*}
which is a $b$-fibration, see figure \ref{fig:pushforward}.

\begin{figure}
\centering
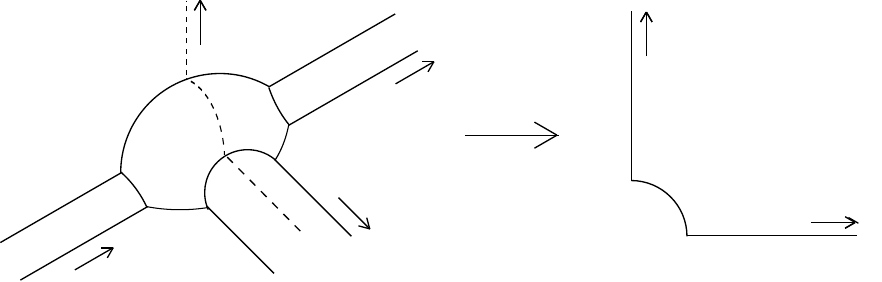
\caption{The pushforward from $\wt\Delta_{HX}\lra\sE\sT$, given by integration in $x$. Note that $\sE\sT$ may be identified with the interior lift of $x=0$ in $\wt\Delta_{HX}$. }
\label{fig:pushforward}
\end{figure}

Let us denote the blow-down maps by
\begin{equation*}
	\xymatrix{
	& \wt\Delta_{HX}\ar[ld]_-{\beta_{\wt\Delta}} \ar[rd]^-{\beta_{\wt\Delta, (1)}} \\
	\Delta_{HX} \ar[rr]^-{\beta_{\Delta, (1)}} & & M\times [0,1]_{\eps} \times \bbR^+_{\tau} } \quad \Mand \quad
	\xymatrix{ \sE\sT \ar[d]^{\beta_{\eps,\tau}} \\ [0,1]_{\eps}\times \bbR^+_{\tau},}
\end{equation*}
so that we have a commutative diagram
\begin{equation}\label{eq:DiagDeltasEsT}
	\xymatrix{
	\wt\Delta_{HX} \ar[rr]^-{\beta_{\wt\Delta, (1)}} \ar[d]_-{\pi_{\eps,\tau}} & & M \times [0,1]_{\eps} \times \bbR^+_{\tau} \ar[d]^{p_{\eps,\tau}} \\
	\sE\sT \ar[rr]^-{\beta_{\eps,\tau}} \ar[dr]_-{\pi_{\eps}} & & [0,1]_{\eps} \times \bbR^+_\tau \ar[dl] \\
	& [0,1]_{\eps} & }
\end{equation}
in which the horizontal arrows are blow-down maps and the vertical arrows are $b$-fibrations.
The trace of $G$ is given by
\begin{equation*}
	p_{\eps,\tau*}\tr (\beta_{H,b*}\cK_G)\rest{\diag_M} 
	= (\beta_{\eps,\tau}\circ \pi_{\eps,\tau})_* \beta_{\wt\Delta}^*(\tr \cK_G\rest{\Delta_{HX}} )
\end{equation*}
and so we see that $\beta_{\eps,\tau}^*\Tr(G)$ will be a polyhomogeneous function on $\sE\sT.$\\

Let us label the new boundary hypersurface of $\wt\Delta_{HX}$ compared to $\Delta_{HX}$ by 
\begin{equation*}
	\bhs{\eps,\tau}= \beta_{\wt\Delta}^{-1}(\bhs{hmf}\cap\bhs{tf})
\end{equation*}
and label the other boundary hypersurfaces of $\wt \Delta_{HX}$ with the same label as the corresponding face in $\Delta_{HX}.$
Let us label the boundary hypersurfaces of $\sE\sT$ by
\begin{equation*}
	\bhs{tf} = \beta_{\eps,\tau}^{\sharp}(\{ \tau =0 \}), \quad
	\bhs{tff} = \beta_{\eps,\tau}^{-1}(0,0), \quad
	\bhs{af} = \beta_{\eps,\tau}^{\sharp}( \{ \eps = 0 \}).
\end{equation*}
Note that the boundary hypersurfaces of $\wt \Delta_{HX}$ are related to those of $\sE\sT$ by
\begin{equation*}
\begin{gathered}
	\pi_{\eps,\tau}^{-1}(\bhs{tf}(\sE\sT)) = \bhs{tf}(\wt \Delta_{HX}), \quad
	\pi_{\eps,\tau}^{-1}(\bhs{tff}(\sE\sT)) = \bhs{tff}(\wt \Delta_{HX}) \cup \bhs{\eps,\tau}(\wt\Delta_{HX}), \\
	\pi_{\eps,\tau}^{-1}(\bhs{af}(\sE\sT)) = \bhs{hbf}(\wt \Delta_{HX}) \cup \bhs{hmf}(\wt \Delta_{HX}).
\end{gathered}
\end{equation*}
So the push-forward along $\pi_{\eps,\tau}$ may introduce log terms at $\bhs{af}$ from accidental multiplicities at $\bhs{hbf} \cap \bhs{hmf},$ and 
log terms at $\bhs{tff}$ from accidental multiplicities at $\bhs{tff} \cap \bhs{\eps,\tau}.$
Let us compute the push-forward for $G$ as above.
First note that
\begin{multline*}
	\cK_G \in \rho^{\gamma}\CI(HX_{s};\Hom(E)\otimes \beta_{H,R}^*(\tfrac1{\sqrt{x^2+\eps^2}}\Omega(M))) \\
	\implies
	\beta_{\wt\Delta}^*(\tr \cK_G\rest{\Delta_{HX}})
	\in \rho^{\wt\gamma}\CI(\wt\Delta_{HX}; \beta_{\wt\Delta,(1)}^*(\tfrac1{\sqrt{x^2+\eps^2}}\Omega(M) ) )
\end{multline*}
where $\wt\gamma_{\eps,\tau} = \gamma_{hmf}+\gamma_{tf}.$
Next if we denote $\cK_G = \wt\cK_G \beta_{H,R}(\frac1{\sqrt{x^2+\eps^2}}\mu_M)$ and multiply by $\beta_{\wt\Delta,(1)}^*(\mu_{[0,1]_{\eps}\times\bbR^+_\tau})$ we get
\begin{multline*}
	\beta_{\wt\Delta}^*(\tr \cK_G\rest{\Delta_{HX}})\beta_{\wt\Delta,(1)}^*(\mu_{[0,1]_{\eps}\times\bbR^+_\tau}) \\
	= 
	\beta_{\wt\Delta}^*(\tr \wt \cK_G\rest{\Delta_{HX}})
	\beta_{\wt\Delta,(1)}^*\lrpar{\frac1{\sqrt{x^2+\eps^2}}}
	\beta_{\wt\Delta,(1)}^*(\mu_{M \times [0,1]_{\eps}\times\bbR^+_\tau}) \\
	\in \rho^{\wt\gamma}(\rho_{hbf}\rho_{tff})^{-1}(\rho_{hbf}\rho_{\eps,\tau}\rho_{tff}^2)\CI(\wt\Delta_{HX};\Omega(\wt\Delta_{HX}))
\end{multline*}
which implies by the pushforward theorem of \cite[Theorem~5]{me1} that 
\begin{equation*}
\begin{gathered}
	\beta_{\eps,\tau}^*(\Tr(G) \; d\eps d\tau)
	\in \cA^{\wt \gamma_{tf}, (\wt \gamma_{tff}\bar \cup \wt\gamma_{\eps,\tau})+1, \wt\gamma_{hbf} \bar\cup \wt\gamma_{hmf}}(\sE\sT;\Omega(\sE\sT)) \\
	=  \cA^{\wt\gamma_{tf}, \wt\gamma_{tff} \bar\cup \wt\gamma_{\eps,\tau}, \wt\gamma_{hbf} \bar\cup \wt\gamma_{hmf}}(\sE\sT;\beta_{\eps,\tau}^*\Omega([0,1]_{\eps}\times \bbR^+_{\tau}))
\end{gathered}
\end{equation*}
and we can cancel out the densities on both sides.
For the heat kernel this yields
\begin{equation}\label{eq:HeatTrace}
	\beta_{\eps,\tau}^*\Tr(e^{-t\Delta_{\ed}})
	\in   \cA^{-m, -h-1 \bar\cup -m, 0 \bar\cup 0}(\sE\sT).
\end{equation}

This gives a complete understanding of the trace of the heat kernel for finite time as $\epsilon\searrow 0$.  We also need to have a uniform control of the trace of the heat kernel $t\to \infty$ and $\epsilon\searrow 0$, which is the object of the next proposition.

\begin{proposition}
Suppose that the family $\eth_{\ed}$ also satisfies Assumption 2 in \S\ref{rc.0}.  Then there exists a positive constant $C$ such that for $t>1$,
\begin{equation}
  \Tr(e^{-t\Delta_{\ed}}-\Pi_{\sma})= c_{-1}(t)\log \epsilon +c_{0}(t) + \mathcal{O}(e^{-Ct}),
\label{rt.2}\end{equation}
where 
\begin{gather}
\label{rt.3} c_{-1}(t)= \frac{1}{2\pi} \int_{-\infty}^{\infty} \Tr( I(e^{-tD_b^2},\lambda))d\lambda= \mathcal{O}(e^{-Ct}), \\
\label{rt.4} c_0(t)= {}^R\!\Tr(e^{-tD^2_{d}}-\Pi_{\ker_{L^2}D_{d}}) + {}^R\!\Tr(e^{-tD_b^2}-\Pi_{\ker_{L^2}D_b})= \mathcal{O}(e^{-Ct}),
\end{gather}
where $D_d= N_{mf}(D_{\ed})$.
\label{rt.1}\end{proposition}
\begin{remark}
If we assume furthermore that $\ker D_v=\{0\}$ so that $e^{-tD_{d}}$ is trace class and $D_{d}$ has a discrete spectrum, then the proof is significantly simpler and leads to a stronger result, namely for $t>1$,
$$
     \Tr(e^{-t\Delta_{\ed}}-\Pi_{\sma})= \Tr(e^{-tD^2_{d}}-\Pi_{\ker_{L^2}D_{d}}) + \mathcal{O}(e^{-Ct}\epsilon)= \mathcal{O}(e^{-Ct})
$$
as $\epsilon\searrow 0$.  
\end{remark}
\begin{proof}
Note that $I(e^{-tD_b^2},\lambda) = e^{-t(\lambda^2+ I(D^2_b,0))}$ and $I(D^2_b,0)$ is a Laplace-type operator on the compact manifold $Y$ which by Assumption 2 in \S\ref{rc.0} has no null space; it follows that $c_{-1}(t)=\mathcal{O}(e^{-Ct})$ for some positive constant $C.$ 
Now, because $\tr( e^{-tD_d^2} )\rest{\diag}$ and $\tr( e^{-tD_b^2} )\rest{\diag}$ are $b$-densities (for $t>0$), we know that it is the leading term of $e^{-t\Delta_{\ed}}$ at $\bhs{hmf}\cap\bhs{hbf}$ which prevents $e^{-tD_{d}^2}$ and $e^{-tD_b^2}$ from being trace class in general.  We will denote the leading term of $e^{-t\Delta_{\ed}}$ at $\bhs{hmf}\cap\bhs{hbf}$ by $\left.e^{-t\Delta_{\ed}}\right|_{\bhs{hmf}\cap{\bhs{hbf}}}.$
We wish to isolate this term so we will define 
\begin{equation*}
	A \in \CI \lrpar{ HX_s\rest{t\geq 1}; \beta_{H,R}^*\lrpar{ \Omega_{\epsilon,b}(X_s)} \otimes \Hom(E) }
	\Mwith \left.A \right|_{\bhs{hmf}\cap{\bhs{hbf}}} = \left.e^{-t\Delta_{\ed}}\right|_{\bhs{hmf}\cap{\bhs{hbf}}}.
\end{equation*}
We can identify $HX_s\rest{t\geq 1} = X^2_{b,s} \times [1,\infty)_t$ and think of $A$ as a map
\begin{equation*}
	[1,\infty) \ni t \mapsto A(t) \in \CI \lrpar{ X^2_{b,s}; \beta_{R}^*\lrpar{ \Omega_{\epsilon,b}(X_s)} \otimes \Hom(E) }.
\end{equation*}

Let $c: \bhs{mf}\cap\bhs{bf}\times [0,1)^2\to X^2_{b,s}$ be a tubular neighborhood of $\bhs{mf}\cap\bhs{bf}$ in $X^2_{b,s}$ and let $\pr_1: \bhs{mf}\cap\bhs{bf}\times [0,1)^2\to \bhs{mf}\cap\bhs{bf}$ be the projection on the first factor.  If $\chi\in\CI_c(\bhs{mf}\cap\bhs{bf}\times [0,1)^2)$ is a choice of cut-off function such that 
$\chi\equiv 1$ in a neighborhood of $\bhs{mf}\cap\bhs{bf}$, then we can define an extension map
$$
       \begin{array}{lccl}
       \Xi: & \CI(\bhs{mf}\cap\bhs{bf}) & \to &\CI(X^2_{b,s}) \\
              & f & \mapsto & c_*(\chi \pr_1^*f)
       \end{array}
$$
such that $\left. \Xi(f)\right|_{\bhs{mf}\cap\bhs{bf}}=f$.  More generally, for any vector bundle $V\to X^2_{b,s}$, we use a smooth vector bundle isomorphism $\nu: c^*V\to \pr_1^*(\left. V\right|_{\bhs{mf}\cap\bhs{bf}})$ to define an extension map
$$
 \begin{array}{lccl}
       \Xi: & \CI(\bhs{mf}\cap\bhs{bf};V) & \to &\CI(X^2_{b,s};V) \\
              & \sigma& \mapsto & c_*(\nu^{-1}(\chi \pr_1^*\sigma))
       \end{array}
$$
such that $\left. \Xi(\sigma)\right|_{\bhs{mf}\cap\bhs{bf}}=\sigma$.  In particular, using such an extension map and the fact that $\left.HX_s\right|_{\{t>1\}}\cong X^2_{b,s}\times (1,\infty)$, we can for $t>1$ define a smooth extension of $\left.e^{-t\Delta_{\ed}}\right|_{\bhs{hmf}\cap{\bhs{hbf}}}$ to $HX_s$ by
$$
        A(t):= \Xi(\left.e^{-t\Delta_{\ed}}\right|_{\bhs{hmf}\cap{\bhs{hbf}}}),
$$
so that $\left. A(t)\right|_{\bhs{hmf}\cap\bhs{hbf}}= \left.e^{-t\Delta_{\ed}}\right|_{\bhs{hmf}\cap{\bhs{hbf}}}$.  Since the extension map $\Xi$ does not depend on time, the advantage of considering $A(t)$ instead of 
$e^{-t\Delta_{\ed}}$ is that the dependence of $A$ in $t$ is completely determined by $\left.e^{-t\Delta_{\ed}}\right|_{\bhs{hmf}\cap{\bhs{hbf}}}.$

Recall that  $\left.e^{-t\Delta_{\ed}}\right|_{\bhs{hmf}\cap{\bhs{hbf}}}$ can be identified with the heat kernel of the indicial family $I_b(\Delta_{\ed})$ which, by Assumption 2 in \S\ref{rc.0}, is a Fredholm $b$-operator. Thus, from \cite[Proposition 7.36]{MelroseAPS}, as elements of $\Psi^{-\infty}_b([M;H];E)$ and $\Psi^{-\infty}_b(H\times[-1,1];E)$,  $\left. A(t)\right|_{\bhs{hmf}}$ and $\left. A(t)\right|_{\bhs{hbf}}$ are uniformly $\mathcal{O}(e^{-Ct})$ for some positive constant $C$. 
In particular, the renormalized traces of $\left. A(t)\right|_{\bhs{hmf}}$ and $\left. A(t)\right|_{\bhs{hbf}}$ are $\mathcal{O}(e^{-Ct})$ for some positive constant $C$.  Thus, to show that ${}^R\!\Tr(e^{-tD^2_{d}}-\Pi_{\ker_{L^2}D_{d}})$ and ${}^R\!\Tr(e^{-tD_b^2}-\Pi_{\ker_{L^2}D_b})$ are $\mathcal{O}(e^{-Ct})$, it suffices to show that 
\begin{gather}
\label{rt.5} \| e^{-tD^2_{d}}-\Pi_{\ker_{L^2}D_{d}}- \left.A(t)\right|_{\bhs{hmf}}\|_{\Tr}= \mathcal{O}(e^{-Ct}), \\
\label{rt.6} \| e^{-tD^2_{b}}-\Pi_{\ker_{L^2}D_{b}}- \left.A(t)\right|_{\bhs{hbf}}\|_{\Tr}= \mathcal{O}(e^{-Ct}),
\end{gather}
where $\|\cdot\|_{\Tr}$ is the trace norm.
For \eqref{rt.5}, writing $B(t)= e^{-tD^2_{d}}-\Pi_{\ker_{L^2}D_{d}}$ and $\widetilde{A}(t)= \left.A(t)\right|_{\bhs{hmf}}$, we have that
\begin{equation}
\begin{aligned}
\| B(t)- \widetilde{A}(t)\|_{\Tr}&= \| B(\tfrac{t}2)(B(\tfrac{t}2)-\widetilde{A}(\tfrac{t}2)) + \wt A(\tfrac t2)(B(\tfrac{t}2)-\widetilde{A}(\tfrac{t}2)) -(\widetilde{A}(t)- \widetilde{A}(\tfrac{t}2) \widetilde{A}(\tfrac{t}2)) \|_{\Tr}\\
&\le \left(  \| B(\tfrac{t}2)\| + \|\widetilde{A}(\tfrac{t}2)\|\right)\|B(\tfrac{t}2)-\widetilde{A}(\tfrac{t}2)\|_{\Tr} + \|\widetilde{A}(t)- \widetilde{A}(\tfrac{t}2)\widetilde{A}(\tfrac{t}2)\|_{\Tr},
\end{aligned}
\label{rt.7}\end{equation}
where $\|\cdot \|$ is the operator norm.  By Corollary~\ref{rc.20}, equation \eqref{rcmm.4} and the spectral theorem, we know that 
$\| B(\frac{t}2)\|= \mathcal{O}(e^{-\frac{Ct}2})$ for some positive constant $C$.  Directly from the definition of $\widetilde{A}(t)$, $\widetilde{A}(\frac{t}2)$ is $\mathcal{O}(e^{-\frac{Ct}2})$ as an element of $\Psi^{-\infty}_b([M;H];E)$ for some positive constant $C$.  Thus, using Schur's lemma, \eg as in the proof of \cite[Theorem~{5.34}]{MelroseAPS}, we have that
$\| \widetilde{A}(\frac{t}2)\|= \mathcal{O}(e^{-\frac{Ct}2})$ for some positive constant $C$.  Similarly, since $\widetilde{A}(t)$ and $\widetilde{A}(\frac{t}2)^2$ have the same indicial family, we have that  their difference is uniformly  $\mathcal{O}(e^{-Ct})$ in $|x|\Psi^{-\infty}_b([M;H];E)$ for some positive constant $C$.  In particular, their difference 
is trace class with
$$
 \|\widetilde{A}(t)- \widetilde{A}(\tfrac{t}2)\widetilde{A}(\tfrac{t}2)\|_{\Tr}= \mathcal{O}(e^{-Ct})
 $$     
 for some positive constant $C.$  Thus, we can find a positive constant $C$ such that 
\begin{equation}
  \| B(t)- \widetilde{A}(t)\|_{\Tr}\le e^{-\frac{Ct}2}\left( \| B(\tfrac{t}2)-\widetilde{A}(\tfrac{t}2)\|_{\Tr}+1 \right), \quad \forall \ t>\frac{1}{C}.
  \label{rt.8}\end{equation}
We can then use this estimate recursively to first show that $\| B(t)- \widetilde{A}(t)\|_{\Tr}$ is $\mathcal{O}(1)$ as $t\to \infty$.  Using this fact, we can then use the estimate once more to show that there is a positive constant $C$ such that  
\begin{equation}
     \| B(t)- \widetilde{A}(t)\|_{\Tr}\le \frac{e^{-Ct}}{C}, \quad \forall \ t>1,
     \label{rt.9}
\end{equation}
establishing \eqref{rt.5}.  To establish \eqref{rt.6}, we can proceed exactly in the same way or use instead \cite[Proposition~7.36]{MelroseAPS}.  By the discussion above, these estimates give \eqref{rt.4}.  

Finally, to see that \eqref{rt.2} holds, notice that by Theorem~\ref{rc.7} and the spectral theorem, there exists $C>0$ such that $\|e^{-t\Delta_{\ed}}-\Pi_{\sma}\|=\mathcal{O}(e^{-Ct})$ uniformly in $\epsilon\ge 0$.  By our choice of $A(t)$, there is a positive constant such that $A(t)=\mathcal{O}(e^{-Ct})$ as an element of $\Psi^{-\infty}_{\epsilon,b}(X_s;E)$, in particular $\| A(t)\|= \mathcal{O}(e^{-Ct})$ uniformly in $\epsilon\ge 0$ for some positive constant $C$.    Thus, replacing $B(t)$ with $e^{-t\Delta_{\ed}}-\Pi_{\sma}$ and $\widetilde{A}(t)$ with $A(t)$ in \eqref{rt.7} and \eqref{rt.8}, 
we can apply the argument leading to \eqref{rt.9} uniformly
 in $\epsilon$ to find a positive constant $C$ such that for all $\epsilon\ge 0$,
\begin{equation}
  \| e^{-t\Delta_{\ed}}-\Pi_{\sma} - A(t)\|_{\Tr}\le \frac{e^{-Ct}}C \quad \forall \ t>1.  
\label{rt.10}\end{equation}
Thus, \eqref{rt.2} follows from this estimate and the fact that for some positive constant $C$,  
$$
   \Tr(A(t))= c_{-1}(t)\log\epsilon + c_2(t)+ \mathcal{O}(e^{-Ct}\epsilon\log\epsilon)
$$
where 
$$
   c_2(t)= {}^R\!\Tr(\left.A(t)\right|_{\bhs{hmf}})+ {}^R\!\Tr(\left. A(t) \right|_{\bhs{hbf}}).
$$

\end{proof}

\subsection{Symmetry for even metrics} \label{sec:EvenMets}

There is a class of metrics for which the asymptotic behavior of $D_{\dR}$ as $\rho = \sqrt{x^2+\eps^2} \to 0$ is particularly well-behaved. 
In this section we describe this class of metrics and the resulting asymptotic behavior and then we use this to simplify the asymptotics of the heat trace.

Notice first that a product-type $\ed$-metric in a tubular neighborhood $\sT$ of $H$ in $M$ induces in $\beta_{(1)}^{-1}(\sT\times [0,1]_{\eps})$ a decomposition
$$
  {}^{\ed}T^*X_s= {}^{\ed} T^*_HX_s \oplus {}^{\ed} T^*_VX_s
$$
in terms of horizontal and vertical forms with respect to the fiber bundle  
$$
[(-1,1)_x\times [0,1]_{\eps};\{0\}\times \{0\}]\times H\to [(-1,1)_x\times [0,1]_{\eps};\{0\}\times \{0\}]\times Y
$$ 
induced by $\phi: H\to Y$.  In particular, $\frac{dx}{\rho^2}$ is a horizontal form in this decomposition.  

\begin{definition}\label{def:EvenEd}
We say that $g_{\ed}$ is an {\bf even $\ed$-metric } if it is an $\ed$-metric differing from  a product-type $\ed$-metric $g_{pt}$ by elements of $\rho^2\CI(\beta_{(1)}^{-1}(\sT\times [0,1]_{\eps}); S^2({}^{\ed}T_{H}^*X_s))$ and $\rho^2\CI(\beta_{(1)}^{-1}(\sT\times [0,1]_{\eps}); S^2({}^{\ed}T_{V}^*X_s))$ having only even powers of $\rho$ in their expansion at $\bhs{\bs}$, where the decomposition in terms of horizontal and vertical forms is the one induced by $g_{pt}$.   We say that $g_{\ed}$ is an {\bf even $\ed$-metric to order $\ell$} if there is an even $\ed$-metric $g_{\ed,\ev}$ such that
\begin{equation*}
	g_{\ed} - g_{\ed,\ev} \in \rho^{\ell}\CI(X_s; S^2(\Ed TX_s)). 
\end{equation*}
\end{definition}

\begin{definition}\label{def:EvenFMet}
Let $F\lra X_s$ be a vector bundle with flat connection $\nabla^F$ and $g_F$ a bundle metric, not necessarily compatible.
We say that $g_F$ is even in $\rho$ if its Taylor expansion at $\bhs{sb}$ has only even powers of $\rho.$
\end{definition}

Let $\CI_{\ev} \lrpar{ HX_s; \beta_{H,R}^*\lrpar{ \Omega_{\epsilon,b}(X_s)} \otimes \Hom(\Lambda(\Ed T^*X_s)\otimes F) }$ be the subspace of 
\begin{equation}
\CI \lrpar{ HX_s; \beta_{H,R}^*\lrpar{ \Omega_{\epsilon,b}(X_s)} \otimes \Hom(\Lambda(\Ed T^*X_s)\otimes F) }
\label{ev.1}\end{equation}  consisting of elements $\kappa$ having an expansion at $\bhs{tff}$ of the form
\begin{equation}
     \kappa \sim \sum_{j=0}^{\infty} \rho_{tff}^j \kappa_j 
\label{ev.2}\end{equation}
with $\kappa_j$ a smooth section on $\bhs{tff}$ such that 
$$
     \kappa_j(\theta,y,z,-S',-U',z',\sigma)= (-1)^{j+\bN_{H/Y}} \kappa_j(\theta,y,z,S',U',z',\sigma)
$$
in the coordinates \eqref{eq:CoordsTff}, where the number operator $N_{H/Y}$ gives the shift in vertical degree induced by $\kappa_j$.  Similarly, let $\CI_{\odd} \lrpar{ HX_s; \beta_{H,R}^*\lrpar{ \Omega_{\epsilon,b}(X_s)} \otimes \Hom(\Lambda(\Ed T^*X_s)\otimes F) }$ be the subspace of  \eqref{ev.2} consisting of elements $\kappa$ having an expansion at $\bhs{tff}$ of the form
\begin{equation}
     \kappa \sim \sum_{j=0}^{\infty} \rho_{tff}^j \kappa_j 
\label{ev.3}\end{equation}
with this time $\kappa_j$ a smooth section on $\bhs{tff}$ such that 
$$
     \kappa_j(\theta,y,z,-S',-U',z',\sigma)= (-1)^{j+\bN_{H/Y}+1} \kappa_j(\theta,y,z,S',U',z',\sigma)
$$
in the coordinates \eqref{eq:CoordsTff}.

\begin{proposition}\label{eveness.1}
Let $\alpha$ be the family of index sets from the construction of the heat kernel \eqref{eq:DefAlpha}.
If $g_{\ed}$ and $g_F$ are even metrics, the heat kernel of $D_{\ed}^2$ has an (infinite-order) parametrix at $\bhs{tff}$ in
\begin{equation}\label{eq:EEETff}
	\rho^{\alpha}\CI_{\ev} \lrpar{ HX_s; \beta_{H,R}^*\lrpar{ \Omega_{\epsilon,b}(X_s)} \otimes \Hom(\Lambda(\Ed T^*X_s)\otimes F) }.
\end{equation}
If $g_F$ is even and $g_{\ed}$ is an even $\ed$-metric to order $\ell \geq 2,$ then the heat kernel has 
a parametrix at $\bhs{tff}$ in \eqref{eq:EEETff} plus a term in 
\begin{equation*}
	\rho^{\alpha}\rho_{tff}^{\ell}\CI \lrpar{ HX_s; \beta_{H,R}^*\lrpar{ \Omega_{\epsilon,b}(X_s)} \otimes \Hom(\Lambda(\Ed T^*X_s)\otimes F) }.	
\end{equation*}
\end{proposition}

\begin{proof}
Let us review our implementation of the Hadamard parametrix construction in constructing the heat kernel, focusing on $\bhs{tff}.$
We first established that the action of $\beta_{H,L}^*(t\pa_t + tD_{\dR}^2)$ at this face was through
\begin{equation*}
	N_{tff}(t\pa_t+tD_{\dR}^2) = \frac12 \sigma\pa_\sigma + \sigma^2 N_{ff}(\rho^2D_{\dR}^2)
\end{equation*}
where $\sigma = \tau/\rho$ is a rescaled time variable and with initial condition given by the model heat kernel at $\bhs{tf}.$
Thus the model heat kernel at $\bhs{tff}$ is the product of the heat kernels on $\bbR^{h+1}$ and the vertical family of heat kernels on $Z.$
We extend this model heat kernel to a section over $HX_s$ in any way we like and when we apply $\beta_{H,L}^*(t\pa_t + t D_{\dR}^2)$ to this extended section we get a section that vanishes to first order at $\bhs{tff}.$ 
The first non-vanishing coefficient $w$ induces our next model problem: we look for a section $u$ such that $N_{tff}(t\pa_t+tD_{\dR}^2)u=w$ and this gives us an improved parametrix. This process is continued until we have a parametrix that vanishes to infinite order at $\bhs{tff}.$\\

Thus, if we denote the space in \eqref{eq:EEETff} by $\sE_{\ev},$ to prove the proposition it suffices to show that the parametrix construction can be done entirely in $\sE_{\ev}.$  
First we point out that the model $N_{tff}(e^{-tD_{\dR}^2})$ given in \eqref{eq:AllNormalOps} is clearly even, so can be extended to an element of $\sE_{\ev}$.  Now, if we also set 
$$
    \sE_{\odd}:= \rho^{\alpha}\CI_{\odd} \lrpar{ HX_s; \beta_{H,R}^*\lrpar{ \Omega_{\epsilon,b}(X_s)} \otimes \Hom(\Lambda(\Ed T^*X_s)\otimes F) }, 
$$ 
then it follows from \eqref{cc.3}, \eqref{cc.4} and \cite[Proposition~10.1]{BGV} that the left action $\tau d_F$ interchanges parity, that is, it induces maps
$$
  \beta_{H,L}^*(\tau d_F): \sE_{\ev}\to \sE_{\odd}, \quad \beta_{H,L}^*(\tau d_F): \sE_{\odd}\to \sE_{\ev}.
 $$ 
 For its adjoint $d^*_F$ given in \eqref{eq:MullerDelta}, notice that since $g_{\ed}$ and $g_F$ are even metrics, the operator $\#$ will preserve parity, while $*$ will preserve or reverse parity depending on whether $v$ is even or odd, so that  $\tau d_F^*$ will also interchange parity.  Hence, we see that $\tau D_{\ed}$ interchanges parity, so that $\tau^2 D_{\ed}$ preserves parity.  On the other hand, the action on the left of $\sigma \pa_{\sigma}$ clearly preserves parity, so that the left action of the heat operator preserves parity and induces a map
\begin{equation}
           \beta_{H,L}^*( t\pa_t + tD^2_{\ed}): \sE_{\ev}\to \sE_{\ev}.
\label{ev.4}\end{equation}              
This means that  the parametrix construction can be done in $\sE_{\ev}$ as claimed.  Now, if we only assume that $g_{\ed}$ is even up to order $\ell$, then \eqref{ev.4} remains true  up to order $\ell$, namely $\beta_{H,L}^*( t\pa_t + tD^2_{\ed})$ maps
\begin{equation}
         \sE_{\ev}+ \rho^{\alpha}\rho_{tff}^{\ell}\CI \lrpar{ HX_s; \beta_{H,R}^*\lrpar{ \Omega_{\epsilon,b}(X_s)} \otimes \Hom(\Lambda(\Ed T^*X_s)\otimes F) }
\label{ev.5}\end{equation}         
to itself, so that the parametrix construction can indeed be performed in \eqref{ev.5} as claimed.
\end{proof}

The key consequence is that we can refine the expansion of the pointwise trace of the heat kernel of an even $\ed$-metric from
\begin{equation*}
	\tr\lrpar{e^{-t\Delta}\rest{\Delta_{HX}}} \sim \lrpar{ \rho_{tff}^{-h-1}\sum A_{k} \rho_{tff}^{k} }\mu
\end{equation*}
which follows from the heat kernel construction, to
\begin{equation}\label{eq:EvenDVanishingTff}
	\tr\lrpar{e^{-t\Delta}\rest{\Delta_{HX}}} \sim 
	\lrpar{ \rho_{tff}^{-h-1}\sum A_{2k} \rho_{tff}^{2k} }
\end{equation}
since odd terms in $S'$ and $U'$ vanish on the diagonal and do not contribute to the trace.
If the metric is even to order $\ell$ then the asymptotics of the pointwise trace are
\begin{equation*}
	\tr\lrpar{e^{-t\Delta}\rest{\Delta_{HX}}} \sim 
	\lrpar{ \rho_{tff}^{-h-1}\sum_{2k<\ell} A_{2k} \rho_{tff}^{2k} 
	+\rho_{tff}^{-h+\ell-1}\sum A'_k \rho_{tff}^k }\mu
\end{equation*}
In particular, this has the following consequence.

\begin{corollary} \label{eveness.2}
If $g_F$ is even in $\rho,$ $g_{\ed}$ is an even $\ed$-metric to order $h+2,$ and $Y$ is even dimensional
then $\tr\lrpar{e^{-t\Delta}\rest{\Delta_{HX}}}$ does not have a constant term at $\bhs{tff}.$
\end{corollary}

Next let us consider the contribution of the expansion of $\beta_{\wt\Delta}^*\tr\lrpar{e^{-t\Delta}\rest{\Delta_{HX}}}$ at $\bhs{\eps,\tau}(\wt\Delta_{HX})$ to the trace of the heat kernel.
We know that 
$\tau^m\tr\lrpar{e^{-t\Delta}\rest{\Delta_{HX}}}$ is smooth in $\tau$ and $\eps$ near $\bhs{hmf} \cap \bhs{tf}.$ The expansion in $\tau$ is of the form
\begin{equation*}
	\sum_{k\geq 0} a_k \tau^{2k}
\end{equation*}
with each $a_k$ a polynomial in the curvature of the product-type metric and its covariant derivatives.
Hence the joint expansion of $\tau^m\tr\lrpar{e^{-t\Delta}\rest{\Delta_{HX}}}$ in $\tau$ and $\eps$ is of the form
\begin{equation*}
	\sum_{k, \ell \geq 0} a_{k,\ell} \tau^{2k}\eps^{\ell}.
\end{equation*}
If we lift now to $\wt\Delta_{HX}$ by introducing the coordinates $\kappa = \sqrt{\tau^2+\eps^2},$ $\gamma = \tan^{-1}(\eps/\tau)$ this expansion becomes
\begin{equation*}
	\kappa^m \cos^m(\gamma) \beta_{\wt\Delta}^*\tr\lrpar{e^{-t\Delta}\rest{\Delta_{HX}}} 
	\sim \sum_{k,\ell \geq 0} a_{k,\ell} \kappa^{2k+\ell}\cos^{2k}(\gamma)\sin^{\ell}(\gamma).
\end{equation*}
Thus we have that, for even $\ed$-metrics, the trace of the heat kernel satisfies
\begin{equation}\label{eq:ImprovedTraceAsymp}
	\beta_{\eps,\tau}^*\Tr(e^{-t\Delta}) \in \cA^{-m, (-m+\bbN_0) \bar\cup (-h-1+2\bbN_0), 0 \bar\cup 0}(\sE\sT)
\end{equation}
and for $g_F$ even and $g_{\ed}$ even to order $\ell\geq 2,$
\begin{multline}\label{eq:ImprovedTraceAsymp.2}
	\beta_{\eps,\tau}^*\Tr(e^{-t\Delta}) \in \cA^{-m, (-m+\bbN_0) \bar\cup (-h-1+Ev(0, \ell)), 0 \bar\cup 0}(\sE\sT)\\
	\Mwith Ev(0,k) = \{ n \in \bbN_0: n \geq k \Mor n\text{ is even and $<k$} \}.
\end{multline}

This shows that the expansion at $\bhs{\eps,\tau}$ induces terms in the expansion of $\beta_{\eps,\tau}^*\Tr(e^{-t\Delta})$ at $\bhs{tff}(\sE\sT).$ 
However the heat kernel is polyhomogeneous on $\Delta_{HX}$ before blowing-up $\bhs{hmf}\cap\bhs{tf}$ and so the induced terms in the expansion at $\bhs{tff}(\sE\sT)$ are also lifts of polyhomogeneous expansions on $[0,1]_{\eps}\times\bbR^+_\tau.$ We record this as a lemma.

\begin{lemma}\label{lem:ImprovedTraceAsymp}
If $\chi$ is a smooth function on $X_s$ supported away from $\bhs{sb}$ then
\begin{equation*}
	\Tr(\chi e^{-t\Delta}) \in \tau^{-m}\CI([0,1]_{\eps} \times \bbR^+_\tau)
\end{equation*}
\end{lemma}

\begin{proof}
The lift of $\chi$ to $\wt \Delta_{HX}$ is supported away from $\bhs{tff} \cup \bhs{hbf}$ hence $\beta_{\wt\Delta}^*((\beta_{H,L}^*\chi e^{-t\Delta})\rest{\Delta_{HX}})$ is the pull-back of a density with coefficients in $\tau^{-m} \CI(M\times[0,1]_{\eps}\times \bbR^+;\hom(E)).$
The result follows from the commutativity of the diagram
\begin{equation*}
	\xymatrix{
	\wt\Delta_{HX} \ar[r]^-{\beta_{\wt\Delta, (1)}} \ar[d]_-{\pi_{\eps,\tau}} &  M \times [0,1]_{\eps} \times \bbR^+_{\tau} \ar[d]^{p_{\eps,\tau}} \\
	\sE\sT \ar[r]^-{\beta_{\eps,\tau}}  & [0,1]_{\eps} \times \bbR^+_\tau }
\end{equation*}
\end{proof}

Below we shall make use of the proof of this lemma to simplify the asymptotics of analytic torsion for even metrics.

\section*{Torsion under degeneration}

\section{The $R$-torsion on manifolds with boundary} \label{sec:RTorsion}

Let us first recall basic facts about Reidemeister torsion.  For more details, the reader is encouraged to look at \cite{Milnor1966, Cheeger1979, Dar, Muller1993}.  Let $C_*$ be a finite dimensional  real chain complex 
\begin{equation}
\xymatrix{
     0\ar[r] & C_n\ar[r]^-{\pa_n}& C_{n-1} \ar[r]^-{\pa_{n-1}}&\cdots \ar[r]^{\pa_{1}} & C_0\ar[r]^-{d_0}& 0.
}
\label{rto.1}\end{equation}
Suppose that we are given a preferred basis $c_i$ of $C_i.$  As usual, let $Z_i$ denote the kernel of $\pa_i$ and let $B_i$ be the image of $\pa_{i+1}$ so that the homology group of degree $i$ is given by
$$
     \mathrm H_i(C_*)= Z_i/ B_i.
$$ 
Choose a basis $b_i$ of $B_i$ and let $h_i$ be a basis of $\mathrm H_i(C_*).$  Under the decomposition $C_i= B_i \oplus \mathrm H_i(C)\oplus B_{i-1},$ we see that the bases $b_i,$ $h_i$ and $b_{i-1}$ induce another basis of $C_i$ that we will denote $b_ih_ib_{i-1}.$
\begin{definition}
For a preferred basis $c_i$ of $C_i,$ a basis $h_i$ of $\mathrm H_i(C_*)$ and any choice of basis $b_i$ of $B_i,$ we define the $R$-torsion of the complex $C_*$ by
$$
    \tau(c_i,h_i):= \prod_{i=0}^n [b_ih_ib_{i-1}|c_i]^{(-1)^i},
$$
where $[u|v]:= |\det W|$ for $u,v$ two basis of a real vector space $V$ related by $\displaystyle u_i= \sum_{j}W_{ij} v_j.$  
\label{rto.2}\end{definition}
As can be checked directly, this definition does not depend on the choice of the basis $b_i.$  
Instead of the complex $C_*,$ we can look at the dual complex of cochains $C^*,$
\begin{equation}
\xymatrix{
     0\ar[r]& C^0 \ar[r]^-{d_0}& C^1\ar[r]^-{d_1}& \cdots \ar[r]^-{d_{n-1}}& C_n\ar[r]^-{d_n}& 0,
}
\label{rto.3}\end{equation}
where $C^i$ is the dual of $C_i.$  As a preferred basis of $C^i,$ we can simply take the basis $c^i$ dual to $c_i.$  Denote by $Z^i$ the kernel $d_i$ and by $B^i$ the image of $d_{i-1}$ so that the cohomology groups of $C^*$ are given by $\mathrm H^i(C^*)=Z^i/B^i.$ 
\begin{definition}
For a preferred basis $c^i$ of $C^i,$ a basis $h^i$ of $\mathrm H^i(C^*)$ and any choice of basis $b^i$ of $B^i,$ we define the $R$-torsion of $C^*$ by
$$
     \tau(c^i,h^i)= \prod_{i=0}^n [c^i|b^i h^i b^{i+1}]^{(-1)^i}.
$$
\label{rto.4}\end{definition}  
As the following example shows, although the $R$-torsion is defined in terms of real vector spaces, it can recapture information about the torsion when those vector spaces come from free abelian groups, explaining in part the origin of the terminology.  
\begin{example}
Let $0\to F^0\to F^{1}\to\cdots \to F^n\to 0$ be a complex of free abelian groups and set $C^i= F^i\otimes_{\bbZ}\bbR.$  In this case, basis for $F^i$ and $\mathrm H^i(F^*,\bbZ)$ induces preferred basis $c^i$ for $C^i$ and $h^i$ for $\mathrm H^i(C^*).$  For such a choice, the $R$-torsion of the complex $C^*$ is given by 
\begin{equation}
    \tau(c^i, h^i)= \prod_{i=0}^n O_i^{(-1)^{i-1}},
\label{rto.6}\end{equation}
where $O_i$ is the order of the torsion subgroup of $\mathrm H^i(F^*,\bbZ),$ see \cite[(1.4)]{Cheeger1979} for details.  
\label{rto.5}\end{example}
In the previous example, we can in particular take $F^*$ to be the complex of cochains of a $CW$-complex, in which case the torsion \eqref{rto.6} is a homotopy invariant (though generally it is not).  More generally, we can twist by flat vector bundles to get other invariants.  More precisely, let $K$ be a finite $CW$-complex and let $C_*(K)$ be the cellular chain complex associated to $K.$  Let $\pi_1(K)$ be the fundamental group of $K.$  Then $\pi_1(K)$ naturally acts on the universal cover $\widetilde{K}$ of $K.$  This induces an action of $\pi_1(K)$ on the corresponding cellular chain complex $C_*(\widetilde{K})$ giving it the structure of a $\bbZ\pi_1(K)$-module.  In fact, for each $i,$ $C_*(\widetilde{K})$ is a free $\bbZ\pi_1(K)$-module finitely generated by a choice of lifts $\{\widetilde{e}^j_i\}$ of the $i$-cells $\{e^j_i \}$ of $K.$  Let now $\alpha: \pi_1(K)\to \GL(k,\bbR)$ be a representation of the fundamental group of $K.$  We will assume that  it is unimodular, which means $|\det(\alpha(\nu)|=1$ for all $\nu\in \pi_1(K)$.  This ensures that the action of $\pi_1(K)$ preserves the volume of $\bbR^k$ for its canonical volume form.  Notice that orthogonal and unitary representations are special examples of unimodular representations.  In particular, $\alpha$ gives $\bbR^k$ the structure of a $\bbZ\pi_1(K)$-module, so that we can consider the finite dimensional real chain complex
\begin{equation}
        C_i(K; \alpha)= C_i(\widetilde{K})\otimes_{\bbZ\pi_1(K)}\bbR^k.
\label{rto.7}\end{equation}
  By taking a basis $\{x_{\ell}\}$ of $\bbR^k,$ we then have a preferred basis $\{\widetilde{e}^j_i\otimes x_{\ell}\}$ of $C_i(K; \alpha).$ We will assume in fact that $\{x_{\ell}\}$ is a \textbf{unimodular} basis of $\bbR^k$, that is, of volume one in the sense that in terms of the canonical basis $\{v_i\}$ of $\bbR^n$, $x_{\ell}=\sum_iA_{\ell i}v_i$ with matrix $A$ such that  $|\det A|=1$. \begin{definition}
For a choice of basis $h_*(\alpha)$ of the homology of the complex $C_*(K; \alpha),$ we define the $R$-torsion $\tau(K; \alpha, h_*(\alpha))$ associated to $K$ and $\alpha:\pi_1(K)\to \GL(k,\bbR)$ to be the $R$-torsion
$\tau(\{\widetilde{e}^j_i\otimes x_{\ell}\},h_*(\alpha))$ of the complex $C_*(K; \alpha).$  Similarly, for a choice of basis $h^*(\alpha)$ of the cohomology of the dual complex $C^*(K; \alpha),$ we define the $R$-torsion 
$\tau(K; \alpha,h^*(\alpha))$ to be the $R$-torsion $\tau(c^*, h^*(\alpha))$ of the complex $C^*(K; \alpha)$ with $c^i$ the basis of $C^i(K; \alpha)$ dual to $\{\widetilde{e}^j_i\otimes x_{\ell}\}.$  
\label{rto.8}\end{definition}
Thanks to the fact that $\alpha$ is unimodular, the definition of the $R$-torsion does not depend on the choice of lifts $\widetilde{e}^j_i$.
On the other hand, for a non-trivial unimodular representation $\alpha,$ there is no reason for the $R$-torsion $\tau(K,\alpha,h_*(\alpha))$ to be a homotopy invariant.  It is however a combinatorial invariant, that is, it is invariant under subdivision of $K,$ see \cite{Milnor1966, Muller1993} for a proof.  In fact, when the complex $C_*(K; \alpha)$ is acyclic, it is shown in \cite{Chapman1974} to be even a topological invariant.

To compute the $R$-torsion of a complex, one important tool is the following formula due to  Milnor. 

\begin{theorem}[Milnor \cite{Milnor1966}]
Let 
\begin{equation}
0\to C'\to C\to C''\to0
\label{rto.8a}\end{equation} be a short exact sequence of finite dimensional real cochain complexes with preferred bases $c^i$ $(c')^i$ and $(c'')^i$ such that $c^i= (c')^i (c'')^i.$ If $h,$ $h'$ and $h''$ are choices of bases for the cohomology 
of $C,$ $C'$ and $C'',$ then
\begin{equation}
   \frac{\tau(c,h)}{\tau(c',h') \tau(c'',h'')}= \tau(\cH)
\label{rto.10}\end{equation}
where $\cH$ is the acyclic complex given by the long exact sequence in cohomology associated to the short exact sequence \eqref{rto.8a} and $\tau(\cH)$ is the $R$-torsion of $\cH$ with preferred basis given by $h,h'$ and $h''.$       
\label{rto.9}\end{theorem}
\begin{remark}
We will not need it, but a similar formula holds for short exact sequences of finite dimensional real chain complexes. 
\end{remark}

This can be used to see what happens to the $R$-torsion when a fibered  cusp surgery is performed on a closed manifold $M$ along a hypersurface $H\subset M$ with trivial normal bundle and equipped with a fibration $\phi: H\to Y$ over a closed manifold $Y.$   For a choice of triangulation on $M$ compatible with the decomposition 
$$
M= \sm \cup_{H\sqcup H} H\times [0,1],
$$
we have a short exact sequence of complexes
\begin{equation}
\xymatrix{
  0 \ar[r] & C^*(\sm,\pa \sm; \alpha) \ar[r] & C^*(M; \alpha) \ar[r] &   C^*(H\times [0,1]; \alpha)\ar[r] & 0.
}
\label{ac.3}\end{equation} 
If we assume that $\mathrm H^*(H; \alpha)=0,$ then the long exact sequence in cohomology associated to \eqref{ac.3} gives a natural identification
\begin{equation}
        \mathrm H^*(\sm,\pa\sm; \alpha)\cong \mathrm H^*(M; \alpha).
\label{ac.4}\end{equation}
Thus, we see from \eqref{ac.4} that a choice of basis $h^*(\alpha)$ for $\mathrm H^*(\sm,\pa\sm; \alpha)$ induces one for $\mathrm H^*(M; \alpha).$  This can be used to relate the $R$-torsion on $M$ and $\sm$ as follows.
\begin{theorem}
Let $M$ be a closed manifold with a hypersurface $H\subset M$ having a trivial normal bundle.  Let $\phi: H\to Y$ be a fibration with $Y$ a closed manifold.  Suppose $\alpha:\pi_1(M)\to \GL(k,\bbR)$ is a unimodular representation such that $\mathrm H^*(\phi^{-1}(y); \alpha)=\{0\}$ for each $y\in Y$.  Then $\mathrm H^*(H;\alpha)=0$ and for any choice of basis $h^*(\alpha)$ for 
$\mathrm H^*(M; \alpha),$ we have the relation
\begin{equation}
\tau(M,\alpha,h^*(\alpha))= \tau(\sm,\pa\sm,\alpha,h^*(\alpha))\tau(H,\alpha)
\label{ac.5c}\end{equation}
where $h^*(\alpha)$ is seen as a basis of $H^*(\sm,\pa\sm;\alpha)$  via the identification \eqref{ac.4}. 
\label{ac.5}\end{theorem}
\begin{proof}
Using the Leray-Serre spectral sequence of the fiberbundle  $\phi: H\to Y$, we see that $\mathrm H^*(H,\alpha)=0$.   Thus,
with the given choices of cohomology bases, we obtain \eqref{ac.5c} by applying Theorem~\ref{rto.9} to  \eqref{ac.3}.
\end{proof}

In this setting, we can also relate $L^2$-cohomology with the cohomology of the manifold with boundary.  More precisely, let $N$ be a smooth manifold with boundary such that $\pa N$ is equipped with a fiber bundle $\phi: \pa N\to Y$, where the base $Y$ and the fibers of $\phi$ are closed manifolds.  Let $g_d$ be a fibered cusp metric on $N\setminus \pa N$ compatible with the fiber bundle $\phi:\pa N\to Y$.   Let $\alpha: \pi_1(N)\to \GL(k,\bbR)$ be a unimodular representation and let $F\to N$ be the corresponding flat vector bundle of rank $k$ on $N$.  Let $g_F$ be a smooth Euclidean metric for $F\to N$.  Finally suppose that $F$ is \textbf{strongly acyclic at infinity}, namely that 
\begin{equation}
        \mathrm H^*(\phi^{-1}(y);F)=0   \quad \forall \ y\in Y.
\label{sac.1}\end{equation}
As in the proof of Theorem~\ref{ac.5}, we see from the Leray-Serre spectral sequence of the fiber bundle $\phi: \pa N\to Y$ that $H^*(\pa N;F)=0$.  Hence, we see from the long exact sequence associated to the pair $(N,\pa N)$ that there is a natural identification
$$
     \mathrm H^*(N;F)\cong \mathrm H^*(N,\pa N;F).
$$
Now, since the vector bundle $F$ is flat, we can consider the $L^2$-complex
\begin{equation}
\xymatrix{
  \cdots \ar[r] & L^2_{d}\Omega^{i-1}(N\setminus\pa N, F, g_d)\ar[r]^-{d} & L^2_{d}\Omega^{i}(N\setminus\pa N,F, g_d)\ar[r]^-{d} &L^2_{d}\Omega^{i+1}(N\setminus\pa N,F, g_d)\ar[r]^-{d} & \cdots} 
\label{l2.1}\end{equation}
where 
$$
    L^2_d\Omega^i(N\setminus \pa N,F,g_d)= \{ \omega\in L^2\Omega^i(N\setminus \pa N,F,g_d) \; | \; d\omega\in L^2\Omega^{i+1}(N\setminus \pa N,F,g_d)  \}
$$
and $L^2\Omega^i(N\setminus\pa N,F,g_d)$ is the space of forms of degree $i$ on $N\setminus \pa N$ with values in $F$ which are $L^2$ with respect to the metric $g_d$ and the metric $g_F$ of $F.$  Recall that the $L^2$-cohomology, denoted $H^*_{(2)}(N,F,g_d),$ is by definition the cohomology of the complex \eqref{l2.1}.  

\begin{proposition}
If $F$ is strongly acyclic at infinity, then there are natural identifications
$$
\mathrm H^*_{(2)}(N,F,g_d)\cong \mathrm H^*(N;F)\cong \mathrm H^*(N,\pa N;F).
$$
\label{l2.2}\end{proposition}
\begin{proof}
For $p\in \pa N$, let $\cU$ be an open neighborhood diffeomorphic to $[0,1)_x\times \bbB^h\times Z$ such that the restriction of $g_d$ to $\cU\setminus (\cU\cap\pa N)$ is quasi-isometric to $ dx^2 + g_{\bbB^h} + x^2 g_Z$, where $h=\dim Y$, $\bbB^h$ is an open ball in $\bbR^h$, $Z=\phi^{-1}(p)$, $g_{\bbB^h}$ is the restriction of the Euclidean metric on $\bbB^h$ and $g_Z$ is some Riemannian metric on $Z$.  Using the $L^2$-Künneth formula of \cite{Z1982} and the assumption that $F$ is strongly acyclic at infinity, we see that 
$$
        \mathrm H^*_{(2)}(\cU\setminus (\cU\cap \pa N),F,g_d)= \{0\}.
$$
The result then follows by applying a standard argument using commutative diagrams of Mayer-Vietoris sequences.  
\end{proof}
In particular, the $L^2$-cohomology is finite dimensional.  Combining with  the Kodaira decomposition gives the following.
\begin{corollary}
If $F$ is strongly acyclic at infinity, then the exterior derivative $d$ and its formal adjoint $\delta$ have closed range and induce the Kodaira decomposition
$$
   L^2\Omega^k(N,F,g_d)=  L^2\cH^k(N,F,g)\oplus  d\left( L^2_d\Omega^{k-1}(N,F,g_d) \right) \oplus  \delta\left( L^2_{\delta}\Omega^{k+1}(N,F,g_d) \right)
$$  
where $L^2\cH^k(N,F,g_d)$ is the space of $L^2$-harmonic forms of degree $k$  taking values in $F$ and 
$$
    L^2_{\delta}\Omega^i(N\setminus \pa N,F,g_d)= \{ \omega\in L^2\Omega^i(N\setminus \pa N,F,g_d) \; | \; \delta\omega\in L^2\Omega^{i-1}(N\setminus \pa N,F,g_d)  \}.  
$$
In particular, $d+\delta$ is a Fredholm operator  and there are canonical identifications
$$
      L^2\cH^k(N,F,g)\cong \mathrm H^k_{(2)}(N,F,g)\cong H^*(N;F)\cong \mathrm H^*(N,\pa N;F).
$$
\label{l2.3}\end{corollary}

\section{The intersection $R$-torsion of Dar and $L^2$-cohomology}

To discuss the intersection $R$-torsion of Dar \cite{Dar}, we need first to make a quick review of the intersection homology of Goresky and MacPherson in \cite{GM1980}.
Its original purpose  was to give an answer to  the following problem posed by
Sullivan \cite{Sullivan}: Can one find a class of spaces with singularities
for which the signature of manifolds extends as a cobordism invariant.
The class of singular spaces that Goresky and MacPherson considered were 
stratified pseudomanifolds with only even codimension strata, for instance
complex varieties.  The signature in this context was defined via intersection
homology, which is defined more generally for every stratified pseudomanifold.
To define intersection homology, we need to work in the piecewise linear 
category.  Therefore, all the spaces involved in the definitions of this
subsection are assumed to be piecewise linear.

In general, a \textbf{stratification} of a pseudomanifold $X$ is a filtration
by closed subspaces
\begin{equation}
   X_{0}\subset X_{1}\subset \cdots \subset X_{m-3}\subset X_{m-2}=X_{m-1}
\subset X_{m}=X
\label{ih.1}\end{equation} 
such that for each $p\in X_{i}\setminus X_{i-1},$ there is a filtered
space
\begin{equation}
 p=V_{i}\subset \cdots \subset V_{m-1}\subset V_{m}=V
\label{ih.2}\end{equation}
and a piecewise linear mapping $V\times B^{i}\to X$ which for each
$j$ takes $V_{j}\times B^{i}$ piecewise linearly homeomorphically to a 
neighbourhood of $p\in X_{j}.$  Here, $B^{i}$ is the piecewise linear
ball of dimension $i$ and $p$ corresponds to $V_{i}\times \{\tilde{p}\}$
where $\tilde{p}\in B^{i}$ is some interior point.

If $C^{T}_{*}(X)$ denotes the chain complex of simplicial chains with respect
to a triangulation $T,$ then one can compare chains 
$c\in C^{T}_{*}(X)$ and $c'\in C^{T'}_{*}(X)$ coming from two different
triangulations by looking at their canonical images in 
$C^{T''}_{*}(X)$ where $T''$ is a common refinement of $T$ and $T'.$
We denote by $C_{*}(X)$ the complex of piecewise linear chains given 
by taking the direct limit of the $C^{T}_{*}(X)$ over all triangulations
of $X$ compatible with the piecewise linear structure.

A \textbf{perversity} is a sequence of numbers $\bp=(p_{2},\ldots,p_{m-2})$
such that $p_{2}=0$ and $p_{k+1}=p_{k}$ or $p_{k}+1.$  If $i$ is an
integer and $\bp$ is a perversity, a subspace $N\subset X$ is 
said to be \textbf{$(\bp,i)$-allowable} if $\dim N\le i$ and 
$\dim (N\cap X_{m-k})\le i-k+p_{k}$ for all $k\ge 2.$  Denote
by $\IC^{\bp}_{i}(X)$ the subgroup of $C_{i}(X)$ consisting of
the chains $\xi$ such that $|\xi |$ is $(\bp,i)$-allowable and
$|\pa \xi|$ is $(\bp,i-1)$-allowable.     
\begin{definition}
The $i$th \textbf{intersection homology group} of perversity $\bp,$
denoted $\IH^{\bp}_{i}(X)$ is the $i$th homology group of the chain
complex $\IC^{\bp}_{*}(X).$
\label{ih.3}\end{definition}
One of the important features of intersection homology is that it is not
in general a homotopy invariant.  On the other hand, a non-trivial fact proved
in \cite{GM1980} is 
that for a given perversity $\bp,$ it does not depend on the choice of  
stratification \eqref{ih.1}.  The choice of terminology `intersection 
homology' comes from the fact that for perversities
$\bp,$ $\bq,$ and $\br$ such that $\bp+\bq\le \br,$ there is an intersection
product
\begin{equation}
  \IH^{\bp}_{i}\times \IH^{\bq}_{j} \to \IH^{\br}_{i+j-n}(X)
\label{ih.4}\end{equation} 
which corresponds to the intersection of two cocycles when they are in
`general position'.  In particular, this gives rise to a generalized
Poincar\'e duality.  If $\bp$ and $\bq$ are complementary perversities,
meaning that $\bp+\bq=\bt$ where $\bt=(0,1,\ldots,m-2)$ is the 
largest possible perversity, and if $i$ and $j$ are of complementary
dimensions ($i+j=m$), then the pairing
\begin{equation}
  \IH^{\bp}_{i}(X)\times \IH^{\bq}_{j}(X) \to \IH^{\bt}_{0}(X)\to \bbZ
\label{ih.5}\end{equation}
is non-degenerate when tensored with $\bbQ.$

If $\alpha: \pi_1(X)\to \GL(k,\bbR)$ is a unimodular representation, we can also form the complex
$$
                     \IC^{\bp}_*(X;\alpha)= \IC^{\bp}_*(\tX)\otimes_{\bbZ\pi_1(X)} \bbR^k,
$$  
\begin{definition}
The intersection homology groups $\IH^{\bp}_*(X; \alpha)$ are the homology groups of the complex $\IC^{\bp}_*(X; \alpha).$  Similarly, the cohomology groups $\IH^*_{\bp}(X; \alpha)$ are the cohomology groups of the complex $\IC^*_{\bp}(X; \alpha)$ dual to the complex $\IC^{\bp}_*(X; \alpha).$  
\label{dt.11}\end{definition}

Suppose now that $X$ admits a stratification of depth at most 1, that is, $X_k=B \; \forall \dim B\le k\le n-2$ and $X_k=\emptyset$ for $k<\dim B$ for some compact manifold $B$ of dimension $b\le n-2.$   Suppose further that there is a smooth manifold with boundary $N$ and a fiber bundle  $\phi: \pa N\to B$ such that $X$ is obtained from $N$ by collapsing the fibers of $\phi$ onto their bases.  In this case, we can relate  the intersection cohomology with local coefficients with the $L^2$-cohomology with local coefficients of a fibered cusp metric on $N\setminus \pa N=X\setminus B$.  First recall that by the Riemann-Hilbert correspondence, we can naturally associate to the unimodular representation $\alpha: \pi_1(X)\to \GL(k,\bbR)$ a flat vector bundle $F$ of rank $k$ on $N.$  Let us also equip $F$ with a metric $g_F$ so that it is a smooth Euclidean vector bundle on the manifold with boundary $N$.  Notice however that unless the representation $\alpha$ is orthogonal, this Euclidean structure cannot be compatible with the flat connection.      

On the other hand, let $gd$ be a fibered cusp metric on  $N\setminus \pa N$ compatible with the fiber bundle $\phi:\pa N\to B$ and a choice of boundary defining function $x\in\CI(N).$  Since the vector bundle $F$ is flat, we can then consider the $L^2$-complex \eqref{12.1} and the corresponding $L^2$-cohomology.
 Assuming that the flat vector bundle $F$ is Witt and that $g_F$ is a smooth metric on the manifold with boundary $N$, we can easily establish the following relation with intersection cohomology.
 \begin{proposition}
 Suppose  that the flat vector bundle $F$ is Witt.  Then there is a natural isomorphism
 $$
         \mathrm H^*_{(2)}(N,F,g)\cong \IH_{\bm}^*(X; \alpha)    
 $$
 where $\bm(k)= \lfloor \frac{k-1}2\rfloor$ is the upper middle perversity.  
 \label{dt.15}\end{proposition}
 \begin{proof}
 When $F$ is trivial, this is a standard result, see \cite{hhm}.  Otherwise, identify $N\setminus \pa N$ with $X\setminus B$.  Then there is a closed manifold $L$ such that any point $p\in B$ has a neighborhood $\cV \subset X$ homeomorphic $\bbB^b\times \cC L$ where $\bbB^b$ is the ball of dimension $b=\dim B$ in $\bbR^b$ and $\cC L $ is the cone over $L$.  Let $\{\cU_i\}$ be an open cover of $B$ in $X$ such that any non-empty intersection $\cU_I= \cap_{i\in I} \cU_i$ is homeomorphic to $\bbB^b\times L$.
 Appealing to the $L^2$-K\"unneth formula of \cite{Z1982} and using the definition of the intersection complex, we have that for any intersection $\cU_I=\cap_{i\in I} \cU_i,$ there are natural isomorphisms
 \begin{equation}
           \mathrm H^i_{(2)}(\cU_I\setminus \cU_I\cap B, F, g)\cong \IH_{\bm}^i(\cU_I; \alpha)\cong \left\{ \begin{array}{ll}\mathrm H^i(L; F),  & i\le \frac{v}2 \\
                               \{0\}, & i\ge \frac{v}2,  \end{array} \right.
 \label{dt.16}\end{equation}
 where $v=\dim L.$  We can then apply the five-lemma to suitable commutative diagrams of Mayer-Vietoris sequences  to conclude that there is a natural isomorphism
 $$
       \mathrm H^*_{(2)}(N,F,g)\cong \IH_{\bm}^*(X; \alpha).
 $$  
 \end{proof}
Since the $L^2$-cohomology is in particular finite dimensional, this has the following immediate consequence.
\begin{corollary}
Suppose that the flat vector bundle $F$ is Witt and $g_F$ is a smooth bundle metric on the manifold with boundary $N$.  Then the exterior derivative $d$ and its formal adjoint $\delta$ have closed range and we have the Kodaira decomposition
$$
   L^2\Omega^k(N,F,g)=  L^2\cH^k(N,F,g)\oplus  d\left( L^2_d\Omega^{k-1}(N,F,g) \right) \oplus  \delta\left( L^2_{\delta}\Omega^{k+1}(N,F,g) \right)
$$  
where $L^2\cH^k(N,F,g)$ is the space of $L^2$-harmonic forms taking values in $F$ and 
$$
    L^2_{\delta}\Omega^i(N\setminus \pa N,F,g)= \{ \omega\in L^2\Omega^i(N\setminus \pa N,F,g) \; | \; \delta\omega\in L^2\Omega^{i-1}(N\setminus \pa N,F,g)  \}.  
$$
In particular, $d+\delta$ is a Fredholm operator  and there is a canonical identification
$$
      \IH_{\bm}^*(X; \alpha) \cong \mathrm H^k_{(2)}(N,F,g)\cong L^2\cH^k(N,F,g).
$$
\label{kd.1}\end{corollary}

Under the assumption that $B$ is zero dimensional, we can relate $L^2$-cohomology and intersection cohomology with absolute and relative cohomology.

\begin{lemma} Suppose  that the flat vector bundle $F$ is Witt with bundle metric $g_F$ smooth on $N$.  If  $B$ is zero dimensional,  there is a natural isomorphism 
$$
         \mathrm H^k_{(2)}(N, F, g)\cong \IH^k_{\bm}(X; F)\cong \left\{  \begin{array}{ll} \mathrm H^k(N;F), &  k\le \frac{n-1}2,  \\
                                                                                                     \Im\left[ \mathrm H^k(N,\pa N;F) \to \mathrm H^k(N;F)\right], & k=\frac{n}2, \\
                                                                                                                                           \mathrm H^k(N,\pa N; F) & k> \frac{n-1}2. \end{array} \right.
$$ 
where $\IH^k_{\bm}(X;F):=  \IH^k_{\bm}(X; \alpha).$  
\label{se.1} \end{lemma}
\begin{proof}
The isomorphism between intersection cohomology and the $L^2$ cohomology is given in Proposition~\ref{dt.15}.
For $k\le \frac{n-1}2,$ the other isomorphism follows by applying the five-lemma to the commutative diagram of long exact sequences
\begin{equation}
\xymatrix{
  \cdots \ar[r] & \mathrm H^k(N,\pa N;F) \ar[r] \ar[d]  &   \mathrm H^k(N;F) \ar[r] \ar[d]  & \mathrm H^k(L;F) \ar[r]\ar[d]  &  \cdots \\
  \cdots  \ar[r] & \mathrm H^k(N,\pa N;F) \ar[r]   & \IH^k_{\bm}(X; F) \ar[r] & \IH^k_{\bm}(\cC_{\phi}\pa N; F) \ar[r]  & \cdots. 
}
\label{se.1a}\end{equation}
Note in particular that the Witt condition on $F$ implies that the map $\mathrm H^{k}(L ;F)\to \IH^k_{\bm}(\cC_{\phi}\pa N;F) $ is an isomorphism for $k\le \frac{n-1}2$ by \eqref{dt.16}.  For $k=\frac{n}2$ when $n$ is even, the result also follows from \eqref{se.1a}.  Finally, for $k>\frac{n}2,$ notice by \eqref{dt.16} that   $\IH^k_{\bm}(\cC_{\phi}\pa N; F)=\{0\}$ for $k\ge\frac{n-1}2,$ so we can just use the bottom long exact sequence in \eqref{se.1a} to obtain the isomorphism.  
\end{proof}

In order to define a $R$-torsion out of intersection homology, one needs to describe it in terms of a finitely generated complex of chains.  A natural choice is to take the basic sets $R^{\bp}_i$ of \cite[\S 3.4]{GM1980}.

\begin{definition}
Let $X$ be a pseudomanifold  of dimension $n$ with a fixed stratification and let $\bp$ be a choice of perversity.  Let $T$ be a triangulation of $X$ subordinate to the stratification, that is, such that $X_k$ is a subcomplex of $T$ for each $k.$  Let $T'$ be the first barycentric subdivision of $T$ and denote by $R^{\bp}_i$ the subcomplex of $T'$ consisting of all simplices which are $(\bp,i)$-allowable.  
\label{dt.1}\end{definition}

For such basic sets $R^{\bp}_i,$ there is a natural identification 
\begin{equation}
\IH^{\bp}_i(X) \cong \Im( \mathrm H_i(R^{\bp}_i)\to \mathrm H_i(R^{\bp}_{i+1})).
\label{dt.1b}\end{equation}
This suggests considering the abelian group
$$
        \cR^{\bp}_i(X)= \mathrm H_i(R^{\bp}_i, R^{\bp}_{i-1})
$$
finitely generated by chains $\{e^j_i\}$ in $R^{\bp}_i$ with boundary in $R^{\bp}_{i-1}.$  These groups form a complex with boundary map $\pa_i: \cR^{\bp}_i(X)\to \cR^{\bp}_{i-1}(X)$  induced by the boundary map of the long exact sequence in homology
$$
\xymatrix{
  \cdots \ar[r] & \mathrm H_i(R^{\bp}_{i-1}) \ar[r] & \mathrm H_i(R^{\bp}_i) \ar[r] 
  	& \mathrm H_i(R^{\bp}_i,R^{\bp}_{i-1}) \ar[r]^{\pa_i} & \mathrm H_{i-1}(R^{\bp}_{i-1}) \ar[r] & \cdots.
}
$$
In terms of this complex, intersection homology is canonically given by $\IH^{\bp}_i(X)\cong \mathrm H_{i}(\cR^{\bp}_*(X)).$  This can be proved directly from \eqref{dt.1b}.  

Let now $\tX$ be the universal cover of $X$ and let $\tR_i^{\bp}$ be the lift of $R^{\bp}_i$ to $\tX.$  The fundamental group of $X,$ denoted $\pi_i(X),$ naturally acts on $\tX$ and $\tR^{\bp}_i$ and we can as before define a complex by
$$
     \cR^{\bp}_i(\tX)= \mathrm H_i(\tR^{\bp}_i, \tR^{\bp}_{i-1}).
$$    
This complex is a free $\bbZ\pi_1(X)$-module generated by lifts of chains $\{e^j_i\}$ in $\cR^{\bp}_i(X).$  Let $\{\widetilde{e}^j_i\}$ be a preferred basis of $\cR^{\bp}_i(\tX)$ coming from a choice of lifts of the chains $\{e^j_i\}.$  
\begin{remark}
It is not completely clear a priori that one can choose the chains $\{e^j_i\}$ to be simply connected, so an obstruction to lift cannot be automatically ruled out.  However, in the case that  $X$ has one singular stratum $B$ of dimension 0, which is the situation considered in \cite{ARS2}, the chains $\{e^j_i\}$ can be chosen to be simply connected, in fact contractible.  Indeed, each of them can be chosen to be either a simplex not intersecting $B$, or else a chain retracting to one of the connected components of $B$.   
\label{gD.1}\end{remark}
Suppose that we are now given a unimodular representation of the fundamental group $\alpha:\pi_1(X)\to \GL(k,\bbR).$  This gives $\bbR^k$ the structure of a $\bbZ\pi_1(X)$-module so that we can introduce a twisted version of the complex $\cR^{\bp}_i(X),$ namely,
$$
      \cR^{\bp}_i(X; \alpha)= \cR^{\bp}_i(\tX)\otimes_{\bbZ\pi_1(X)} \bbR^k.
$$
Clearly, $\cR^{\bp}_i(X; \alpha)$ is a finite dimensional real vector space with preferred basis given by $\{\widetilde{e}^j_i\otimes x_{\ell}\}$ where $\{x_{\ell}\}$ is a choice of unimodular basis for $\bbR^k.$  
\begin{definition}[Dar \cite{Dar}]
With a choice of preferred basis in homology $h_*(\alpha),$ the \textbf{intersection $R$-torsion} $I\tau^{\bp}(X,\alpha,h_*(\alpha))$ associated to the perversity $\bp$ and the unimodular representation $\alpha$  is the torsion of the complex $\cR^{\bp}_i(X; \alpha)$ with preferred basis $\{\widetilde{e}^j_i\otimes x_{\ell}\}$ and basis in homology given by $h_*(\alpha).$  Similarly,  we can define the intersection torsion in cohomology $I\tau^{\bp}(X,\alpha, h^*(\alpha))$ by specifying a basis $h^*(\alpha)$ of the cohomology of the complex $\cR_{\bp}^i(X; \alpha)$ dual to $\cR^{\bp}_i(X; \alpha)$ and by using as a preferred basis  of $\cR_{\bp}^i(X; \alpha)$ the basis  dual to $\{\widetilde{e}^j_i\otimes x_{\ell}\}.$
\label{dt.2}\end{definition}
Using the fact that  the representation $\alpha$ is unimodular, we can proceed as in \cite{Dar} to show that  the intersection $R$-torsion does not depend on the choice of triangulation $T$ and is therefore a combinatorial invariant.

\section{Analytic torsion conventions} \label{sec:ATConventions}

We briefly recall the definition of analytic torsion for closed manifolds and manifolds with fibered  cusps.

Let $(M,g)$ be a closed oriented  Riemannian manifold of dimension $m,$ $F$ a flat bundle over $M$ with connection $\nabla$ and bundle metric $g_F,$ not necessarily compatible.
This data defines a twisted de Rham complex $\Omega^*(M;F)$ and a corresponding Hodge Laplacian 
\begin{equation*}
	\Delta_q: \Omega^q(M;F) \lra \Omega^q(M;F)
\end{equation*}
on forms of degree $q.$ The zeta function of this operator is, for $\Re s\gg 0,$
\begin{equation*}
	\zeta_q(s) = \frac1{\Gamma(s)} \int_0^{\infty} t^s\; \Tr(e^{-t\Delta_q}-\Pi_{\ker \Delta_q}) \; \frac{dt}t
\end{equation*}
where $\Pi_{\ker \Delta_q}$ denotes the orthogonal projection onto the null space of $\Delta_q.$ 
The short-time asymptotic expansion of the heat kernel can be used to see that this integral converges when $\Re s>m/2$ and that the resulting holomorphic function has a meromorphic continuation to the entire complex plane with at worst simple poles, for which $s=0$ is a regular point.
The (logarithm of) analytic torsion is defined in terms of the derivatives of these zeta functions at the origin,
\begin{equation*}
	\lAT(M,g,F, g_F) = \frac12 \sum_{q=0}^m (-1)^q q \zeta_q'(0).
\end{equation*}
If the null spaces $\cH^q(M;F)$ of the Hodge Laplacians are all trivial, and the representation is orthogonal \cite{Ray-Singer} or unimodular \cite{Muller1993}, then the analytic torsion is independent of the choice of metric.
Otherwise, as shown in \cite{Ray-Singer2}, we can remove the dependence on the metric by choosing a basis $\{ \mu_j^q \}$ of $\cH^q(M;F)$ for each $q,$ 
and an orthonormal basis of harmonic forms $\omega$ with respect to the metric $g,$ and then defining
\begin{equation*}
	\bar{\lAT}(M,\{\mu_j^q\},F) = \lAT(M,g;F)  - \log \left(  \prod_{q=0}^{n} [\mu^q |\omega^q]^{(-1)^q} \right),
\end{equation*}
where $[\mu^q|\omega^q ] =|\det W^q|$ with $W^q$ the matrix satisfying $$\displaystyle \mu^q_i= \sum_j W^q_{ij} \omega^q_j.$$ 
Another way of saying this is that $\lAT$ can be thought of invariantly as a metric on the determinant line
\begin{equation*}
	\det \cH^*(M;F) = \bigwedge_{q=0}^m \lrpar{\wedge^{\mathrm{top}} \cH^q(M;F)}^{(-1)^{q+1}},
\end{equation*}
where the inverse of a vector space denotes its dual.\\

For a manifold with  cusp ends, the definition of analytic torsion is similar but slightly more complicated.
Let $M'$ be a manifold with boundary $H$ and $x$ a boundary defining function for $H.$
Assume that $H$ participates in a fiber bundle
\begin{equation*}
	Z \fib H \xlra{\phi} Y
\end{equation*}
and fix a connection for $\phi.$
As above there is a bundle, the $d$-tangent bundle, ${}^d TM'$ whose sections are locally generated by $x\pa_x, \pa_y, \tfrac1x\pa_z$ and a dual bundle, the $d$-cotangent bundle, ${}^d T^*M'$ whose sections are locally generated by $\tfrac{dx}x, dy, x\; dz.$
(If $M$ is a closed manifold with hypersurface $H$ and $M' = M\setminus H$ then ${}^d TM' \cong \Ed TX_s\rest{\bhs{sm}}.$)

A general fibered cusp metric, or $d$-metric, is a bundle metric on ${}^d TM'$ but it is convenient to single out particularly well-behaved $d$-metrics.  First, a \textbf{product-type $d$-metric} is a metric of the form 
\begin{equation*}
	g_d = \frac{dx^2}{x^2} + x^2g_{H/Y} + \phi^*g_Y
\end{equation*}
in a collar neighborhood $\cC\cong [0,1)_x\times H$ of $H$,
with $g_Y$ a metric independent of $x$  on $Y$ and $g_{H/Y}$ a symmetric 2-tensor on $H$ restricting to a Riemannian metric on each fiber of $\phi$.  We say a $d$-metric is \textbf{product-type to order $k$} if it differs from a product-type metrics by a smooth section of $S^2({}^dT^*M')$ which vanishes to order $k$ at $\pa M'$.
 Analogously to Definition \ref{def:EvenEd}, we say that $g_d$ is an \textbf{even} $d$-metric if it is a $d$-metric which in a collar neighborhood $\cC$ of $\pa M'$ differs from a product-type $d$-metric $g_{pt}$ by elements of 
 $x^2\CI(\cU; S^2({}^{d}T_{H}M'))$ and  $x^2\CI(\cU; S^2({}^{d}T_{H}M'))$ having only even powers of $x$ in their expansion at $\pa M'$, where here the decomposition  in $\cC$
  $$
      {}^dT^*M'= {}^dT^*_HM' \oplus {}^dT^*_HM'
 $$
 in terms of vertical and horizontal forms is the one induced by $g_{pt}$ for the fiber bundle $[0,1)_x\times H\to [0,1)_x\times Y$ defined using $\phi:H\to Y$.  More generally, we say it is \textbf{even to order $\ell$} if it differs from even $d$-metric by a smooth section of $S^2({}^dT^*M')$ vanishing to order $\ell$ at $\pa M'$.

Let $(F,\nabla^F)$ be a flat vector bundle over $M'$ and let $g_F$ be a bundle metric on $F,$ not necessarily compatible with $\nabla^F.$
We say that $g_F$ is {\bf even in $x$} if its Taylor expansion at $H$ has only even powers of $x.$
Equivalently, $g_F$ is even in $x$ if when we extend $F$ to the double of $M'$ over $H,$ $g_F$ extends to a smooth bundle metric over the double.\\

We form the twisted $d$-differential forms
\begin{equation*}
	{}^d\Omega^*(M';F) = \CI(X_s; \Lambda^*({}^d T^*M') \otimes F),
\end{equation*}
with corresponding de Rham operators $\eth_{\dR}$
and Hodge Laplacians $\Delta_q.$  
We assume that $(F,\nabla^F, g_F)$ satisfy the hypotheses of Lemma \ref{lem:WittIsFredholm} so that $\eth_{\dR}$ is Fredholm on $L^2_d.$

We can not define the zeta function directly as above because the heat kernel of such a Laplacian is not trace-class.
However, as in \cite{MelroseAPS, Melrose-Nistor, Albin:RenInt}, we can define a renormalized trace by showing that the function
\begin{equation*}
	z\mapsto \Tr(x^z e^{-t\Delta_q})
\end{equation*}
is finite for $\Re z\gg 0$ and has a meromorphic continuation to $\bbC,$ so that we can set
\begin{equation*}
	\RTr{e^{-t\Delta_q}} = \FP_{z=0}\Tr(x^z e^{-t\Delta_q}).
\end{equation*}
This follows from the construction of the heat kernel above (and also from \cite[\S 4]{v}), as does the important fact that this renormalized trace has asymptotic behavior similar to that on closed manifolds (see \cite[Theorem A.1]{Albin-Rochon:ModSpace})
\begin{equation*}
\begin{gathered}
	\Mas t \to \infty, \quad 
	\RTr{ e^{-t\Delta_q} - \Pi_{\ker \Delta_q} } \to 0 \text{ exponentially } \\
	\Mas t \to 0, \quad
	\RTr{e^{-t\Delta_q}} \sim t^{-m/2}\sum_{k\geq 0} a_{k/2} t^{k/2} + t^{-(h+1)/2} \sum_{k\geq 0} b_{k/2}t^{k/2}\log t 
\end{gathered}
\end{equation*}
as long as $(F, g_F)$ satisfy the hypotheses of Lemma \ref{lem:WittIsFredholm}. 
The corresponding zeta function
\begin{equation*}
	\zeta_q(s) = \frac1{\Gamma(s)} \int_0^{\infty} t^s\; \RTr{ e^{-t\Delta_q} - \Pi_{\ker \Delta_q} }\; \frac{dt}t
\end{equation*}
is hence holomorphic for $\Re s>m/2$ with a meromorphic continuation to $\bbC$ with at worst double poles. 
There is a pole at $s=0,$ with residue $b_{h+1/2},$ but if $Y$ is even dimensional our analysis of the heat kernel for even metrics in \S\ref{sec:EvenMets} shows that $b_{h+1/2}$ vanishes if $g_d$ is even to order $\ell \geq h+2.$
For surfaces with cusps, this way of defining the zeta function was used for example in \cite{Albin-Aldana-Rochon:Ricci, Albin-Aldana-Rochon:Compact}.

We define
\begin{equation*}
	\lAT(M',g_d,F) = \frac12\sum_{q=0}^m (-1)^q q \lrpar{ \FP_{s=0} \frac{\zeta_q(s)}s }.
\end{equation*}
Just as on a closed manifold we have the following variation formula.
\begin{lemma}\label{lem:ATVariation}
Let $F\lra M'$ be a bundle with flat connection $\nabla^F$ and bundle metric $g_F,$ not necessarily compatible with $\nabla^F,$ but even in $x.$
Let $\sigma \mapsto g_{\sigma}$ be a family of fibered cusp metrics that are even to order $\ell\geq \dim Y + 2$ and whose difference is $\cO(x)$ as sections of $S^2({}^d T^*M'),$ and let
\begin{equation*}
	\alpha = *^{-1}\pa_{\sigma}\rest{\sigma =0} *.
\end{equation*}
If $m=\dim M$ is odd and $h= \dim Y$ is even then the variation of the analytic torsion is given by
\begin{equation*}
	\pa_\sigma\rest{\sigma=0} \lAT(M',g_{\sigma}, F) 
	=  \frac12\sum_{q=0}^m (-1)^q  \Tr( \alpha_q\Pi_{\ker \Delta_q})
\end{equation*}
where $\alpha_q$ is the restriction of $\alpha$ to forms of degree $q,$
and hence analytic torsion defines a metric on the Hodge cohomology determinant line independent of $\sigma.$
If the flat bundle $F$ is also acyclic, then analytic torsion is a number independent of $\sigma.$
\end{lemma}
\begin{proof}
First, since $g_{\sigma}$ are even to order $\ell\geq 2$ and differ by $\cO(x),$ they actually differ by $\cO(x^2).$
Next let us determine the structure of $\alpha.$
We know that there is a collar neighborhood $\sC = [0,1)\times H$ of $\pa M'$ with respect to which
\begin{equation*}
	g_{\sigma} -\left( \frac{dx^2}{x^2} + x^2 g_{H/Y} + \phi^*g_Y\right) \in x^{2}\CI_{\ev}(\sC; S^2({}^d T^*M'))+ x^{\ell}\CI(\sC; S^2({}^d T^*M')),
\end{equation*}
where $g_{H/Y}$ and $\phi^*g_Y$ are constant in  $x$ and smooth in $\sigma,$ and where $\CI_{\ev}(\sC; S^2({}^d T^*M'))$ is the space of smooth sections with even expansion at $\pa M'$.  
Let $\{ e_i\}$ be a frame for $TH/Y$ orthonormal with respect to $g_{H/Y}\rest{\sigma=0}$ and $\{\eta_\mu\}$ be a frame for $TY$ orthonormal with respect to $g_Y\rest{\sigma=0}.$ Let $\{\bar e_i\}$ and $\{\bar \eta_\mu \}$ be their horizontal lifts to $\sC$ with respect to the connection induced by the metric.
Let $\cF =\{\tfrac1x\bar e_i, \bar \eta_\mu, x\pa_x \}$ be the resulting frame for ${}^dTM'$ (not necessarily orthonormal for any metric) and 
let $\cF^{\flat}$ be the dual frame for ${}^d T^*M'.$
We can write 
\begin{equation*}
	g_\sigma = \sum_{\theta^i, \theta^j \in \cF^{\flat}} g_{ij}(\sigma) \theta^i \wedge \theta^j 
	\Mwith g_{ij} \in \CI_{\ev}(\sC) + x^{\ell}\CI(\sC)
\end{equation*}
where $\CI_{\ev}$ denotes functions even in $x$
and hence, for any $s \in \CI(\sC;F),$
\begin{equation*}
	*( s(\zeta) \theta^{i_1} \wedge \cdots \wedge \theta^{i_p}) 
	= \frac{\sqrt{\det (g_{ij})}}{(m-p)!}g^{i_1\ell_1}\cdots g^{i_p\ell_p} 
	s(\zeta) \epsilon_{\ell_1\cdots \ell_m} \theta^{\ell_{p+1}}\wedge\cdots\wedge\theta^{\ell_{m}}.
\end{equation*}
This shows that
\begin{equation}\label{eq:Alpha}
	\alpha \in 
	\CI_{\ev}(\sC; \Hom(\Lambda^*({}^dT^*M') ) 
	+ x^{\ell}
	\CI(\sC; \Hom(\Lambda^*({}^dT^*M' ) ).
\end{equation}
In particular, since $\pa_{\sigma}g_{ij} = \cO(x^2),$ we point out that $\alpha$ is $\cO(x^2)$ as a section of $\Hom(\Lambda^*({}^d T^*M')).$


Duhamel's formula yields
\begin{equation*}
	\pa_{\sigma}e^{-t\Delta_q} = 
	-\int_0^t e^{-s\Delta_q}\pa_{\sigma}\Delta e^{-(t-s)\Delta_q} \; ds
\end{equation*}
and a direct computation using \eqref{eq:MullerDelta} shows that
\begin{equation*}
	\pa_{\sigma}\rest{\eps=0}\Delta = -\alpha \delta d + \delta \alpha d - d\alpha \delta + d\delta \alpha.
\end{equation*}
Since $\alpha$ is $\cO(x^2),$ the construction of the heat kernel above implies that $\alpha e^{-t\Delta}$ is trace-class for all $t>0$ (as are 
$\alpha de^{-t\Delta},$ $\alpha \delta e^{-t\Delta},$ and $\alpha d\delta e^{-t\Delta}$) and hence $\pa_{\sigma}e^{-t\Delta_q}$ is trace-class for all $t>0.$
Thus we can compute as in \cite[Theorem 2.1]{Ray-Singer} and find that
\begin{equation*}
	\pa_{\sigma} \sum (-1)^q q \; \RTr{e^{-t\Delta_q}} = t\pa_t \sum (-1)^q \; \Tr\lrpar{\alpha_q e^{-t\Delta_q}}
\end{equation*}
where $\alpha_q$ is the restriction of $\alpha$ to forms of degree $q$
and hence, for $\Re s>m/2,$
\begin{equation*}
	\pa_{\sigma} \sum (-1)^q q \zeta_q(s)
	= \frac s{\Gamma(s)} \int_0^{\infty} t^s \sum (-1)^{q+1} \; \Tr\lrpar{\alpha_q (e^{-t\Delta_q}-\Pi_{\ker \Delta_q})} \frac{dt}t.
\end{equation*}

Now our construction of the heat kernel in \S\ref{sec:HeatKerLaplaceType} and its refinement in \S\ref{sec:EvenMets} (by restricting to $\bhs{mf}$)
together with \eqref{eq:Alpha} imply that 
\begin{equation*}
	\tr(\alpha e^{-t\Delta}) \sim 
	\lrpar{ \rho_{tff}^{-h-1}\sum_{2k<\ell} A_{2k}(\alpha) \rho_{tff}^{2k} 
	+\rho_{tff}^{-h+\ell-1}\sum A'_k(\alpha) \rho_{tff}^k }\mu
\end{equation*}
and so $\ell>h+1$ guarantees that there is no constant term in this expansion.
 
Applying Melrose's push-forward theorem to $\alpha e^{-t\Delta}$ on $[M' \times \bbR^+_{\sqrt t}; H \times \{ 0 \}],$ yields
\begin{multline*}
	\Tr\lrpar{\alpha e^{-t\Delta}} 
	\sim 
	t^{-m/2} \sum_{k\geq 0}a_{k}(\alpha) t^k \\  
	+t^{-(h+1)/2} \lrspar{ \sum_{0\leq k < (\ell-1)/2} \wt a_{k}(\alpha) t^{k} 
	+t^{\ell/2} \sum_{k\geq 0} \wt a_k'(\alpha) t^{k/2}
	+\sum_{k\geq 0} b_{k/2}(\alpha_q)t^{k/2}\log t }
	\Mas t \to 0.
\end{multline*}
where the $\{a_k(\alpha_q)\}$ are the contribution from the interior lift of $\{ \sqrt t = 0\},$ the $\{\wt a_{k}(\alpha_q)\}$ are the contribution from the front face of the blow-up, and the $\{ b_{k/2}(\alpha_q) \}$ are the contribution from the corners (see the appendix of \cite{Albin-Rochon:ModSpace} for details). In particular this lets us deduce that, since $m$ is odd and $\alpha \in \cO(x^2),$ the coefficient of $t^0\log t$ vanishes since neither corner has a term at order $t^0.$

Thus we can meromorphically continue $\pa_{\sigma} \sum (-1)^q q \zeta_q(s)$ to the complex plane and we find that, near $s=0,$ it equals
\begin{equation*}
	s {\sum_q (-1)^{q} \Tr(\alpha_q \Pi_{\ker \Delta_q})} + \cO(s^2)
\end{equation*}
which proves the variation formula in the lemma. 
\end{proof}

We will prove our main results for product-type $d$ metrics, it follows from this lemma that they will hold for $d$-metrics that are of even to order $h+2.$\\

To relate the analytic torsion of $M$ with a smooth metric to the analytic torsion of $[M;H]$ with a fibered cusp metric, we will use our uniform construction of the resolvent and the heat kernel of  the Hodge Laplacian as $\eps \to 0.$  Putting these together we will see that the limit as $\eps \to 0$ of the analytic torsion is equal to the sum of the analytic torsion of $[M;H]$ and a contribution from each of the model operators.

In order to compute these limits, it will be very useful to rewrite $\displaystyle \FP_{s=0} \tfrac1s\zeta(s)$ as a renormalized integral. 
From
\begin{equation*}
	\RTr{e^{-t\Delta}} \sim t^{-m/2}\sum_{k \geq 0} a_{k/2} t^{k/2} + t^{-(h+1)/2} \sum_{k\geq 0} b_{k/2}t^{k/2}\log t
\end{equation*}
(where for closed manifolds we have ${}^R\mathrm{Tr} = \Tr,$ $a_{k/2} =0$ if $k$ is odd and $b_j=0$ for all $j$),
we see that, for $\Re s>m/2$ and any $L>-m/2,$
\begin{multline*}
	\int_0^1 t^s \; \RTr{e^{-t\Delta} - \Pi_{\ker \Delta}} \; \frac{dt}t
	=  f_L(s)\\
	+ \int_0^1 t^s \; \lrspar{-\dim\ker\Delta + \sum_{-m/2 \leq k < L} a_{(k+m)/2} t^{k/2} + \sum_{-(h+1)/2 \leq k < L} b_{(h+1+k)/2}t^{k/2}\log t } \; \frac{dt}t
\end{multline*}
with $f_L(s)$ holomorphic for $\Re s> -L.$
Since we have
\begin{equation*}
	\int_0^1 t^{s+k/2} \frac{dt}t = \frac1{s+k/2}, \quad
	\int_0^1 t^{s+k/2} \log t \frac{dt}t =-\frac1{(s+k/2)^2},
\end{equation*}
the integral above is a meromorphic function on $\bbC$ with at worst double poles and explicit singular parts.
Near $s=0$ we can write $\frac1{\Gamma(s)} = s + \gamma s^2 + \frac1{12}(6\gamma^2-\pi^2)s^3 + \cO(s^4)$ where $\gamma$ is the Euler-Mascheroni constant. It follows that the meromorphically continued zeta function is equal, near $s=0,$ to
{\small
\begin{equation*}
	\lrpar{ s + \gamma s^2 + \frac{6\gamma^2-\pi^2}{12}s^3 + \cO(s^4) }
	\lrpar{ -\frac{b_{(h+1)/2}}{s^2} + \frac{a_{m/2} - \dim \ker \Delta}s + \sideset{^R}{_0^\infty}\int\RTr{e^{-t\Delta}}\;\frac{dt}t + \cO(s)}
\end{equation*}
}
and so the coefficient of $s$ is 
\begin{equation*}
	\sideset{^R}{_0^\infty}\int\RTr{e^{-t\Delta}}\;\frac{dt}t - \frac{6\gamma^2-\pi^2}{12}b_{(h+1)/2} + \gamma (a_{m/2} - \dim \ker \Delta),
\end{equation*}
so that 
\begin{equation}
\zeta'(0) = \sideset{^R}{_0^\infty}\int\RTr{e^{-t\Delta}}\;\frac{dt}t - \frac{6\gamma^2-\pi^2}{12}b_{(h+1)/2} + \gamma (a_{m/2} - \dim \ker \Delta).
\label{zp.1}\end{equation}
In particular for a closed manifold (with $b_{(h+1)/2}=0$), this yields  
\begin{equation}\label{eq:ExplicitZetaPrime}
	\zeta'(0) =
	\sideset{^R}{_0^\infty}\int\RTr{e^{-t\Delta}}\;\frac{dt}t 
	+  \gamma (a_{m/2} - \dim \ker \Delta).
\end{equation}
More generally, we have seen that when $h= \dim Y$ is even, $g_F$ is even and $g_d$ is even to order $h+2,$ we have $b_{(h+1)/2}=0$ for the Hodge Laplacian, and hence
\eqref{eq:ExplicitZetaPrime} holds just as for a closed manifold in this case.

\section{Asymptotics of analytic torsion}\label{sec:AsympAT}

Having determined the behavior under fibered  cusp surgery of the resolvent of Dirac-type operators, and hence their small eigenvalues, and the heat kernels of Laplace-type operators, we can now determine the behavior of the corresponding zeta functions. Specializing to twisted de Rham operators we will find the asymptotics of analytic torsion. In particular we will show that its finite part as $\eps\to 0$ decomposes into a contribution from the various model operators. We will first describe the general case and then impose extra conditions on the representations to determine these contributions.\\

Let $D_{\ed}$ be a Dirac-type operator as in \S\ref{rc.0} satisfying Assumptions 1 and 2 of that section.
For positive $\eps,$ the derivative of the zeta function of $D_{\ed}^2$ at $s=0$ is given by \eqref{eq:ExplicitZetaPrime}
\begin{equation*}
	\zeta'(0) 
	=
	\gamma( a_{m/2} - \dim\ker D_{\ed}^2) + \sideset{^R}{_0^\infty}\int \Tr(e^{-tD_{\ed}^2}) \; \frac{dt}t
\end{equation*}
and we wish to determine the finite part as $\eps\to0.$

Recall from \S\ref{sec:TraceHeatKer} that the trace of the heat kernel of $D_{\ed}^2$ is polyhomogeneous on 
\begin{equation*}
	\sE\sT = [ [0,1]_{\eps} \times \bbR^+_{\tau}; \{ \eps, \tau = 0 \} ] 
\end{equation*}
where $\tau = \sqrt t,$ which has two boundary hypersurfaces over $\eps=0,$ $\bhs{af}$ and $\bhs{tff}.$
By the renormalized push-forward theorem of \cite[Lemma~2.4]{hmm}, at $\bhs{af},$ the trace of the heat kernel behaves like
\begin{equation*}
	\Tr(e^{-tD_{\ed}^2}) = a\log \rho_{af} + \RTr{e^{-t\Delta_d}} + \RTr{e^{-t\Delta_b}} + o(1),
\end{equation*}
while at $\bhs{tff},$ the expansion of the trace of the heat kernel has the form
\begin{equation}
	\Tr(e^{-tD_{\ed}^2}) \sim \rho_{tff}^{-m}\sum_{k= 0}^{m} A_k^{tff} \rho_{tff}^{k} + \rho_{tff}^{-h-1} \sum_{\ell=0}^{h+1} \wt A_k^{tff} \rho_{tff}^\ell \log\rho_{tff}+ o(1),
\label{ne.1}\end{equation}
and we have seen that for product-type metrics many of these coefficients vanish.

First let us consider the renormalized integral for large time.
On a manifold undergoing surgery we denote the set of small eigenvalues by $\Spec_{\sma}(\Delta)$ and we
can write
\begin{equation*}
	\Tr\lrpar{ e^{-tD_{\ed}^2} }
	= \Tr \Big( \sum_{\lambda \in \Spec_{\sma}(\Delta)} e^{-t\lambda}\Pi_{\lambda} \Big) + \cO(e^{-t\delta})
	= \Tr \lrpar{ e_{\sma}^{-tD_{\ed}^2} } + \cO(e^{-t\delta}) \Mas t\to\infty
\end{equation*}
with the advantage that $\delta$ is independent of $\eps,$ where $e_{\sma}^{-tD_{\ed}^2}= \Pi_{\sma}  e^{-tD^2_{\eps,d}}\Pi_{\sma}$.  Next notice that
\begin{equation*}
	\FP_{s=0} \int_1^\infty t^s e^{-t\lambda}\frac{dt}t= 	\begin{cases}
	\displaystyle
0, & \Mif \lambda=0, \\
	 & \\
	\displaystyle
	 \FP_{s=0} \frac1{\lambda^s}\int_\lambda^{\infty} t^s e^{-t} \frac{dt}t
	= \int_\lambda^\infty e^{-t}\frac{dt}t
	=: \Gamma(0,\lambda), 
	& \Mif \lambda \neq 0,
	\end{cases}
\end{equation*}
and so
\begin{multline*}
	\sideset{^R}{_1^\infty}\int \Tr(e^{-tD_{\ed}^2}) \; \frac{dt}t
	= \sideset{^R}{_1^\infty}\int \Tr(e^{-tD_{\ed}^2} - e_{\sma}^{-tD_{\ed}^2}) \; \frac{dt}t
	 + \sideset{^R}{_1^\infty}\int \Tr(e_{\sma}^{-tD_{\ed}^2}) \; \frac{dt}t \\
	= \int_1^{\infty} \Tr(e^{-tD_{\ed}^2} - e_{\sma}^{-tD_{\ed}^2}) \; \frac{dt}t
	 + \sum_{\lambda \in \Spec_{\sma}(D_{\ed}^2)\setminus \{0\}}
	 \Gamma(0,\lambda) \dim \ker(D_{\ed}^2-\lambda).
\end{multline*}
Now as $\eps \to 0$ we have that 
\begin{equation*}
	\FP_{\eps=0} \int_1^{\infty} \Tr(e^{-tD_{\ed}^2} - e_{\sma}^{-tD_{\ed}^2}) \; \frac{dt}t
	=
	\sideset{^R}{_1^\infty}\int {}^{R}\!\Tr(e^{-t\Delta_d}) \; \frac{dt}t
	+ \sideset{^R}{_1^\infty}\int {}^{R}\!\Tr(e^{-t\Delta_b}) \; \frac{dt}t,
\end{equation*}
while, as $\lambda \to 0,$ $\Gamma(0,\lambda) = -\log\lambda - \gamma + \cO(\lambda).$ 
Let us abbreviate 
\begin{equation*}
\begin{gathered}
	N_\eps = \dim \ker D_{\ed}^2, \quad
	N_b = \dim \ker \Delta_b, \quad
	N_d = \dim \ker \Delta_d, \\
	\log \zeta_{\sma}(D_{\ed}^2) = 
	 \log \prod_{\substack{\lambda \in \Spec_{\sma}(D_{\ed}^2)\setminus \{0\} \\ \text{repeated with multiplicity}}} \lambda.
\end{gathered}
\end{equation*}
Then $\FP_{\eps=0}\log \zeta_{\sma}(D_{\ed}^2)$ makes sense provided the logarithm of the product of the positive small eigenvalues has a polyhomogeneous expansion at $\epsilon=0$.  
In this case, we have that 
\begin{multline*}
	\FP_{\eps=0} \sideset{^R}{_1^\infty}\int \Tr(e^{-tD_{\ed}^2}) \; \frac{dt}t \\
	=
	\sideset{^R}{_1^\infty}\int \Tr(e^{-t\Delta_d}) \; \frac{dt}t
	+ \sideset{^R}{_1^\infty}\int \Tr(e^{-t\Delta_b}) \; \frac{dt}t
	-\gamma(N_b + N_d - N_{\eps}) -\FP_{\eps=0}\log \zeta_{\sma}(D_{\ed}^2).
\end{multline*}

Next let us consider the renormalized integral for small time
\begin{equation*}
	\sideset{^R}{_0^1}\int \Tr(e^{-tD_{\ed}^2}) \frac{dt}t,
\end{equation*}
for which we will use the following diagram
\begin{equation*}
	\xymatrix{
	\sE\sT\rest{\tau\leq1} \ar[rr]^-{\beta_{\eps,\tau}} \ar[rd]^{\pi_{\eps}} & & [0,1]_{\eps}\times [0,1]_\tau \ar[ld]_{p_{\eps}} \\
	& [0,1]_{\eps}. & }
\end{equation*}

\begin{lemma} 
i) If, in the push-forward along $\pi_{\eps},$ we renormalize using $\tau$ (as opposed to $\rho_{tf}$) the resulting map satisfies
\begin{equation*}
	\begin{array}{ccc}
	\cA^{(E_{tf}, E_{tff}, E_{af})}(\sE\sT\rest{\tau\leq 1}) & \longrightarrow & \cA^{F}([0,1]_{\eps}) \\
	f & \longmapsto & \displaystyle \FP_{s=0}\pi_{\eps*}(f\beta_{\eps,\tau}^*(t^s\tfrac{dt}t)) 
	\end{array}
\end{equation*}
where
\begin{equation*}
	F = E_{af} \bar\cup E_{tff} \cup
	\bigcup_{\{\ell : (0,\ell) \in E_{tf}\}} \{ (z,q): \exists (z,p) \in E_{tff}, q \in \{ p+1, \ldots, p+\ell+1\} \}.
\end{equation*}

ii)  We have that
\begin{equation*}
	\FP_{\eps=0}
	\sideset{^R}{_0^1}\int \Tr(e^{-tD_{\ed}^2}) \frac{dt}t
	=
	\sideset{^R}{_0^1}\int \; \RTr{e^{-t\Delta_d}} \frac{dt}t
	+\sideset{^R}{_0^1}\int \; \RTr{e^{-t\Delta_b}} \frac{dt}t
	+\sideset{^R}{_0^\infty}\int \; A^{tff}_{m} \frac{d\sigma}\sigma,
\end{equation*}
where $A^{tff}_{m}$ is defined in equation \eqref{ne.1}.
iii) As $\eps\searrow 0$, we have the following polyhomogeneous expansion 
\begin{equation}
\sideset{^R}{_0^1}\int \Tr(e^{-tD_{\ed}^2}) \frac{dt}t\sim \sum_{k\ge -h-1}^0 \eps^k(c_k+d_k\log\eps) + e_0(\log\eps)^2 +o(1). 
\label{phge.1}\end{equation}
Moreover, $d_{-h-1}=0$ when  the pushforward of $N_{tff}(e^{-tD_{\ed}^2})$ in \eqref{eq:AllNormalOps} to $\bhs{tff}(\sE\sT)$ has no term of order zero at $\bhs{tf}(\sE\sT)$, which holds in particular when $m$ is odd.
\label{ae.1}\end{lemma}

\begin{proof} $ $

i) This is a small generalization of \cite[Theorem A.1]{Albin-Rochon:ModSpace}.
We may use a partition of unity to restrict the support of $f.$ If $f$ is supported away from $\bhs{tf}$ then the renormalized push-forward coincides with the usual push-forward and the index set $E_{af} \bar\cup E_{tff}$ follows from Melrose's push-forward theorem \cite{me1}. Thus it suffices to consider $f$ supported near the corner $\bhs{tff}\cap \bhs{tf},$ say in $\{ \rho_{tf}, \rho_{tff} \leq \delta\}.$
We can use projective coordinates
\begin{equation*}
	\eps, \sigma = \frac{\tau}\eps
\end{equation*}
in which $\rho_{tff} = \eps$ and $\rho_{tf} = \sigma.$
We can write the pull-back of $t^{s-1}\; dt$ as $2(\sigma\eps)^{2s} \; (\eps^{-1} d\eps + \sigma^{-1} d\sigma).$
The integral along the slice $\eps=\eps_0$ of $\rho_{tff}^j (\log \rho_{tff})^p \rho_{tf}^k (\log \rho_{tf})^{\ell} \beta_{\eps,\tau}^*(t^s\frac{dt}t)$ is equal to
\begin{equation*}
	2 \eps_0^{2s+j} (\log \eps_0)^p 
	\int_0^\delta  \sigma^{k+2s} (\log \sigma)^{\ell} \; \frac{d\sigma}{\sigma}
	=
	2 \eps_0^{2s+j}  (\log \eps_0)^p 
	\sum_{j=0}^{\ell} (-1)^j \frac{ \ell!}{(\ell-j)!}\frac{\delta^{k+2s}(\log \delta)^{\ell-j}}{(k+2s)^{j+1}}.
\end{equation*}
If $k \neq 0$ then this has finite part at $s=0$ equal to $C \eps_0^j (\log \eps_0)^p$ for some constant $C.$
If $k=0$ then we can replace $(\eps_0\delta)^{2s}$ with its Taylor expansion $\sum_{n \geq 0} (2s\log \eps_0\delta)^n/n!$ and write this integral as
\begin{equation*}
	\eps_0^j (\log \eps_0)^p \sum_{n=1}^{\ell+1} C_n (\log \eps_0)^n 
\end{equation*}
for some constants $C_n.$
Thus every term of the form $\rho_{tff}^j (\log \rho_{tff})^p (\log \rho_{tf})^{\ell}$ contributes terms $\eps^j (\log \eps)^q,$ $q \in \{p+1, \ldots, p+\ell+1\}$ to the final expansion.

Note that if we had renormalized using $\rho_{tf}$ instead of $\tau$ the final index set would have been $E_{af} \bar\cup E_{tff}$ just as in the usual push-forward theorem. This was pointed out in \cite[page 128]{hmm}.

ii) It follows from ($i$) that the contribution to the expansion of $\sideset{^R}{_0^1}\int \Tr(e^{-tD_{\ed}^2}) \frac{dt}t$ in $\eps$ from the corners of $\sE\sT$ involves positive powers of $\log \eps$ and so will not contribute to the constant term in $\eps.$ Thus the constant term comes from the constant term in the expansion of $\Tr(e^{-tD_{\ed}^2}) \frac{dt}t$ at $\bhs{af},$ which is
\begin{equation*}
	\sideset{^R}{_0^1}\int \; \RTr{e^{-t\Delta_d}} \frac{dt}t
	+\sideset{^R}{_0^1}\int \; \RTr{e^{-t\Delta_b}} \frac{dt}t
\end{equation*}
and the constant term in its expansion at $\bhs{tff}$ which is 
\begin{equation*}
	\sideset{^R}{_0^\infty}\int \; A^{tff}_{m} \frac{d\sigma}\sigma.
\end{equation*}

iii) Recall that the coefficients in the expansion of $\RTr{e^{-tD_{\ed}^2}}$ at $\bhs{tff}(\sE\sT)$  comes from the coefficients of the expansion of $e^{-tD_{\ed}^2}$ at $\bhs{\eps,\tau}(\Delta_{HX})$ and $\bhs{tff}(\Delta_{HX})$.  However, as explained in Lemma~\ref{lem:ImprovedTraceAsymp}, the terms coming $\bhs{\eps,\tau}(\Delta_{HX})$ yields terms on $\sE\sT$ that are in fact lifts of polyhomogeneous terms on $[0,1]_{\eps}\times \bbR^+_{\tau},$  with singular expansion at $\tau=0$ but smooth expansion at $\eps=0$.  Thus, when we take the renormalized integral in $\tau$, these terms are bounded as $\eps
\searrow 0$.  Thus, when we want to compute the singular terms in the polyhomogeneous expansion of the renormalized integral in \eqref{phge.1}, instead of what is given in \eqref{eq:ImprovedTraceAsymp.2}, we can effectively assume that $\RTr{e^{-tD_{\ed}^2}}$ is an element of 
$$
\cA^{-m, -h-1, 0 \bar\cup 0}(\sE\sT)
$$
to apply  i), which yields \eqref{phge.1} as desired.   Moreover, since coefficient $d_{-h-1}$ of the term of order $\eps^{-h-1}\log \eps$ comes from the coefficient of the term of order zero at $\bhs{tf}(\sE\sT)$ in the pushforward of $N_{tff}(e^{-tD_{\ed}^2})$ in \eqref{eq:AllNormalOps} to $\bhs{tff}(\sE\sT)$, we see that $d_{-h-1}=0$ when this term is zero.  In particular, we see from the formula for $N_{tff}(e^{-tD_{\ed}^2})$ in \eqref{eq:AllNormalOps} that this the case when $m$ is odd.

\end{proof}

Applying these considerations to the twisted de Rham operator, we will obtain an expression for the asymptotics of the analytic torsion.
This will involve the analytic torsion of $[M;H]$ with the induced $d$-metric, and also an expression related to, but different from, the analytic torsion of $Y \times [-\pi/2, \pi/2]$ with the induced $b$-metric and the flat bundle $\rho^{\bN}\cH^*(H/Y;F).$ Namely, recall that
\begin{equation*}
	D_b \in \Diff^1_b(Y \times [-\pi/2,\pi/2]; \Lambda^*(Y \times [-\pi/2,\pi/2]) \otimes \rho^{\bN}\cH^*(H/Y;F))
\end{equation*}
is such that its square $D_b^2$ splits into a direct sum of operators $D_b^2 = \oplus (D_b^2)_{j,k},$ with
\begin{equation*} 
	(D_b^2)_{j,k} \in \Diff_b^2(Y \times [-\pi/2,\pi/2]; \Lambda^j(Y \times [-\pi/2,\pi/2]) \otimes \rho^{k}\cH^k(H/Y;F)).
\end{equation*}
Let $\zeta_{j,k}(s)$ denote the zeta function of $(D_b^2)_{j,k}$ defined using the renormalized trace of its heat kernel and define
\begin{equation*}
	\lAT(Y \times [-\pi/2,\pi/2], D_b, \cH^*(H/Y;F)) 
	= \frac12\sum_{j,k} (-1)^{j+k}(j+k)\zeta_{j,k}'(0).
\end{equation*}
Now we can state our conclusion regarding the asymptotics of the zeta function and analytic torsion.

\begin{theorem}\label{thm:FinalAsymp}
Let $M$ be a closed oriented manifold with a two-sided hypersurface $H \subseteq M$ participating in a fibration $H\xlra{\phi}Y$ and $x$ a fixed choice of defining function for $H.$ Let $g_{\ed}$ be a fibered  cusp surgery metric of product-type to order two, $E\lra X_s$ a $\Cl(\Ed TX_s)$-bundle with associated Dirac-type operator $\eth_{\ed}$ satisfying Assumptions 1 and 2 from \S\ref{rc.0}.   Then the zeta function of $D_{\ed}^2$ is such that
\begin{equation}
\zeta_{\ed}'(0) +\log \zeta_{\sma}(D_{\ed}^2)\sim \sum_{k=-h-1}^0 \eps^k(\widetilde{c}_k + \widetilde{d}_k \log\eps) + \widetilde{e}_0(\log\eps)^2 + o(1) \quad
\label{phge.2}\end{equation}
as $\eps\searrow 0$ with $\widetilde{d}_{-h-1}=0$ if  the pushforward of $N_{tff}(e^{-tD_{\ed}^2})$ in \eqref{eq:AllNormalOps} to $\bhs{tff}(\sE\sT)$ has no term of order zero at $\bhs{tf}(\sE\sT)$, which holds in particular when $m$ is odd.  Moreover, 
\begin{multline}\label{eq:ZetaAsymp}
	\FP_{\eps=0}\left( \zeta_{\ed}'(0) +\log \zeta_{\sma}(D_{\ed}^2)\right)	= \zeta_{d}'(0) + \zeta_b'(0)-\gamma a_{(h+1)/2}(D^2_b) \\ + \frac{6\gamma^2-\pi^2}{12} b_{(h+1)/2}(D^2_d) 
	+ \sideset{^R}{_{\bhs{tff}\cap\bhs{tf}}^{}}\int A^{tff\cap tf}
	+\sideset{^R}{_0^\infty}\int \; A^{tff}_{m} \frac{d\sigma}\sigma,
\end{multline}
where $\zeta_d$ and $\zeta_b$ are the zeta functions defined using $\RTr{e^{-t\Delta_d}}$ and $\RTr{e^{-t\Delta_b}},$ $a_{(h+1)/2}(D^2_b)$ is the coefficient of order $t^0$ as $t\searrow 0$ of ${}^R\!\Tr(e^{-t D^2_b})$,  $b_{(h+1)/2}(D^2_d)$ is the coefficient of order $\log t$ as $t\searrow 0$ of ${}^R\!\Tr(e^{-t D^2_d})$, $A^{tff\cap tf}$ is the coefficient order $\rho^0_{tff}\rho_{tf}^0$ of $\tr(e^{-t\Delta^2_{\ed}})$ at $\bhs{tff}\cap\bhs{tf}$ in $\Delta_{HX}$, and $A_m^{tff}$ was introduced in \eqref{ne.1}.
\end{theorem}

\begin{proof} 
The only term that we have yet to consider in the explicit formula \eqref{eq:ExplicitZetaPrime} for $\zeta'(0)$  is $a_{m/2}(\Delta_{\ed}),$ the constant term in the short time asymptotics of the trace of the heat kernel. Directly from our construction of the heat kernel and the pushforward theorem, we have that 
$$
  a_{m/2}(\Delta_{\ed})\sim q_{(0,1)}\log \eps + q_0 + o(1)     
$$
as $\eps\searrow 0$ and
\begin{equation}\label{eq:LimitCstTerm}
	q_0=\FP_{\eps=0} a_{m/2}(\Delta_{\ed}) =  a_{m/2}(\Delta_d)+  \sideset{^R}{_{\bhs{tff}\cap\bhs{tf}}^{}}\int A^{tff\cap tf}.
\end{equation}
Thus, combining \eqref{eq:ExplicitZetaPrime} with Lemma~\ref{ae.1} and Proposition~\ref{rt.1} yields the polyhomogeneous expansion \eqref{phge.2}.
To compute the finite part at $\eps=0$, recall that for $\eps>0$ we have
\begin{equation*}
	\zeta_{\ed}'(0) =
	\sideset{^R}{_0^\infty}\int\RTr{e^{-t\Delta_{\ed}}}\;\frac{dt}t 
	+  \gamma (a_{m/2}(\Delta_{\ed}) - N_{\eps}),
\end{equation*}
and we have computed that 
\begin{multline*}
	\FP_{\eps=0} 
	\sideset{^R}{_0^\infty}\int\RTr{e^{-t\Delta_{\ed}}}\;\frac{dt}t 
	=
	\sideset{^R}{_0^\infty}\int \Tr(e^{-t\Delta_d}) \; \frac{dt}t
	+ \sideset{^R}{_0^\infty}\int \Tr(e^{-t\Delta_b}) \; \frac{dt}t \\
	-\gamma(N_b + N_d - N_{\eps}) -\FP_{\eps=0}\log \zeta_{\sma}(D_{\ed}^2)
	+\sideset{^R}{_0^\infty}\int \; A^{tff}_{m} \frac{d\sigma}\sigma,
\end{multline*}
which together with \eqref{eq:LimitCstTerm} and \eqref{zp.1} establishes \eqref{eq:ZetaAsymp}. 
\end{proof}

\begin{corollary}
Let $F \lra X_s$ be a vector bundle with flat connection $\nabla^F$ and metric $g_F,$ not necessarily compatible.
If $h = \dim Y$ is even, $g_F$ is even in $\rho,$ $g_{\ed}$ is an even $\ed$-metric to order $h+2,$ and $F\lra X_s$ is Witt over $\bhs{\bs},$ then we have%
\begin{multline}\label{eq:GralATAsymp}
	\FP_{\eps=0}\left( \lAT(M,g_{\ed},F)  +\frac12\sum (-1)^q q
	\log \zeta_{\sma}(\Delta_q)\right)
	= \lAT([M;H],g_{d},F) \\
	+ \lAT(Y \times [-\pi/2,\pi/2], D_b, \cH^*(H/Y;F)).
\end{multline}
If moreover $F$ is such that 
\begin{equation*}
	\mathrm H^*(H/Y;F) = 0, 
\end{equation*}
then 
\begin{equation}\label{eq:AcyclicATAsymp}
	\FP_{\epsilon=0}\lAT(M,g_{\ed},F) =  \lAT([M; H],g_{d},F).
\end{equation}
\label{aat.1}\end{corollary}
\begin{proof}
The first three terms in \eqref{eq:ZetaAsymp} yield the terms in \eqref{eq:GralATAsymp}.
To see that the final terms do not contribute, first note that since $g_{\ed}$ is even to order $h+2$, the trace of the heat kernel simplifies as explained in \S\ref{sec:EvenMets}.
In particular, Corollary~\ref{eveness.2} and the fact that $Y$ is even-dimensional show that $A^{tff\cap tf}=a_{(h+1)/2}(\Delta_b)= b_{(h+1)/2}(\Delta_d)=0$ and that $A^{tff}_m$ has no contribution from $\bhs{tff}(\wt \Delta_{HX}).$
Next, Lemma \ref{lem:ImprovedTraceAsymp} shows that whenever $\chi$ is a smooth function on $X_s$ supported away from $\bhs{sb},$
$\Tr(\chi e^{-t\Delta}) \in \tau^{-m}\CI([0,1]_{\eps} \times \bbR^+_\tau)$ and hence
\begin{equation*}
	\FP_{\eps=0} \sideset{^R}{_0^1}\int \Tr(\chi e^{-t\Delta}) = \sideset{^R}{_0^1}\int \Tr(\chi\rest{\bhs{sm}}e^{-t\Delta_d}).
\end{equation*}
It follows that $\sideset{^R}{_0^\infty}\int \; A^{tff}_{m} \frac{d\sigma}\sigma$ has no contribution from $\bhs{\eps,\tau}(\wt\Delta_{HX})$ and so $\sideset{^R}{_0^\infty}\int \; A^{tff}_{m} \frac{d\sigma}\sigma=0.$
Finally if $\mathrm H^*(H/Y;F) = 0$ then $\lAT(Y \times [-\pi/2,\pi/2], D_b, \cH^*(H/Y;F))=0$ and $N_b=0.$  By Theorem~\ref{ac.5}, $\mathrm H_{(2)}^*(M\setminus H;F)=\mathrm H^*(M; F)$, so $N_d=\dim \mathrm H^*(M; F)$ and there are no positive small eigenvalues, from which the result follows. \end{proof}

\section{A Cheeger-M\"uller theorem for fibered cusp manifolds} \label{sec:CheegerMuller}

We now use the analysis of the previous sections to deduce a Cheeger-M\"uller theorem for strongly acyclic flat bundles on manifolds with fibered  cusp ends.

\begin{theorem}\label{thm:ATRTSurgery}
Let $M$ be a closed oriented odd-dimensional manifold with a hypersurface $H\subset M$ whose normal bundle is trivial. Let $\phi: H\to Y$ be a fiber bundle with $Y$ a closed oriented even-dimensional manifold and  let $g_{\ed}$ be a corresponding fibered  cusp surgery metric even to order $\dim Y +2.$
Suppose $\alpha:\pi_1(M)\to \GL(k,\bbR)$ is a unimodular representation and endow the associated flat vector bundle $F \lra X_s$ with an metric $g_F$ even in $\rho.$
Assume that $F \lra X_s$ is such  that
\begin{equation*}
	\mathrm H^*(H/Y;F) = 0.
\end{equation*}
Then 
 we have the equality of analytic and Reidemeister torsions
\begin{equation*}
	\bar \lAT([M;H], g_d,F, g_F,\mu^*)= \log \tau([M;H],\pa[M;H],F,\mu^*)+ \log \tau(H,\alpha),
\end{equation*}
where $\mu^j$ is any choice of basis of 
$$
\mathrm H^{j}([M;H],\pa [M;H];F)\cong \mathrm H^{j}([M;H];F)\cong L^2\cH^j([M;H],F, g_{d}).
$$
\end{theorem}

\begin{proof} 
Since $\mathrm H^*(H/Y;F)=0$, Theorem \ref{ac.5} can be applied and gives that $\mathrm H^*(H;F)=0$, so that $\mathrm H^{j}([M;H],\pa [M;H];F)\cong \mathrm H^{j}([M;H];F)$.  More importantly, it tells us that
\begin{equation}
	\tau(M,\alpha,\mu^*)=  \tau([M;H],\pa[M;H],F,\mu^*) \tau(H,\alpha).
\label{fin.1}\end{equation}
We can also apply Corollary~\ref{aat.1} to obtain that
\begin{equation}
	\FP_{\epsilon=0}\lAT(M,g_{\ed},F,g_F) = \lAT(M_0,g_{d},F,g_F).
\label{fin.2}\end{equation}
On the other hand, the Cheeger-M\"uller theorem \cite{Cheeger1979, Muller1978, Muller1993} implies that for all $\epsilon>0$
\begin{equation}
	\log \tau(M,\alpha,\mu^*) = \bar\lAT (M,g_{\ed},F,g_F,\mu^*)= \lAT (M,g_{\ed},g_F,F)-\log \left(  \Pi_{q=0}^{n} [\mu^q |\omega_{\epsilon}^q]^{(-1)^q} \right),
\label{fin.3}\end{equation}
where $\omega^q_{\epsilon}$ is any orthonormal basis of the harmonic forms of degree $q$ with respect to $g_{\ed}$.  By Corollary~\ref{rc.29} and the fact $H^{*}(H/Y;F)=0$, we see that 
\begin{equation}
\lim_{\epsilon\to 0} [\mu^q|\omega_{\epsilon}^q]= [\mu^q| \omega^q_{0}]
\label{fin.4}\end{equation}
with $\omega^q_{0}$ an orthonormal basis of $L^2$-harmonic forms of degree $q$ with respect to the metric $g_d$.  Combining \eqref{fin.1}, \eqref{fin.2}, \eqref{fin.3} and \eqref{fin.4} gives the result.  
\end{proof}

\begin{corollary}
Let $(M_0',g_d)$ be an oriented odd-dimensional non-compact Riemannian manifold with fibered  cusp ends and $\bar M_0'$ its compactification to a manifold with fibered boundary,
\begin{equation*}
	Z\fib H \xlra{\phi} Y,
\end{equation*}
where $Y$ is an oriented even-dimensional compact manifold.  
Assume that $g_{d}$ is an even fibered cusp metric to order $\dim Y +2.$
Let $\alpha:\pi_1(M_0') \lra \GL(k,\bbR)$ be a unimodular representation whose associated flat  vector bundle $F'\lra \bar M_0'$ is such that
\begin{equation}
	\mathrm H^*(H/Y;F') = 0.
\label{acyc.1}\end{equation}
Suppose $F'$ is equipped with a smooth bundle metric $g_{F'}$ on $\bar M_0'$ with having an even expansion in terms of the boundary defining function.
Then we have equality of analytic and Reidemeister torsions
\begin{equation*}
	\bar\lAT(M_0', g_d,F',\mu^*)
	= \log \tau(\bar M_0',\pa\bar M_0',\alpha,\mu^*)+\frac12 \log\tau(H,\alpha)
\end{equation*}
for any choice of basis $\mu^q$ of $ \mathrm H^q(\bar M_0', \pa \bar M_0';F)$ for $q=0,\ldots, \dim M_0$.  
\label{even.2}\end{corollary}

\begin{proof}
By Lemma \ref{lem:ATVariation} we can assume that  the metric is a product-type fibered  cusp metric at $H.$

Let $M$ be the double of $M_0'$ across $H$ and let $g_{\ed}$ be a $\ed$-metric on $X_s = [M \times [0,1]_{\eps};H\times \{ 0\}]$ that restricts to $\bhs{\ms} = M_0' \sqcup M_0'$ to be $g_d$ on each copy.
Let $F \lra M$ be the double of $F'\lra M_0'.$ Note that $F \lra M$ is a flat bundle (since, e.g., near $H,$ $F'$ is the pull-back to a collar neighborhood of a flat Euclidean bundle over $H,$ and the same will be true for $F$ in a collar neighborhood of $H$), and hence so is the pull-back to $X_s,$ which we continue to denote $F.$ 
Hence, applying Theorem \ref{thm:ATRTSurgery} we have that 
\begin{equation*}
\begin{aligned}
	\bar\lAT(M_0', g_d,F',\mu^*)& = \frac12 \bar\lAT(\bhs{\ms}, g_d,F,\mu^*\sqcup\mu^*) \\
	&=  \log \tau(\bar M_0',\pa\bar M_0',\alpha,\mu^*) +\frac12\log\tau(H,\alpha).
\end{aligned}	
\end{equation*}
\end{proof}
\begin{remark}
If we assume instead that $M_0'$ is even-dimensional in Corollary~\ref{even.2}, then the condition $\cH^*(H/Y;F') = 0$ ensures by Remark~\ref{ds.1} that the Hodge Laplacian associated to $F'$ has compact resolvent and discrete spectrum.  Thus, if $\alpha$ is an orthogonal representation, we can apply the argument of \cite[Theorem~2.3]{Ray-Singer} to see that $\lAT(M_0', g_d,F')=0$. 
\label{even.1}\end{remark}
Notice that the condition \eqref{acyc.1} imposes some restrictions on the type of bundle $F$ one might consider, in particular it cannot be trivial.  Nevertheless, the condition is often satisfied.  Indeed, acyclicity is very common, since for any compact manifold $W$, acyclic representations are Zariski open in the set of representations $\pi_1(W)\to \GL(n,\bbC)$, so as soon as there is one acyclic representation, there are many.  Here is a simple way to construct flat Euclidean bundles on $3$-dimensional hyperbolic manifolds with cusp satisfying condition \eqref{acyc.1}.
\begin{example}
Pick a compact hyperbolic $3$-manifold $W=\Gamma \setminus \bbH^3$ with $\Gamma$ a discrete, torsion free co-compact subgroup of $\SL(2,\bbC)$.  By \cite[Lemma~4.6]{Muller1993}, we can find an acyclic representation $\rho: \pi_1(W)\to \GL(n,\bbC)$ for some $n\in\bbN$.  Since the set of acyclic representations is Zariski open, we can in fact assume that $\rho$ is unitary. Let $F$ be the corresponding unitary flat vector bundle.  Now, pick a simple closed geodesic $\gamma$ on which $F$ has a non-trivial holonomy and let $M_0'$ be the complement of $\gamma$ in $W$.  Then by a result of Sakai, see\cite[Proposition~4]{Kojima}, $M_0'$ naturally admits a finite volume hyperbolic metric $g_{\operatorname{hyp}}$ with a cusp.  The link $Z$ of the cusp is a torus and by our choice of $\gamma$, $F$ will be non-trivial on $Z$, so that $H^*(Z;F)=0$  
and Corollary~\ref{even.2} does apply to $(M_0',g_{\operatorname{hyp}},F)$.
\label{acyc.2}\end{example}

 \appendix

\section{Model cases: Euclidean Laplacians and Dirac operators} \label{euc.0}

In order to perform the resolvent construction, it is necessary to have a fairly refined understanding of the mapping properties of Laplacians and Dirac operators on Euclidean space, as well as on Euclidean space with a point removed. In this section, we develop the necessary theory.

\subsection{Mapping properties near infinity}

Let $\Delta$ be the standard Euclidean Laplacian on $\bbR^n$ with positive spectrum.  In polar coordinates, it is of the form
\begin{equation}
    \Delta= \frac{1}{r^2}\left( \Delta_{\bbS^{n-1}}- \left(r\frac{\pa}{\pa r}\right)^2 - (n-2)r\frac{\pa}{\pa r} \right),
\label{euc.1a}\end{equation}
where $r$ is the distance function from the origin and $\Delta_{\bbS^{n-1}}$ is the Laplacian with positive spectrum for the induced metric on the unit sphere.  If $\varrho$ is a boundary defining function for the radial compactification $\RC(\bbR^n)$ of $\bbR^n$ such that $\varrho=\frac{1}r$ near $\pa \RC(\bbR^n)$, this means that away from the origin,
$$
     \Delta= \varrho^2\left(    \Delta_{\bbS^{n-1}} -\left(  \varrho\frac{\pa}{\pa \varrho}\right)^2  + (n-2) \varrho \frac{\pa}{\pa \varrho}\right).
$$
Then let $A=\varrho^{-2}\Delta$. The operator $A$ is an elliptic $b$-differential operator of order 2 on $\RC(\bbR^n)$, with indicial family given by
$$
    \widehat{A}(\tau)= \Delta_{\bbS^{n-1}} +\tau^2 +i(n-2)\tau,
$$ 
so that, in the notation of \cite{MelroseAPS},
$$
\Spec_b(\widehat{A})=\left\{  -i\frac{n-2}2\pm i\sqrt{\nu+\left( \frac{n-2}2 \right)^2} \; | \; \nu \in \Spec(\Delta_{\bbS^{n-1}})\right\}.
$$
In particular, we have $\Spec_b(\widehat{A})=\{0,i\}$ when $n=1$ and $\Spec_b(\widehat{A})=i\bbZ$ when $n=2$.  When $n>2$, $\Spec_b(\widehat{A})\subset i\bbR$, with
\begin{equation}
0,n-2\in i\Spec_b(\widehat{A}),  \quad (0,n-2)\cap i\Spec_b(\widehat{A})=\emptyset.
\label{euc.1b}\end{equation}

Following \cite[(5.118)]{MelroseAPS}, we have, for each $\alpha\in \bbR$, two natural index sets associated to the indicial family $\widehat{A}$:
\begin{multline}
    \widehat{H}_n^{\pm}(\alpha)= \left\{  (z,k)\in \bbC\times \bbN_0\; | \; \exists m\in \bbN_0, \Re z> \pm\alpha+m,  \right.  \\
    \left. \pm(z-m) \in i\Spec_b(\widehat{A}), \quad 
   k+1 \le \sum_{j=0}^m \ord(-i(z-j)) \right\}.
\end{multline}
From \cite[Proposition~5.61]{MelroseAPS}, we have the following standard regularity result.

\begin{proposition}
If $G$ is an index set and if $u\in\varrho^{\alpha}H^m_b(\RC(\bbR^n))$ is such that $Au\in \cA_{\phg}^G(\RC(\bbR^n))$, then in fact $u\in \cA_{\phg}^{G\overline{\cup}\widehat{H}_n^+(\alpha)}(\RC(\bbR^n))$.
\label{euc.1}\end{proposition}

The parametrix construction in \cite{MelroseAPS} can also be used to obtain the following result about the Euclidean Laplacian.  

\begin{proposition}
For $m\in\bbN_0$, the map
\begin{equation}
    \Delta: \varrho^{\epsilon}H^{m+2}_b(\RC(\bbR^n))\to \varrho^{\epsilon+2} H^m_b(\RC(\bbR^n)
\label{euc.2b}\end{equation}
is Fredholm if and only if $\epsilon\notin i\Spec_b(\widehat{A})$.  Moreover, it is injective for $\epsilon\ge 0$ and surjective for $\epsilon<n-2$ with $\epsilon\notin i\Spec_b(\widehat{A})$.  In particular, if $n>2$, it is bijective for $\epsilon\in (0,n-2)$.  
\label{euc.2}\end{proposition}
\begin{proof}
It is standard that $\Delta$ has no kernel in $L^2(\bbR^n; g_{\euc})$, where $g_{\euc}$ is the Euclidean metric.  Since $L^2(\bbR^n;g_{\euc})=\varrho^{\frac{n}2}L^2_b(\RC(\bbR^n))$, this means that the map \eqref{euc.2b} is injective for $\epsilon\ge \frac{n}2$.  But if $n\ne 2$, by Proposition~\ref{euc.1} and \eqref{euc.1b} (and, when $n=1$, the fact that $i\Spec_b(\widehat{A})=\{-1,0\}$), we see that in fact the map 
\eqref{euc.2b} is injective for $\epsilon>0$.  If $n=2$ and $\epsilon>0$, then we know by Proposition~\ref{euc.1} that any element in the kernel tends to zero at infinity, and therefore must be zero by the maximum principle. So, in all cases, the map \eqref{euc.2b} is injective for $\epsilon>0$.

Now, since $\Delta$ is essentially self-adjoint on $L^2(\bbR^n,g_{\euc})$, one computes that the formal adjoint of $A=\varrho^{-2}\Delta$ is given by 
$$
       A^*= \varrho^{-n}\Delta \varrho^{n-2}. 
$$
Consequently, the map 
$$
A^*: \varrho^{\epsilon} H^{m+2}_b(\RC(\bbR^n)) \to \varrho^{\epsilon} H^{m}_b(\RC(\bbR^n))
$$
is injective for $\epsilon+n-2>0$.  By the duality of $\varrho^{\epsilon}L^2_b(\RC(\bbR^n))$ and $\varrho^{-\epsilon}L^2_b(\RC(\bbR^n))$, this means that the map \eqref{euc.2} is surjective for $\epsilon<n-2$ with $\epsilon\notin i\Spec_b(\widehat{A})$.  

Finally, to see that it is also injective for $\epsilon=0$, notice that by the relative index theorem of \cite{MelroseAPS}, the kernel of the map \eqref{euc.2b} is precisely given by the constant functions when $\epsilon<0$ is sufficiently close to zero.  Since the constant functions are not in $L^2_b(\RC(\bbR^n))$, this means the map \eqref{euc.2b} is injective when $\epsilon=0$.  
\end{proof}

This has the following interesting application.  

\begin{corollary}
Let $H_n=\widehat{H}_n^{+}(\alpha)$, where $\alpha<n-2$ is chosen sufficiently large so that $(\alpha,n-2)\cap i\Spec_b(\widehat{A})=\emptyset$.Then there exists a natural linear continuous map 
$$
             \Delta^{-1}: \cS(\bbR^n) \to \cA^{H_n}_{\phg}(\RC(\bbR^n))
$$
such that $\Delta\circ \Delta^{-1} f=f$ for all $f\in \cS(\bbR^n)$, where $\cS(\bbR^n)$ is the space of Schwartz functions on $\bbR^n$. Here the norm on the image space is $C^{\infty}$ on the coefficients and $C^N$ on the remainders, as in the discussion after (5.150) of \cite{MelroseAPS}.
\label{euc.3}\end{corollary}
\begin{proof}
First, for $n>2$, we can take any $\alpha\in(0,n-2)$.  By \cite[Proposition 5.64]{MelroseAPS}, the inverse of the bijective map
\begin{equation}
            A: \varrho^{\epsilon} H^{m+2}_b(\RC(\bbR^n)) \to \varrho^{\epsilon} H^{m}_b(\RC(\bbR^n))
\label{euc.3a}\end{equation}
is given by an element $A^{-1}\in \Psi^{-2, \widehat{H}_n^+(\alpha),\widehat{H}_n^-(\alpha) }_{b, os}(\RC(\bbR^n))$. Thus,  $\Delta^{-1}= A^{-1}\rho^{-2}$ has the desired properties. If instead $n\le2$ and $\alpha<n-2$ is such that  $(\alpha,n-2)\cap i\Spec_b(\widehat{A})=\emptyset$, then the map \eqref{euc.3a} is surjective with kernel given by the constant functions.  Thus we can still define $A^{-1}\in \Psi^{-2, \widehat{H}_n^+(\alpha),\widehat{H}_n^-(\alpha) }_{b, os}(\RC(\bbR^n))$ unambiguously by projecting off the constants, so that we can take again $\Delta^{-1}= A^{-1}\rho^{-2}$. 
\end{proof}

The corresponding property for Dirac operators follows. Let $E$ be a representation of the Clifford algebra $\Cl(\bbR^n,g_{\euc})$ and $D\in \Diff^1(\bbR^n;E)$ the corresponding Dirac operator.
\begin{corollary}
There is a natural continuous linear map
$$
      D^{-1}: \cS(\bbR^n;E)\to \cA^{J_n}_{\phg}(\RC(\bbR^n);E)
$$
such that $D\circ D^{-1} f=f$ for all $f\in \cS(\bbR^n;E)$, where
$$
        J_n=\left\{  \begin{array}{ll}
                 H_n+1, &  n\ne 2, \\
                 (H_n+1)\setminus \{(1,1)\}, & n=2,
        \end{array}
        \right.
$$
is such that $\inf J_n= n-1$.   
 
\label{euc.4}\end{corollary}
\begin{proof}
In this case, $\Delta:=D^2$ is just the Euclidean Laplacian acting on a trivial bundle with trivial connection on $\bbR^n$.  The previous corollary gives us a natural continuous linear map
$$
     \Delta^{-1}: \cS(\bbR^n;E)\to \cA^{H_n}_{\phg}(\RC(\bbR^n);E)
$$
such that $\Delta\circ \Delta^{-1} f=f $ for all $f\in \cS(\bbR^n;E)$.  Thus, it suffices to take $D^{-1}= D\Delta^{-1}$. Since $D$ is a scattering differential operator, it adds 1 to each element of the index set, which immediately completes the proof when $n\neq 2$.

When $n=2$, $\inf (H_n+1)=(n-1,1)$, so we still have to show we can take $J_2= (H_2+1)\setminus (1,1)$; that is, that there is no term of the form $\frac{\log r}{r}$ in the expansion at infinity of any element of the image of $D^{-1}$. To do this, let $f\in \cS(\bbR^n;E)$. The first term in the asymptotic expansion of $\Delta^{-1}f$ at infinity is of the form 
$$
       f_{\pa}\log r
$$
with $f_{\pa}\in \ker \Delta_{\bbS^1}$.  On the other hand, in polar coordinates, $D$ is of the form
$$
        D= \gamma_x\pa_x+ \gamma_y \pa y =\frac{1}{r}(\gamma_{\theta} \frac{\pa}{\pa \theta}+ \gamma_r r\frac{\pa}{\pa r});
$$
here $\gamma_x$ and $\gamma_y$ are fixed elements of the Clifford algebra ,while
$$
\gamma_{r}= \gamma_x\cos \theta+ \gamma_y \sin\theta\quad \mbox{and}  \quad \gamma_{\theta}= -\gamma_x \sin\theta+ \gamma_y\cos \theta
$$
are endomorphisms of $E$ (seen as a bundle on $\bbR^2$) depending only on $\theta$.  Since $f_{\delta}\in \ker \Delta_{\bbS^1}$, we must have $\frac{\pa f_{\pa}}{\pa \theta}=0$.  Consequently, 
$$
      D(f_{\delta}\log r)= \gamma_r \frac{f_{\delta}}{r}
$$ 
is the first term in the asymptotic expansion of $D^{-1}f= D\Delta^{-1}f$.  In particular, there are no term of the form $\frac{\log r}{r}$ in the asymptotic expansion of $D^{-1}f$, showing that we can take the smaller index set $J_2= (H_2+1)\setminus \{(1,1)\}$ when $n=2$. 
\end{proof}

\subsection{Mapping properties near a blown-up point}
We now introduce the space 
\begin{equation}
  \widetilde{\bbR}^n= [\bbR^n ; \{0\}] 
\label{bu.1}\end{equation}
formed by blowing up the origin in $\bbR^n$.   Let $\beta: \widetilde{\bbR}^n\to \bbR^n$ be the blow-down map. As before, let $\Delta$ be the standard Euclidean Laplacian. However, this time, we are interested in the mapping properties near the origin rather than those near infinity. From \eqref{euc.1a}, we see that, in polar coordinates, the operator 
$$
      B= r^2 \Delta = \Delta_{\bbS^{n-1}}- \left(r\frac{\pa}{\pa r}\right)^2 - (n-2)r\frac{\pa}{\pa r}
$$
is a $b$-operator near the boundary component $\beta^{-1}(0)$ in $\widetilde{\bbR}^n$. Its indicial family there is given by
\begin{equation}
     \widehat{B}(\tau)= \Delta_{\bbS^{n-1}} +\tau^2 -i(n-2)\tau
\label{bu.2}\end{equation}
so that
$$
\Spec_b(\widehat{B})=\left\{  i\frac{n-2}2\pm i\sqrt{\nu+\left( \frac{n-2}2 \right)^2} \; | \; \nu \in \Spec(\Delta_{\bbS^{n-1}})\right\}.
$$
From \cite[(5.118)]{MelroseAPS}, we have for each fixed $\alpha\in \bbR$ two natural index sets associated to the indicial family $\widehat{B}$:
\begin{multline}
    \widehat{G}_n^{\pm}(\alpha)= \left\{  (z,k)\in \bbC\times \bbN_0\; | \; \exists m\in \bbN_0, \Re z> \pm\alpha+m,  \right.  \\
    \left. \pm(z-m) \in i\Spec_b(\widehat{B}), \quad 
   k+1 \le \sum_{j=0}^m \ord(-i(z-j)) \right\}.
\label{bu.2a}\end{multline}

Since any Laplacian is locally modelled on the Euclidean Laplacian, it is not surprising that we obtain the same index sets if we apply the same blow-up construction to a Laplacian not necessarily coming from the Euclidean metric.

\begin{lemma}
Let $g$ be a Riemannian metric on $\bbR^n$ and $\Delta_g$ the corresponding Laplacian.  Then $r^2 \Delta_g$ is a $b$-operator near $\beta^{-1}(0)$, with indicial family having the same zeros with the same ranks and orders as the indicial family \eqref{bu.2}. 
\label{bu.3}\end{lemma}   
\begin{proof}
Choose linear coordinates on $\bbR^n$ so that $g_{ij}(0)=\delta_{ij}$ at the origin.  With this choice, $\Delta_g= \Delta+ f\Delta'$, where $f$ is a smooth function with $f=\mathcal{O}(r)$ near the origin and $\Delta'$ is a  smooth differential operator of order $2$.  Thus,  $r^2\Delta_g$ has the same indicial family as $r^2\Delta$, from which the result follows.    
\end{proof}

As before, we need to convert this to a statement about Dirac operators. Let $g$ be a (possibly incomplete) smooth Riemannian metric on $\bbR^n$, $\Cl_g(T\bbR^n)$ the corresponding bundle of Clifford algebras, and $E\to \bbR^n$ a Clifford bundle with Clifford connection.  Let $\eth$ be the corresponding Dirac-type operator.  We will be interested in the operator $D:=\eth+a$ for some $a\in\CI(\bbR^n;\End(E))$.
\begin{lemma}
The operator $D:=\eth+a$ above is such that 
$r^2D^2$ is a $b$-operator near $\beta^{-1}(0)$ in $\widetilde{\bbR^n}$ with indicial family having the same zeros as \eqref{bu.2} with the same orders and ranks.  In particular, for $\alpha\in\bbR$, $r^2D^2$ has the same associated index sets \eqref{bu.2a}.
\label{bu.4}\end{lemma}
\begin{proof}
As before, choose linear coordinates on $\bbR^n$ such that $g_{ij}(0)=\delta_{ij}$ and choose a trivialization of $E$ extending the Clifford action of $\left.\Cl_g(T\bbR^n)\right|_{0}$ at the origin to all of $E$.  We then have an Euclidean Dirac-type operator $\eth_{\euc}$ such that $\eth-\eth_{\euc}= f\eth'$, where $f$ is a smooth function with $f=\mathcal{O}(r)$ near the origin and $\eth'\in\Diff^1(\bbR^n;E)$.  Therefore, since $r(f\eth' +a)$ is a $b$-operator with vanishing indicial family at $\beta^{-1}(0)$, we see that the $b$-operator
$$
  r^2D^2= r^2(\eth_{\euc}+f\eth' +a)^2= r^2\eth_{\euc}^2+ r^2( (f\eth'+a)^2 + \eth_{\euc}(f\eth' +a) +(f\eth'+a)\eth_{\euc})
$$
has the same indicial family as $r^2\eth_{\euc}^2$.  But $\eth_{\euc}^2$ is just the Euclidean Laplacian acting on the sections of the trivialized vector bundle $E$.  This means that the indicial family of $r^2\eth_{\euc}^2$ (and therefore of $r^2D^2$) has the same zeros with the same orders as the indicial family \eqref{bu.2}.
\end{proof}

The fact that the zeros of the indicial family and their orders do not depend on the particular choice of $\eth$ and $a$ is a useful fact when we have a family of such objects. We will be particularly interested in the following situation. Let $\nu: F\to X$ be a real vector bundle of rank $n$ over a smooth compact manifold $X$, possibly with boundary.  Let $\widetilde{F}$ be obtained from the total space $F$ of the vector bundle by blowing up the zero section, and denote by $\beta:\widetilde{F}\to F$ the blow down map.  Let $E\to F$ be a vector bundle over $F$, and consider a family of fiberwise operators $D\in \Diff^1(F/X;E)$ which are of the form considered in Lemma~\ref{bu.4}.  

\begin{proposition}
For $\lambda\in\bbC$, let $\lambda\mapsto f_{\lambda}\in\varrho\CI_c(\widetilde{F};E)$ be a holomorphic family of sections, where $\varrho\in\CI(\widetilde{F})$ is a boundary defining function for the boundary face $\beta^{-1}(0)$ in $\widetilde{F}$.  Then there exists a holomorphic family $\lambda\mapsto u_{\lambda} \in\cA^{G_n}_{\phg}(\widetilde{F})$ with uniform compact support on $\widetilde{F}$ such that
\begin{equation}
       (D-\lambda)u_{\lambda}-f_{\lambda}
\label{bu.5a}\end{equation}
has a trivial polyhomogeneous expansion at the boundary face $\beta^{-1}(0)$ in $\widetilde{F}$.
Here
$$
    G_n=\left(\widehat{G}^{+}_n(\alpha)\overline{\cup} (\bbN_0+3)\right)-1,
$$
with $\alpha<2$ chosen so that $z\in i\Spec_b(\widehat{B})\Longrightarrow \Re z\notin (\alpha,2)$.  In particular, $\inf G_n=(1,0)$.   
\label{bu.5}\end{proposition}
\begin{proof}
We need to find a holomorphic family $v_\lambda\in \cA^{G_n +1}_{\phg}(\widetilde{F})$ such that
$$
     \varrho^2(D-\lambda)^2 v_{\lambda}-\varrho^2 f_{\lambda}
$$
has a trivial polyhomogeneous expansion at $\beta^{-1}(0)$, for then
$$
     u_{\lambda}= (D-\lambda) v_{\lambda}
$$
would be the desired solution. The index set of $\varrho^2 f_{\lambda}$ is $\bbN_0+3$, and $\rho^2(D-\lambda)^2$ is a family of $b$-operators with indicial families having their poles specified by Lemma~\ref{bu.4}.  Thus, in a fixed fiber $\widetilde{F}_p=\widetilde{\nu}^{-1}(p)$ of $\widetilde{\nu}:\widetilde{F}\to X$ and for fixed $\lambda$, we can apply \cite[Lemma~5.44]{MelroseAPS} to find $v_{\lambda,p}\in\cA^{G_n}(\widetilde{F}_p)$ such that  $\varrho(D-\lambda)v_{\lambda,p}$ has the same polyhomogeneous asymptotic expansion as $f$ on $\widetilde{F}_p$.  But since $G_n$ is fixed, the proof of \cite[Lemma~5.44]{MelroseAPS} generalizes immediately, allowing us to do this smoothly on $X$ and in a holomorphic way in $\lambda$.
\end{proof}

 \section{Geometric microlocal preliminaries} In this appendix, we give a brief summary of some key concepts in geometric microlocal analysis; these definitions are originally due to Melrose. Further details may be found in, e.g., \cite{me1,MelroseAPS,maz91,gr}.

\subsection{Manifolds with corners and blow-up} 

A \textbf{manifold with corners} $W$ of dimension $n$ is loosely speaking a topological space with differentiable structure  locally modelled on $\mathbb R_{+}^k\times\mathbb R^{n-k};$ the $n$-dimensional unit cube is a simple example. The boundary of $W$ is a union of \textbf{boundary hypersurfaces}, and we assume throughout that each boundary hypersurface is embedded. For any boundary hypersurface $H\subset W,$ we let $\rho_H$ denote a \textbf{boundary defining function} for $H$; that is, $\rho_H$ is a smooth non-negative function on $W$ which is zero precisely at $H$ and with nonvanishing differential on $H.$

Blow-up is a procedure for creating new manifolds with corners from old ones. We say that $S\subset W$ is a \textbf{p-submanifold} if for every point $p\in S,$ there exists a local model $\mathbb R_+^k(x_1,\ldots,x_k)\times\mathbb R^{n-k}(x_{k+1},\ldots,x_n)$ on $W$ centered at $p$ in which $S$ is the zero set of a fixed subset of the coordinates $\{x_i\}.$ Given such a p-submanifold, we define the \textbf{(radial) blow-up} $[W;S]$ by replacing $S\subset W$ with a copy of its spherical normal bundle. This procedure can be seen as the introduction of polar coordinates around $S.$

\subsection{$b$-maps and $b$-fibrations}

Let $W$ and $Z$ be manifolds with corners and let $\mathcal M(W)$ and $\mathcal M(Z)$ be the collections of their respective boundary hypersurfaces.
\begin{definition}\cite{me1} A \textbf{$b$-map} $f:W\rightarrow Z$ is a smooth map with the property that for each $H\in\mathcal M(Z),$ there exist nonnegative integers $e_f(G,H)$ such that
\[f^*\rho_H= a_G\prod_{G\in\mathcal M(W)}\rho_G^{e_f(G,H)}\] for $a_G$ a smooth positive function. \end{definition}
Technically, this is what Melrose calls an \textbf{interior $b$-map}, but we may restrict our attention to this case and hence we drop the word `interior.' Note that this definition is independent of the choice of boundary defining functions. The numbers $e_f(G,H)$ are called the \textbf{exponent matrix} for the $b$-map $f.$ Any $b$-map also induces a well-defined map from the boundary faces (not just hypersurfaces) of $W$ to the boundary faces of $Z,$ which we also call $f$ (see \cite{gr}).

Note that this condition ensures in particular that no boundary hypersurface is mapped into a corner.

Any $b$-map also induces a natural map $^bf_*,$ called the \textbf{$b$-differential}, between the `$b$-tangent bundles' $^bT^*W$ and $^bT^*Z.$ Since we do not need to use the explicit form of the $b$-differential, we refer to \cite{me1} for details. The only key fact that we need is that the $b$-differential respects composition; given $b$-maps $f:W\rightarrow Z$ and $g:Z\rightarrow Y,$ $^bg_*\circ\ ^bf_*=\ ^b(g\circ f)_*.$

\begin{definition} A $b$-map $f:W\rightarrow Z$ is a \textbf{$b$-submersion} if $^bf_*(w)$ is surjective for all $w\in W.$\end{definition}

\begin{definition} A \textbf{$b$-fibration} $f:W\rightarrow Z$ is a surjective  $b$-submersion such that  for each $G\in \cM(G)$, either $f(G) \in \cM(Z)$ or $f(G)=Z$.  \end{definition}
\begin{remark}
For a $b$-submersion, the second condition is equivalent to the $b$-normality condition of \cite{me1}, see \cite{gr}.
\label{ne.2}\end{remark}
\begin{proposition}\label{prop:bsubcomp} Suppose that $f:W\rightarrow Z$ and $g:Z\rightarrow Y$ are $b$-maps which are $b$-submersions. Then $g\circ f$ is a $b$-map and a $b$-submersion.\end{proposition}
\begin{proof} The composition of $b$-maps is well-known to be a $b$-map; see remark 2.12 in \cite{gr}. As for the $b$-submersion condition, it follows immediately from the fact that the $b$-differential respects composition. \end{proof}

\subsection{Polyhomogeneous conormal functions, pull-back, and push-forward}
In order to define the calculi of pseudodifferential operators that we need, we recall the definition of \textbf{polyhomogeneous conormal functions} from \cite{me1} (also see \cite{MelroseAPS,maz91,gr}). First recall that smooth functions on a manifold with corners $W$ have Taylor expansions in terms of the boundary defining functions at each boundary hypersurface, with joint Taylor expansions at the corners. For polyhomogeneous conormal functions, we simply allow more terms in the Taylor expansion. In particular, we require asymptotic expansions at each boundary hypersurface $H,$ with each term in the expansion having the form
\[c\rho_H^z(\log\rho_H)^p\]
for $(z,p)\in\mathbb C\times\mathbb N_0,$ and with only finitely many terms in each expansion having $\Re z<N$ for any fixed $N.$ We also require joint asymptotic expansions at each corner. If $u$ is polyhomogeneous conormal on $W,$ the \textbf{index set} at each $H$ is the set of $(z,p)$ with a term in the expansion of $u$ at $H$ of the form $c\rho_H^z(\log\rho_H)^p.$ We also adopt the convention that if $(z,p)$ is in the index set, so is $(z+1,p),$ and if $p>0$ so is $(z,p-1).$ The \textbf{index family} of $u$ is a collection of index sets, one for each boundary hypersurface. If the index family of $u$ is contained in an index family $\mathcal F,$ we write
\[u\in\mathcal A^{\mathcal F}(W).\]

There are various operations on index sets, from \cite{me1} (see also \cite{maz91}). First, given two index sets $E_1$ and $E_2,$ we define
\[E_1+E_2=\{(z_1+z_2,p_1+p_2):\ (z_1,p_1)\in E_1,\ (z_2,p_2)\in E_2\}.\]
The operation of \textbf{extended union} is defined by
\[E_1\extu E_2=E_1\cup E_2\cup\{(z,p_1+p_2+1):(z,p_1)\in E_1,\ (z,p_2)\in E_2\}.\]
For any index set $E,$ $\inf E=(z_0,p_0)$ is defined by asking $z_0$ to be  the infimum of $\Re z$ over all $z$ for which $(z,0)\in E$ and $p_0$ to be the greatest $p\in \bbN_0$ such that $(z,p_0)\in E$ with $\Re z=z_0$. We say an index set $E$ is \textbf{positive} if $\inf E>0=(0,0).$ For a positive index set $E,$ we write
\begin{equation}\label{posinfsum} E|_{\infty}=\bigcup_{k=1}^{\infty}\sum_{i=1}^{k} E.
\end{equation}
Since $E$ is positive, $E|_{\infty}$ is also an index set.

We also need to consider polyhomogeneous conormal distributions with interior conormal singularities at p-submanifolds. These conormal singularities are locally modelled on the singularities of the Schwarz kernel of a differential or pseudodifferential operator on a compact manifold; we omit the complete definition, which may be found in \cite{gr}. If $u$ is a polyhomogeneous conormal distribution on $W$ with index family $\mathcal F$ and with an interior conormal singularity of order $m$ at a p-submanifold $S\subset W,$ we write
\[u\in\mathcal A^{\mathcal F}I^m(W,S).\]

Polyhomogeneous conormal distributions have nice properties under pull-back and push-forward by $b$-maps between manifolds with corners.
\begin{proposition}[Melrose's pull-back and push-forward theorems] Suppose that $W_1$ and $W_2$ are manifolds with corners and $f:W_1\rightarrow W_2$ is a $b$-map. Then:

a) If $u$ is polyhomogeneous conormal on $W_2,$ $f^*u$ is polyhomogeneous conormal on $W_1.$

b) If $f$ is also a $b$-fibration and $v$ is polyhomogeneous conormal on $W_1,$ and $f_*v$ is well-defined, then $f_*v$ is polyhomogeneous conormal on $W_2.$

Moreover, the index families of the pullback and pushforward distributions may be computed from the index families of the original distributions and the exponent matrix of $f.$ See \cite{MelroseAPS,maz91,gr} for the specific formulas.
\end{proposition}

There are also versions of the pull-back and push-forward theorems for polyhomogeneous conormal distributions with interior conormal singularities; see the appendix to \cite{emm}.

\section{Proof of composition formula}\label{appendix:composition}
\subsection{Triple surgery space}
In order to discuss the composition of the operators in our pseudodifferential calculus, we now define a `triple surgery space' $X_s^3.$ This construction combines key features of the constructions in \cite{mame1,v}. In fact, as with the double space construction, our triple space is constructed by taking the triple space from \cite{mame1} and performing additional blow-ups.  Additionally, we will be able to identify the $\phi$-calculus triple space from \cite{mame2} with a submanifold of our surgery triple space. Our notation is taken mostly from the explanation of the $\phi$-triple space given in \cite{v}.

We start with the space $M^3\times[0,1]_{\epsilon},$ and label its boundary hypersurface $\bhs{Z}.$ Boundary defining functions for each copy of $H$ are $x,$ $x',$ and $x''$ respectively. Now we perform a series of iterated blow-ups: 

1) Blow up $H\times H\times H\times\{0\}.$ Call the new face $\bhs{T}.$

2) Blow up $H\times H\times M\times \{0\},$ $H\times M\times H\times\{0\},$ and $M\times H\times H\times\{0\}.$ Call the new faces $\bhs{F},$ $\bhs{C},$ and $\bhs{S}$ respectively.

3) Blow up $H\times M\times M\times\{0\},$ $M\times H\times M\times\{0\},$ and $M\times M\times H\times\{0\},$ calling the new faces $\bhs{N_1},$ $\bhs{N_2},$ and $\bhs{N_3}$ respectively. At this point, we have reconstructed the triple space from \cite{mame1}, which we call $X_{b,s}^3.$ 

4) Let $D$ be the lifted triple fiber diagonal in $X_{b,s}^3$ - that is, the interior lift of $\{x=x'=x'',y=y'=y''\}.$ Then blow up $\bhs{T}\cap D.$ This creates a new face which we call $\bhs{\phi_{TT}},$ corresponding to the central face in the $\phi$-triple space \cite{v}. 

5) Let $D_S,$ $D_C,$ $D_F$ be the lifted double fiber diagonals given by the interior lifts of $\{x=x',y=y'\},$ $\{x=x'',y=y''\},$ and $\{x'=x'',y'=y''\}$ respectively. Then create new faces $\bhs{\phi_{ST}},$ $\bhs{\phi_{CT}},$ $\bhs{\phi_{FT}}$ by blowing up $\bhs T\cap D_S,$ $\bhs T\cap D_C,$ $\bhs T\cap D_F$ respectively (in any order - the submanifolds are disjoint). These correspond to the `spokes' coming from the center of the scattering triple space \cite{v}. 

6) Finally, viewing $D_S,$ $D_C,$ and $D_F$ as p-submanifolds of the new space (by taking interior lifts in each case), we create new faces $\bhs{\phi_S},$ $\bhs{\phi_C},$ and $\bhs{\phi_S}$ by blowing up $\bhs S\cap D_S,$ $\bhs C\cap D_C,$ and $\bhs F\cap D_F$ respectively.

\begin{figure}
    \centering

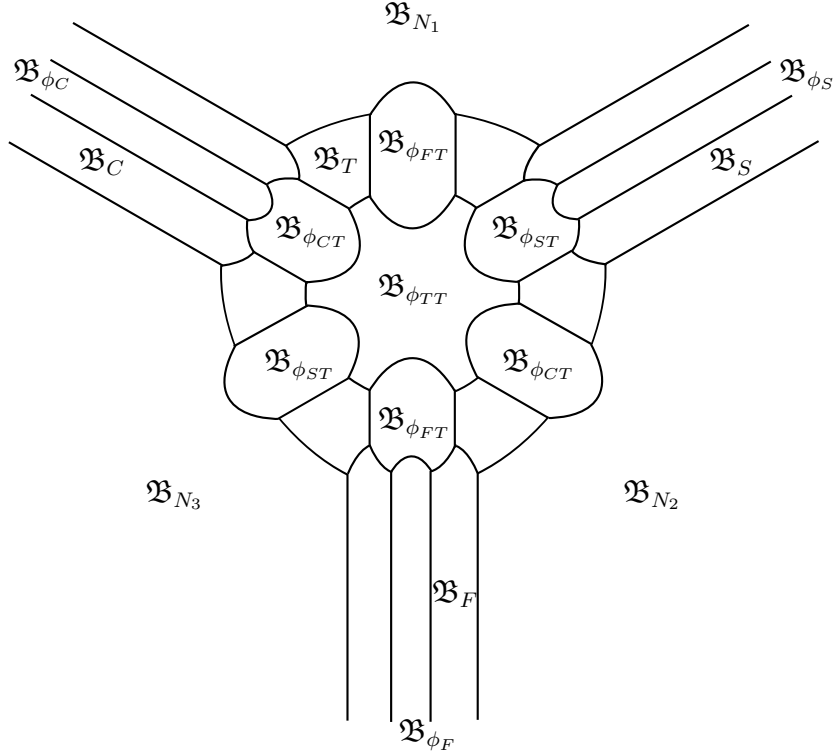

    \caption{The triple surgery space $X_s^3,$ with $\epsilon$ suppressed.}
\label{fig:triplespace}
      \end{figure}

This final space is $X_s^3,$ illustrated with $\epsilon$ suppressed in Figure \ref{fig:triplespace}. It has fifteen boundary hypersurfaces, labeled as discussed in the construction; we let each boundary hypersurface keep its label under the lift via each successive blow-down map. Additionally, the original $\epsilon=0$ boundary hypersurface $\bhs{Z}$ may be identified with the $\phi$-triple space of \cite{mame2,v}. As before, let $\beta_{(3)}$ be the blow-down map to $M^3\times[0,1]_{\epsilon}.$

\subsection{Properties of the projection maps}

First we show that there are natural projection maps from the double space to the single space and that these maps are $b$-fibrations.

\begin{proposition}\label{propc1} Projection off $x,y,z$ (resp. $x',y',z'$) extends by continuity to a well-defined map $\beta_{(2),R}:X^2_{s}$ to $X_s$ (resp. $\beta_{(2),L}$), and each map is a $b$-fibration. \end{proposition}
\begin{proof} Analogous projection maps $\pi^2_{b,s,R}$ and $\pi^2_{b,s,L}$ from $X^2_{b,s}$ to $X_s$ appear in \cite{mame1}; they are shown to be well-defined $b$-fibrations. Let $\beta_{\ff}$ be the blow-down map from $X^{2}_{\epsilon}$ to $X^2_{b,s}$ given by blowing down $\bhs{\ff}.$ Then
\[\beta_{(2),R}=\pi^2_{b,s,R}\circ\beta_{\ff};\ \beta_{(2),L}=\pi^2_{b,s,L}\circ\beta_{\ff},\]
and therefore $\beta_{(2), R}$ and $\beta_{(2), L}$ are well-defined. Moreover, from Proposition \ref{prop:bsubcomp} and the well-known fact that all blow-down maps are $b$-maps and $b$-submersions (a reference may be found in the proof of Lemma 12 in \cite{hmm}), both $\beta_{(2), R}$ and $\beta_{(2), L}$ are $b$-maps and $b$-submersions. It remains only to check that each map is $b$-normal. We know from \cite{mame1} that $\pi^2_{b,s,R}$ and $\pi^2_{b,s,L}$ are both $b$-fibrations and hence $b$-normal, and the images of $\bhs{\mf},$ $\bhs{\lf},$ and $\bhs{\rf}$ remain unchanged after composition with $\beta_{\ff}.$ Finally, it is easy to see in local coordinates (see Figure \ref{fig:doublespace}) that $\beta_{(2),R}(\phibf)=\beta_{(2), R}(\bhs{\ff})=\bhs{\bs}.$ Since an analogous statement is true for $\beta_{(2),L},$ both maps are $b$-normal and hence $b$-fibrations. 
\end{proof}
The exponent matrices may be computed directly; for $\pi_{s,R}^2,$ all entries are zero except for the following, which are equal to $1:$
\[\{(\bhs{\rf},\bhs{\bs}),(\phibf,\bhs{\bs}),(\bhs{\ff},\bhs{\bs}),(\bhs{\lf},\bhs{\ms}),(\bhs{\mf},\bhs{\ms})\}.\] 
Similarly, for $\pi_{s,L}^2,$ the entries which are equal to $1$ are:
\[\{(\bhs{\lf},\bhs{\bs}),(\phibf,\bhs{\bs}),(\bhs{\ff},\bhs{\bs}),(\bhs{\rf},\bhs{\ms}),(\bhs{\mf},\bhs{\ms})\}.\]

Now we show that the natural projection maps from the triple space to the double space are $b$-fibrations; this analysis is the critical technical tool used in the proof of the composition formula.

\begin{lemma} Projection off $x,y,z$ (resp. $x',y',z'$ or  $x'',y'',z''$) extends by continuity to a well-defined map $\pi_F:X^3_{s}\rightarrow X^2_{s}$ (resp. $\pi_C,$ $\pi_S$), and each map is a $b$-fibration. \end{lemma}
\begin{proof}By symmetry, it suffices to prove that $\pi_F$ is well-defined and a $b$-fibration. The image of $\pi_F$ should be thought of as $X^2_{s}$ with the coordinates $x'$ and $x''$ in place of $x$ and $x'.$

Regard the fiber diagonal $D_F$ (i.e. the closure of the lift of the interior of $\{x=x',y=y'\}$) as a p-submanifold of $X_{b,s}^3.$ Define an auxiliary space \[\tilde X_s^3=[[X_{b,s}^3;\ \bhs T\cap D_F];\ \bhs F\cap D_F],\] and let $\bhs{\tilde\phi_{FT}}$ be the face created by the first blow-up. 
\begin{proposition} $\tilde X_s^3$ is a blow-down of $X_s^3.$
\end{proposition}
\begin{proof} Consider $[\tilde X_s^3;\bhs{\tilde\phi_{FT}}\cap D_C\cap D_S].$ By \cite[Lemma 5]{hmm}, blow-ups which are nested, transverse, or disjoint commute. Since $\bhs F\cap D_F$ and $\bhs{\tilde\phi_{FT}}\cap D_C\cap D_S$ are disjoint p-submanifolds of $[X_{b,s}^3; T\cap D_T],$
\[[\tilde X_s^3;\bhs{\tilde\phi_{FT}}\cap D_C\cap D_S]=
[[[X_{b,s}^3;\ \bhs T\cap D_F];\ \bhs{\tilde\phi_{FT}}\cap D_C\cap D_S];\ \bhs F\cap D_F].\]
Now $\bhs{\tilde\phi_{FT}}\cap D_C\cap D_S$ is the lift of $\bhs T\cap D\subset X_{b,s}^3$ to $[X_{b,s}^3;\ \bhs T\cap D_F].$ Since $\bhs T\cap D\subset \bhs T\cap D_F$ and nested blow-ups commute,
\[[\tilde X_s^3;\bhs{\tilde\phi_{FT}}\cap D_C\cap D_S]=[[[X_{b,s}^3;\ \bhs T\cap D];\ \bhs T\cap D_F];\ \bhs F\cap D_F].\]
The final space in this chain is a blow-down of $X_s^3,$ obtained by blowing down $\bhs{\phi_C}$ and $\bhs{\phi_S},$ then $\bhs{\phi_{CT}}$ and $\bhs{\phi_{ST}}.$ \end{proof}
\begin{proposition} Projection off $x$ extends by continuity to a map $\hat{\pi}_F$ from $\tilde X_s^3$ to $X_s^2,$ and $\hat{\pi}_F$ is a $b$-fibration.
\end{proposition}
\begin{proof} One could write everything out in local coordinates and prove the proposition directly, but it is far simpler to use a technical lemma from \cite{hmm}. In \cite{mame1}, an analogous projection map which we call $\pi_{b,F}$ is shown to be a well-defined $b$-fibration from $X_{b,s}^3$ to $X_{b,s}^2.$ By equation (102) in \cite{mame1}, each entry in the exponent matrix of $\pi_{b,F}$ is either zero or one. Now remember that $X_s^2$ is obtained from $X_{b,s}^2$ by blowing up $\bhs{\bff}\cap D_{fib}.$ The inverse image of $\bhs{\bff}\cap D_{fib}$ under $\pi_{b,F}$ is not a p-submanifold itself, but it is the union of the two p-submanifolds $\bhs T\cap D_F$ and $\bhs F\cap D_F.$ Therefore, by Lemma 10 of \cite{hmm} and the definitions of $\tilde X_s^3$ and $X_s^2,$ the $b$-fibration $\pi_{b,F}$ lifts to a well-defined map which is a $b$-fibration from $\tilde X_s^3$ to $X_s^2.$ This lift is precisely $\hat\pi_F,$ completing the proof of the proposition.
\end{proof}

We conclude from these two propositions that $\pi_F$ is the composition of a blow-down map and $\hat{\pi}_F$ and is hence well-defined. Since both the blow-down map and $\hat{\pi}_F$ are $b$-maps and $b$-submersions, so is $\pi_F$ (Proposition \ref{prop:bsubcomp}). It remains only to check $b$-normality. This may be done directly and is not hard, but for later purposes we first compute the exponent matrix $e_{\pi_F}(G,H).$ From Figure \ref{fig:triplespace}, as well as the analogous analysis in \cite{mame1} and \cite{v}, the pullbacks of boundary defining functions are given by
\[(\pi_F)^*\rho_{\ff}=\rho_{\phi_{TT}}\rho_{\phi_{FT}}\rho_{\phi_F}; (\pi_F)^*\rho_{\fbf}=\rho_{\phi_{CT}}\rho_{\phi_{ST}}\rho_{T}\rho_{F};\]\[(\pi_F)^*\rho_{\lf}=\rho_{\phi_S}\rho_S\rho_{N_2};\ (\pi_F)^*\rho_{\rf}=\rho_{\phi_C}\rho_C\rho_{N_3};\ (\pi_F)^*\rho_{\mf}=\rho_{N_1}\rho_Z.\]
So the exponent matrix $e_{\pi_F}(G,H)$ is zero unless $(G,H)$ is one of the following pairs, in which case $e_{\pi_F}(G,H)=1$:
\begin{multline}\label{expmatrixf}\left\{(\bhs{\phi_{TT}},\bhs{\ff}), (\bhs{\phi_{FT}},\bhs{\ff}), (\bhs{\phi_F},\bhs{\ff}), (\bhs{\phi_{CT}},\phibf), (\bhs{\phi_{ST}},\phibf), \right.\\ \left.(\bhs T,\phibf), (\bhs F,\phibf),(\bhs{\phi_S},\bhs{\lf}),(\bhs S,\bhs{\lf}),(\bhs{N_2},\bhs{\lf}),\right.\\ \left.(\bhs{\phi_C},\bhs{\rf}),(\bhs C,\bhs{\rf}),(\bhs{N_3},\bhs{\rf}),(\bhs{N_1},\bhs{\mf}),(\bhs Z,\bhs{\mf})\right\}.\end{multline}
Notice that each $G\in\mathcal M(X_s^3)$ has exactly one $H\in\mathcal(X_s^2)$ for which $e_{\pi_F}(G,H)$ is nonzero. By Remark~\ref{ne.2} and the fact that $\pi_F$ is already known to be a $b$-submersion, we conclude that $\pi_F$ is $b$-normal, and hence a $b$-fibration.
\end{proof}

The same arguments apply for $\pi_C$ and $\pi_S,$ and their exponent matrices may be read off from (\ref{expmatrixf}) by symmetry; all entries are zero, except the ones listed below, which are 1. For $\pi_C$:
\begin{multline}\label{expmatrixc}\left\{(\bhs{\phi_{TT}},\bhs{\ff}), (\bhs{\phi_{CT}},\bhs{\ff}), (\bhs{\phi_C},\bhs{\ff}), (\bhs{\phi_{FT}},\phibf), (\bhs{\phi_{ST}},\phibf), \right.\\ \left.(\bhs T,\phibf), (\bhs C,\phibf),(\bhs{\phi_S},\bhs{\lf}),(\bhs S,\bhs{\lf}),(\bhs{N_1},\bhs{\lf}),\right.\\ \left.(\bhs{\phi_F},\bhs{\rf}),(\bhs F,\bhs{\rf}),(\bhs{N_3},\bhs{\rf}),(\bhs{N_2},\bhs{\mf}),(\bhs Z,\bhs{\mf})\right\}.\end{multline}
And for $\pi_S$:
\begin{multline}\label{expmatrixs}\left\{(\bhs{\phi_{TT}},\bhs{\ff}), (\bhs{\phi_{ST}},\bhs{\ff}), (\bhs{\phi_S},\bhs{\ff}), (\bhs{\phi_{FT}},\phibf), (\bhs{\phi_{CT}},\phibf), \right.\\ \left.(\bhs T,\phibf), (\bhs S,\phibf),(\bhs{\phi_C},\bhs{\lf}),(\bhs C,\bhs{\lf}),(\bhs{N_1},\bhs{\lf}),\right.\\ \left.(\bhs{\phi_F},\bhs{\rf}),(\bhs F,\bhs{\rf}),(\bhs{N_2},\bhs{\rf}),(\bhs{N_3},\bhs{\mf}),(\bhs Z,\bhs{\mf})\right\}.\end{multline}

\subsection{Densities and blow-up}
As a preliminary step in the proofs, we note that the canonical density bundles on manifolds with corners transform nicely under blow-up. The following proposition, due to Melrose, may be proved by writing out the expression for a blow-up in local coordinates (see also \cite[Lemma 2.2]{v}):
\begin{proposition}\label{prop:densityblowup} Suppose that $F$ is a p-submanifold of a manifold with corners $W,$ $\dim F=f$ and $\dim W=w.$ Then if $\beta$ is the blow-down map from $[W;F]$ to $W,$
\[\beta^*\Omega_b(W)=\Omega_b([W;F]);\ \ \beta^*\Omega(W)=(\rho_F)^{w-f-1}\Omega([W;F]).\]
\end{proposition}

The following corollary of Proposition \ref{prop:densityblowup} will be useful in the proof of the composition theorem:

\begin{corollary}\label{cor:densityblowup}[Blow-up density transformations]
The density bundles transform under blow-ups as follows:
\begin{equation}\label{dentranssingle}(\beta_s)^*\nu(M\times[0,1]_{\epsilon})=\rho_{\bs}\nu(X_s);\end{equation}
\begin{equation}\label{dentransdouble}(\beta_s^2)^*\nu(M^2\times[0,1]_{\epsilon})=\rho_{\lf}\rho_{\rf}\rho_{\fbf}^{2}\rho_{\ff}^{h+3}\nu(X_s^2);\end{equation}
\begin{multline}\label{dentranstriple}(\beta_{(3)})^*\nu(M^3\times[0,1]_{\epsilon})=(\rho_{N_1}\rho_{N_2}\rho_{N_3})(\rho_F\rho_C\rho_S)^{2}(\rho_T)^{3}(\rho_{\phi_F}\rho_{\phi_C}\rho_{\phi_S})^{h+3}\cdot\\ (\rho_{\phi_{FT}}\rho_{\phi_{CT}}\rho_{\phi_{ST}})^{h+4}(\rho_{\phi_{TT}})^{2h+5}\nu(X_s^3).\end{multline}
\end{corollary}
\begin{proof} The proof of each statement follows from repeated applications of Proposition \ref{prop:densityblowup}. The proof of (\ref{dentranssingle}) requires only one application of Proposition \ref{prop:densityblowup} and is left to the reader.

To prove (\ref{dentransdouble}), recall that the space $X_s^2$ is obtained by iterated blow-up of $M^2\times [0,1]_{\epsilon}.$ The first blow-up creates the face $\bhs{\bff},$ and with $w-f-1=2,$ which appears at first to give a factor of $\rho_{\bff}^{2}.$ However, under the iterated blow-ups that follow, the defining function for the front face of this first blow-up actually lifts to $\rho_{\fbf}\rho_{\ff},$ so the first blow-up gives us a factor of $(\rho_{\fbf}\rho_{\ff})^{2}.$ Similarly, the blow-ups creating $\bhs{\lf}$ and $\bhs{\rf}$ give factors of $\rho_{\lf}$ and $\rho_{\rf}$ respectively, and the blow-up creating $\bhs{\ff}$ gives a factor of $\rho_{\ff}^{h+1}.$ Multiplying the various factors together gives (\ref{dentransdouble}).

The proof of (\ref{dentranstriple}) proceeds in an analogous manner. In the first step in the creation of the triple space, we blow up $\{x=x'=x''=\epsilon=0\}$ to create a new face $\bhs T.$ This blow-up has $w-f-1= 3$, so we expect a factor of $\rho_T^{3}.$ However, under the remaining iterated blow-ups, a boundary defining function for $\bhs T$ at this stage will lift to $\rho_T\rho_{\phi_{TT}}\rho_{\phi_{FT}}\rho_{\phi_{CT}}\rho_{\phi_{ST}},$ so we instead get a factor of
\[(\rho_T\rho_{\phi_{TT}}\rho_{\phi_{FT}}\rho_{\phi_{CT}}\rho_{\phi_{ST}})^{3}\]
from the first step. The remainder of the proof proceeds analogously and is again left to the reader.
\end{proof}

\subsection{Proof of mapping theorem} Now we prove Theorem \ref{mappingprops}. Assume for simplicity that $m=-\infty.$ By the Schwarz kernel theorem and the definition of the surgery calculus, the kernel of $A$ is a distribution $K(A)\in\mathcal A(X_s)^2$ with index sets as above, and we have
\begin{equation}\label{mapping}g=(\beta_{(2),L})_*(K(A)\kappa_{\phi}(\beta_{(2),R})^*f).\end{equation}
We will now apply the pullback and pushforward theorems to analyze this expression. 

When using the pullback theorem it is convenient to work with functions, but when using the pushforward theorem it is easier to work with the canonical full densities. So we transform (\ref{mapping}) into an equation where the left-hand side is a multiple of $\nu(X_s)$ and the interior of the push-forward on the right is a multiple of $\nu(X_s^2).$ First multiply both sides by $\nu_{\ephi}|d\epsilon|$ and rewrite, using the definition of $\kappa_{\phi}$:
\[g\nu_{\ephi}|d\epsilon|=(\beta_{(2),L})_*(K(A)(\beta_{(2),R})^*f((\beta_{(2),L})^*\nu_{\ephi}(\beta_{(2),R})^*\nu_{\ephi}|d\epsilon|)).\]
Using the definitions of the various density bundles, the left-hand side is $g\rho^{-(h+2)}\beta_s^*\nu(M\times[0,1]_{\epsilon})$ and the right-hand side is
\[(\beta_{(2),L})_*((\rho\rho')^{-(h+2)}K(A)(\beta_{(2),R})^*f((\beta_{(2),L})^*\beta_s^*\nu(M\times[0,1]_{\epsilon})(\beta_{(2),R})^*\beta_s^*\nu(M\times[0,1]_{\epsilon})|d\epsilon|^{-1}).\]
Since $\nu(M\times[0,1]_{\epsilon})=|dx\ dy\ dz\ d\epsilon|,$ we have that on $X_s^2,$
\[(\beta_{(2),L})^*\beta_s^*\nu(M\times[0,1]_{\epsilon})(\beta_{(2),R})^*\beta_s^*\nu(M\times[0,1]_{\epsilon})=(\beta_s^2)^*\nu(M^2\times[0,1]_{\epsilon})|d\epsilon|.\]
Using this identity together with the facts that $\rho=\rho_{\bs}$ on the single space, $\rho=\rho_{\lf}\rho_{\fbf}\rho_{\ff}$ on the double space, and $\rho'=\rho_{\rf}\rho_{\fbf}\rho_{\ff}$ on the double space, we have
\[g\rho_{\bs}^{-(h+2)}\beta_s^*\nu(M\times[0,1]_{\epsilon})=(\beta_{(2),L})_*((\rho_{\lf}\rho_{\rf}\rho^2_{\fbf}\rho^2_{\ff})^{-(h+2)}K(A)(\beta_{(2),R})^*f((\beta_s^2)^*\nu(M^2\times[0,1]_{\epsilon})).\]
Applying Corollary \ref{cor:densityblowup} yields
\begin{equation}\label{mappingtrans} g\rho_{\bs}^{-(h+1)}\nu(X_s)=(\beta_{(2),L})_*(K(A)(\beta_{(2),R})^*f(\rho_{\lf}\rho_{\rf}\rho_{\fbf}^2\rho_{\ff})^{-(h+1)}\nu(X_s^2)).\end{equation}
This is the form we wanted, and we are now in a position to apply the pullback and pushforward theorems. 

By the pullback theorem, the interior of the expression on the right-hand side of (\ref{mappingtrans}), $K(A)(\beta_{(2),R})^*f(\rho_{\lf}\rho_{\rf}\rho_{\fbf}^2\rho_{\ff})^{-(h+1)},$ is polyhomogeneous on $X_s^2$ with index family
\[E_{\ff}+F_{\bs}-(h+1)\textrm{ at }\bhs{\ff},\ E_{\fbf}+F_{\bs}-2(h+1)\textrm{ at }\phibf,\ E_{\lf}+E_{\ms}-(h+1)\textrm{ at }\bhs{\lf},\]
\[E_{\rf}+F_{\bs}-(h+1)\textrm{ at }\bhs{\rf},\ E_{\mf}+F_{\ms}\textrm{ at }\bhs{\mf}.\]
Then applying the pushforward theorem, we see that $g\rho^{-(h+1)}_{\bs}$ is polyhomogeneous on $X_s$ with index family 
\[(E_{\ff}+F_{\bs}-(h+1))\extu (E_{\fbf}+F_{\bs}-2(h+1))\extu(E_{\lf}+F_{\ms}-(h+1))\textrm{ at }\bhs{\bs},\]
\[(E_{\mf}+F_{\ms})\extu (E_{\rf}+F_{\bs}-(h+1))\textrm{ at }\bhs{\ms}.\]
Subtracting $h+1$ from the index set at $\bhs{\bs}$ gives the desired index family for $g,$ completing the proof. 

The case of general $m$ is no harder.   Then, $K(A)(\beta_{(2),R})^*f$ is still polyhomogeneous on $X_s^2$ with the same index family, but also with a conormal singularity at the lifted diagonal. However, the fibers of the pushforward by $\pi_{s,L}$ are transverse to the lifted diagonal, so the singularity `integrates out' and we get the same conclusion.

\subsection{Proof of composition formula}

We are now ready to prove Theorem \ref{composition}.

\begin{proof} Assume for simplicity that $m=m'=-\infty$; we will discuss the extension to arbitrary $m$ and $m'$ at the end of the proof. Then if the Schwarz kernel of $A$ is $K(A),$ et cetera, we have (following \cite{mame1}):
\[K(C)\kappa_{\phi}=(\pi_C)_*((\pi_F)^*(K(B)\kappa_{\phi})(\pi_S)^*(K(A)\kappa_{\phi})).\]

As in the proof of the mapping properties lemma, we want to write the left-hand side as a section of $\nu(X_s^2)$ and the interior of the pushforward on the right as a section of $\nu(X_s^3).$ First multiply both sides by $(\beta_{(2),L})^*\nu_{\ephi}$ and re-write everything in terms of $\nu(M^{i}\times[0,1]_{\epsilon}),$ $i=1,$ 2, 3. Using coordinates $(x,x',x'')$ and the corresponding $(\rho,\rho',\rho''),$ we have
\begin{multline}\label{thiseq}K(C)(\rho\rho'')^{-(h+2)}(\beta_s^2)^*\nu(M^2\times[0,1]_{\epsilon})=(\pi_C)_*((\rho\rho'\rho'')^{-(h+2)}(\pi_F)^*(K(B))(\pi_S)^*(K(A))\\ \cdot (\beta_{(3)})^*\nu(M^3\times[0,1]_{\epsilon})).\end{multline}
We rewrite the left-hand side of (\ref{thiseq}), using the fact that $\rho\rho''=(\rho_{\lf}\rho_{\rf}\rho_{\fbf}^2\rho_{\ff}^2)$ together with Corollary \ref{cor:densityblowup}, as
\[K(C)(\rho_{\lf}\rho_{\rf}\rho_{\fbf}^2\rho_{\ff})^{-(h+1)}\nu(X_s^2).\]
To analyze the right-hand side of (\ref{thiseq}), let
\[\nu_{RHS}=(\rho\rho'\rho'')^{-(h+2)}(\beta_{(3)})^*\nu(M^3\times[0,1]_{\epsilon}).\]
As functions on $X_s^3,$ $\rho,$ $\rho',$ $\rho''$ are the pullbacks of $\rho_{\bs}$ under $\pi_{F}\circ\beta_{(2),L},$ $\pi_{F}\circ\beta_{(2),R},$ and $\pi_S\circ\beta_{(2),R}$ respectively. So we may use the pullback theorem along with the exponent matrices for $\pi_F,$ $\pi_S,$ $\beta_{(2),L},$ and $\beta_{(2),R}$ to compute that
\begin{multline}\nu_{RHS}=(\rho_{\phi_{TT}}\rho_{\phi_{FT}}\rho_{\phi_{CT}}\rho_{\phi_{ST}}\rho_{T})^{-3(h+2)}(\rho_{\phi_F}\rho_{\phi_C}\rho_{\phi_S}\rho_F\rho_C\rho_S)^{-2(h+2)}\\ \cdot(\rho_{N_1}\rho_{N_2}\rho_{N_3})^{-(h+2)}\beta_{(3)}^*\nu(M^3\times[0,1]_{\epsilon}).\end{multline}
Then using Corollary \ref{cor:densityblowup} gives
\[\nu_{RHS}=(\rho_T)^{-3(h+1)}(\rho_{\phi_{FT}}\rho_{\phi_{CT}}\rho_{\phi_{ST}}\rho_F\rho_C\rho_S)^{-2(h+1)}(\rho_{\phi_{TT}}\rho_{\phi_F}\rho_{\phi_C}\rho_{\phi_S}\rho_{N_1}\rho_{N_2}\rho_{N_3})^{-(h+1)}\nu(X_s^3).\]

Now we put everything together by applying the pullback theorem to compute the index sets of $(\pi_S)^*K(A)$ and $(\pi_F)^*K(B).$ We find that $K(C)(\rho_{\lf}\rho_{\rf}\rho_{\fbf}^2\rho_{\ff})^{-(h+1)}\nu(X_s^2)$ is the pushforward by $\pi_C$ of a full density on $X_s^3$ with index sets, as a section of $\Omega(X_s^3),$ given by
\begin{multline}E_{\ff}+F_{\ff}-(h+1) \textrm{ at $\bhs{\phi_{TT}},$ } E_{\fbf}+F_{\ff}-2(h+1)\textrm{ at $\bhs{\phi_{FT}},$ } E_{\fbf}+F_{\fbf}-2(h+1)\textrm{ at $\bhs{\phi_{CT}},$ }\\
E_{\ff}+F_{\fbf}-2(h+1)\textrm{ at $\bhs{\phi_{ST}},$ } E_{\rf}+F_{\ff}-(h+1)\textrm{ at $\bhs{\phi_{F}},$ } E_{\lf}+F_{\rf}-(h+1)\textrm{ at $\bhs{\phi_{C}},$ }\\
E_{\ff}+F_{\lf}-(h+1)\textrm{ at $\bhs{\phi_{S}},$ } E_{\fbf}+F_{\fbf}-3(h+1)\textrm{ at $\bhs T,$ } E_{\rf}+F_{\fbf}-2(h+1)\textrm{ at $\bhs F,$ }\\
E_{\lf}+F_{\rf}-2(h+1)\textrm{ at $\bhs C,$ } E_{\fbf}+F_{\lf}-2(h+1)\textrm{ at $\bhs S,$ } E_{\lf}+F_{\mf}-(h+1)\textrm{ at $\bhs{N_1},$ }\\
E_{\rf}+F_{\lf}-(h+1)\textrm{ at $\bhs{N_2},$ } E_{\mf}+F_{\rf}-(h+1)\textrm{ at $\bhs{N_3},$ } E_{\mf}+F_{\mf}\textrm{ at $\bhs Z.$}
\end{multline}
We now apply the pushforward theorem to conclude that, as an element of $\Omega(X_s^2),$ the full density $K(C)(\rho_{\lf}\rho_{\rf}\rho_{\fbf}^2\rho_{\ff})^{-(h+1)}\nu(X_s^2)$ has index sets
\[\tilde G_{\ff}=(E_{\ff}+F_{\ff}-(h+1))\extu(E_{\fbf}+F_{\fbf}-2(h+1))\extu(E_{\lf}+F_{\rf}-(h+1));\] 
\[\tilde G_{\fbf}=(E_{\fbf}+F_{\ff}-2(h+1))\extu(E_{\ff}+F_{\fbf}-2(h+1))\extu(E_{\fbf}+F_{\fbf}-3(h+1))\]\[\extu (E_{\lf}+F_{\rf}-2(h+1));\] 
\[\tilde G_{\lf}=(E_{\ff}+F_{\lf}-(h+1))\extu(E_{\fbf}+F_{\lf}-2(h+1))\extu(E_{\lf}+F_{\mf}-(h+1));\]
\[\tilde G_{\rf}=(E_{\rf}+F_{\ff}-(h+1))\extu(E_{\rf}+F_{\fbf}-2(h+1))\extu(E_{\mf}+F_{\rf}-(h+1));\]
\[\tilde G_{\mf}=(E_{\mf}+F_{\mf})\extu(E_{\rf}+F_{\lf}-(h+1)).\]
Finally, multiply by $(\rho_{\lf}\rho_{\rf}\rho_{\fbf}^2\rho_{\ff})^{h+1}$; this adds $(h+1)$ to the index sets at $\bhs{\lf},$ $\bhs{\ff},$ and $\bhs{\rf}$ and adds $2(h+1)$ to the index set at $\phibf.$ We see immediately that $K(C)$ has the index sets claimed in the theorem. This completes the proof for $m=m'=-\infty.$

The generalization of this argument to arbitrary $m$ and $m'$ is standard and follows the arguments in \cite{mame1,v}. The only difference is that instead of applying the pullback and pushforward theorems for distributions which are smooth in the interior, we apply the pullback and pushforward theorems for distributions with interior conormal singularities. These theorems may be found in the appendix to \cite{emm}; in particular, see Propositions B7.6 for pullback and B7.20 for pushforward. To use the pullback theorem, we need to know that the $b$-fibration $\pi_F: X_s^3\rightarrow X_s^2$ is transversal to $D_{fib}\subset X_s^2$ - i.e. that for every point $p\in\pi_F^{-1}(D_{fib}),$ we have $(^b\pi_F)_*(^bT_pX_s^3)+\ ^bT_{\pi_F(p)}D_{fib}=\ ^bT_{\pi_F(p)}X_s^2.$ To use the pushforward theorem, we must show that the intersection $D_F\cap D_S=D\subset X_s^3$ is transverse, that $\pi_C(D)=D_{fib}$ is an embedded submanifold of $X_s^2,$ and that $\pi_C$ is transversal to both $D_F$ and $D_S.$ The definition of that transversality is that for all $p\in D_F$ (respectively $D_S$), $\ker((^b\pi_C)_*)+\ ^bT_pD_F=\ ^bT_pX_s^3.$ All of these statements are obvious in the interior and may be checked near the boundaries of $X_s^3$ and $X_s^2$ using explicit local coordinates. 
\end{proof}

\bibliographystyle{amsalpha}

\bibliography{cuspsurgres}

\end{document}

%% file: ARSsinglespace.pdf_tex
\begingroup%
  \makeatletter%
  \providecommand\color[2][]{%
    \errmessage{(Inkscape) Color is used for the text in Inkscape, but the package 'color.sty' is not loaded}%
    \renewcommand\color[2][]{}%
  }%
  \providecommand\transparent[1]{%
    \errmessage{(Inkscape) Transparency is used (non-zero) for the text in Inkscape, but the package 'transparent.sty' is not loaded}%
    \renewcommand\transparent[1]{}%
  }%
  \providecommand\rotatebox[2]{#2}%
  \ifx\svgwidth\undefined%
    \setlength{\unitlength}{336.45665283bp}%
    \ifx\svgscale\undefined%
      \relax%
    \else%
      \setlength{\unitlength}{\unitlength * \real{\svgscale}}%
    \fi%
  \else%
    \setlength{\unitlength}{\svgwidth}%
  \fi%
  \global\let\svgwidth\undefined%
  \global\let\svgscale\undefined%
  \makeatother%
  \begin{picture}(1,0.45468164)%
    \put(0,0){\includegraphics[width=\unitlength]{ARSsinglespace.pdf}}%
    \put(0.38610909,0.41808703){\color[rgb]{0,0,0}\makebox(0,0)[lb]{\smash{$H$}}}%
    \put(-0.00246131,0.07462095){\color[rgb]{0,0,0}\makebox(0,0)[lb]{\smash{$\epsilon=0$}}}%
    \put(0.40686285,0.13322515){\color[rgb]{0,0,0}\makebox(0,0)[lb]{\smash{$\bhs{sb}$}}}%
    \put(0.1139324,0.01092266){\color[rgb]{0,0,0}\makebox(0,0)[lb]{\smash{$\bhs{sm}$}}}%
    \put(0.70118736,0.01092266){\color[rgb]{0,0,0}\makebox(0,0)[lb]{\smash{$\bhs{sm}$}}}%
  \end{picture}%
\endgroup%

%% file: ARSdoublespace.pdf_tex
\begingroup%
  \makeatletter%
  \providecommand\color[2][]{%
    \errmessage{(Inkscape) Color is used for the text in Inkscape, but the package 'color.sty' is not loaded}%
    \renewcommand\color[2][]{}%
  }%
  \providecommand\transparent[1]{%
    \errmessage{(Inkscape) Transparency is used (non-zero) for the text in Inkscape, but the package 'transparent.sty' is not loaded}%
    \renewcommand\transparent[1]{}%
  }%
  \providecommand\rotatebox[2]{#2}%
  \ifx\svgwidth\undefined%
    \setlength{\unitlength}{368.85625bp}%
    \ifx\svgscale\undefined%
      \relax%
    \else%
      \setlength{\unitlength}{\unitlength * \real{\svgscale}}%
    \fi%
  \else%
    \setlength{\unitlength}{\svgwidth}%
  \fi%
  \global\let\svgwidth\undefined%
  \global\let\svgscale\undefined%
  \makeatother%
  \begin{picture}(1,0.88299981)%
    \put(0,0){\includegraphics[width=\unitlength]{ARSdoublespace.pdf}}%
    \put(0.06108409,0.43156448){\color[rgb]{0,0,0}\makebox(0,0)[lb]{\smash{$\bhs{lf}$}}}%
    \put(0.11530576,0.61591814){\color[rgb]{0,0,0}\makebox(0,0)[lb]{\smash{$\bhs{mf}$}}}%
    \put(0.41894708,0.6701398){\color[rgb]{0,0,0}\makebox(0,0)[lb]{\smash{$\bhs{rf}$}}}%
    \put(0.41894708,0.19298915){\color[rgb]{0,0,0}\makebox(0,0)[lb]{\smash{$\bhs{rf}$}}}%
    \put(0.76596574,0.43156448){\color[rgb]{0,0,0}\makebox(0,0)[lb]{\smash{$\bhs{lf}$}}}%
    \put(0.48115863,0.39242096){\color[rgb]{0,0,0}\makebox(0,0)[lb]{\smash{$\phibf$}}}%
    \put(0.33450141,0.50967743){\color[rgb]{0,0,0}\makebox(0,0)[lb]{\smash{$\phibf$}}}%
    \put(0.71174407,0.61591814){\color[rgb]{0,0,0}\makebox(0,0)[lb]{\smash{$\bhs{mf}$}}}%
    \put(0.11530576,0.24721081){\color[rgb]{0,0,0}\makebox(0,0)[lb]{\smash{$\bhs{mf}$}}}%
    \put(0.71174407,0.24721081){\color[rgb]{0,0,0}\makebox(0,0)[lb]{\smash{$\bhs{mf}$}}}%
    \put(0.35093305,0.39447976){\color[rgb]{0,0,0}\makebox(0,0)[lb]{\smash{$\bhs{ff}$}}}%
  \end{picture}%
\endgroup%

%% file: ARSheatspace.pdf_tex
\begingroup%
  \makeatletter%
  \providecommand\color[2][]{%
    \errmessage{(Inkscape) Color is used for the text in Inkscape, but the package 'color.sty' is not loaded}%
    \renewcommand\color[2][]{}%
  }%
  \providecommand\transparent[1]{%
    \errmessage{(Inkscape) Transparency is used (non-zero) for the text in Inkscape, but the package 'transparent.sty' is not loaded}%
    \renewcommand\transparent[1]{}%
  }%
  \providecommand\rotatebox[2]{#2}%
  \ifx\svgwidth\undefined%
    \setlength{\unitlength}{489.8bp}%
    \ifx\svgscale\undefined%
      \relax%
    \else%
      \setlength{\unitlength}{\unitlength * \real{\svgscale}}%
    \fi%
  \else%
    \setlength{\unitlength}{\svgwidth}%
  \fi%
  \global\let\svgwidth\undefined%
  \global\let\svgscale\undefined%
  \makeatother%
  \begin{picture}(1,0.55186141)%
    \put(0,0){\includegraphics[width=\unitlength]{ARSheatspace.pdf}}%
    \put(0.08710256,0.07547265){\color[rgb]{0,0,0}\makebox(0,0)[lb]{\smash{$x'$}}}%
    \put(0.89040588,0.07547265){\color[rgb]{0,0,0}\makebox(0,0)[lb]{\smash{$x$}}}%
    \put(0.4745049,0.39825173){\color[rgb]{0,0,0}\makebox(0,0)[lb]{\smash{$\bhs{hbf}$}}}%
    \put(0.4745049,0.20191342){\color[rgb]{0,0,0}\makebox(0,0)[lb]{\smash{$\bhs{tff}$}}}%
    \put(0.48103818,0.01547447){\color[rgb]{0,0,0}\makebox(0,0)[lb]{\smash{$\bhs{tf}$}}}%
    \put(0.67664945,0.49877852){\color[rgb]{0,0,0}\makebox(0,0)[lb]{\smash{$t$}}}%
    \put(0.30653954,0.06662138){\color[rgb]{0,0,0}\makebox(0,0)[lb]{\smash{$\bhs{tb}$}}}%
    \put(0.21072376,0.33645928){\color[rgb]{0,0,0}\makebox(0,0)[lb]{\smash{$\bhs{hlf}$}}}%
    \put(0.72883608,0.33645928){\color[rgb]{0,0,0}\makebox(0,0)[lb]{\smash{$\bhs{hrf}$}}}%
  \end{picture}%
\endgroup%

%% file: ARSDeltaHX.pdf_tex
\begingroup%
  \makeatletter%
  \providecommand\color[2][]{%
    \errmessage{(Inkscape) Color is used for the text in Inkscape, but the package 'color.sty' is not loaded}%
    \renewcommand\color[2][]{}%
  }%
  \providecommand\transparent[1]{%
    \errmessage{(Inkscape) Transparency is used (non-zero) for the text in Inkscape, but the package 'transparent.sty' is not loaded}%
    \renewcommand\transparent[1]{}%
  }%
  \providecommand\rotatebox[2]{#2}%
  \ifx\svgwidth\undefined%
    \setlength{\unitlength}{261.72512207bp}%
    \ifx\svgscale\undefined%
      \relax%
    \else%
      \setlength{\unitlength}{\unitlength * \real{\svgscale}}%
    \fi%
  \else%
    \setlength{\unitlength}{\svgwidth}%
  \fi%
  \global\let\svgwidth\undefined%
  \global\let\svgscale\undefined%
  \makeatother%
  \begin{picture}(1,0.69625472)%
    \put(0,0){\includegraphics[width=\unitlength]{ARSDeltaHX.pdf}}%
    \put(0.91286182,0.43250682){\color[rgb]{0,0,0}\makebox(0,0)[lb]{\smash{$x$}}}%
    \put(0.85044776,0.13106285){\color[rgb]{0,0,0}\makebox(0,0)[lb]{\smash{$t$}}}%
    \put(0.52103526,0.60413772){\color[rgb]{0,0,0}\makebox(0,0)[lb]{\smash{$\epsilon$}}}%
    \put(0.59090138,0.14782476){\color[rgb]{0,0,0}\makebox(0,0)[lb]{\smash{$\bhs{hbf}$}}}%
    \put(0.33108681,0.0364756){\color[rgb]{0,0,0}\makebox(0,0)[lb]{\smash{$\bhs{hmf}$}}}%
    \put(0.39221965,0.29629017){\color[rgb]{0,0,0}\makebox(0,0)[lb]{\smash{$\bhs{tff}$}}}%
    \put(0.18262139,0.46003889){\color[rgb]{0,0,0}\makebox(0,0)[lb]{\smash{$\bhs{tf}$}}}%
  \end{picture}%
\endgroup%

%% file: ARSpushforward.pdf_tex
\begingroup%
  \makeatletter%
  \providecommand\color[2][]{%
    \errmessage{(Inkscape) Color is used for the text in Inkscape, but the package 'color.sty' is not loaded}%
    \renewcommand\color[2][]{}%
  }%
  \providecommand\transparent[1]{%
    \errmessage{(Inkscape) Transparency is used (non-zero) for the text in Inkscape, but the package 'transparent.sty' is not loaded}%
    \renewcommand\transparent[1]{}%
  }%
  \providecommand\rotatebox[2]{#2}%
  \ifx\svgwidth\undefined%
    \setlength{\unitlength}{423.27363281bp}%
    \ifx\svgscale\undefined%
      \relax%
    \else%
      \setlength{\unitlength}{\unitlength * \real{\svgscale}}%
    \fi%
  \else%
    \setlength{\unitlength}{\svgwidth}%
  \fi%
  \global\let\svgwidth\undefined%
  \global\let\svgscale\undefined%
  \makeatother%
  \begin{picture}(1,0.31927404)%
    \put(0,0){\includegraphics[width=\unitlength]{ARSpushforward.pdf}}%
    \put(0.47106358,0.21610548){\color[rgb]{0,0,0}\makebox(0,0)[lb]{\smash{$x$}}}%
    \put(0.40478707,0.07935163){\color[rgb]{0,0,0}\makebox(0,0)[lb]{\smash{$t$}}}%
    \put(0.23239324,0.28500115){\color[rgb]{0,0,0}\makebox(0,0)[lb]{\smash{$\epsilon$}}}%
    \put(0.17420663,0.14621399){\color[rgb]{0,0,0}\makebox(0,0)[lb]{\smash{$\bhs{tff}$}}}%
    \put(0.03292447,0.22449203){\color[rgb]{0,0,0}\makebox(0,0)[lb]{\smash{$\bhs{tf}$}}}%
    \put(0.34364423,0.22067356){\color[rgb]{0,0,0}\makebox(0,0)[lb]{\smash{$\bhs{\epsilon,\tau}$}}}%
    \put(0.05833089,0.0555261){\color[rgb]{0,0,0}\makebox(0,0)[lb]{\smash{$\bhs{\epsilon,\tau}$}}}%
    \put(0.34126335,0.02497865){\color[rgb]{0,0,0}\makebox(0,0)[lb]{\smash{$\bhs{hbf}$}}}%
    \put(0.17707051,0.0049318){\color[rgb]{0,0,0}\makebox(0,0)[lb]{\smash{$\bhs{hmf}$}}}%
    \put(0.10735736,0.00513679){\color[rgb]{0,0,0}\makebox(0,0)[lb]{\smash{$x$}}}%
    \put(0.74219313,0.27179014){\color[rgb]{0,0,0}\makebox(0,0)[lb]{\smash{$\epsilon$}}}%
    \put(0.94215543,0.0747321){\color[rgb]{0,0,0}\makebox(0,0)[lb]{\smash{$t$}}}%
    \put(0.72224641,0.2105004){\color[rgb]{0,0,0}\makebox(0,0)[lb]{\smash{$\bhs{tf}$}}}%
    \put(0.75004522,0.1120989){\color[rgb]{0,0,0}\makebox(0,0)[lb]{\smash{$\bhs{tff}$}}}%
    \put(0.83194973,0.06444514){\color[rgb]{0,0,0}\makebox(0,0)[lb]{\smash{$\bhs{af}$}}}%
  \end{picture}%
\endgroup%

%% file: ARStriplespace.pdf_tex
\begingroup%
  \makeatletter%
  \providecommand\color[2][]{%
    \errmessage{(Inkscape) Color is used for the text in Inkscape, but the package 'color.sty' is not loaded}%
    \renewcommand\color[2][]{}%
  }%
  \providecommand\transparent[1]{%
    \errmessage{(Inkscape) Transparency is used (non-zero) for the text in Inkscape, but the package 'transparent.sty' is not loaded}%
    \renewcommand\transparent[1]{}%
  }%
  \providecommand\rotatebox[2]{#2}%
  \ifx\svgwidth\undefined%
    \setlength{\unitlength}{355.45444336bp}%
    \ifx\svgscale\undefined%
      \relax%
    \else%
      \setlength{\unitlength}{\unitlength * \real{\svgscale}}%
    \fi%
  \else%
    \setlength{\unitlength}{\svgwidth}%
  \fi%
  \global\let\svgwidth\undefined%
  \global\let\svgscale\undefined%
  \makeatother%
  \begin{picture}(1,0.78695791)%
    \put(0,0){\includegraphics[width=\unitlength]{ARStriplespace.pdf}}%
    \put(0.39169049,0.47806759){\color[rgb]{0,0,0}\makebox(0,0)[lb]{\smash{$\bhs{\phi_{TT}}$}}}%
    \put(0.14524541,0.25978065){\color[rgb]{0,0,0}\makebox(0,0)[lb]{\smash{$\bhs{N_3}$}}}%
    \put(0.65163941,0.25978065){\color[rgb]{0,0,0}\makebox(0,0)[lb]{\smash{$\bhs{N_2}$}}}%
    \put(0.39844241,0.76617466){\color[rgb]{0,0,0}\makebox(0,0)[lb]{\smash{$\bhs{N_1}$}}}%
    \put(0.07430072,0.61149941){\color[rgb]{0,0,0}\makebox(0,0)[lb]{\smash{$\bhs{C}$}}}%
    \put(0.74306301,0.61149837){\color[rgb]{0,0,0}\makebox(0,0)[lb]{\smash{$\bhs{S}$}}}%
    \put(0.44887205,0.15333241){\color[rgb]{0,0,0}\makebox(0,0)[lb]{\smash{$\bhs{F}$}}}%
    \put(0.39169049,0.33659878){\color[rgb]{0,0,0}\makebox(0,0)[lb]{\smash{$\bhs{\phi_{FT}}$}}}%
    \put(0.39169049,0.6307896){\color[rgb]{0,0,0}\makebox(0,0)[lb]{\smash{$\bhs{\phi_{FT}}$}}}%
    \put(0.51209973,0.53912376){\color[rgb]{0,0,0}\makebox(0,0)[lb]{\smash{$\bhs{\phi_{ST}}$}}}%
    \put(0.28253445,0.53915638){\color[rgb]{0,0,0}\makebox(0,0)[lb]{\smash{$\bhs{\phi_{CT}}$}}}%
    \put(0.27128125,0.39845876){\color[rgb]{0,0,0}\makebox(0,0)[lb]{\smash{$\bhs{\phi_{ST}}$}}}%
    \put(0.52335293,0.39845876){\color[rgb]{0,0,0}\makebox(0,0)[lb]{\smash{$\bhs{\phi_{CT}}$}}}%
    \put(0.81701242,0.70634677){\color[rgb]{0,0,0}\makebox(0,0)[lb]{\smash{$\bhs{\phi_S}$}}}%
    \put(0.00517442,0.70634629){\color[rgb]{0,0,0}\makebox(0,0)[lb]{\smash{$\bhs{\phi_C}$}}}%
    \put(0.41350497,0.00543319){\color[rgb]{0,0,0}\makebox(0,0)[lb]{\smash{$\bhs{\phi_F}$}}}%
    \put(0.32304597,0.61313113){\color[rgb]{0,0,0}\makebox(0,0)[lb]{\smash{$\bhs{T}$}}}%
  \end{picture}%
\endgroup%